\documentclass{amsart}
\usepackage[margin=1in]{geometry} 
\usepackage{lipsum}
\usepackage{tikz-cd}
\usepackage{stackrel}
\usepackage[hidelinks]{hyperref} 
\usepackage[capitalise]{cleveref} 
\usepackage{parskip}
\usepackage{graphicx,color}
\usepackage{amssymb,amsmath}
\usepackage{mathtools}
\usepackage{enumerate}
\usepackage{mathrsfs}
\usepackage{enumerate}
\usepackage{ragged2e}
\usepackage{enumitem,pifont}
\usepackage{xypic}
\usepackage{tikz}
\usetikzlibrary{arrows,decorations.pathmorphing,automata,backgrounds}
\usetikzlibrary{backgrounds,positioning}
\usepackage{caption}
\usepackage{subcaption}
\usepackage{float}

\usepackage{egothic}
\usepackage[T1]{fontenc}
\usepackage{yfonts}

\graphicspath{{Figures/}}
\usepackage{calc}
\usepackage{comment}

\newcommand{\C}{\mathbb{C}}

\newcommand{\K}{\mathbb{K}}

\newcommand{\undn}{\underline{n}}

\newcommand{\calA}{\mathcal{A}}

\newcommand{\calM}{\mathcal{M}}

\newcommand{\calO}{\mathcal{O}}
\newcommand{\calP}{\mathcal{P}}
\newcommand{\calQ}{\mathcal{Q}}
\newcommand{\calT}{\mathcal{T}}

\newcommand{\id}{\operatorname{id}}

\newcommand{\Hom}{\operatorname{Hom}}

\newcommand{\Aut}{\operatorname{Aut}}

\newcommand{\TAut}{\operatorname{TAut}}
\newcommand{\SAut}{\operatorname{SAut}}

\newcommand{\trace}{\operatorname{tr}}
\newcommand{\ob}{\operatorname{Ob}}

\newcommand{\Simp}{\operatorname{Simp}}
\newcommand{\gr}{\operatorname{gr}}
\newcommand{\grp}{\mathcal{G}}

\newcounter{dummy} \numberwithin{dummy}{section}
 \newtheorem{theorem}[dummy]{Theorem}
\newtheorem{thm}[dummy]{Theorem}
\newtheorem{prop}[dummy]{Proposition}

\newtheorem{lemma}[dummy]{Lemma}
\newtheorem{cor}[dummy]{Corollary}
\newtheorem*{thm*}{Theorem}
\newtheorem*{prop*}{Proposition}

\theoremstyle{definition}
\newtheorem{definition}[dummy]{Definition}

\newtheorem{notation}[dummy]{Notation}
\numberwithin{equation}{section}
\newtheorem{construction}[dummy]{Construction}

\theoremstyle{remark}
\newtheorem{remark}[dummy]{Remark}
\newtheorem{example}[dummy]{Example}

\usepackage{multicol}

\newcommand{\gt}{\mathsf{GT}}
\newcommand{\grt}{\mathsf{GRT}}
\newcommand{\gtm}{\mathsf{GTM}}
\newcommand{\grtm}{\mathsf{GRTM}}
\newcommand{\kv}{\mathsf{KV}}
\newcommand{\kvs}{\mathsf{KV}^\mathrm{sym}}
\newcommand{\krv}{\mathsf{KRV}}
\newcommand{\krvs}{\mathsf{KRV}^\mathrm{sym}}

\newcommand{\solkv}{\mathsf{SolKV}}

\newcommand{\Br}{\mathsf{B}} 
\newcommand{\PB}{\mathsf{PB}} 
\newcommand{\F}{\mathsf{F}} 
\newcommand{\PaB}{\mathsf{PaB}}

\newcommand{\hPaB}{\mathsf{PaB}_{\mathbb{K}}}
\newcommand{\CD}{\mathsf{CD}}
\newcommand{\PaCD}{\mathsf{PaCD}}

\newcommand{\lie}{\mathfrak{lie}} 
\newcommand{\ass}{\mathfrak{ass}} 
\newcommand{\tder}{\mathfrak{tder}}
\newcommand{\sder}{\mathfrak{sder}}
\newcommand{\cyc}{\mathrm{cyc}}
\newcommand{\bch}{\mathfrak{bch}}

\newcommand{\ib}{\mathfrak{t}}
\newcommand{\lkrv}{\mathfrak{krv}}
\newcommand{\lkv}{\mathfrak{kv}}

\newcommand{\inner}{\mathbf{t}} 
\newcommand{\tw}{b} 
\newcommand{\etw}{B} 

\newcommand{\barphi}{\bar{\varphi}}

\newcommand{\Op}{\mathrm{Op}}
\newcommand{\Mop}{\mathrm{MOp}}

\let\oldtocsection=\tocsection
\let\oldtocsubsection=\tocsubsection

\renewcommand{\tocsection}[2]{\hspace{0em}\oldtocsection{#1}{#2}}
\renewcommand{\tocsubsection}[2]{\hspace{1em}\oldtocsubsection{#1}{#2}}

\definecolor{darkblue}{rgb}{0,0,0.7} 
\definecolor{green}{RGB}{57,181,74} 
\definecolor{violet}{RGB}{147,39,143} 

\newcommand{\darkblue}{\color{darkblue}} 

\newcommand{\defn}[1]{{\darkblue \emph{#1}}}

\newcommand{\eqdef}{:=}

\usepackage{todonotes}
\title{Genus zero Kashiwara--Vergne solutions from braids}

\author[Z. Dancso]{Zsuzsanna Dancso}
\address{School of Mathematics and Statistics\\ The University of Sydney\\ Sydney, NSW, Australia}
\email{zsuzsanna.dancso@sydney.edu.au}

\author[I. Halacheva]{Iva Halacheva}
\address{Department of Mathematics \\ Northeastern University \\ Boston, Massachusetts, USA}
\email{i.halacheva@northeastern.edu}

\author[G. Laplante-Anfossi]{Guillaume Laplante-Anfossi}
\address{Centre for Quantum Mathematics, Syddansk Universitet, Campusvej 55, Odense, Denmark}
\email{glaplanteanfossi@imada.sdu.dk}

\author[M. Robertson]{Marcy Robertson}
\address{School of Mathematics and Statistics \\ The University of Melbourne \\ Melbourne, Victoria, Australia}
\email{marcy.robertson@unimelb.edu.au}

\author[C. Singh]{Chandan Singh}
\address{School of Mathematics and Statistics \\ The University of Melbourne \\ Melbourne, Victoria, Australia}
\email{chandans@student.unimelb.edu.au}

\thanks{IH was supported by the National Science Foundation Grant No.\ DMS-2302664. GLA, MR and CS were supported by the Australian Research Council Future Fellowship FT210100256. GLA was supported by the Andrew Sisson Fund, Novo Nordisk Foundation grant NNF20OC0066298, Villum Fonden Villum Investigator grant 37814, and Danish National Research Foundation grant DNRF157. CS acknowledges the support of an Australian Government Research Training Program (RTP) Scholarship and the Dr Albert Shimmins Fund.}

\begin{document}

\begin{abstract} 
Using the language of moperads — monoids in the category of right modules over an operad — we reinterpret the Alekseev–Enriquez–Torossian construction of Kashiwara--Vergne (KV) solutions from associators. 
We show that any equivalence between the moperad of parenthesized braids with a frozen strand and the moperad of chord diagrams gives rise to a family of genus zero KV solutions operadically generated by a single classical KV solution. 
We show that the Grothendieck—Teichm\"uller module groups act on the latter, intertwining the actions of the KV symmetry groups. 
In the other direction, we show that any symmetric KV solution gives rise to a module map from parenthesized braids with a frozen strand to tangential automorphisms of free Lie algebras. This map factors through the moperad of chord diagrams if and only if the associated KV associator is a Drinfeld associator.
\end{abstract}
\maketitle
\section*{Introduction}

The Kashiwara–Vergne (KV) equations were posed in \cite{KV78} as part of an effort to use the combinatorial and structural properties of the Baker–Campbell–Hausdorff (BCH) formula to give an algebraic proof of the Duflo isomorphism \cite{Duflo77}. Informally, a solution to the KV equations consists of a pair of Lie series satisfying compatibility and divergence conditions that, if fulfilled, yield the desired algebraic framework. 
A general solution to the equations was eventually constructed by Alekseev and Meinrenken~\cite{AM06}, building on prior work by Torossian~\cite{Torossian02} using configuration space integrals defining Kontsevich’s formality morphism~\cite{Kontsevich99,Kontsevich03}.

In~\cite{AT12}, Alekseev and Torossian reformulated the KV equations in terms of \emph{tangential automorphisms}, which are automorphisms of the prounipotent group $\exp(\lie_n)$ associated to the completed free Lie algebra on $n$ generators, $\lie_n$. Tangential automorphisms act by conjugating each free generator by elements in $\exp(\lie_n)$ (\cref{def: tangential automorphism}). Tangential automorphisms of $\lie_n$ form a group $\TAut_n$. 
In this formulation, a solution to the KV equations is a tangential automorphism \[F:\exp(\lie_2)\rightarrow \exp(\lie_2)\] satisfying two identities which rectify the failure of associativity and commutativity of the exponential map at the Lie algebra level. Recall that the product of exponentials satisfies
\[
e^{x_1}e^{x_2} = e^{\bch(x_1,x_2)},
\]
where $
\bch(x_1,x_2) \eqdef x_1 + x_2 + \tfrac{1}{2}[x_1,x_2] + \tfrac{1}{12}[x_1,[x_1,x_2]] - \tfrac{1}{12}[x_2,[x_1,x_2]] + \cdots$ is the Baker--Campbell--Hausdorff (BCH) series in $\lie_2$. 
The first KV equation~\eqref{SolKVI} requires the tangential automorphism $F \in \TAut_2$ to satisfy  $F(e^{x_1}e^{x_2})=e^{x_1+x_2}$. For more detail and the second KV equation~\eqref{SolKVII} see \cref{sec: background on tder and KV solutions}. 

The associativity property of the Baker--Campbell--Hausdorff series gives rise to coherence relations that constrain how KV solutions behave under nested compositions. 
In~\cite{AT12}, Alekseev and Torossian 
introduce the class of \emph{KV associators}. The KV associators are tangential automorphisms $G:\exp(\lie_3)\rightarrow \exp(\lie_3)$ of the form 
\[
G = F^{1,23}F^{2,3}(F^{12,3}F^{1,2})^{-1},
\]
which encode the behavior of a KV solution $F$ under nested applications of the BCH formula (\cref{defn: KV Associator}). 
The KV associators satisfy a pentagon identity that reflects the associativity of repeated applications of the BCH formula, together with a pair of hexagon equations encoding symmetry under cyclic permutation of inputs. These identities mirror those satisfied by Drinfeld associators (\cref{def:DrinfAss}). Indeed, Alekseev and Torossian  \cite{AT12} showed that every Drinfeld associator gives rise to a KV solution and corresponding KV associator. Later, Alekseev, Enriquez and Torossian \cite{AET10} gave a now-famous explicit formula constructing KV solutions from Drinfeld associators. 
It remains an open question whether every KV associator arises from a Drinfeld associator (cf. just after Theorem 4.6 of \cite{AT12}).

\vspace{1mm}

\textbf{Expansions.} 
Both Drinfeld associators and Kashiwara--Vergne solutions correspond to specific topological invariants -- called {\em homomorphic expansions} -- for different knotted structures. Below we give an intuitive introduction to homomorphic expansions. 

Homomorphic expansions are particularly useful when defined on a finitely presented algebraic structure $\calT$, typically of topological nature, such as the pure braid group on $n$ strands, or the operad of parenthesized braids. We allow formal linear combinations over a field $\K$ of characteristic 0, and denote the resulting linear extension by $\K\calT$. The {\em augmentation ideal} $\K\calT \supseteq I=\{\sum c_i T_i : c_i\in \K, T_i \in \calT, \sum c_i=0\}$ gives rise to a decreasing filtration on $\K\calT$, called the {\em I-adic} of {\em prounipotent} filtration:
$$\K\calT = I^0 \supseteq I^1 \supseteq I^2 \supseteq \cdots $$

We denote by $\calA$ the degree completed associated graded structure with respect to this filtration: $$\calA= \Pi_{n=0}^\infty I^n/I^{n+1}.$$
A {\em homomorphic expansion} is an isomorphism between the \textit{I}-adic completion $\widehat{\K\calT} = \calT_\K$, and the associated graded structure $\calA$ (see Section \ref{subsec:compform}). We note that in \cite{WKO2,BD13} a homomorphic expansion is defined to be a filtered homomorphism $Z: \K\calT \to \calA$ with the property that the associated graded map of $Z$ is the identity on $\calA$. Such a map induces a unique isomorphism between the \textit{I}-adic completion and the associated graded structure.

The {\em expansion problem} for $\calT$ is the problem of finding a homomorphic expansion for $\K\calT$. The papers \cite{AKKN18highergenus, AKKN_GT_Formality} refer to expansion problems as {\em formality problems}, as homomorphic expansions are related to, although weaker than formality isomorphisms in rational homotopy theory. 
Expansion problems are often difficult, and homomorphic expansions do not always exist. For a finitely presented structure $\calT$, the search for a homomorphic expansion comes down to solving a finite set of equations in $\calA$, in order to determine the $Z$-values of the generators of $\calT$. Often, these equations turn out to be independently interesting.

The prototype for a topological expansion problem is the expansion problem for the {\em operad of parenthesized braids}. The set of homomorphic expansions 
\[
\left\{ \, \varphi\colon \hPaB \overset{\cong}{\longrightarrow} \PaCD \, \right\}
\]
between the completed operad of parenthesized braids $\hPaB$, and its associated graded operad, the operad of parenthesized chord diagrams $\PaCD$, is in one-to-one correspondence with {\em Drinfeld associators}  ~\cite{BN98, FresseBook1, Calaque2025}. Drinfeld associators (\cref{def:DrinfAss}) are solutions of the pentagon and hexagon equations in the exponentiated Drinfeld--Kohno Lie algebras \cite{BN98,FresseBook1,Tamarkin03,Calaque2025}, and they 
are of independent interest in quantum algebra as they encode the coherence conditions for braided monoidal categories.

Solutions to the Kashiwara--Vergne equations also correspond to homomorphic expansions of multiple topological objects: a class of knotted surfaces in four-dimensions called {\em welded foams} \cite{WKO2,WKO2Cor}, and the Goldman--Turaev Lie bialgebra of loops in a punctured disc \cite{AKKN_genus_zero}. 
However, the key construction of KV solutions from Drinfeld associators \cite{AET10} is not phrased in the language of expansions. In fact, it is an interesting question whether it constructs a homomorphic expansion, and if so, for what structure. 

The motivating goal of this paper is to answer this question by uncovering the relevant structure -- the {\em moperad} of braids with a frozen strand -- and reinterpreting the Alekseev--Enriquez--Torossian construction as a construction of a homomorphic expansion for this structure.

\vspace{1mm}

\textbf{Main results.}
The Alekseev--Enriquez--Torossian construction builds on the correspondence between Drinfeld associators and operad homomorphic expansions $\varphi\colon \hPaB \to \CD$, 
showing that the local behavior of such a homomorphic expansion $\varphi$ on a family of elementary morphisms in $\hPaB$ can be used to produce a tangential automorphism $F_{\varphi}$, which satisfies the KV equations. We revisit this construction and show that it amounts to extending $\varphi$ to a homomorphic expansion for the \emph{moperad of parenthesized braids with a frozen strand},~$\PaB^1$ (\cref{def:Moperad of Parenthesized Braids}).

Moperads, which are monoids in the category of right modules over an operad (\cref{def:moperad}), extend the classical operadic framework by allowing for compositions with one distinguished input~\cite{Willwatcher2016Moperads,calaque20moperadic,CamposIdrissiWillwacher24}).  The moperad~$\PaB^1$ consists of groupoids~$\PaB^1(n)$ which model the fundamental groupoids of ordered configuration spaces of~$n$ points in the punctured plane. Topologically, these configurations can be pictured as lying in a cylinder, where the puncture is a marked boundary component. 
The two defining operations on~$\PaB^1$ reflect this geometry:
\begin{enumerate}
  \item a monoid structure given by composition along the frozen strand, corresponding to gluing the unmarked boundary of a cylinder into the marked boundary of the other;
  \item a right $\PaB$-module structure encoding the operadic composition of configurations away from the puncture, analogous to the classical structure encoded by the little discs operad.
\end{enumerate}
A homomorphic expansion for $\hPaB^1$ can be encoded as a pair $(\varphi^1,\varphi)$, where $\varphi$ is an operad equivalence between $\hPaB$ and $\CD$, and $\varphi^1$ is a $\varphi$-equivariant equivalence $\varphi^1:\hPaB^1 \to \CD^+$.

It was shown by Calaque and Gonzalez~\cite[Theorem~3.4]{calaque20moperadic} that the moperad $\PaB^1$ admits an explicit finite presentation. Therefore, homomorphic expansions for $\hPaB^1$ are determined by their values on the generators. This implies that homomorphic expansions of $\hPaB^1$ are in bijection with tuples $(\mu,f,g)\in \K\times\exp(\ib_3)\times\exp(\ib^+_{2})$, where $\ib_n$ denotes the Drinfeld--Kohno Lie algebra and the superscript~$+$ denotes a shift, i.e. $\ib_n^+\cong \ib_{n+1}$. The values an expansion takes on the generators of $\PaB^1$ must satisfy the equations arising from the relations in the presentation of $\PaB^1$. These are the pentagon and hexagon equations of the operad $\PaB$, which ensure that the pair $(\mu,f)$ is a Drinfeld associator (\cref{cor: operadic associators}); and the additional {\em mixed pentagon}, {\em right pentagon} and {\em octagon} equations which arise in the moperad structure of $\PaB^1$ for the pair $(f,g)$ (\cref{presenation of PaB1}, \cref{thm: assoc^1}).
 
Building on this characterization, we show that any operad homomorphic expansion $\varphi:\PaB \to \CD$ has a one-parameter family of extensions to a moperad homomorphic expansion $(\varphi^1,\varphi)$. We re-interpret the Alekseev--Enriquez--Torossian construction to show that such an extension $(\varphi^1,\varphi)$ with coupling constant~$1$ always restricts to a tangential automorphism $F_{\varphi^1} : \exp(\lie_n) \to \exp(\lie_n)$, which satisfies the KV equations and is symmetric, that is, invariant under the Alekseev--Torossian involution \cite{AT12}: 

\begin{thm*}[\cref{thm: the construction is good}]
Any moperad homomorphic expansion $(\varphi^1,\varphi):\hPaB^1 \rightarrow \CD^+$ with $\mu=1$ restricts to a tangential automorphism $F_{\varphi^1}$ on the free group $F_2 \subseteq \PaB^1((0(12)),(0(12)))$, and $F_{\varphi^1}$ is a symmetric KV solution.
\end{thm*}

There is a natural shift functor $(-)^+$ which produces a moperad morphism from any operad morphism (\cref{example: shifted moperad}).
Applying this functor to the homomorphic expansion corresponding to a Drinfeld associator $\varphi : \hPaB \to \CD$, one obtains a homomorphic expansion of moperads $(\varphi^+,\varphi) : \hPaB^+ \to \CD^+$. Precomposing with the natural inclusion $\hPaB^1 \hookrightarrow \hPaB^+$, one recovers the Alekseev–Enriquez–Torossian formula \cite[Theorem~4]{AET10} for $F_{\varphi^1}$. This highlights that incorporating the defining relations of $\PaB^1$ as a $\PaB$-moperad, the KV equations arise from the pentagon and hexagon equations as a natural consequence of extending an operad map to a moperad map.

The set of KV solutions admits commuting left and right actions of the Kashiwara--Vergne groups~$\kv$ and~$\krv$, analogous to the commuting left and right actions of the Grothendieck--Teichm\"uller groups~$\gt$ and $\grt$ on the set of Drinfeld associators \cite{Drin}. 
It was shown in~\cite{AET10} that the Alekseev--Enriquez--Torossian construction intertwines these actions, in other words, constructs a map of torsors between Drinfeld associators and KV solutions. Indeed, the groups~$\gt_1$ and~$\grt_1$ embed as subgroups into the KV symmetry groups~$\kv$ and~$\krv$, respectively. 
Moreover, when a KV solution $F_{\Phi}$ arises from an associator $\Phi$, the actions of $\kv$ and $\krv$ on $F_\Phi$ coincide with the actions of~$\gt_1$ and~$\grt_1$ on $\Phi$ (\cite[Proposition~9.13]{AT12}, \cite[Theorem~9]{AET10}).

In \cite{calaque20moperadic}, Calaque and Gonzalez showed that the group of object-fixing automorphisms of the moperad $\hPaB^1$ is isomorphic to a cyclotomic variant of the Grothendieck--Teichm\"uller group, Enriquez’s \emph{Grothendieck--Teichm\"uller module group} denoted $\mathsf{GTM}$.
This group was introduced in the context of braided module categories~\cite{Enriquez2007cyclotomic}, which are module categories over braided monoidal categories (\cite{Enriquez2007cyclotomic}, \cite{ben2018quantum}, \cite{kolb2020braided}). 
It governs deformation symmetries of these module categories, analogous to how the Grothendieck–Teichm\"uller group $\gt$ governs deformations of braided monoidal categories themselves. 

In Section~\ref{sec: GT1 actions on constructed KV solutions} we show that suitable local data extracted from an automorphism of $\hPaB^1$ yields an element of the symmetry group $\kv$. 
Concretely, this gives an injective group homomorphism
\[
  \gtm_{1} \longrightarrow \kv,
\]
where $\gtm_{1}$ is the subgroup of $\gtm$ with coupling constant $\lambda=1$.  In other words, the Grothendieck–Teichm\"uller module group $\gtm_{1}$ acts as a group of KV symmetries.  One can similarly define a right action of the graded version of the Grothendieck–Teichm\"uller module group $\grtm_{1}$ on KV solutions. 
In Section~\ref{sec: GT1 actions on constructed KV solutions} we conclude the following.  
\begin{thm*}[Theorems~\ref{thm:KV-action-F-012} and \ref{thm:KRV}]
Let $\gtm_{1}$ and $\grtm_{1}$ denote the subgroups of $\gtm$ and $\grtm$ with~$\lambda=1$. 
Then the assignment of a KV solution $F_{\varphi^1}$ to a moperad equivalence $(\varphi^1,\varphi):\hPaB^1 \rightarrow \CD^+$ with $\mu=1$ intertwines the commuting left and right actions of $\gtm_{\lambda=1}$ and $\grtm_{\lambda=1}$ with the left and right actions of the KV symmetry groups $\kv$ and $\krv$.
\end{thm*}

In contrast to the case of Drinfeld associators, there are intrinsic obstacles to formulating an operadic theory of KV solutions. 
Nonetheless, in~\cite{WKO2,DHR2} KV solutions are characterized as homomorphic expansions of a {\em circuit algebra} or {\em wheeled prop} \cite{DHR1}. 
However, these underlying algebraic structures present fundamental limitations.
In particular, while tangential automorphisms $F \in \TAut_n$ of \(\exp(\lie_n)\) assemble into an operad in the category of sets (\cref{operad of TAut}), the partial compositions are not group homomorphisms. 
Nevertheless, this minimal operadic structure on the collection of groups \(\TAut_n\) is sufficient to obtain a family of generalized KV solutions of type \((0,n+1)\), or \emph{genus zero KV solutions}, as introduced in \cite{AKKN_genus_zero}. 

The genus zero extension of the classical KV problem allows operations involving multiple inputs. 
In particular, a genus zero KV solution is a tangential automorphism \(F\colon \exp(\lie_n) \to \exp(\lie_n)\) satisfying the condition
\[
F(e^{x_1} \cdots e^{x_n}) = e^{x_1 + \cdots + x_n},
\]
as well as a generalized second KV equation.
In~\cite{AKKN18highergenus}, Alekseev, Kawazumi, Kuno, and Naef construct such higher-arity solutions by inductively ``gluing together'' classical KV solutions. We use their construction to define a non-symmetric, group-colored operad \(\solkv\) of genus zero KV solutions, colored by Duflo functions~(\cref{thm:operad-SolKV}). 
Then, given a homomorphic expansion $(\varphi^1,\varphi)$ of the moperad $\hPaB^1$, we identify a family of tangential automorphisms \(\{F_{0w}\}\), indexed by parenthesizations \(w\) of the identity permutation \(12\cdots n\). This coincides with the suboperad of \(\solkv\) generated by the binary solution \(F_{\varphi^1}\) (\cref{thm:SolKV-from-moperad}). 
As a result, we obtain explicit formulas for genus zero KV solutions of all arities arising from a Drinfeld associator.

A class of better behaved automorphisms of the free group are the \emph{special automorphisms}: those tangential automorphisms of $\exp(\lie_n)$ which act trivially on the abelianization. Special automorphisms in particular form an operad in the category of prounipotent groups: that is, restricting the partial compositions from $\TAut$ to $\SAut$ makes them group homomorphisms. A source of special automorphisms comes from Kohno’s isomorphism
\[
  \widehat{\PB}_n \;\xrightarrow{\;\cong\;}\; \exp(\ib_n)
\]
which identifies the prounipotent completion of the pure braid group with the prounipotent group corresponding to the Drinfeld–Kohno Lie algebra $\ib_n$ of infinitesimal braids (\cite{Kohno1985}).
The semidirect product decomposition of the pure braid group
\(
  \PB_n \cong \F_{n-1}\rtimes \PB_{n-1}
\)
induces, after completion, an action of $\widehat{\PB}_{n-1}$ on the prounipotent free group $\widehat{\F}_{n-1}\cong \exp(\lie_{n-1})$.  This action defines an injective group homomorphism
$
  \exp(\ib_{\,n-1}) \hookrightarrow  \SAut_{\,n-1},
$
 as in \cite[Section~4]{AET10}, \cite[Proposition~3.11]{AT12}.

In \cref{sec: moperad of dervations}, we define two $\SAut$-moperads in prounipotent groups, $\TAut^1$ and $\SAut^1$, for tangential and special automorphisms, respectively. The moperad $\SAut^1$ arises as a natural target for the embedding of the Drinfeld--Kohno operad in $\SAut$ (see \cref{lemma: image of CD+ in SAut1}).
The moperad $\TAut^1$ realizes the strongest natural operadic structure on tangential automorphisms.

Finally, we study symmetric KV solutions, which are KV solutions invariant under the Alekseev--Torossian involution (Definition~\ref{def:SolKVsymmetric}), and which are   
most closely related to Drinfeld associators. Each symmetric KV solution has a corresponding {\em KV associator}, which satisfies the pentagon and hexagon equations in $\SAut_3$.
If a symmetric KV solution arises from a Drinfeld associator, the corresponding KV associator 
lies in the image of the canonical inclusion $\exp(\ib_3) \hookrightarrow \SAut_3$, and thus recovers the original Drinfeld associator. 
In \cref{section: symmetric KV solutions}, we show the following.
\begin{thm*}[{\cref{prop: symmetric KV solutions give KV solutions,cor: moperad map factors if KV Ass is Drinf Ass}}]
    Every symmetric KV solution $F$ defines a map of  $\SAut$-modules 
\[
  (\varphi^1_F, \varphi_F): \hPaB^1 \longrightarrow \TAut^+
\] 
which  factors through an equivalence $\hPaB^1\rightarrow\CD^+$ if and only if the KV associator $G = F^{1,23}F^{2,3}(F^{12,3}F^{1,2})^{-1}$ is in the image of $\exp(\ib_3)$.
\end{thm*}


\vspace{1mm}

\textbf{Convention.} Throughout we fix a field $\K$ of characteristic zero. In diagrams, all morphisms are drawn from bottom to top.

\textbf{Acknowledgements.}
We thank Anton Alekseev, Dror Bar-Natan, Benjamin Enriquez, Florian Naef, Muze Ren, and Pavol \v{S}evera for useful discussions. 
We would like to thank Najib Idrissi and Hidekazu Furusho for comments on an earlier draft.
We are grateful to the Universit\'e de Gen\`eve and the SwissMAP research station in Les Diablerets where part of this work was carried out. The second author is grateful to the Sydney Mathematical Research Institute for a two month research visit during which this project was initiated.

\vspace{1mm}

\textbf{Plan of the paper.}
The structure of the paper is as follows. Section 1 reviews the well known example of Drinfeld associators realized as homomorphic expansions for the operad of parenthesized braids, and the Grothendieck--Teichm\"uller groups as symmetries of these homomorphic expansions. 
Section 2 introduces moperads, the moperad structure of parenthesized braids with a frozen strand, and the expansion problem for this structure. 
Section 3 introduces genus zero Kashiwara--Vergne solutions and the associated Lie algebras and prounipotent groups from an operadic perspective.
In Section 4, we present the main construction of KV solutions from Drinfeld associators through the lens of moperad expansions. We establish the compatibility of these KV solutions with the actions of the Grothendieck--Teichm\"uller groups. 
Section 5 constructs module maps on $\hPaB^1$ from KV associators, and shows that when a KV associator is in the image of $\exp(\ib_3)$, the corresponding module map factors through a moperad equivalence $\hPaB^1 \to \CD^+$. 
Finally, two appendices provide supplementary background and technical details on operadic composition of tangential and special derivations, and cohomological computations.



\tableofcontents

\section{Operad expansions: the operad of parenthesized braids and Drinfeld associators}
In this section we review the correspondence between Drinfeld associators and homomorphic expansions for the operad of parenthesized braids. To do so, we recall the necessary definitions of operads, the operad of parenthesized braids, completions, chord diagrams, and associators. Although much of this material is standard, we include it in order to fix notation and establish the conceptual framework for the new results that follow.

\subsection{Operads}
Let $(\mathrm{C},\otimes, I)$ be a cocomplete closed symmetric monoidal category such that the monoidal product $\otimes$ commutes with small colimits in each variable. The examples of such categories used in this paper are the categories of Lie algebras,  vector spaces, Hopf algebras, (prounipotent) groups and groupoids, and sets. 

Let $\Sigma_n \eqdef \Aut(\undn)$ denote the symmetric group on the set $\undn \eqdef\{1,\ldots,n\}$.  Given two permutations $\sigma \in \Sigma_m$ and $\tau \in \Sigma_n$ and $1\leq i\leq m$, one can \emph{insert} $\tau$ into the entry of $\sigma$ labeled $i$ (in one-line notation), and shift all the larger entries appropriately to maintain bijectivity. This defines a new permutation $\sigma \circ_i \tau \in \Sigma_{m+n-1}$. In formulas: 
\begin{equation}
\label{def: composition of permutations}
    (\sigma \circ_i\tau)(k) \eqdef 
    \begin{cases}
        \sigma(k) & \text{if }  1 \leqslant k < \sigma^{-1}(i), \,\sigma(k)<i, \\
        \sigma(k)+n-1 & \text{if }  1 \leqslant k < \sigma^{-1}(i), \,\sigma(k)>i, \\
        \tau(k-\sigma^{-1}(i)+1)+(i-1) & \text{if } \sigma^{-1}(i)\leqslant k \leqslant \sigma^{-1}(i)+n-1, \\
        \sigma(k-n+1)  & \text{if } \sigma^{-1}(i)+n\leqslant k \leqslant m+n-1, \,\sigma(k-n+1)<i, \\
        \sigma(k-n+1)+n-1  & \text{if } \sigma^{-1}(i)+n\leqslant k \leqslant m+n-1, \,\sigma(k-n+1)>i.
    \end{cases}
\end{equation}

\medskip 
A \defn{sequence} in a symmetric monoidal category $\mathrm{C}$ is an $\mathbb{N}$-graded collection of objects in $$\calP=(\calP(0),\calP(1),\ldots, \calP(n),\ldots).$$ A \defn{symmetric sequence} in $\mathrm{C}$ is a sequence $\calP=\{\calP(n)\}_{n\geq 0}$ in which each entry is equipped with a right action of the symmetric group \[\begin{tikzcd}
\calP(n)\otimes \Sigma_n\arrow[r] & \calP(n).
\end{tikzcd}\]
\begin{definition}
\label{defn: operad}
A (symmetric) \defn{operad} $\mathcal{P}$ in $\mathrm{C}=(\mathrm{C},\otimes, I)$ consists of a (symmetric) sequence $\mathcal{P}=\{ \mathcal{P}(n) \}_{n\geq 0}$, along with:
\begin{itemize}[leftmargin=*,label=\labelitemii]
    \item A distinguished unit map $\id:I \to \mathcal{P}(1)$;
    \item A family of a partial composition maps
    \[
    \circ_i : \mathcal{P}(n) \otimes \mathcal{P}(m) \longrightarrow \mathcal{P}(n+m-1),\; 1\leq i\leq n.
    \]  
\end{itemize}    
These partial composition operations are required to satisfy 
\begin{itemize}
    \item[(i)] unitality: $\mu \circ_i \id = \id \circ_1 \mu = \mu \quad$ for all $\mu \in \calP(n)$ ,
    \end{itemize}

\begin{itemize}
    \item [(ii)] associativity: For all $\lambda \in \mathcal{P}(\ell)$, $\mu \in \mathcal{P}(m)$, and $\nu \in \mathcal{P}(n)$, we have
\[(\lambda\circ_{j}\mu) \circ_{i} \nu=\begin{cases}
(\lambda \circ_i \nu) \circ_{j+n-1} \mu, & \text{if } 1 \leq i \leq j - 1, \\
\lambda \circ_j (\mu \circ_{i-j+1} \nu), & \text{if } j \leq i \leq j + m - 1, \\
(\lambda \circ_{i-m+1} \nu) \circ_j \mu, & \text{if } j + m \leq i \leq \ell + m - 1.
\end{cases}\]
\end{itemize}
\begin{itemize}
    \item[(iii)]and, if $\mathcal{P}$ is symmetric, equivariance : $\mu^\sigma \circ_{\sigma^{-1}(i)} \nu^\tau
              = (\mu \circ_{i} \nu)^{\sigma \circ_i \tau}$ for all $\mu \in \calP(m)$, $\nu \in \calP(n)$, $\sigma\in \Sigma_m$ and $\tau\in \Sigma_n$\ .
\end{itemize}

\end{definition}

\medskip 

In order to define the operadic composition on KV-solutions we require a group-colored refinement of the usual notion of operad. In this setting, operations come equipped with input and output types, and composition is defined only when these types coincide.  Let $\mathfrak{C}$ be a small groupoid (or group, or finite set), whose objects we refer to as \defn{colors}.  A \defn{$\mathfrak{C}$-colored sequence} in a monoidal category $\mathrm{C}$ is a collection of objects in $\mathrm{C}$, indexed by finite lists of colors $(c_1,\ldots,c_n;c_0)$ in $\mathfrak{C}$:
\[
\mathcal{P} = \{ \mathcal{P}(c_1, \ldots, c_n; c_0) \}_{(c_1,\ldots,c_n;c_0)\in \mathrm{Ob}(\mathfrak{C})^{n+1}}.
\] If $\mathfrak C$ is a groupoid containing nontrivial isomorphisms, we require covariant functoriality in each input color and contravariant functoriality in the output color. In other words, for morphisms $f_i:c_i\to c_i'$ and $g:c_0\to c_0'$ in $\mathfrak C$, there are structure maps
\[
\mathcal P(c_1,\dots,c_n;c_0')
\longrightarrow
\mathcal P(c_1',\dots,c_n';c_0),
\] functorial in all variables.

\medskip

\begin{definition}
\label{defn: colored operad}
A (symmetric) \defn{$\mathfrak{C}$-colored operad} $\mathcal{P}$ in $\mathrm{C}=(\mathrm{C},\otimes, I)$ consists of a $\mathfrak{C}$-colored (symmetric) sequence $\mathcal{P}=\{ \mathcal{P}(c_1, \ldots, c_n; c_0) \}$, along with:
\begin{itemize}[leftmargin=*,label=\labelitemii]
    \item For each $c \in \mathrm{Ob}(\mathfrak{C})$, a unit map $\id_c:I \to \mathcal{P}(c; c)$;
    \item A family of a partial composition maps
    \[
    \circ_i : \mathcal{P}(c_1,\dots,c_n;c_0) \otimes \mathcal{P}(d_1,\dots,d_m;c_i) \longrightarrow \mathcal{P}(c_1,\dots,c_{i-1},d_1,\dots,d_m,c_{i+1},\dots,c_n; c_0),\; 1\leq i\leq n,
    \]  which satisfy the $\mathfrak{C}$-colored versions of the unitality, associativity and, if relevant, equivariance axioms in Definition~\ref{defn: operad} (cf. \cite[Definition 1.1]{bm_coloured_operads}, \cite[Section 3]{Petersen_groupoid_coloured_operads})
\end{itemize}    

\end{definition}

\begin{remark}
In the case where the groupoid $\mathfrak{C}$ has a single object and no non-identity morphisms (i.e., $\mathfrak{C} = \{*\}$), a $\mathfrak{C}$-colored (symmetric) operad reduces to an ordinary (symmetric) operad. 
\end{remark}

A \defn{morphism} $\varphi : \mathcal{P} \to \mathcal{Q}$ of $\mathfrak{C}$-colored (symmetric) operads in $\mathrm{C}$ is a morphism of $\mathfrak{C}$-colored (symmetric) sequences in $\mathrm{C}$ which preserves all unit and composition maps.

\begin{example}
Let $G$ be a group, viewed as a groupoid with one object. 
Define a non-symmetric $G$-colored operad $\mathcal{P}$ in $\mathrm{Set}$ as follows. For a tuple $(g_1,\ldots,g_n; g_0)$ in $G^{n+1}$, let
\[
\mathcal{P}(g_1,\ldots,g_n; g_0) := 
\begin{cases}
\{*\}, & \text{if } g_1 \cdots g_n = g_0, \\
\emptyset, & \text{otherwise}.
\end{cases}
\]
Composition is defined by substitution: if $\mu \in \mathcal{P}(g_1,\ldots,g_n; g_0)$ and $\nu \in \mathcal{P}(h_1,\ldots,h_m; g_i)$, then $\mu \circ_i \nu$ is defined and equal to the unique element of
\[
\mathcal{P}(g_1,\ldots,g_{i-1}, h_1,\ldots,h_m, g_{i+1},\ldots,g_n; g_0),
\]
provided $h_1 \cdots h_m = g_i$. Unit elements are given by $\id_g \in \mathcal{P}(g; g)$ for all $g \in G$. This operad has no symmetric group actions, and is thus a non-symmetric operad colored by the group $G$.
\end{example}

\subsection{The operad of parenthesized braids}
In this section we recall the braid groups and their associated operadic structures. 
The \defn{braid group} on $n$ strands, denoted $\Br_n$, is the fundamental group of the space of unordered configurations of $n$ points in the complex plane. The \defn{pure braid group} on $n$ strands, denoted $\PB_n$, is the fundamental group of the space of ordered configurations of $n$ points in the complex plane. Equivalently, $\PB_n$ is the kernel of the natural surjection
\[
\Br_n \to \Sigma_n
\]
sending a braid to its underlying permutation. The braid group $\Br_n$ admits the standard Artin presentation
\[
\Br_n=\left<\beta_1,\ldots,\beta_{n-1}\mid  \beta_i\beta_j=\beta_j\beta_i 
\ \text{when} \ |i-j|\geq 2,\ \beta_i\beta_{i+1}\beta_i =\beta_{i+1}\beta_i\beta_{i+1}\right>.
\]
A standard generating set for $\PB_n$ is given by the braids
\[
x_{ij}\eqdef \beta_{j-1}\cdots \beta_{i+1}\beta_i^2\beta_{i+1}^{-1}\cdots \beta_{j-1}^{-1},
\] for $1\leq i<j\leq n.$ Braids are usually drawn in three dimensions, with the vertical direction representing the parameter of the loop, as shown in Figure~\ref{fig:Braids}. The standard generator $\beta_i$ is drawn as the braid in which strand $i$ crosses over strand $i+1$. Composition in the braid group is represented by drawing braids from bottom to top and writing braid words from left to right.

\begin{figure}[ht!]
    \includegraphics[width=12cm]{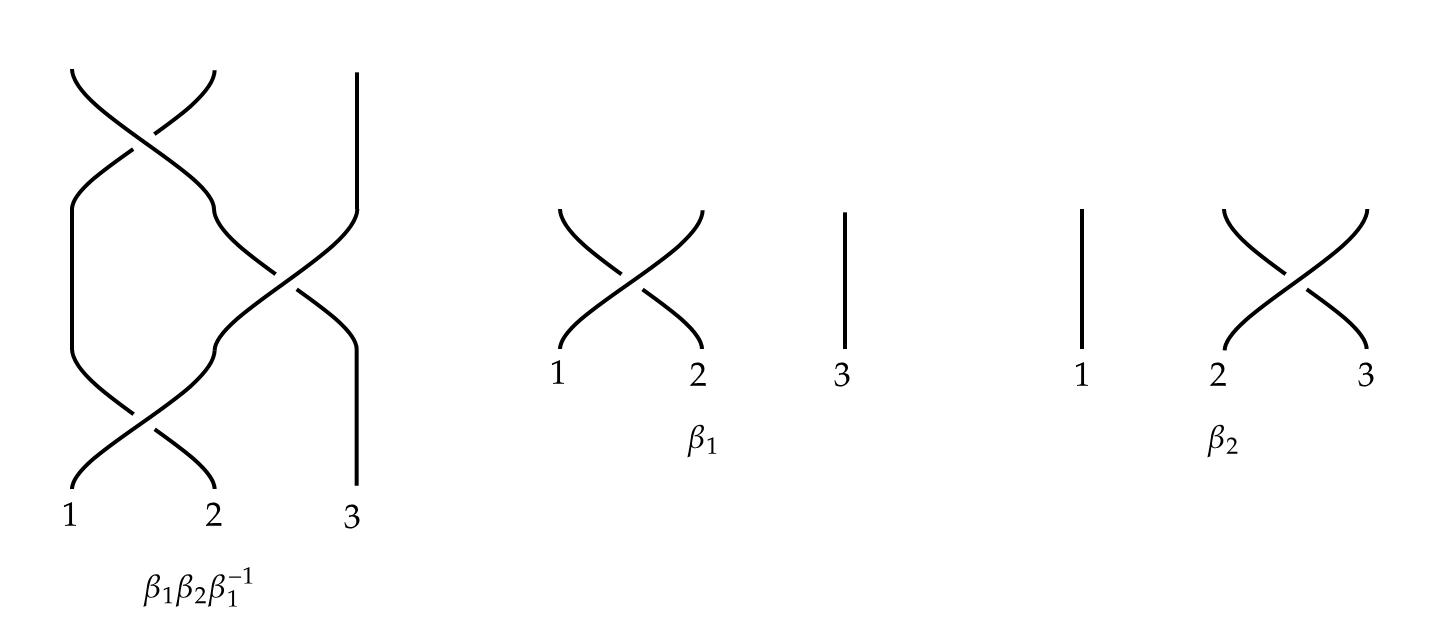}
    \caption{A braid in $\Br_3$, and the generators of $\Br_3$. The braid on the left is drawn from bottom to top and can be written in terms of generators as $\beta_1\beta_2\beta_1^{-1}$.}
    \label{fig:Braids}
\end{figure}

For $n\geq 1$, let $\Omega(n)$ denote the set of fully parenthesized permutations of the set $\{1,\ldots,n\}$. For example, $(1(32))$ and $((12)3)$ are elements of $\Omega(3)$. Alternatively, the set $\Omega(n)$ is in bijection with the set of planar binary rooted trees whose leaves are labeled with the set $\{1,\ldots,n\}$, see Figure~\ref{fig:Tree} and \cite[Chapter~6.1.2]{FresseBook1}. 

\begin{figure}[ht!]
    \includegraphics[width=5cm]{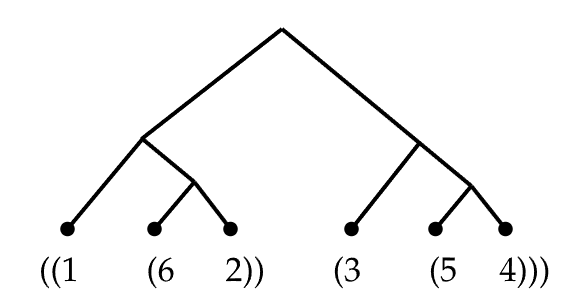}
    \caption{A parenthesized permutation represented by a planar binary rooted tree with labeled leaves.}\label{fig:Tree}
\end{figure}

For each $n\geq 1$, there is a natural action of the symmetric group $\Sigma_n$ on $\Omega(n)$ given by the multiplication of permutations under the natural map 
\[u:\Omega(n)\rightarrow \Sigma_n,\] 
which forgets parentheses. For instance, $u\big((53)((42)1)\big)=(53421)$ -- keeping in mind 
that permutations are written in one-line notation and the parentheses denote parenthesizations, not cycles.
The symmetric sequence $\Omega=\{\Omega(n)\}_{n\geq 1}$ forms an operad in the category of sets with operadic composition 
\[\begin{tikzcd}
\Omega(n)\times \Omega(m)\arrow[r, "\circ_i"] & \Omega(n+m-1)
\end{tikzcd}
\] 
given by inserting a parenthesized permutation $w'\in\Omega(m)$ into the entry labeled $i$ in a parenthesized permutation $w\in\Omega(n)$. At the level of permutations insertion is defined as in \eqref{def: composition of permutations}, and the parenthesization of $w'$ is declared innermost. 
For example, we have $(1(32))\circ_3 (21) = (1((43)2))$. 
See \cite[Chapter~6.1]{FresseBook1} for more details.

\begin{definition}
\label{def: operad PaB} 
The \defn{operad of parenthesized braids}, denoted~$\PaB$, is an operad in groupoids defined as follows.
Let $\PaB(0) \eqdef \{*\}$ be the trivial groupoid. 
For each $n\geq 1$, the groupoid $\PaB(n)$ has object set
\[
\ob(\PaB(n)) \eqdef \Omega(n),
\]
and for $w_1,w_2\in \Omega(n)$ the morphism set $\Hom_{\PaB(n)}(w_1,w_2)$
consists of those braids $\beta\in \Br_n$ whose underlying permutation is $u(w_2)^{-1}u(w_1)$. See Figure~\ref{fig:ParBraid} for an example.  Composition in the groupoid $\PaB(n)$ is induced by multiplication in the braid group $\Br_n$.

The symmetric group $\Sigma_n$ acts on each groupoid $\PaB(n)$ via composition of permutations.  
The resulting symmetric sequence $\PaB=\{\PaB(n)\}_{n\geq 0}$ forms an operad in groupoids. The operadic composition functors
\[
\begin{tikzcd}
\PaB(m)\times \PaB(n) \arrow[r, "\circ_i"] & \PaB(m+n-1)
\end{tikzcd}
\]
are defined on objects by the operadic composition of parenthesized permutations. On morphisms, composition is defined by inserting a parenthesized braid on $n$ strands into the strand labeled $i$ of a parenthesized braid on $m$ strands, as shown in Figure~\ref{fig:ParBraidOperad}. See \cite[Chapter~6]{FresseBook1} or~\cite{Calaque2025} for more details.
\end{definition}

\begin{figure}[ht!]
    \includegraphics[width=8cm]{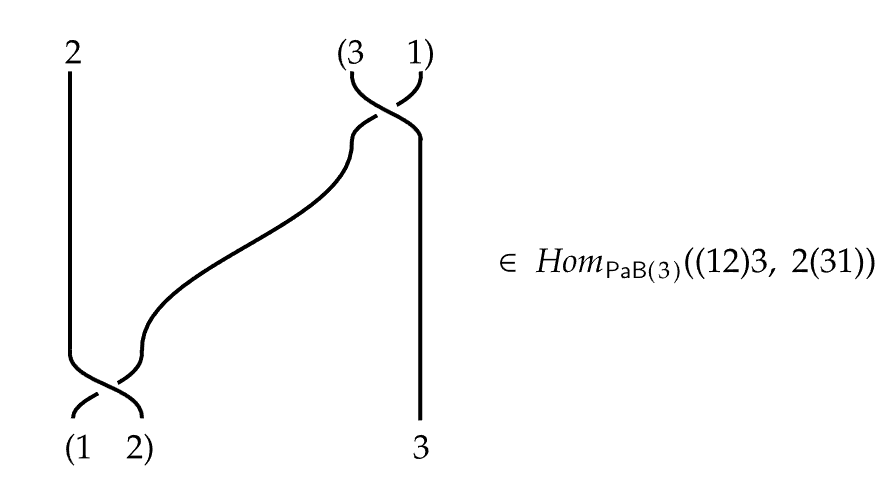}
    \caption{An example of a parenthesized braid.}\label{fig:ParBraid}
\end{figure}

\begin{figure}[ht!]
    \includegraphics[width=12cm]{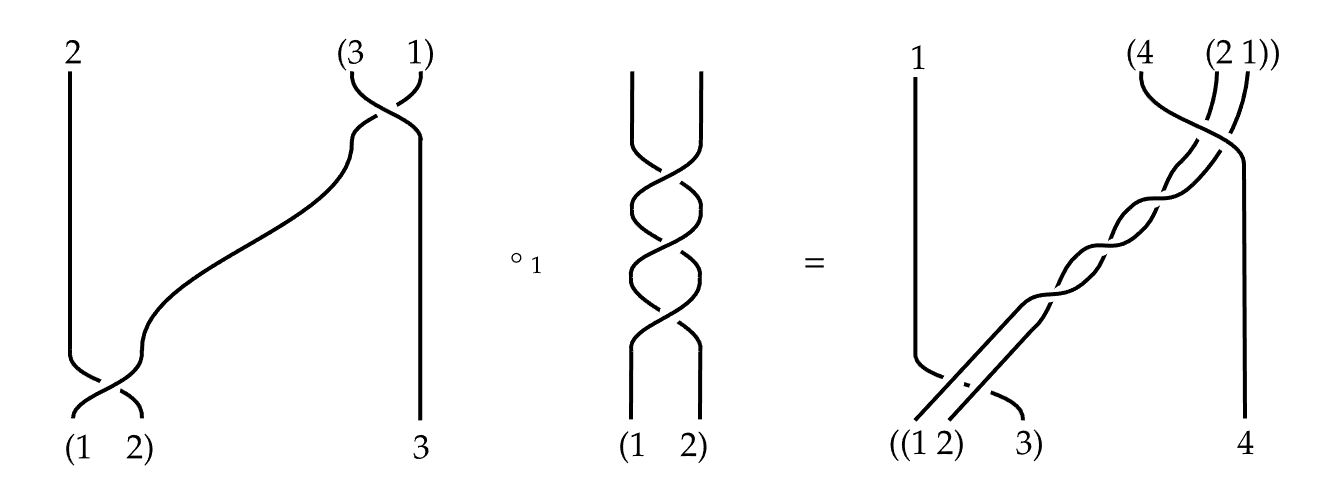}
    \caption{The operadic composition of parenthesized braids.}\label{fig:ParBraidOperad}
\end{figure}

\begin{remark}
\label{rem:composition}
A note on conventions: in the groupoid $\PaB$ we write composition in function notation, from right to left. Thus, for composable braids $\beta_1$ and $\beta_2$, the composite $\beta_2\circ\beta_1$ means ``first apply $\beta_1$, then apply $\beta_2$.'' In a braid diagram, this corresponds to drawing $\beta_1$ below $\beta_2$. In the usual multiplicative notation for braid groups, the same composite is written $\beta_1\beta_2$. Accordingly, for any object $w\in \Omega(n)$, the automorphism group $\Aut_{\PaB(n)}(w)$ is canonically identified with $\PB_n^{op}$.
\end{remark}

The operad $\PaB$ admits a finite presentation with generators
\[
R=R^{1,2}\in\Hom_{\PaB(2)}((12),(21))
\qquad\text{and}\qquad
\Phi=\Phi^{1,2,3}\in\Hom_{\PaB(3)}((12)3,1(23)),
\]
shown in Figure~\ref{fig:PaBGens}. 
In other words, every morphism in the groupoid $\PaB(n)$ can be written as a combination of operadic and categorical compositions of the generating braid $R^{1,2}$, the associativity isomorphism $\Phi^{1,2,3}$, identity morphisms, and the permutations and inverses of these morphisms. 
Note that $R^{1,2}$ is the generating braid $\beta_1\in\Br_2$, viewed as a morphism in the groupoid $\PaB(2)$, while $\Phi^{1,2,3}$ is the identity braid viewed as a non-identity morphism $((12)3)\to(1(23)).$

\begin{figure}[ht!]
    \includegraphics[width=8cm]{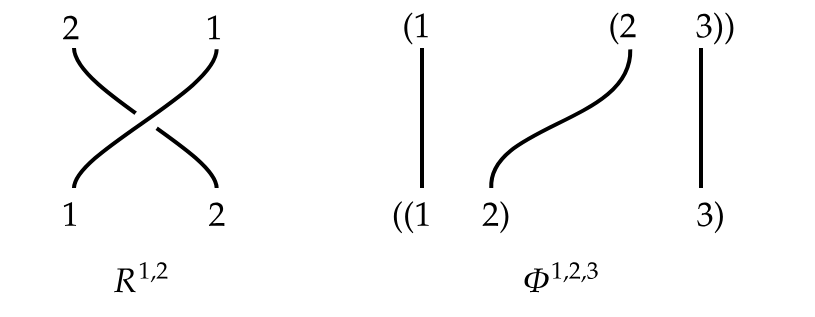}
    \caption{The generating morphisms of the operad $\PaB$.}\label{fig:PaBGens}
\end{figure}

\begin{remark}
Throughout, we use cosimplicial notation for morphisms in $\PaB$ in order to align with key sources and simplify the bookkeeping of permutations. In this notation, the superscript records the indices of the strands present at the source of a braid, omitting identity strands. In figures, this means that the strands are numbered at the bottom, as shown in Figure~\ref{fig:Cosimp}. For example, the notation $R^{1,2}$ indicates in particular that $R^{1,2}$ has source $(12)$, whereas $R^{2,1}$ has source $(21)$, and $R^{2,3}=\id_{(12)}\circ_2 R$ has source $1(23)$. Double indices indicate doubled strands; for example,
\[
R^{1,23}=R^{1,2}\circ_2 \id_{(12)},
\]
as in Figure~\ref{fig:Cosimp}. Including $\emptyset$ in the superscript, as in $\Phi^{\emptyset,1,2}$, denotes composition with the trivial object $*\in\PaB(0)$, namely
\[
\Phi^{\emptyset,1,2}= \Phi^{1,2,3}\circ_1 *.
\]
This has the effect of deleting a strand in the relevant braid. For more details, see the discussion of the unitary extension of $\PaB$ in \cite[Chapter 6]{FresseBook1}. 
\end{remark}

\begin{figure}[ht!]
    \includegraphics[width=8cm]{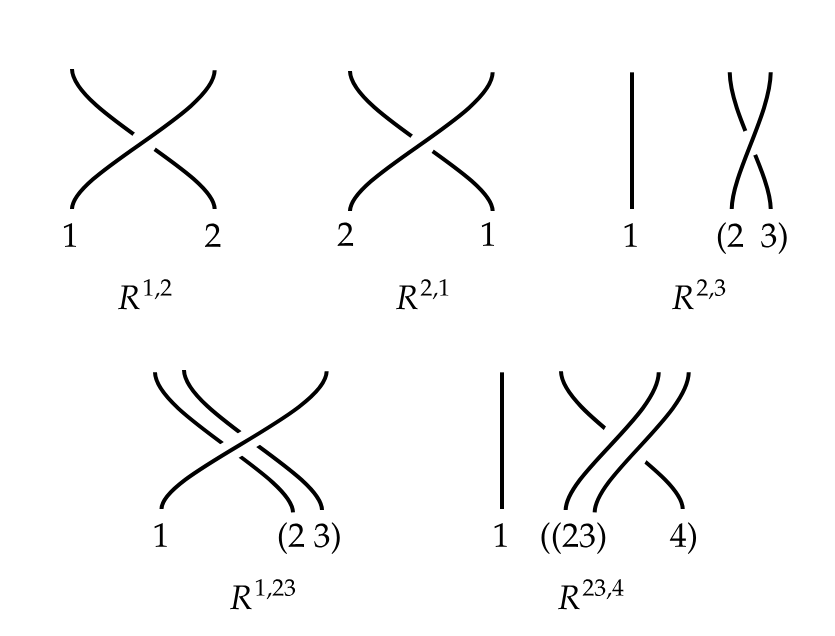}
    \caption{Examples of cosimplicial notation.}\label{fig:Cosimp}
\end{figure}

Now we are ready to state the finite presentation or coherence theorem for $\PaB$:

\begin{theorem}[{\cite[Theorem~6.2.4]{FresseBook1}}]
\label{thm: presentation pab}
The parenthesized braids operad $\PaB$ is generated by the object $(12)\in \PaB(2)$ and the morphisms
\[
R \eqdef R^{1,2}\in \Hom_{\PaB(2)}((12),(21))
\qquad\text{and}\qquad
\Phi \eqdef \Phi^{1,2,3}\in \Hom_{\PaB(3)}(((12)3),1(23)),
\]
subject to the following relations:
\begin{align}
\tag{U}
\label{eqn:U}
\Phi^{\emptyset,1,2}=\Phi^{1,\emptyset,2}=\Phi^{1,2,\emptyset}=\id_{1,2}
& \quad \mathrm{in}~\Hom_{\PaB(2)}\big((12),(12)\big), \\
\Phi^{1,2,34}\Phi^{12,3,4}
=\Phi^{2,3,4} \Phi^{1,23,4} \Phi^{1,2,3}
& \quad \mathrm{in}~\Hom_{\PaB(4)}\left(((12)3)4,1(2(34))\right), \label{eqn:P}\tag{P} \\
R^{1,3}\Phi^{2,1,3}R^{1,2}= \Phi^{2,3,1} R^{1,23} \Phi^{1,2,3}
& \quad \mathrm{in}~\Hom_{\PaB(3)}\left((12)3,2(31)\right),
\label{eqn:H1}\tag{H1} \\
R^{1,3}(\Phi^{1,3,2})^{-1} R^{2,3} = (\Phi^{3,1,2})^{-1} R^{12,3} (\Phi^{1,2,3})^{-1}
& \quad \mathrm{in}~\Hom_{\PaB(3)}\left(1(23),(31)2\right). \label{eqn:H2}\tag{H2}
\end{align}
\end{theorem}

\subsection{Completion and homomorphic expansions}\label{subsec:compform}

Let $\K$ be a field containing $\mathbb{Q}$, and let $\mathsf{G}$ be a discrete group. The \defn{prounipotent completion} of $\mathsf{G}$ over $\K$, denoted $\mathsf{G}_{\K}$, is the universal prounipotent group over $\K$ equipped with a homomorphism from $\mathsf{G}$.

This completion is constructed from the group algebra $\K[\mathsf{G}]$, which is naturally a cocommutative Hopf algebra. Let $I \subset \K[\mathsf{G}]$ denote the augmentation ideal, that is, the kernel of the counit
\[
\epsilon \colon \K[\mathsf{G}] \to \K,
\qquad
\epsilon(g)=1 \ \text{for all } g\in \mathsf{G}.
\]
The \defn{$I$-adic completion} of $\K[\mathsf{G}]$ is the inverse limit
\[
\K[\mathsf{G}]^\wedge_I := \varprojlim_n \K[\mathsf{G}]/I^n.
\]
The prounipotent completion of $\mathsf{G}$ is then the group of \defn{group-like elements} in the completed Hopf algebra $\K[\mathsf{G}]^\wedge_I$:
\[
\mathsf{G}_{\K}
:=
\{g \in \K[\mathsf{G}]^\wedge_I \mid \Delta(g)=g\otimes g,\ \epsilon(g)=1\}.
\]

The associated Lie algebra is the Lie algebra of primitive elements in $\K[\mathsf{G}]^\wedge_I$:
\[
\mathfrak{g}
:=
\operatorname{Prim}(\K[\mathsf{G}]^\wedge_I)
=
\{x \in \K[\mathsf{G}]^\wedge_I \mid \Delta(x)=x\otimes 1 + 1\otimes x,\ \epsilon(x)=0\}.
\]
Since $\K[\mathsf{G}]^\wedge_I$ is complete and cocommutative, the exponential and logarithm series converge and define mutually inverse bijections between primitive and group-like elements. In particular, every element $g\in \mathsf{G}_{\K}$ can be written uniquely in the form $g=e^x$ for some $x\in \mathfrak{g}$. The prounipotent completion of $\mathsf{G}$ can equivalently be realized as the unipotent radical of the proalgebraic completion of $\mathsf{G}$.

\begin{remark}
There is a closely related Lie algebra associated to any group $\mathsf{G}$, defined using its lower central series. Recall that the \defn{lower central series} of $\mathsf{G}$ is given inductively by
\[
\Gamma_1(\mathsf{G}) := \mathsf{G},
\qquad
\Gamma_{n+1}(\mathsf{G}) := [\Gamma_n(\mathsf{G}),\mathsf{G}].
\]
Each quotient $\Gamma_n(\mathsf{G})/\Gamma_{n+1}(\mathsf{G})$ is abelian, and the associated graded object
\[
\gr_\Gamma(\mathsf{G})
:=
\bigoplus_{n\geq 1}\Gamma_n(\mathsf{G})/\Gamma_{n+1}(\mathsf{G})
\]
inherits a Lie bracket induced by the group commutator.

The relationship between this graded Lie algebra and the Lie algebra $\mathfrak g=\operatorname{Prim}(\K[\mathsf{G}]^\wedge_I)$
of the prounipotent completion depends on additional hypotheses. In particular, if $\mathsf{G}$ is $1$-formal over $\K$, then one can identify the Lie algebra $\mathfrak g=\operatorname{Prim}(\K[\mathsf{G}]^\wedge_I)$ with the degree completion of the Lie algebra determined by the lower central series.
\end{remark}

\begin{example}
\label{example: completed free groups}

Let $\mathsf{F}_n$ denote the free group on $n$ generators. Its lower central series quotients $\Gamma_k(\mathsf{F}_n)/\Gamma_{k+1}(\mathsf{F}_n)$
are torsion-free, and the associated graded Lie algebra
\[
\gr_\Gamma(\mathsf{F}_n)
=
\bigoplus_{k \geq 1} \Gamma_k(\mathsf{F}_n)/\Gamma_{k+1}(\mathsf{F}_n)
\]
is isomorphic to the free Lie algebra $\mathfrak{lie}_n$ on $n$ generators, with each generator in degree one. It follows that the prounipotent completion $(\mathsf{F}_n)_\K$ is naturally isomorphic to the prounipotent group associated to the degree completion of $\mathfrak{lie}_n$, namely
\[
(\mathsf{F}_n)_\K \cong \exp(\widehat{\mathfrak{lie}}_n).
\]
Here $\widehat{\mathfrak{lie}}_n$ denotes the degree completion of the free Lie algebra $\mathfrak{lie}_n$. The exponential map identifies $\widehat{\mathfrak{lie}}_n$ with the group-like elements in the completed universal enveloping algebra, which in this case agrees with the completed group algebra $\K[\mathsf{F}_n]^\wedge_I$. See, for example, \cite[Proposition~8.4.1]{FresseBook1} for further details.
\end{example}

\subsubsection{Associated graded of a group}
Given a filtered algebra $A = F_0 \supseteq F_1 \supseteq F_2 \supseteq \cdots,$ its \defn{associated graded algebra} is
\[
\gr A \eqdef \bigoplus_{n\geq 0} F_n/F_{n+1},
\]
with multiplication induced by the filtration. When the filtration is complete, one may also consider the completed \defn{associated graded algebra}
\[
\widehat{\gr}A \eqdef \prod_{n\geq 0} F_n/F_{n+1}.
\]

Now let $\mathsf{G}$ be a group, and let $I\subset \K[\mathsf{G}]$ be the augmentation ideal. The powers $I^n$ define a decreasing filtration on the group algebra $\K[\mathsf{G}]$, and this induces a complete filtration on its $I$-adic completion $\K[\mathsf{G}]^\wedge_I$. The corresponding completed associated graded algebra is
\[
\widehat{\gr}\big(\K[\mathsf{G}]\big) \eqdef \prod_{n\geq 0} I^n/I^{n+1}.
\]
It inherits the structure of a graded cocommutative Hopf algebra. We write $\grp\big(\widehat{\gr}(\K[\mathsf{G}])\big)$ for its group of group-like elements.

\begin{definition}\label{def:HomExp}
A \defn{homomorphic expansion} for a group $\mathsf{G}$ is a filtered Hopf algebra isomorphism
\[
\widehat{Z}\colon \K[\mathsf{G}]^\wedge_I \longrightarrow \widehat{\gr}\big(\K[\mathsf{G}]\big).
\]
Such an isomorphism induces an isomorphism of prounipotent groups
\[
\widehat{Z}\colon \mathsf{G}_\K \xrightarrow{\cong}
\grp\big(\widehat{\gr}(\K[\mathsf{G}])\big).
\]
\end{definition}

Informally, the existence of a homomorphic expansion means that the prounipotent completion of $\mathsf{G}$ is completely determined by the associated graded structure coming from the
augmentation filtration.

\begin{example}
\label{example: associated graded of completed free groups}
Let $\mathsf{F}_n$ be the free group on $n$ generators. As in Example~\ref{example: completed free groups}, its prounipotent completion is naturally isomorphic to the prounipotent group associated to the degree-completed free Lie algebra: $(\mathsf{F}_n)_\K \cong \exp(\widehat{\mathfrak{lie}}_n).$ On the other hand, the associated graded Hopf algebra of the group algebra is naturally isomorphic to the universal enveloping algebra of the graded Lie algebra arising from the lower central series: $\gr_I \K[\mathsf{F}_n] \cong  U\!\left(\gr_\Gamma(\mathsf{F}_n)\otimes \K\right).$
Since $\gr_\Gamma(\mathsf{F}_n)\cong \mathfrak{lie}_n$, passing to completions gives an isomorphism of completed graded Hopf algebras $\widehat{\gr}\,\K[\mathsf{F}_n] \cong \widehat{U}(\mathfrak{lie}_n).$ Taking group-like elements and using the exponential map therefore yields an isomorphism
\[
(\mathsf{F}_n)_\K \xrightarrow{\cong}
\grp\!\left(\widehat{\gr}\,\K[\mathsf{F}_n]\right).
\]
Thus free groups admit homomorphic expansions in the sense of Definition~\ref{def:HomExp}.
\end{example}

\subsubsection{Prounipotent completion of groupoids and operads}

Prounipotent completion extends from groups to groupoids. Let $\mathrm{G}$ be a groupoid. Its prounipotent completion $\mathrm{G}_{\K}$ is defined objectwise: it has the same objects as $\mathrm{G}$, and for each object $x$ the automorphism group is replaced by its prounipotent completion
\[
\Aut_{\mathrm{G}_{\K}}(x):=\Aut_{\mathrm{G}}(x)_{\K}.
\]
More generally, if $x$ and $y$ lie in the same connected component of $\mathrm{G}$, then $\Hom_{\mathrm{G}}(x,y)$ is a torsor under the automorphism groups of $x$ and $y$, and $\Hom_{\mathrm{G}_{\K}}(x,y)$ is obtained by completing this torsor accordingly. 
The resulting completion is described in detail in \cite[\S9.1]{FresseBook1}. For the purposes of this paper, the main point is that prounipotent completion is functorial and behaves well with respect to monoidal products.

\begin{prop}[{\cite[Proposition~9.2.2]{FresseBook1}}]
\label{prop: completion is monoidal}
The completion functor
\[
(-)_{\K}\colon \mathsf{Grpd}\to \mathsf{Grpd}_{\K}
\]
is symmetric monoidal.
\end{prop}

Proposition~\ref{prop: completion is monoidal} allows us to apply prounipotent completion aritywise to obtain completions of colored operads in groupoids.

\begin{definition}
Let $\calP=\{\calP(n)\}_{n\ge 0}$ be an operad in groupoids. The \defn{prounipotent completion} of $\calP$ is the operad in prounipotent groupoids defined aritywise by $\calP_{\K}:=\{\calP(n)_{\K}\}_{n\ge 0}.$ Since prounipotent completion is symmetric monoidal, the operadic composition maps
\[
\circ_i\colon \calP(n)\times \calP(m)\longrightarrow \calP(n+m-1)
\]
induce composition maps
\[
\begin{tikzcd}
\calP(n)_{\K}\times \calP(m)_{\K}\arrow[r,"\widehat{\circ}_i"] & \calP(n+m-1)_{\K}
\end{tikzcd}
\]
which endow $\calP_{\K}$ with the structure of an operad.
\end{definition}

\subsubsection{Associated graded structures for groupoids and operads}
Let $\calP=\{\calP(n)\}_{n\ge 0}$ be a (colored) operad in groupoids. Applying the $\K$-linear envelope aritywise gives a (colored) operad in $\K$-linear categories
\[
\K[\calP]:=\{\K[\calP(n)]\}_{n\ge 0}.
\]
The augmentation filtrations on the endomorphism groups induce aritywise completed associated graded objects, and we write
\[
\widehat{\gr}\,\K[\calP]:=\{\widehat{\gr}\,\K[\calP(n)]\}_{n\ge 0}
\]
for the resulting (colored) operad in completed graded $\K$-linear categories.

\begin{definition}
Let $\calP$ be a (colored) operad in groupoids. A \defn{homomorphic expansion} for $\calP$ is a filtered isomorphism of (colored) operads in $\K$-linear categories
\[
\widehat{Z}\colon \K[\calP]^\wedge \xrightarrow{\cong} \widehat{\gr}\,\K[\calP].
\]
Passing to group-like morphisms in each arity induces an isomorphism of (colored) operads in prounipotent groupoids
\[
\mathcal Z\colon \calP_{\K}\xrightarrow{\cong} \grp\bigl(\widehat{\gr}\,\K[\calP]\bigr).
\]
\end{definition}
\subsection{Expansions in action: chord diagrams and Drinfeld associators}\label{subsec:chorddiag}

Kohno proved in \cite{Kohno1985} that the Lie algebra of the prounipotent completion of the pure braid group $\PB_n$ is the Lie algebra now known as the \defn{Drinfeld--Kohno Lie algebra} $\mathfrak{t}_n$, also called the Lie algebra of \defn{infinitesimal braids}. 

The Lie algebra $\mathfrak{t}_n$ is generated by symbols $t_{ij}=t_{ji}$ for $1\leq i\neq j\leq n$, subject to the relations
\[
[t_{ij},t_{ik}+t_{jk}] = 0
\qquad\text{and}\qquad
[t_{ij},t_{kl}] = 0
\]
whenever $i,j,k,l$ are distinct. The grading is determined by declaring each generator to have degree one. Throughout the paper we work with the graded completion, which we denote by $\ib_n$ rather than $\widehat{\ib}_n$ in order to simplify notation. Accordingly, the prounipotent completion $(\PB_n)_\K$ is naturally identified with the prounipotent group $\exp(\ib_n)$, with $x_{ij}$ corresponding to $e^{t_{ij}}$; see \cite{Kohno1985} or the sketch of Theorem~10.0.7 in \cite{FresseBook1}.

\begin{definition}
\label{def: operad of inft braids}
By convention, we set $\ib_0=\ib_1=\{0\}$. For $n\geq 2$, there is a natural right action of the symmetric group $\Sigma_n$ on each Lie algebra $\ib_n$ which permutes the indices of the generators, that is, ~$(t_{ij})^{\sigma} \eqdef t_{\sigma^{-1}(i)\sigma^{-1}(j)}$, for some $\sigma\in\Sigma_n$.  The symmetric sequence $\ib=\{\ib_n\}_{n\geq 0}$ forms and operad in the category of Lie algebras with partial composition maps 
\[\begin{tikzcd}
\ib_n\oplus \ib_m\arrow[r, "\circ_{\alpha}"] & \ib_{n+m-1},
\end{tikzcd}\] $1\leq \alpha\leq n$ defined on generators $t_{ij}\in\ib_n$ and $t_{kl}\in\ib_m$ as $t_{ij} \circ_\alpha t_{kl} \eqdef t_{ij} \circ_\alpha 0 +0 \circ_\alpha t_{kl}$ with
\[ 
t_{ij} \circ_\alpha 0 
\eqdef 
\begin{cases}
    t_{(i+m-1)(j+m-1)} & \text{ if } \alpha < i,\\
    \sum_{\beta=1}^{m}t_{(i+\beta-1)(j+m-1)} & \text{ if } \alpha = i,\\
    t_{i(j+m-1)} & \text{ if } i<\alpha<j,\\
    \sum_{\beta=1}^{m}t_{i(j+\beta-1)} & \text{ if } \alpha=j,\\
    t_{ij} & \text{ if } \alpha > j,
\end{cases}
\]
and $0 \circ_\alpha t_{kl} \eqdef t_{(k+\alpha-1)(l+\alpha-1)}$ for all $\alpha$. Note that $0 \in \ib_1$ acts as an identity for the operadic $\circ_\alpha$ compositions.
\end{definition}

Passing from a degree-complete Lie algebra to the group-like elements in its completed universal enveloping algebra defines a functor from degree-complete Lie algebras to prounipotent groups. Equivalently, one may pass from a degree-complete Lie algebra $\mathfrak{g}$ to the prounipotent group $\exp(\mathfrak{g})$, with group law given by the Baker--Campbell--Hausdorff formula. Applied arity-wise, this induces a functor from operads in degree-complete Lie algebras to operads in prounipotent groups.

In particular, the prounipotent groups $\exp(\ib_n)$, for $n\geq 0$, assemble into an operad called the \defn{operad of chord diagrams}. This terminology comes from the usual diagrammatic realization of the completed universal enveloping algebra of $\ib_n$, whose elements are represented by chord diagrams on $n$ vertical strands, with each generator $t_{ij}$ drawn as a horizontal chord joining strand $i$ to strand $j$, and multiplication given by vertical stacking.

\begin{definition}
For $n\geq 0$, let $\mathsf{CD}(n)\eqdef \exp(\ib_n)$ denote the prounipotent group of group-like elements in the completed universal enveloping algebra of $\ib_n$.\footnote{Since $\ib_n$ is degree-complete, the exponential map identifies $\ib_n$ with the group-like elements in its completed universal enveloping algebra.} The right action of $\Sigma_n$ on $\ib_n$ induces a right action on $\mathsf{CD}(n)$. The resulting symmetric sequence $\mathsf{CD}=\{\mathsf{CD}(n)\}_{n\geq 0}$ forms an operad in prounipotent groups, with operadic composition given by \[e^x\circ_i e^y \eqdef e^{x\circ_i y}.\]
\end{definition}

The associated graded operad of the operad $\PaB$ is the operad of {\em parenthesized} chord diagrams built on $\CD$.

\begin{definition}\label{defn: PaCD}
Let $\PaCD(0)=\{*\}$ be the trivial groupoid. For $n\geq 1$, define $\PaCD(n)$ to be the prounipotent groupoid with object set $\ob(\PaCD(n))=\Omega(n)$ and with morphism spaces \[\Hom_{\PaCD(n)}(w_1,w_2)=\exp(\ib_n)\] for all $w_1,w_2\in\Omega(n)$. Categorical composition is given by multiplication in the prounipotent group $\exp(\ib_n)$, and identity morphisms are given by the unit element. The symmetric group $\Sigma_n$ acts on $\PaCD(n)$ by permuting objects and, on morphisms, through its action on $\ib_n$. The resulting symmetric sequence $\PaCD=\{\PaCD(n)\}_{n\geq 0}$ forms an operad in prounipotent groupoids. On objects, the operadic composition is induced by the operadic composition of parenthesized permutations, and on morphisms it is given by the operadic composition in $\CD$, namely \[e^x\circ_i e^y \eqdef e^{x\circ_i y}.\]
\end{definition}

In \cite{BN98}, Bar-Natan defines a \emph{homomorphic expansion} for $\PaB$ to be an object-fixing isomorphism of operads $\hPaB \to \PaCD$, and shows that such expansions are classified by Drinfeld associators. On the other hand, every operad equivalence $\varphi\colon \hPaB\to \CD$ extends uniquely to an operad isomorphism
\[
\varphi'\colon \hPaB\to \PaCD
\]
which is the identity on objects (\cite[Theorem~10.3.12]{FresseBook1}). Thus, operad equivalences $\hPaB\to \CD$ and homomorphic expansions of $\PaB$ encode the same data. Since in later sections we work only with the operad $\CD$, we will formulate the correspondence with Drinfeld associators in terms of operad equivalences $\hPaB\to \CD$. We now recall the definition of a Drinfeld associator.

\begin{definition}
\label{def:DrinfAss}
A \defn{Drinfeld associator} is a pair $(\mu,f)\in \K^{\times}\times \exp(\lie_2)$ satisfying the following equations:
\begin{gather}
    f(\xi_1, \xi_2)f(\xi_2, \xi_1)=1, \label{inversion}\tag{I} \\
    e^{\mu \xi_1/2} f(\xi_3, \xi_1) e^{\mu \xi_3/2} f(\xi_2, \xi_3) e^{\mu \xi_2/2} f(\xi_1, \xi_2) = 1, \label{hexagon}\tag{H}
\end{gather}
whenever $e^{\xi_1+\xi_2+\xi_3}=1$ in $\exp(\mathfrak{t}_3)$, and
\begin{equation}
\label{pentagon}
\tag{P}
f(t_{23}, t_{34})f(t_{12}+ t_{13}, t_{24}+t_{34})f(t_{12}, t_{23})
=
f(t_{12}, t_{23}+ t_{24}) f(t_{13} + t_{23}, t_{34})
\end{equation}
in $\exp(\mathfrak{t}_4)$. Equation~\eqref{pentagon} is called the \defn{pentagon} equation, and equation~\eqref{hexagon} is called the \defn{hexagon} equation.
\end{definition}

\begin{remark}
\label{rmk: Hexagon in Associator}
The hexagon equation is often written in the form
\begin{equation}
\label{eq:DA-Hex} \tag{H'}
e^{\mu t_{12}/2} f(t_{13}, t_{12}) e^{\mu t_{13}/2} f(t_{23}, t_{13}) e^{\mu t_{23}/2} f(t_{12}, t_{23}) = 1
\end{equation}
in $\exp(\ib_3)$, which is equivalent to \eqref{hexagon}. Since the element $t_{12}+t_{13}+t_{23}$ generates the center of $\ib_3$, one obtains further equivalent forms of the hexagon equation, and similarly of the pentagon equation, by adding central elements of $\ib_3$ or $\ib_4$ to the inputs of $f$. For example, a version of the hexagon equation involving only the variables $t_{12}$ and $t_{13}$ is
\[
e^{\mu t_{12}/2} f(t_{13}, t_{12}) e^{\mu t_{13}/2} f(-t_{12}-t_{13}, t_{13}) e^{\mu (-t_{12}-t_{13})/2} f(t_{12}, -t_{12}-t_{13}) = 1.
\]

The hexagon equation \eqref{eq:DA-Hex} is also equivalent to the pair of equations
\begin{gather}
e^{(t_{12}+t_{13})/2}=f(t_{23},t_{31})^{-1} e^{t_{13}/2} f(t_{21}, t_{13}) e^{t_{12}/2} f(t_{12},t_{23})^{-1}, \label{eq:DA-Hex1}\tag{H1} \\
e^{(t_{13}+t_{23})/2} = f(t_{31},t_{12}) e^{t_{13}/2} f(t_{13},t_{32})^{-1} e^{t_{23}/2} f(t_{12},t_{23}), \label{eq:DA-Hex2}\tag{H2}
\end{gather}
provided \eqref{inversion} holds. Conversely, the equations \eqref{eq:DA-Hex1} and \eqref{eq:DA-Hex2} together imply the inversion relation \eqref{inversion} (\cite[\S 4]{Drin}). Thus one may equivalently define Drinfeld associators as pairs $(\mu,f)$ satisfying \eqref{eq:DA-Hex1}, \eqref{eq:DA-Hex2}, and \eqref{pentagon}, or as pairs satisfying \eqref{inversion}, \eqref{eq:DA-Hex1}, \eqref{eq:DA-Hex2}, and \eqref{pentagon}.
\end{remark}

\medskip

The following theorem is proved in \cite{BN98} (in a slightly different context), and also in \cite[Proposition 10.2.7]{FresseBook1}. Conceptually, it is a consequence of Theorem~\ref{thm: presentation pab}, since an operad equivalence is determined by its values on the generators of $\hPaB$.

\begin{thm}\label{cor: operadic associators}
There is a bijection between the set of operad equivalences $\varphi\colon \hPaB \longrightarrow \mathsf{CD}$ and the set of Drinfeld associators. Under this correspondence,
\[
\varphi(R^{1,2}) = e^{\frac{\mu t_{12}}{2}} \in\exp(\ib_2)
\qquad \text{and} \qquad
\varphi(\Phi^{1,2,3}) = f(t_{12},t_{23}) \in \exp(\ib_3),
\]
and these elements satisfy the equations of Definition~\ref{def:DrinfAss}. In particular, for any operad equivalence $\varphi$, the pair $(\mu, f(t_{12},t_{23}))$ is a Drinfeld associator, and conversely every Drinfeld associator arises in this way.
\end{thm}

\subsection{The Grothendieck--Teichm\"{u}ller groups}\label{subsec:GTgroups}

The set of Drinfeld associators is a bi-torsor under the actions of two pro-algebraic groups: the Grothendieck--Teichm\"uller group $\gt$ and the graded Grothendieck--Teichm\"uller group $\grt$. Their prounipotent parts are denoted $\gt_1$ and $\grt_1$, respectively. We follow the conventions of \cite[Chapters~10 and~11]{FresseBook1}, and we will freely identify $(\mathsf{F}_2)_{\K}$ with $\exp(\lie_2)$ via the isomorphism $(\mathsf{F}_2)_{\K} \cong \exp(\lie_2)$ given by sending 
\[x_1 \mapsto e^{\xi_1}, \quad \text{and} \quad x_2 \mapsto e^{\xi_2} 
\] (cf. Example~\ref{example: completed free groups}).

\begin{definition}\label{defn: GT}
The \defn{Grothendieck--Teichm\"uller group}, denoted $\gt$, is the pro-algebraic group consisting of all pairs $(\lambda,f)\in \K^{\times}\times (\mathsf{F}_2)_{\K}$ satisfying
\begin{gather}
    f(x_1,x_2) f(x_2,x_1)= 1, \label{H1 for GT} \\
    f(x_1,x_2)x_1^{\mu}f(x_3,x_1)x_3^{\mu}f(x_2,x_3)x_2^{\mu} =1
    \quad \text{in } (\mathsf{F}_2)_{\K},
    \quad \text{where } x_3x_2x_1 = 1 \text{ and } \lambda = 2\mu + 1,  \label{H2 for GT} \\
    f(x_{23},x_{34})f(x_{13}x_{12},x_{34}x_{24})f(x_{12},x_{23})
    =
    f(x_{12},x_{24}x_{23})f(x_{23}x_{13},x_{34})
    \quad \text{in } (\PB_{4})_{\K}. \label{P for GT}
\end{gather}
The group law is given by
\[
(\lambda_1,f_1)\ast (\lambda_2,f_2)
=
\bigl(\lambda_1 \lambda_2,\,
f_{1}(x_1,x_2)\, f_2(x_1^{\lambda_{1}}, f_1^{-1}x_2^{\lambda_1}f_{1})\bigr).
\]
\end{definition}

The subgroup of $\gt$ consisting of all pairs $(\lambda,f)$ with $\lambda=1$ is denoted by $\gt_1$. This is the prounipotent part of $\gt$. This subgroup is sometimes referred to as the \defn{acyclotomic} Grothendieck--Teichmüller group, as it coincides with the kernel of the cyclotomic character of the profinite version of the Grothendieck-Teichm\"uller group (\cite{LS15}, \cite{Drin}). 

The group $\gt$ acts freely and transitively on the set of Drinfeld associators. Given a Drinfeld associator $(\mu,f(\xi_1,\xi_2))\in \K^{\times}\times \exp(\lie_2)$ and an element $(\lambda,g(x_1,x_2))\in \gt$, the action is
\begin{equation}\label{eqn: action of gt on Drinf ass}
(\mu, f ) \ast (\lambda,g)
=
(\mu\lambda,\, f(\xi_1,\xi_2)\, g(e^{\mu \xi_1}, f^{-1}e^{\lambda \xi_2}f)).
\end{equation}
Note that this is a left group action, but we write the group element on the right in accordance with our function-composition convention.

Via the identification of operad equivalences $\varphi\colon \hPaB\to \CD$ with Drinfeld associators, the group $\gt$ identifies with the automorphism group of the operad $\hPaB$.

\begin{thm}[{\cite{BN98}, \cite[Theorem~11.1.7]{FresseBook1}}]
\label{thm: GT as automorphisms}
Let $\Aut_{0}(\hPaB)$ denote the set of automorphisms $\vartheta\colon\hPaB\to \hPaB$ that act as the identity on objects. Then there is an isomorphism of pro-algebraic groups
\[
\Aut_{0}(\hPaB)\cong \gt.
\]
\end{thm}

The pair $(\lambda, f) \in \gt$ uniquely determines the automorphism $\vartheta \in \Aut_{0}(\hPaB)$ given on generators by
\begin{align*}
    \vartheta(R^{1,2}) &= (R^{2,1}R^{1,2})^{\nu} \cdot R^{1,2} , \\
    \vartheta(\Phi^{1,2,3}) &= f(x_{12},x_{23})\cdot \Phi^{1,2,3},
\end{align*}
where $\nu=(\lambda -1)/2$. Conversely, any such assignment defines an automorphism of $\hPaB$.

The \defn{graded Grothendieck--Teichm\"uller group}, denoted $\grt$, also acts on the set of Drinfeld associators. Its prounipotent part is denoted $\grt_1$.

\begin{definition}
The \defn{graded Grothendieck--Teichm\"uller group} $\grt$ is the pro-algebraic group defined as the semidirect product
\[
\grt=\K^{\times}\ltimes \grt_1,
\]
where $\grt_1$ consists of all elements $g(\xi_1,\xi_2)\in \exp(\lie_2)\subset \exp(\ib_3)$ satisfying
\begin{align}
g(\xi_1, \xi_2) &= g(\xi_2, \xi_1)^{-1}, \\
g(\xi_3, \xi_1)\, g(\xi_2, \xi_3)\, g(\xi_1, \xi_2)&= 1
\quad \text{in } \exp(\ib_3),
\quad \text{with } \xi_1 + \xi_2 +\xi_3 = 0, \\
g(t_{12}, t_{23} + t_{24})\, g(t_{13} + t_{23}, t_{34})
&=
g(t_{23}, t_{34})\, g(t_{12} + t_{13}, t_{24} + t_{34})\, g(t_{12}, t_{23})
\quad \text{in } \exp(\ib_4).
\end{align}
The action of $\K^{\times}$ on $\grt_1$ is given by
\[
(\lambda\cdot g)(\xi_1,\xi_2)=g(\lambda \xi_1,\lambda \xi_2).
\]
The group law in $\grt=\K^{\times}\ltimes \grt_1$ is
\[
(\lambda_1,g_1)\ast (\lambda_2,g_2)
=
\bigl(\lambda_1 \lambda_2,\,
g_{1}(\xi_1,\xi_2)\,g_2(\lambda_1 \xi_1, g_1^{-1}\cdot \lambda_1 \xi_2 \cdot g_{1})\bigr).
\]
\end{definition}

The group $\grt$ also acts freely and transitively on the set of Drinfeld associators. Given a Drinfeld associator $(\mu,f)$ and an element $(\lambda, g)\in \grt$, the action is
\begin{equation}
 (\lambda, g)\ast (\mu, f)= (\lambda \mu,\, g(\xi_1,\xi_2)\, f(\lambda \xi_1, g^{-1}\cdot \lambda \xi_2 \cdot g)).
\end{equation}

The following theorem characterizes $\grt$ as the automorphism group of the operad of parenthesized chord diagrams.

\begin{thm}[{\cite{BN98}, \cite[Theorem~10.3.10]{FresseBook1}}]
Let $\Aut_0(\PaCD)$ denote the set of operad isomorphisms $\vartheta\colon\PaCD\to\PaCD$ that act as the identity on objects. Then there is an isomorphism of pro-algebraic groups
\[
\Aut_{0}(\PaCD)\cong \grt.
\]
\end{thm}

\section{Moperad expansions: The moperad of parenthesized braids with a frozen strand}
Expansions of the moperad of parenthesized braids with a frozen strand were described in \cite{calaque20moperadic} in the study of an operadic interpretation of cyclotomic associators. In the present paper, we use a closely related variant of this moperad, adapted for our study of Kashiwara--Vergne solutions. The precise relationship with the formulation of \cite{calaque20moperadic} is explained in Remark~\ref{remark: difference in presentations of PaB1} later in this section.

\subsection{Moperads}
In this section we recall some preliminaries on moperads. For further details, see for instance \cite{Willwatcher2016Moperads}.

\begin{definition}
Let $\mathcal{P}$ be a one-colored symmetric operad in a symmetric monoidal category $\mathrm{C}$. A \defn{right $\mathcal{P}$-module} consists of a symmetric sequence $\mathcal{M} = \{\mathcal{M}(n)\}_{n \geq 0}$ in $\mathrm{C}$ together with structure maps
\[
\circ_i : \mathcal{M}(k) \otimes \mathcal{P}(m) \longrightarrow \mathcal{M}(k + m - 1), \quad \text{for } 1 \leq i \leq k,
\]
called \emph{partial right actions}, which satisfy the following axioms:
\begin{itemize}
    \item[(i)] For all $\mu \in \mathcal{M}(k)$, we have $\mu \circ_i \id = \mu$.
    
    \item[(ii)] For all $\lambda \in \mathcal{M}(\ell)$, $\mu \in \mathcal{P}(m)$, and $\nu \in \mathcal{P}(n)$, we have
    \[
    (\lambda\circ_{j}\mu) \circ_{i} \nu=\begin{cases}
    (\lambda \circ_i \nu) \circ_{j+n-1} \mu, & \text{if } 1 \leq i \leq j - 1, \\
    \lambda \circ_j (\mu \circ_{i-j+1} \nu), & \text{if } j \leq i \leq j + m - 1, \\
    (\lambda \circ_{i-m+1} \nu) \circ_j \mu, & \text{if } j + m \leq i \leq \ell + m - 1,
    \end{cases}
    \]
    
    \item[(iii)] For all $\mu \in \mathcal{M}(k)$, $\nu \in \mathcal{P}(m)$, and $\sigma \in \Sigma_k$, $\tau \in \Sigma_m$, we have
    \[
    \mu^\sigma \circ_{\sigma^{-1}(i)} \nu^\tau
              = (\mu \circ_{i} \nu)^{\sigma \circ_i \tau},
    \]
    where $\sigma \circ_i \tau$ denotes the block permutation obtained by inserting $\tau$ into the $i$th slot of $\sigma$.
\end{itemize}
\end{definition}

\begin{example}
\label{ex:operad-self-module}
Let $\mathcal{P}$ be an operad in the category $\mathrm{C}$. Then $\mathcal{P}$ can be regarded as a right $\mathcal{P}$-module via its own operadic composition. Explicitly, we define the structure maps
\[
\circ_i : \mathcal{P}(k) \otimes \mathcal{P}(m) \longrightarrow \mathcal{P}(k + m - 1),
\]
for $1 \leq i \leq k$ using the operad composition. This action satisfies all the axioms of a right $\mathcal{P}$-module.
\end{example}

\begin{definition}
A \defn{morphism of right $\mathcal{P}$-modules} $\varphi: \mathcal{M} \to \mathcal{N}$ is a morphism of symmetric sequences such that the diagram
\[
\begin{tikzcd}
\mathcal{M}(k) \otimes \mathcal{P}(m) \arrow[r, "\circ_i"] \arrow[d, "\varphi(k) \otimes \id"'] & \mathcal{M}(k+m-1) \arrow[d, "\varphi(k+m-1)"] \\
\mathcal{N}(k) \otimes \mathcal{P}(m) \arrow[r, "\circ_i"] & \mathcal{N}(k+m-1)
\end{tikzcd}
\]
commutes for all $k,m \geq 0$ and $1 \leq i \leq k$.

If $\mathcal{P}$ and $\mathcal{Q}$ are operads in $\mathrm{C}$, then a \defn{morphism} from a right $\mathcal{P}$-module $\mathcal{M}$ to a right $\mathcal{Q}$-module $\mathcal{N}$ is a pair $(\varphi',\varphi)$, where $\varphi : \mathcal{P} \to \mathcal{Q}$ is an operad morphism and $\varphi' : \mathcal{M} \to \mathcal{N}$ is a morphism of symmetric sequences compatible with the module structures via $\varphi$. That is, for all $k,m \geq 0$ and $1 \leq i \leq k$, the diagram
\[
\begin{tikzcd}
\mathcal{M}(k)\otimes\mathcal{P}(m) \arrow[d, "\varphi'(k)\otimes \varphi(m)"'] \arrow[r, "\circ_i"] & \mathcal{M}(k+m-1) \arrow[d, "\varphi'(k+m-1)"] \\
\mathcal{N}(k)\otimes \mathcal{Q}(m) \arrow[r, "\circ_i"] & \mathcal{N}(k+m-1)
\end{tikzcd}
\]
commutes. In other words, $\varphi'$ intertwines the right $\mathcal{P}$-action on $\mathcal{M}$ with the right $\mathcal{Q}$-action on $\mathcal{N}$ via the operad map $\varphi$. For an extensive treatment of right modules over operads and their morphisms, see~\cite{FresseModule09}.
\end{definition}

\begin{definition}
\label{def:moperad} 
Let $\mathcal{P}$ be a one-colored symmetric operad in $\mathrm{C}$. A $\mathcal{P}$-\defn{moperad} is a monoid object $\calM$ in the category of right $\mathcal{P}$-modules. Equivalently, it consists of a symmetric sequence $\calM=\{\calM(n)\}$ together with two types of partial composition maps:
\begin{enumerate}[leftmargin=*]
\item An associative and unital monoid composition
\[
\begin{tikzcd}
 \circ_0: \calM(k) \otimes \calM(m) \longrightarrow \calM(k+m)
\end{tikzcd}
\]
which makes $\calM$ into a monoid object in the category of symmetric sequences.

\item A right $\mathcal{P}$-module structure
\[
\begin{tikzcd}
 \circ_i: \calM(k) \otimes \mathcal{P}(m) \longrightarrow \calM(k+m-1)
\end{tikzcd}
\]
for $1 \leq i \leq k$.
\end{enumerate}
These two operations are required to be compatible in the following sense:
\begin{enumerate}[leftmargin=*]
    \item[(3)] For any $\lambda \in \calM(k)$, $\mu \in \calM(p)$, and $\nu \in \calP(n)$, we have
   \[
   (\lambda\circ_0\mu) \circ_{i} \nu=
   \begin{cases}
    \lambda \circ_0 (\mu \circ_i \nu), & \text{if } 1 \leq i \leq p, \\
    (\lambda \circ_{i-p} \nu) \circ_0 \mu, & \text{if } p+1 \leq i \leq p+k.
    \end{cases}
    \]
\end{enumerate}

See \cite[Definition 9]{Willwatcher2016Moperads} for more details. 
\end{definition}

\begin{example}
\label{example: shifted moperad}
Given an operad $\mathcal{P}$, one may build a $\mathcal{P}$-moperad $\mathcal{P}^+$ by shifting the operad structure. More precisely, set $\mathcal{P}^+(k) \eqdef \mathcal{P}(k+1)$ for $k\geq 0$, where we visualize operations in $\mathcal{P}(k+1)$ as rooted trees with leaves labeled $\{0,1,\ldots,k\}$. The $\Sigma_k$-action on $\mathcal{P}^+(k)$ is obtained by restricting the $\Sigma_{k+1}$-action on $\mathcal{P}(k+1)$ to permutations that fix $0$. The monoid composition $\circ_0:\mathcal{P}^+(k)\otimes\mathcal{P}^+(n)\to \mathcal{P}^+(k+n)$ is given by operadic composition in the $0$th input, that is, by the map $\circ_0:\mathcal{P}(k+1)\otimes\mathcal{P}(n+1)\to \mathcal{P}(k+n+1)$. Similarly, the right $\mathcal{P}$-module structure maps $\circ_i:\mathcal{P}^+(k)\otimes\mathcal{P}(n)\to \mathcal{P}^+(k+n-1)$ are given by the operadic compositions $\mathcal{P}(k+1)\otimes\mathcal{P}(n)\to \mathcal{P}(k+n)$ for $1\leq i\leq k$.
\end{example}

In a moperad $\calM$, an operation in $\calM(k)$ may be visualized as a planar rooted tree whose leaves are labeled by $\{0,1,\ldots,k\}$, with the distinguished leaf labeled by $0$ drawn as the leftmost leaf. The $\circ_0$-composition is given by inserting an operation of $\calM$ into the distinguished leaf labeled by $0$. For $1\leq i\leq k$, the $\circ_i$-compositions are given by inserting an operation of $\calP$ into the leaf labeled by $i$.

\begin{figure}[H]
    \includegraphics[width=10cm]{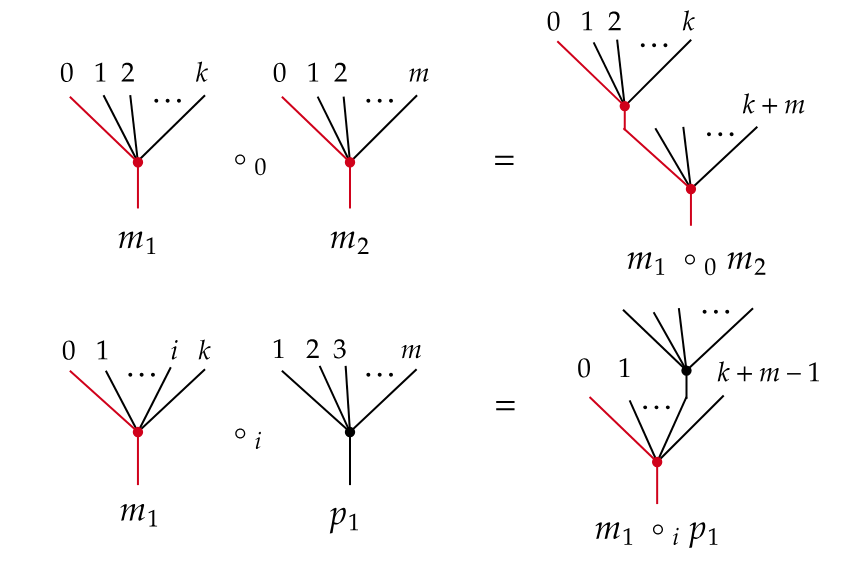}
    \caption{Moperad composition illustrated on rooted trees: here $m_1$ is an operation in $\calM(k)$, $m_2$ is an operation in $\calM(m)$, and $p_1$ is an operation in $\mathcal{P}(m)$.}
    \label{fig:moperad composition}
\end{figure}

\begin{remark}
The picture in Figure~\ref{fig:moperad composition} suggests an equivalent description of a $\calP$-moperad $\calM$ as a two-colored operad in which one color is distinguished, or frozen. Let $\mathfrak{C}=\{\textcolor{red}{a},m\}$ be a discrete set of two colors, and define a $\mathfrak{C}$-colored symmetric sequence $\calO=\{\calO(c_1,\ldots,c_n;c_0)\}$ by
\begin{itemize}[label=\labelitemii]
    \item $\calO(c_1,\ldots,c_n;c_0)=\calP(n)$ if $(c_1,\ldots,c_n;c_0)=(m,\ldots,m;m)$;
    \item $\calO(c_1,\ldots,c_n;c_0)=\calM(n)$ if $(c_1,\ldots,c_n;c_0)=(\textcolor{red}{a},m,\ldots,m;\textcolor{red}{a})$;
    \item $\calO(c_1,\ldots,c_n;c_0)=0$ otherwise, where $0$ denotes the initial object of $\mathrm{C}$.
\end{itemize}
The symmetric group $\Sigma_n$ acts on $\calP(n)$ in the usual way. On $\calM(n)$, the action factors through the subgroup of permutations fixing the distinguished input of color $\textcolor{red}{a}$, that is, through the subgroup $\{\sigma\in \Sigma_n\mid \sigma(1)=1\}\cong \Sigma_{n-1}$. In other words, $\calO$ has two types of operations: algebra operations in the color $m$, governed by the operad $\calP$, and module operations with one distinguished input and output of color $\textcolor{red}{a}$ and all remaining inputs of color $m$. Composition is then defined exactly as for a $\mathfrak{C}$-colored operad, subject to these color and symmetry constraints; see \cite[Definition~8]{Willwatcher2016Moperads}.
\end{remark}

\begin{definition}\label{def:MopMap}
A \defn{map of $\calP$-moperads} $\varphi:\calM\rightarrow \mathcal{N}$ is a map of right $\mathcal{P}$-modules that is compatible with the monoid structures, that is,
\[
\varphi(x\circ_0 y)=\varphi(x)\circ_0\varphi(y)
\]
for all $x\in \calM(k)$ and $y\in \calM(m)$. We write $\Mop_{\calP}(\mathrm{C})$ for the category of $\calP$-moperads in $\mathrm{C}$.

More generally, if $\calP$ and $\calQ$ are operads, then a \defn{map} from a $\calP$-moperad $\calM$ to a $\calQ$-moperad $\mathcal{N}$ is a pair $(\varphi',\varphi)$, where $\varphi:\calP\rightarrow \calQ$ is a map of operads and $\varphi':\calM\rightarrow \mathcal{N}$ is a map of symmetric sequences such that $(\varphi',\varphi)$ is a map of right modules and
\[
\varphi'(x\circ_0 y)=\varphi'(x)\circ_0\varphi'(y)
\]
for all $x\in \calM(k)$ and $y\in \calM(m)$. We write $\Mop(\mathrm{C})$ for the category of moperads in $\mathrm{C}$, where the underlying operad is allowed to vary.
\end{definition}

\begin{remark}
The shift construction of Example~\ref{example: shifted moperad} induces a functor
\[
(-)^{+}:\Op(\mathrm{C})\rightarrow \Mop(\mathrm{C}),
\]
which sends an operad $\calP$ to the $\calP$-moperad $\calP^+$. This functor is faithful: every map of operads $\varphi:\calP\rightarrow \calQ$ induces a map of moperads
\[
(\varphi^+, \varphi):\calP^+\rightarrow\calQ^+,
\]
and distinct operad maps induce distinct moperad maps. However, not every moperad map $\calP^+\rightarrow \calQ^+$ arises by shifting an operad map $\calP\rightarrow\calQ$. This is analogous to the fact that a module homomorphism between rings, viewed as modules, need not preserve multiplication.
\end{remark}

\subsection{Completion and homomorphic expansions for moperads}

Let $\calP$ be an operad and let $\calM=\{\calM(n)\}_{n\geq 0}$ be a $\calP$-moperad. The prounipotent completion of $\calM$ is the $\calP_{\K}$-moperad $\calM_{\K}$ obtained by applying prounipotent completion aritywise:
\[
\calM_{\K}=\{\calM(n)_{\K}\}_{n\geq 0}.
\]
The structure maps of $\calM$ induce the corresponding structure maps on $\calM_{\K}$. In particular, the monoid composition gives maps
\[
\circ_0:\calM(k)_{\K}\otimes \calM(m)_{\K}\longrightarrow \calM(k+m)_{\K},
\]
and the right $\calP$-module structure gives, for each $1\leq i\leq k$, maps
\[
\circ_i:\calM(k)_{\K}\otimes \calP(m)_{\K}\longrightarrow \calM(k+m-1)_{\K}.
\]
These endow $\calM_{\K}$ with the structure of a $\calP_{\K}$-moperad. 

Similarly, applying the $\K$-linear envelope aritywise gives a moperad in $\K$-linear categories
\[
\K[\calM]:=\{\K[\calM(n)]\}_{n\geq 0}
\]
over the operad $\K[\calP]$. The augmentation filtrations induce aritywise completed associated graded objects, and we write
\[
\widehat{\gr}\,\K[\calM]:=\{\widehat{\gr}\,\K[\calM(n)]\}_{n\geq 0}
\]
for the resulting moperad in completed graded $\K$-linear categories over $\widehat{\gr}\,\K[\calP]$.

\begin{definition}
Let $\calP$ be an operad in groupoids, and let $\calM$ be a moperad in groupoids over $\calP$. A \defn{homomorphic expansion} for $\calM$ is a pair $(\widehat{Z}_{\calM},\widehat{Z}_{\calP})$, where $\widehat{Z}_{\calP}$ is a homomorphic expansion for $\calP$ and
\[
\widehat{Z}_{\calM}\colon \K[\calM]^\wedge \xrightarrow{\cong} \widehat{\gr}\,\K[\calM]
\]
is a filtered isomorphism of moperads in $\K$-linear categories compatible with $\widehat{Z}_{\calP}$. Passing to group-like morphisms in each arity then induces an isomorphism of moperads in prounipotent groupoids
\[
\calM_{\K}\xrightarrow{\cong} \grp\bigl(\widehat{\gr}\,\K[\calM]\bigr).
\]
\end{definition}

\subsection{Parenthesized braids with a frozen strand}
The \defn{braid group with a frozen strand}, denoted $\Br_n^1$, is the fundamental group of the configuration space of $n$ unordered points in the punctured plane $\C\setminus\{0\}$. It admits the following finite presentation; see \cite[\S~3]{AET10}.

\begin{prop}
\label{prop:PresForBn1}
The group $\Br_n^1$ is generated by $X_1,X_2,\ldots,X_n$ and $\beta_1,\ldots,\beta_{n-1}$, subject to the relations
\begin{align}
\beta_{i}\beta_{i+1}\beta_{i}&=\beta_{i+1}\beta_{i}\beta_{i+1}, \label{eq:BraidRel1}  \\
\beta_i\beta_j&=\beta_j\beta_i  \quad \text{when} \quad |i-j|\geq 2, \label{eq:BraidRel2} \\
\beta_iX_i\beta_i^{-1}&=X_{i+1},\label{eq:XiXi+1}  \\
\beta_iX_{i+1}\beta_i^{-1}&=X_{i+1}^{-1}X_{i}X_{i+1}, \label{eq:XiComp}\\
\beta_iX_{j}\beta_{i}^{-1}&=X_{j}  \quad \text{when} \quad j\notin\{i,i+1\}. \label{eq:XBetaComm}
\end{align}
\end{prop}

\begin{figure}[ht!]
    \includegraphics[width=15cm]{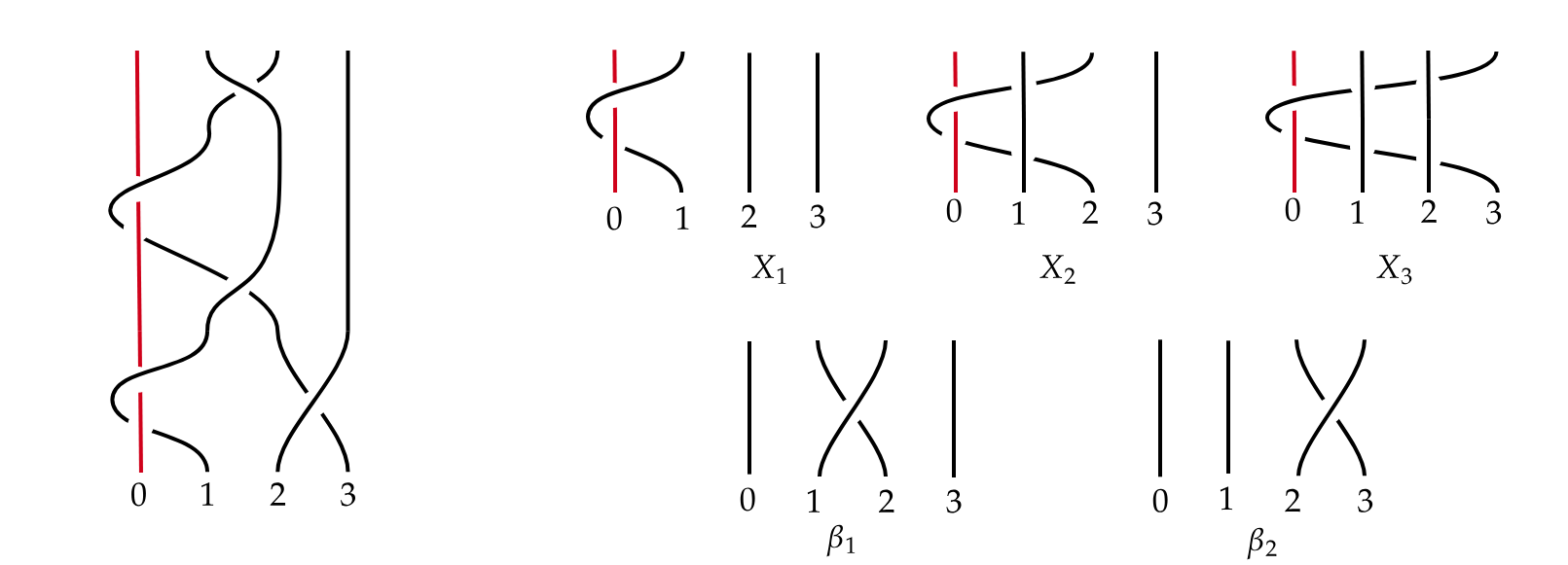}
    \caption{A braid with a frozen strand in $\Br^1_3$, and the generators of $\Br^1_3$. The frozen strand (puncture) is drawn leftmost and shown in red. The generator $X_1$ represents the loop in which the first point circles once around the puncture in the positive direction, and $X_k=\beta_{k-1}X_{k-1}\beta_{k-1}^{-1}$ for $k\geq 2$. The braid on the left can, for example, be written as $\beta_2X_1X_2$.}
    \label{fig:FrozenBraids}
\end{figure}

There is a natural homomorphism $\iota:\Br_n^1\to \Br_{n+1}$ obtained by viewing the puncture as an additional strand. To describe it, consider the standard presentation of $\Br_{n+1}$ with generators $\beta_0,\beta_1,\ldots,\beta_{n-1}$. Define
\[
\iota(\beta_i)\eqdef \beta_i \qquad \text{for } 1\leq i\leq n-1,
\]
and
\[
\iota(X_i)\eqdef \beta_{i-1}\beta_{i-2}\cdots\beta_1\beta_0^{2}\beta_1^{-1}\cdots\beta_{i-2}^{-1}\beta_{i-1}^{-1}
\qquad \text{for } 1\leq i\leq n.
\]
See Figure~\ref{fig:FrozenBraids} for the geometric intuition. On configuration spaces, this corresponds to replacing the puncture by an additional marked point.

\begin{prop}
\label{prof: Bn1 includes in Bn+1}
The map $\iota : \Br^1_n \hookrightarrow \Br_{n+1}$ defined above is an injective group homomorphism.
\end{prop}

\begin{remark}\label{rmk:AltPres}
From the presentation of $\Br^1_n$ in \cref{prop:PresForBn1}, relation~\eqref{eq:XiXi+1} shows that it is enough to take $X_1$ together with the generators $\beta_1,\ldots,\beta_{n-1}$. This yields an equivalent presentation with generators $\{X_1,\beta_1,\ldots,\beta_{n-1}\}$, in which relations \eqref{eq:BraidRel1} and \eqref{eq:BraidRel2} remain unchanged, relation~\eqref{eq:XiXi+1} is omitted, and relations \eqref{eq:XiComp} and \eqref{eq:XBetaComm} are replaced by
\begin{align}
\beta_1X_1\beta_1^{-1} &=X_1^{-1}\beta_1^{-1}X_1\beta_1X_1, \label{eq:XiCompMod}\\
\beta_iX_{1}\beta_{i}^{-1}&=X_{1}  \quad \text{when} \quad i\neq 1. \label{eq:XBetaCommMod}
\end{align}
\end{remark}

Notice that in the presentation of \cref{prop:PresForBn1}, relations~\eqref{eq:BraidRel1} and~\eqref{eq:BraidRel2} are the standard braid relations, while relations~\eqref{eq:XiXi+1}--\eqref{eq:XBetaComm} describe how the generators $X_i$ interact with the crossings $\beta_i$. In particular, there are no relations among the generators $X_1,\ldots,X_n$. Consequently, these elements generate a normal subgroup of $\Br_n^1$ isomorphic to the free group $\F_n$. On the other hand, the subgroup generated by $\beta_1,\ldots,\beta_{n-1}$ is naturally isomorphic to $\Br_n$, and its intersection with the subgroup generated by the $X_i$ is trivial. It follows that
\begin{equation}
\label{eq:iso-Bn1-FnBn}
\Br_n^1\cong \F_n\rtimes \Br_n.
\end{equation}
This is a classical result, also discussed in \cite[Proposition~15]{AET10}, and it will play an important role in later constructions.

\begin{figure}[ht!] 
\includegraphics[width=7cm]{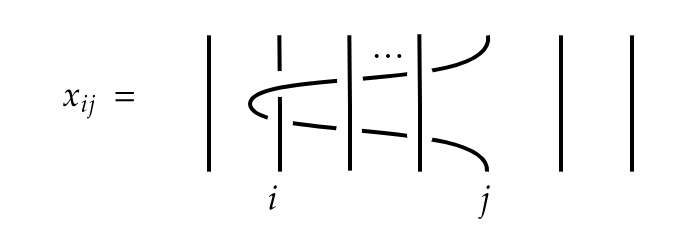} 
\caption{A pure braid generator.} \label{fig:PureGen} 
\end{figure}
The decomposition \eqref{eq:iso-Bn1-FnBn} is closely related to the classical action of the pure braid group $\PB_n$ on the free group $\F_n$. Recall that $\PB_n$ is generated by the elements
\begin{equation}\label{eq:PureBraidGens}
x_{ij}=\beta_{j-1}\cdots\beta_{i+1}\beta_i^2\beta_{i+1}^{-1}\cdots\beta_{j-1}^{-1},
\qquad 1\leq i<j\leq n,
\end{equation}
shown in Figure~\ref{fig:PureGen}. The standard split extension
\[
1\longrightarrow \F_n \longrightarrow \PB_{n+1}\longrightarrow \PB_n \longrightarrow 1
\]
gives the classical semidirect product decomposition
\[
\PB_{n+1}\cong \F_n\rtimes \PB_n,
\]
where the action of $\PB_n$ on $\F_n=\langle y_1,\dots,y_n\rangle$ is given by the usual Artin action:
\[
x_{ij}(y_k)=
\begin{cases}
y_i y_j y_i y_j^{-1} y_i^{-1}, & \text{if } k=i,\\[4pt]
y_i y_j y_i^{-1}, & \text{if } k=j,\\[4pt]
y_k, & \text{if } k\neq i,j.
\end{cases}
\]

This is directly analogous to \eqref{eq:iso-Bn1-FnBn}. In fact, the decomposition $\Br_n^1\cong \F_n\rtimes \Br_n$ restricts to an isomorphism
\begin{equation}\label{eq:PB1}
\PB_n^1\cong \F_n\rtimes \PB_n,
\end{equation} 
where $\PB_n^1$ denotes the pure braid group with a frozen strand.

\begin{lemma}
\label{lemma: iso PB frozen to PB n+1}
There is a natural isomorphism of groups $\PB_n^1\cong \PB_{n+1}$.
\end{lemma}

Let $\Sigma_n^+$ denote the group of automorphisms of the set $\{0,1,\ldots,n\}$. Let $\Sigma_n^1\subset \Sigma_n^+$ denote the stabilizer subgroup of $0$, which is canonically isomorphic to $\Sigma_n$. Given $\sigma \in \Sigma_m^1$ and $\tau \in \Sigma_n$, we can \emph{insert} $\tau$ into the position labeled $i$ of $\sigma$, for $1\leq i\leq m$, using the formulas \eqref{def: composition of permutations}. If $\sigma \in \Sigma_m^1$ and $\tau\in\Sigma_n^1$, then $\tau$ can be inserted into the position labeled $0$ of $\sigma$:
\begin{equation}\label{def: moperad composition of permutations}
(\sigma \circ_0 \tau)(k) \coloneqq
\begin{cases}
0 & \text{if } k = 0, \\
\tau(k) & \text{if } 1 \leq k \leq n, \\
\sigma(k - n) + n & \text{if } n+1 \leq k \leq n+m.
\end{cases}
\end{equation}
The result is a permutation of $\{0,1,\ldots,m+n\}$ which again fixes $0$, and hence lies in $\Sigma^1_{m+n}$. Let $\Omega^1(n)$ denote the set of maximally parenthesized permutations in $\Sigma_n^1$. For example, $0((23)1)$, $(01)(32)$, and $(0(12))3$ are elements of $\Omega^1(3)$.

\begin{definition}
\label{def:MoperadParMerm} 
The symmetric sequence $\Omega^1=\{\Omega^1(n)\}_{n\geq 0}$ forms an $\Omega$-moperad in sets, called the \defn{moperad of parenthesized permutations}. The monoid structure
\[
\begin{tikzcd}
\Omega^1(m) \times \Omega^1(n) \arrow[r, "\circ_0"] & \Omega^1(m+n)
\end{tikzcd}
\]
is given by
\[
(0p)\circ_0(0q)\eqdef ((0q)\,\overline{p}),
\]
for $p\in \Omega(m)$ and $q\in \Omega(n)$, where $\overline{p}$ denotes the parenthesized permutation obtained from $p$ by shifting each label by $n$. The right $\Omega$-module structure
\[
\begin{tikzcd}
\Omega^1(m) \times \Omega(n) \arrow[r, "\circ_i"] & \Omega^1(m+n-1)
\end{tikzcd}
\]
for $1 \leq i \leq m$, is given by replacing the entry labeled $i$ in an element of $\Omega^1(m)$ by a parenthesized permutation in $\Omega(n)$, and then relabeling appropriately. The moperad $\Omega^1$ comes with maps of symmetric sequences
\[
\Omega \longrightarrow \Omega^1 \overset{u}{\longrightarrow} \Sigma^1,
\]
where the first map adjoins the label $0$ in front of an element of $\Omega$, and the second map forgets the parenthesization.
\end{definition}

Next, we define the moperad of parenthesized braids, $\PaB^1$, one of our central objects of study.

\begin{definition}\label{def:Moperad of Parenthesized Braids}
The \defn{moperad of parenthesized braids}, denoted $\PaB^1$, is the $\PaB$-moperad in groupoids defined as follows. Let $\PaB^1(0)\eqdef\{*\}$ be the trivial groupoid. For each $n\geq 1$, let $\PaB^1(n)$ be the groupoid with object set
\[
\ob(\PaB^1(n))\eqdef \Omega^1(n),
\]
and with morphisms $\beta\in\Hom_{\PaB^1(n)}(p,q)$ given by elements of $\Br_n^1$ whose underlying permutation is $u(q)^{-1}u(p)$. Composition in $\PaB^1(n)$ is given by multiplication in $\Br_n^1$ (written right to left; see Remark~\ref{rem:composition}).

The group $\Sigma_n$, identified with the subgroup of $\Sigma_n^+$ fixing $0$, acts on $\PaB^1(n)$ by permuting the nonzero labels. The resulting symmetric sequence $\PaB^1=\{\PaB^1(n)\}_{n\geq 0}$ forms a $\PaB$-moperad in groupoids as follows.
\begin{enumerate}[leftmargin=*]
\item The monoid structure
\[
\begin{tikzcd}
\PaB^1(n)\times \PaB^1(m)\arrow[r, "\circ_0"] & \PaB^1(n+m)
\end{tikzcd}
\]
is given intuitively by inserting the second braid into the frozen $0$-th strand of the first. On objects, $\circ_0$ is the monoid composition of parenthesized permutations. On morphisms, identifying a morphism in arity $n$ with a braid on $n+1$ strands, $\circ_0$ is given by insertion into the $0$-th strand. See Figure~\ref{fig:PaBMonoid} for an example.

\item The right $\PaB$-module structure is given by functors
\[
\begin{tikzcd}
\PaB^1(n)\times \PaB(m)\arrow[r, "\circ_i"] & \PaB^1(n+m-1),
\end{tikzcd}
\]
for $1\leq i\leq n$. On objects, $\circ_i$ is the right action of $\Omega$ on $\Omega^1$ from Definition~\ref{def:MoperadParMerm}. On morphisms, $\circ_i$ is defined by viewing a morphism in arity $n$ as a braid on $n+1$ strands indexed from $0$ to $n$, and inserting into the strand labeled $i$.
\end{enumerate}
\end{definition}

\begin{figure}[ht!]
\includegraphics[width=9cm]{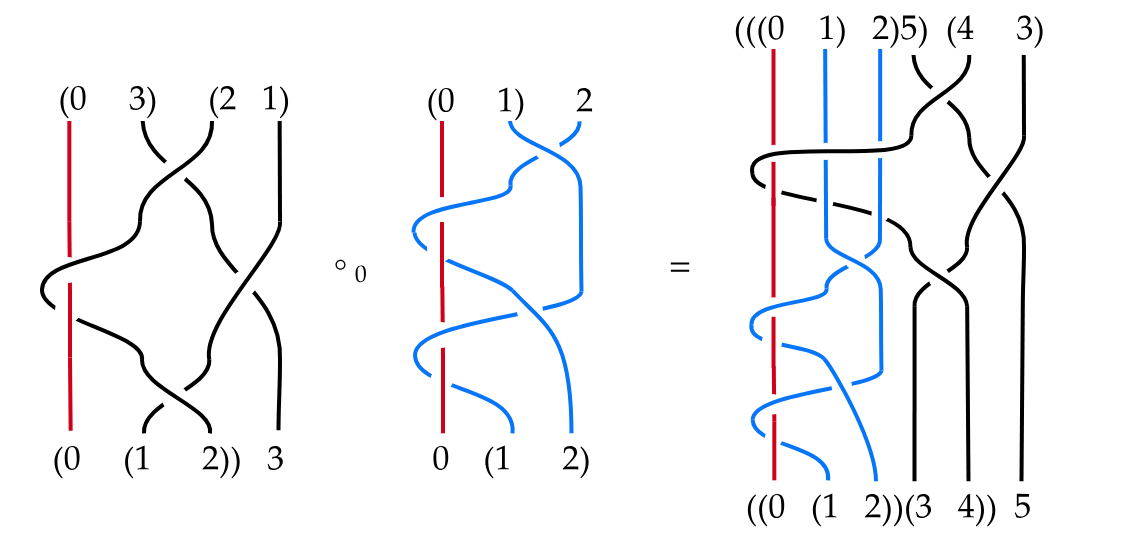}
    \caption{The monoid structure of  $\PaB^1$.}\label{fig:PaBMonoid}
\end{figure}

Like the operad $\PaB$, the $\PaB$-moperad $\PaB^1$ admits a finite presentation \cite[Theorem 3.4]{calaque20moperadic}. In other words, every morphism in $\PaB^1$ can be obtained from a finite number of generating morphisms via operadic, monoidal, and categorical compositions, permutations, and $\PaB$-actions; subject to finitely many relations. 

There is a natural inclusion of symmetric sequences $\PaB\hookrightarrow \PaB^1$ obtained by adjoining a frozen strand labeled $0$ on the left. Under this inclusion, the generator $R^{1,2}$ is sent to $\id_{(01)}\circ_1 R^{1,2}$. Visually, this amounts to adding a frozen strand indexed by $0$, outermost in the parenthesization. We denote the image of $\PaB\subset \PaB^1$, and in particular the generators $R^{1,2}$ and $\Phi^{1,2,3}$, by the same symbols.

The two generating morphisms introduced below describe how the frozen strand interacts with the remaining strands. The groupoid $\PaB^1(1)$ has a single object $(01)$, and its automorphism group
\[
\Hom_{\PaB^1(1)}((01),(01))\cong \Br_1^1\cong \mathbb{Z}
\]
is freely generated by an automorphism $E^{0,1}$, shown in Figure~\ref{fig:PaB1Gens}. The second generator is the associativity isomorphism
\[
\Psi^{0,1,2}\in \Hom_{\PaB^1(2)}((01)2,0(12)),
\]
also shown in Figure~\ref{fig:PaB1Gens}.

\begin{figure}[ht!]
    \includegraphics[width=7cm]{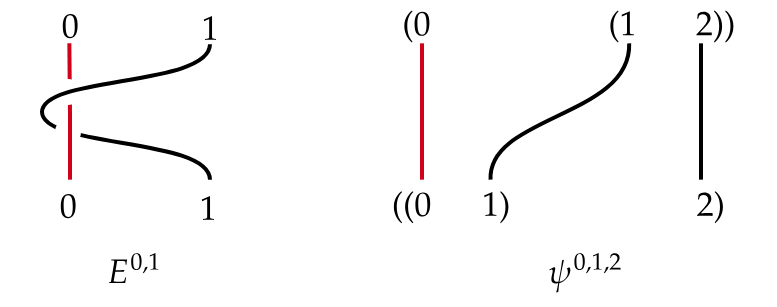}
\caption{The generators of $\PaB^1$ as a $\PaB$-moperad.}\label{fig:PaB1Gens}
\end{figure}

\begin{theorem}[{\cite[Theorem~3.4]{calaque20moperadic}}]
\label{presenation of PaB1}
As a $\PaB$-moperad, $\PaB^1$ is generated by the object $(01)$ and the morphisms
\[
E^{0,1}\in \Hom_{\PaB^1(1)}((01),(01))
\qquad\text{and}\qquad
\Psi^{0,1,2}\in \Hom_{\PaB^1(2)}((01)2,0(12)),
\]
subject to the relations
\begin{align}
\label{eqn:cU}
\tag{U}
\Psi^{0,\emptyset,2} = \Psi^{0,1,\emptyset}= \id_{(01)}
& \quad \mathrm{in}~\Hom_{\PaB^1(1)}\big((01),(01)\big), \\
\label{eqn:MP}
\tag{MP}
\Psi^{0,1,23} \Psi^{01,2,3} =  \Phi^{1,2,3} \Psi^{0,12,3} \Psi^{0,1,2}
& \quad \mathrm{in}~\Hom_{\PaB^1(3)}\big(((01)2)3,0(1(23))\big), \\
\label{eqn:RP}
\tag{RP}
(\Psi^{0,1,2})^{-1}R^{2,1}E^{0,21}R^{1,2}\Psi^{0,1,2}= E^{01,2} E^{0,1}
& \quad \mathrm{in}~\Hom_{\PaB^1(2)}\big((01)2,(01)2\big), \\
\label{eqn:O}
\tag{O}
E^{01,2} = (\Psi^{0,1,2})^{-1} R^{2,1} \Psi^{0,2,1} E^{0,2} (\Psi^{0,2,1})^{-1} R^{1,2} \Psi^{0,1,2}
& \quad \mathrm{in}~\Hom_{\PaB^1(2)}\big((01)2,(01)2\big).
\end{align}
\end{theorem}

\medskip

\begin{remark}
Here, as above, we use cosimplicial notation to describe monoid and operadic compositions more concisely. For example, $E^{01,2}$ denotes the result of applying a $\circ_0$-composition to $E^{0,1}$ with an identity braid: \[E^{01,2}=E^{01}\circ_0 \id_{(01)}.\] As before, the notation $\emptyset$ in relation~\eqref{eqn:cU} denotes composition with the arity-zero operation $*\in \PaB(0)$, which has the effect of deleting a strand.
\end{remark}
\medskip

\begin{remark}\label{remark: difference in presentations of PaB1}
Our presentation of $\PaB^1$ differs slightly from that of \cite[Theorem~3.4]{calaque20moperadic} because we use the blackboard framing in the definition of $E^{0,21}$, and more generally when inserting into strands with nonzero winding number around the puncture. By contrast, \cite{calaque20moperadic} uses the annular framing, which adds a full twist to $E^{0,21}$ relative to the blackboard framing; see Figure~\ref{fig:framing}. The corresponding operadic composition on the associated graded side is therefore also slightly different; see Remarks~3.1,~3.6, and~5.1 of \cite{calaque20moperadic}. As a result, our relation~\eqref{eqn:RP} contains additional half-twists compared to the formulation in \cite{calaque20moperadic}. The two versions are equivalent. The annular framing is more natural from the perspective of cyclotomic actions, whereas the blackboard framing is better suited to compatibility with the KV problem.
\end{remark}

\begin{figure}[ht!]
    \centering
    \includegraphics[width=0.5\linewidth]{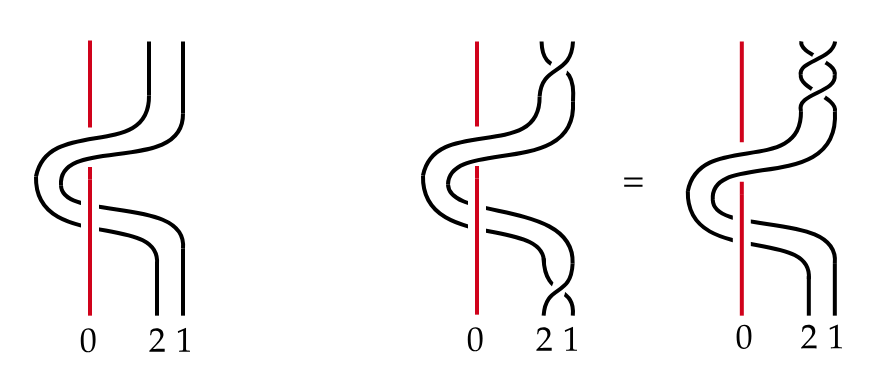}
    \caption{On the left, $E^{0,21}$ using the blackboard framing convention. On the right $E^{0,21}$ in the annular framing convention of \cite{calaque20moperadic}.}
    \label{fig:framing}
\end{figure}

At each parenthesized permutation $w\in \ob(\PaB^1(n))$ there is an isomorphism
\[
\Aut_{\PaB^1(n)}(w)\cong \PB_n^1 \cong \F_n\rtimes \PB_n.
\]
Elements of $\PB_n^1$, viewed as morphisms in $\PaB^1(n)$, can be expressed in terms of the moperad generators of $\PaB^1$, although such expressions are only unique modulo the relations of Theorem~\ref{presenation of PaB1}. Recall that the free group $\F_n$ embeds in $\PB_n^1$, with generators represented by the elements $X_1,\ldots,X_n$, as in Figure~\ref{fig:Xis}. These elements will play an important role, and the following lemma describes how to express them in terms of the generators of $\PaB^1$.

\begin{figure}
   \includegraphics[width=12cm]{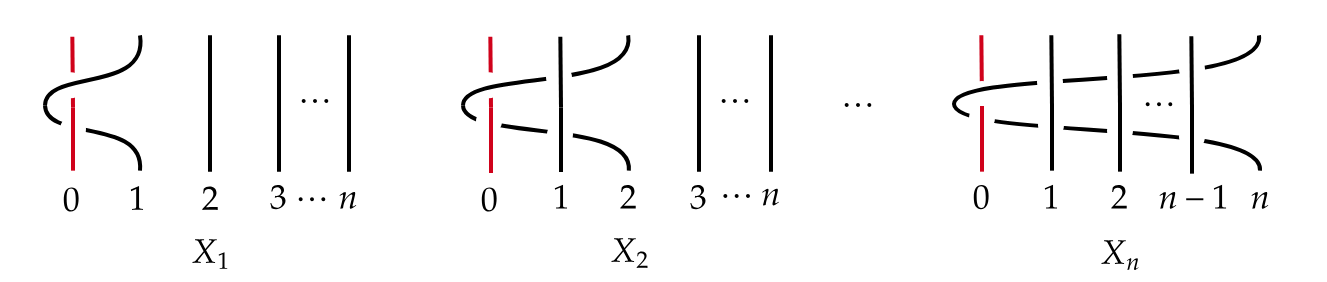}
    \caption{The free group generators $X_1,\ldots,X_n$ in $\PB^1_n$.}\label{fig:Xis}
\end{figure}

\begin{notation}
For a permutation $\sigma\in \Sigma_n$, write
\[
r_n(\sigma)\eqdef (i_1(i_2\ldots(i_{n-1}i_n)))
\qquad\text{and}\qquad
l_n(\sigma)\eqdef (((i_1i_2)i_3\ldots)i_n),
\]
where $\sigma=(i_1,\ldots,i_n)$ in one-line notation. Thus $r_n(\sigma)$ and $l_n(\sigma)$ denote the rightmost and leftmost parenthesizations of $\sigma$, respectively.

Define
\[
r_n^1(\sigma)\eqdef \id_{(01)}\circ_1 r_n(\sigma)\in \ob(\PaB^1(n))
\]
and
\[
l_n^1(\sigma)\eqdef \id_{(01)}\circ_1 l_n(\sigma)
= \left(0\left((i_1i_2)i_3\ldots\right)i_n\right).
\]
Finally, let $l_n^0(\sigma)$ denote the fully left-parenthesized form of the concatenated word $0\sigma$.

When $\sigma$ is the identity permutation, we omit it from the notation; for example, we write $r_n^1$ for $r_n^1(\id)$.
\end{notation}

\begin{lemma}
\label{l:generators-normal-form}
For each $1\leq k\leq n$, let $P_k$ denote the reassociation morphism from $l_n^0(k\,1\,2\ldots\,k-1\,k+1\ldots n)$ to $l_n^1(k\,1\,2\ldots\,k-1\,k+1\ldots n),$
given by
\begin{equation}
\label{eq:Pk}
P_k \eqdef \Psi^{0,l_{n-1}(k\,1\,2\ldots\,k-1\,k+1\ldots n-1),n}\cdots \Psi^{0,((k\,1)\,2),3}\Psi^{0,(k\,1),2}\Psi^{0,k,1}.
\end{equation}
Let $Q_k$ denote the morphism from $l_n^1(k\,1\,2\ldots\,k-1\,k+1\ldots n)$ to $r_n^1$ which reassociates to the right while moving strand $k$ under the strands $1,\ldots,k-1$, namely
\begin{multline}
\label{eq:Qk}
Q_k \eqdef \Phi^{n-2,n-1,n}\cdots(\Phi^{k-1,k,k+1}\cdots\Phi^{2,(3(4\ldots(k-1\;k))),k+1}\Phi^{1,(2(3\ldots(k-1\;k))),k+1})\cdot\\
(R^{k-1,k})^{-1}\cdots((R^{4,k})^{-1}\Phi^{3,k,4}\Phi^{2,(3\;k),4}\Phi^{1,2(3\;k),4})((R^{3,k})^{-1}\Phi^{2,k,3}\Phi^{1,(2\;k),3})((R^{2,k})^{-1}\Phi^{1,k,2})(R^{1,k})^{-1}.
\end{multline}
Then the free group generators $X_1,\ldots,X_n$, viewed as automorphisms in $\Hom_{\PaB^1(n)}(r_n^1,r_n^1)$, are given by
\[
X_k
=
Q_kP_k\bigl(\id_{l_n^0(1\,2\ldots k-1\,k+1\ldots n)}\circ_0 E^{0,k}\bigr)P_k^{-1}Q_k^{-1}.
\]
\end{lemma}

\begin{proof}
The identity
\[
X_k=Q_kP_k\bigl(\id_{l_n^0(1\,2\ldots k-1\,k+1\ldots n)}\circ_0 E^{0,k}\bigr)P_k^{-1}Q_k^{-1}
\]
is an equality in $\Hom_{\PaB^1(n)}(r_n^1,r_n^1)\cong \PB_n^1$. Thus there are two things to check:
\begin{enumerate}[leftmargin=*]
    \item the morphisms on the right-hand side, as defined in \eqref{eq:Pk} and \eqref{eq:Qk}, form a well-defined composition in $\PaB^1$; and
    \item under the identification $\Hom_{\PaB^1(n)}(r_n^1,r_n^1)\cong \PB_n^1$, the resulting morphism represents the same braid as $X_k$.
\end{enumerate}

The first point is a straightforward combinatorial verification, illustrated in Figure~\ref{fig:Xk}.

For the second point, note that all associators are trivial as braids. After forgetting the parenthesization, the claimed identity therefore reduces to
\[
X_k=(R^{k-1,k})^{-1}(R^{k-2,k})^{-1}\cdots(R^{1,k})^{-1}E^{0,k}R^{1,k}\cdots R^{k-2,k}R^{k-1,k}.
\]
This is precisely the usual expression for $X_k\in \Br_n^1$: it is obtained by conjugating the basic loop $E^{0,k}$ by the successive crossings that move strand $k$ past strands $1,\ldots,k-1$. Hence both sides define the same element of $\PB_n^1$.
\end{proof}

\begin{figure}[ht!]
    \includegraphics[width=9cm]{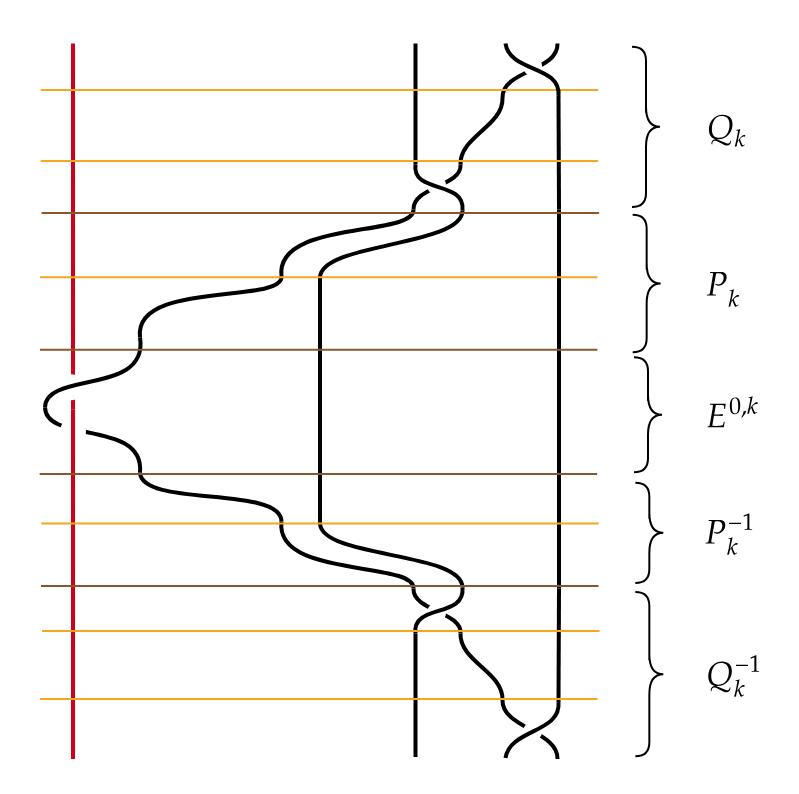}
    \caption{The expression of $X_3$ in $\PaB^1_3$ in terms of the generators.}
    \label{fig:Xk}
\end{figure}

\subsection{Associators and moperad expansions} \label{subsec:assocmoperad}
The isomorphism $\PB_n^1\cong \PB_{n+1}$ is mirrored on the graded side by an isomorphism of Lie algebras $\ib_n^1\cong \ib_{n+1}$ obtained by reindexing the generators of the Drinfeld--Kohno Lie algebra discussed in Section~\ref{subsec:chorddiag}. Thus $\ib_n^1$ may be described as the Lie algebra generated by $\{t_{ij}\mid 0\le i,j\le n,\ i\neq j\}$, subject to the usual Drinfeld--Kohno relations. In fact, with the shift notation of Example~\ref{example: shifted moperad}, we identify $\ib_n^+$ with $\ib_n^1$.

\begin{definition}\label{def: moperad of shifted chord diagrams}
Let $\Sigma_n^+$ denote the group of automorphisms of the set $\{0,1,\ldots,n\}$, and let $\Sigma_n\subset \Sigma_n^+$ be the subgroup fixing $0$. The action of $\Sigma_n$ on $\ib_n^+$ is given by
\[
(t_{ij})^\sigma \eqdef t_{\sigma^{-1}(i)\sigma^{-1}(j)},
\]
$\sigma\in \Sigma_n.$ With this action, the symmetric sequence $\ib^+=\{\ib_n^+\}_{n\ge 0}$ forms an $\ib$-moperad in degree-complete Lie algebras. Both the monoid product
\[
\ib_n^+\circ_0 \ib_m^+ \longrightarrow \ib_{n+m}^+
\]
and the right $\ib$-module structure maps
\[
\ib_n^+\circ_i \ib_m \longrightarrow \ib_{n+m-1}^+,
\]
$1\le i\le n,$ are induced by the operadic composition in $\ib$ from Definition~\ref{def: operad of inft braids}.

Now define $\mathsf{CD}^+(n)\eqdef \exp(\ib_n^+) \cong \exp(\ib_{n+1}).$ The resulting symmetric sequence $\mathsf{CD}^+=\{\mathsf{CD}^+(n)\}_{n\ge 0}$ is the $\mathsf{CD}$-\defn{moperad of chord diagrams}. It is a $\mathsf{CD}$-moperad in prounipotent groups, with monoid product and right $\mathsf{CD}$-action induced from the moperad structure on $\ib^+$ via
\[
e^x\circ_0 e^y \eqdef e^{x\circ_0 y}
\qquad \text{and} \qquad
e^x\circ_i e^y \eqdef e^{x\circ_i y}.
\]
\end{definition}

As in the operadic description of associators (\cref{cor: operadic associators}), we use the presentation of $\PaB^1$ as a $\PaB$-moperad to classify moperad equivalences
\[
(\varphi^1,\varphi):\hPaB^1\longrightarrow \mathsf{CD}^+.
\]
Such equivalences induce unique homomorphic expansions valued in $\PaCD^+$, by an argument similar to that of \cite[Theorem~10.3.12]{FresseBook1}. The following theorem is a simplified version of \cite[Theorem~5.5]{calaque20moperadic}, where the authors give a moperadic description of Enriquez's cyclotomic associators.

\begin{thm}
\label{thm: assoc^1}
A moperad equivalence
\[
(\varphi^1,\varphi):\hPaB^1\longrightarrow \mathsf{CD}^+
\]
is uniquely determined by the values
\[
\varphi(R^{1,2}) = e^{\frac{\mu t_{12}}{2}} \in \exp(\ib_2)
\qquad \text{and} \qquad
\varphi(\Phi^{1,2,3}) = f(t_{12}, t_{23}) \in \exp(\ib_3),
\]
which determine an operad equivalence $\varphi: \hPaB\to \mathsf{CD}$ and hence a Drinfeld associator, together with the values
\[
\varphi^1(E^{0,1}) = e^{\mu t_{01}} \in \exp(\ib_1^+)
\qquad \text{and} \qquad
\varphi^1(\Psi^{0,1,2}) = g(t_{01},t_{12}) \in \exp(\ib_2^+),
\]
where $g$ satisfies the equations
\begin{equation}
\label{eq: MP moperad iso}
\tag{MP}
g(t_{01}, t_{12} + t_{13}) \cdot g(t_{02} + t_{12}, t_{23})
= f(t_{12}, t_{23}) \cdot g(t_{01} + t_{02}, t_{13} + t_{23}) \cdot g(t_{01}, t_{12})
\quad \text{in } \exp(\mathfrak{t}_3^+).
\end{equation}

\begin{equation}
\label{eq: hexagons moperad iso}
\tag{O}
g(t_{01}, t_{12})^{-1} \cdot e^{\frac{\mu t_{12}}{2}} \cdot g(t_{02}, t_{12}) \cdot e^{\mu t_{02}}
\cdot g(t_{02}, t_{12})^{-1} \cdot e^{\frac{\mu t_{12}}{2}} \cdot g(t_{01}, t_{12}) \cdot e^{\mu t_{01}}
= 1
\quad \text{in } \exp(\mathfrak{t}_2^+).
\end{equation}
\end{thm}

\medskip

Every moperad equivalence $(\varphi^1,\varphi):\hPaB^1 \longrightarrow \mathsf{CD}^+$ lies over an operad equivalence $\varphi:\hPaB\rightarrow \CD$, which is in turn indexed by a Drinfeld associator. Conversely, we now show that any operad equivalence $\varphi:\hPaB\rightarrow \CD$ gives rise to a one-parameter family of moperad equivalences $(\varphi^1,\varphi):\hPaB^1 \longrightarrow \mathsf{CD}^+$. The first ingredient is the following inclusion of moperads.

\begin{lemma}
\label{lemma: inclusion PaB1 to PaB+}
The group inclusion $\Br_n^1 \hookrightarrow \Br_{n+1}$ extends to an injective map of $\PaB$-moperads $$(\rho^1,\id):\PaB^1 \hookrightarrow \PaB^+.$$ 
\end{lemma}

\begin{proof}
We define the map of $\PaB$-moperads $(\rho^1,\id): \PaB^1\rightarrow \PaB^+$ to be the identity on objects. Guided by the presentation of the $\PaB$-moperad $\PaB^1$ in \cref{presenation of PaB1} and by the subgroup inclusion $\Br_n^1\hookrightarrow \Br_{n+1}$ from Proposition~\ref{prof: Bn1 includes in Bn+1}, we set
\[
\rho^1(E^{0,1}) \eqdef R^{1,0}R^{0,1}
\quad \text{and} \quad
\rho^1(\Psi^{0,1,2}) \eqdef \Phi^{0,1,2}
\]
in $\Hom_{\PaB^+(1)}((01),(01))$ and $\Hom_{\PaB^+(2)}((01)2,0(12))$, respectively. Here $R^{0,1}$ in $\PaB^+(1)$ corresponds to $R^{1,2}$ in $\PaB(2)$, and $\Phi^{0,1,2}$ in $\PaB^+(2)$ corresponds to $\Phi^{1,2,3}$ in $\PaB(3)$.

To show that these assignments define a well-defined map of moperads, it suffices to check that the defining relations of $\PaB^1$ are preserved. Under $\rho^1$, the mixed pentagon relation \eqref{eqn:MP} becomes the pentagon relation \eqref{eqn:P}. Likewise, the relations \eqref{eqn:O} and \eqref{eqn:RP} reduce to consequences of the hexagon relations \eqref{eqn:H1} and \eqref{eqn:H2}.

More precisely, using~\eqref{eqn:H1} and~\eqref{eqn:H2} we have
\begin{align*}
\rho^{1}\left(E^{01,2}\right) 
= (R^{2,01})(R^{01,2}) 
&=
\left(
(\Phi^{0,1,2})^{-1}R^{2,1}\Phi^{0,2,1}R^{2,0}(\Phi^{2,0,1})^{-1}
\right) 
\left(\Phi^{2,0,1}R^{0,2}(\Phi^{0,2,1})^{-1}R^{1,2}\Phi^{0,1,2} 
\right)\\
&=
(\Phi^{0,1,2})^{-1}R^{2,1}\Phi^{0,2,1}R^{2,0}R^{0,2}(\Phi^{0,2,1})^{-1}R^{1,2}\Phi^{0,1,2} \\
&= \rho^{1}\left((\Psi^{0,1,2})^{-1} R^{2,1} \Psi^{0,2,1} E^{0,2} (\Psi^{0,2,1})^{-1} R^{1,2} \Psi^{0,1,2})\right),
\end{align*} 
which shows that $\rho^1$ preserves the relation~\eqref{eqn:O} from Theorem~\ref{presenation of PaB1}.
To show that $\rho^{1}$ preserves equation \eqref{eqn:RP}, using first the braid relations, we compute:
\begin{align*}
    \rho^1\left(E^{01,2}E^{0,1}\right) 
    &= R^{2,01}R^{01,2}R^{1,0}R^{0,1} = R^{2,01}R^{1,0}R^{0,1}R^{01,2} \\
    &=\underbrace{(\Phi^{0,1,2})^{-1}R^{2,1}\Phi^{0,2,1}R^{2,0}(\Phi^{2,0,1})^{-1}}_{\eqref{eqn:H1}}R^{1,0}R^{0,1}\underbrace{\Phi^{2,0,1}R^{0,2}(\Phi^{0,2,1})^{-1}R^{1,2}\Phi^{0,1,2}}_{\eqref{eqn:H2}} \\
    &=(\Phi^{0,1,2})^{-1}R^{2,1}\underbrace{\Phi^{0,2,1}R^{2,0}(\Phi^{2,0,1})^{-1}R^{1,0}\Phi^{2,1,0}}_{\eqref{eqn:H2}}\underbrace{(\Phi^{2,1,0})^{-1}R^{0,1}\Phi^{2,0,1}R^{0,2}(\Phi^{0,2,1})^{-1}}_{\eqref{eqn:H1}}R^{1,2}\Phi^{0,1,2}\\
    &=(\Phi^{0,1,2})^{-1}R^{2,1}R^{21,0}R^{0,21}R^{1,2}\Phi^{0,1,2} =
    \rho^{1}\left((\Psi^{0,1,2})^{-1}R^{2,1}E^{0,21}R^{1,2}\Psi^{0,1,2}\right).
\end{align*}
The lemma follows. 
\end{proof}

\begin{remark}
The moperad $\PaB^1$ is not equivalent to the shifted moperad $\PaB^+$. Indeed, $\PaB^+$ contains morphisms that permute the frozen strand with the other strands, whereas $\PaB^1$ does not. For example, the morphism
\[
R^{0,1}\colon (01)\to (10)
\]
exists in $\PaB^+$ but not in $\PaB^1$.
\end{remark}

\begin{prop}
\label{prop: moperad isos are associators}
Every equivalence of operads $\varphi:\hPaB\to \CD$ extends to an equivalence of moperads
\[
(\varphi^1,\varphi):\hPaB^1\longrightarrow \mathsf{CD}^+.
\]
\end{prop}

\begin{proof}
Given an equivalence of operads $\varphi:\hPaB \longrightarrow \mathsf{CD}$, the shift construction yields an equivalence of moperads $(\varphi^+,\varphi):\hPaB^+\rightarrow \CD^+.$  Composing with the inclusion of \cref{lemma: inclusion PaB1 to PaB+}, we define
\[
(\varphi^1,\varphi)\eqdef (\varphi^{+},\varphi)\circ (\rho^1,\id):
\hPaB^1\longrightarrow \CD^+.
\]
To show that this is a moperad equivalence, it remains to verify that the values of $\varphi^1(E^{0,1})$ and $\varphi^1(\Psi^{0,1,2})$ satisfy the equations of \cref{thm: assoc^1}.

Since $(\rho^{1},\id):\hPaB^1\rightarrow \hPaB^+$ sends the associativity morphism $\Psi^{0,1,2}$ to the shifted associator $\Phi^{0,1,2}$, we have
\[
\varphi^1(\Psi^{0,1,2})=\varphi^+(\Phi^{0,1,2})=f(t_{01},t_{12})\in\exp(\ib_2^+).
\]
It follows that $\varphi^1(\Psi^{0,1,2})$ satisfies equation~\eqref{eq: MP moperad iso}.

Moreover, using the hexagon equation~\eqref{eq:DA-Hex}, we obtain
\[
\varphi^1(E^{0,1})
=
\varphi^+(R^{1,0})\varphi^+(R^{0,1})
=
e^{\frac{\mu t_{10}}{2}}e^{\frac{\mu t_{01}}{2}}
=
e^{\mu t_{01}}
\in \exp(\ib_1^+),
\]
since $t_{10}=t_{01}$. This satisfies
\begin{align*}
 e^{\mu t_{01}}
&=e^{\frac{\mu t_{01}}{2}}e^{\frac{\mu t_{10}}{2}}\\
&=\underbrace{f(t_{01},t_{12})^{-1}e^{\frac{-\mu t_{12}}{2}}f(t_{12},t_{02})^{-1}e^{\frac{-\mu t_{02}}{2}}f(t_{02},t_{01})^{-1}}_{\eqref{eq:DA-Hex}}
\cdot
\underbrace{f(t_{01},t_{02})^{-1}e^{\frac{-\mu t_{02}}{2}}f(t_{02},t_{12})^{-1}e^{\frac{-\mu t_{12}}{2}}f(t_{12},t_{01})^{-1}}_{\eqref{eq:DA-Hex}}\\
&=
f(t_{01},t_{12})^{-1}e^{\frac{-\mu t_{12}}{2}}f(t_{12},t_{02})^{-1}e^{-\mu t_{02}}f(t_{02},t_{12})^{-1}e^{\frac{-\mu t_{12}}{2}}f(t_{12},t_{01})^{-1}.
\end{align*}
Using $f(t_{12},t_{01})^{-1}=f(t_{01},t_{12})$ and $f(t_{12},t_{02})^{-1}=f(t_{02},t_{12})$, we conclude that $\varphi^1(E^{0,1})$ satisfies equation~\eqref{eq: hexagons moperad iso}. Hence $(\varphi^1,\varphi)$ defines the required moperad equivalence.
\end{proof}

In fact, for a fixed operad equivalence $\varphi:\hPaB \to \CD$, the set of moperad equivalences
\[
(\varphi^1,\varphi):\hPaB^1 \to \CD^+
\]
extending $\varphi$ forms a one-parameter family. Indeed, by Theorem~\ref{thm: assoc^1}, once the value of $\varphi(\Phi^{1,2,3})$ is fixed, there remains only a one-parameter family of possible values for $\varphi^1(\Psi^{0,1,2})$, as we explain below. This observation appears in a more specialized form in \cite[Section~5]{Enriquez2007cyclotomic}, in the context of cyclotomic associators, and was communicated to the authors by Benjamin Enriquez.

\begin{prop}
\label{Enriquez-wisdom}
If $f(x,y), g(x,y) \in \exp(\lie_2)$ satisfy equation (\ref{eq: MP moperad iso}) from Theorem~\ref{thm: assoc^1}, then $g(x,y)=e^{sy}f(x,y)$ for some $s \in \K$.
\end{prop}

\begin{proof}
Consider the projection $\pi_0: \exp(\mathfrak{t}_3^+) \rightarrow \exp(\mathfrak{t}_3)$ defined by sending $t_{ij}\mapsto t_{ij}$ if $i>0$, and $t_{0j}\mapsto 0$ if $i=0$. Applying $\pi_0$ to equation~\eqref{eq: MP moperad iso}, we obtain
\[
g(0, t_{12} + t_{13}) \cdot g(t_{12}, t_{23}) 
= f(t_{12}, t_{23}) \cdot g(0, t_{13} + t_{23}) \cdot g(0, t_{12}).
\]
Since $g(x,y)\in \exp(\lie_2)$, the specialization $g(0,y)$ lies in the one-generator subgroup $\exp(\K y)$, and hence has the form $g(0,y)=e^{sy}$ for some $s\in \K$. Substituting this into the previous equation gives
\[
e^{s(t_{12} + t_{13})} \cdot g(t_{12}, t_{23})
= f(t_{12}, t_{23}) \cdot e^{s(t_{13} + t_{23})} \cdot e^{st_{12}}.
\]
Now $z=t_{12}+t_{13}+t_{23}$ is central in $\ib_3$, so this becomes
\[
e^{s(z-t_{23})}\cdot g(t_{12}, t_{23})
=
f(t_{12}, t_{23})\cdot e^{s(z-t_{12})}\cdot e^{st_{12}},
\]
and therefore $g(t_{12}, t_{23}) = e^{st_{23}} \cdot f(t_{12}, t_{23}).$
\end{proof}

As an immediate consequence, moperad equivalences $\hPaB^1\to \CD^+$ are classified by Drinfeld associators together with one additional scalar parameter.

\begin{cor}
\label{associators1-associators-k}
There is a natural bijection
\[
\bigl\{(\varphi^1,\varphi): \hPaB^1 \overset{\simeq}{\longrightarrow} \CD^+\bigr\}
\cong
\bigl\{\varphi: \hPaB \overset{\simeq}{\longrightarrow} \CD\bigr\}\times \K.
\]
\end{cor}
\medskip

As in the previous section, we identify compositions of morphisms in $\Aut_{\PaB^1(2)}\bigl(0(12)\bigr)\cong \PB_2^1$ with braid multiplication. The following key relation will later be essential in relating moperad maps to Kashiwara--Vergne solutions.
\begin{lemma}\label{lemma: KV1 for braids}
In the automorphism group $\Aut_{\PaB^1(2)}\bigl(0(12)\bigr)\cong \PB_2^1,$
we have $X_2\circ X_1=E^{0,12}.$
\end{lemma}

\begin{proof}
Since
\[
\Aut_{\PaB^1(2)}\bigl(0(12)\bigr)
=
\Hom_{\PaB^1(2)}\bigl(0(12),0(12)\bigr)
\cong \PB_2^1 \subseteq \Br_2^1,
\]
it suffices to show that the two automorphisms correspond to the same braid in $\Br_2^1$. Geometrically, this is illustrated by the loop homotopy in Figure~\ref{fig:BraidKV1}. For an algebraic verification, we use the inclusion $\PaB^1\hookrightarrow \PaB^+$ from Lemma~\ref{lemma: inclusion PaB1 to PaB+}, and hence work in $\Br_3$ with strands indexed by $0,1,2$. Then
\begin{align*}
X_2X_1
&= \beta_1^{-1}\beta_0^2\beta_1\beta_0^2 \\
&= \beta_1^{-1}\beta_0\beta_1\beta_0\beta_1\beta_0 \\
&= \beta_1^{-1}\beta_1\beta_0\beta_1^2\beta_0 \\
&= \beta_0\beta_1^2\beta_0 \\
&= E^{0,12},
\end{align*}
as required.
\end{proof}

\begin{figure}[ht!]
\includegraphics[width=6cm]{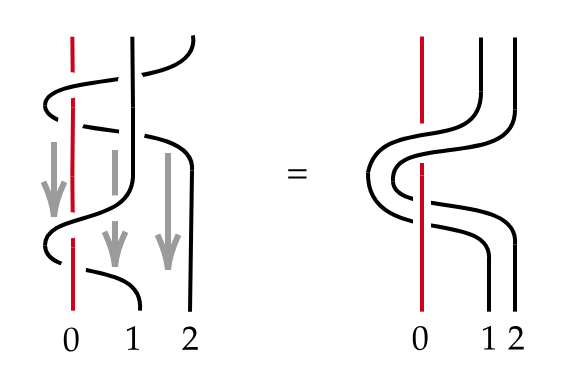}
\caption{A loop homotopy showing that $X_2X_1=E^{0,12}$.}
\label{fig:BraidKV1}
\end{figure}

\subsection{A variation on the Grothendieck--Teichm\"uller groups}\label{sec: GT1}
The Grothendieck--Teichm\"uller group $\gt$ is defined by pairs $(\lambda,f)$ satisfying Drinfeld's duality, hexagon, and pentagon relations (Definition~\ref{defn: GT}). Such a pair induces an automorphism of the free group $\F_2$ via
\[
x \mapsto f(x,y)x^\lambda f(x,y)^{-1}
\qquad \text{and} \qquad
y \mapsto y^\lambda .
\]
In the discrete setting, these relations are highly rigid: Drinfeld~\cite[Proposition~4.1]{Drin} shows that if $\lambda\in 2\mathbb Z+1$ and $f\in F_2$, then necessarily $(\lambda,f)=(1,1)$ or $(-1,1)$. After prounipotent completion, however, one obtains the much richer pro-algebraic Grothendieck--Teichm\"uller group from Definition~\ref{defn: GT}, whose prounipotent part is the subgroup with $\lambda=1$.

Enriquez introduced a cyclotomic analogue of the Grothendieck--Teichm\"uller group by adjoining an additional module-type datum to the pair $(\lambda,f)$. This Grothendieck--Teichm\"uller module group, denoted $\gtm$, is defined by quadruples $(\lambda,\mu,f,g)$ satisfying the usual Grothendieck--Teichm\"uller relations on $(\lambda,f)$, together with an octagon relation and a mixed pentagon relation. As observed just before Proposition~5.3 of~\cite{Enriquez2007cyclotomic}, the octagon relation implies that the assignment
\[
x \mapsto g(x,y)x^\mu g(x,y)^{-1},
\qquad
y \mapsto y^\lambda,
\]
defines an automorphism of $\F_2$, sending $x^{-1}y^{-1}$ to a conjugate of $(x^{-1}y^{-1})^\mu$. Moreover, the defining relations force $\mu=\lambda$, so that one may equivalently describe $\gtm$ using triples $(\lambda,f,g)$. This leads to the following prounipotent version of~\cite[Proposition~5.4]{Enriquez2007cyclotomic}.

\begin{definition}
\label{defn: GTRM}
The group $\gtm$ consists of invertible triples
\[
(\lambda,f,g)\in \K^\times \times (\F_2)_\K \times (\F_2)_\K
\]
such that $(\lambda,f)\in \gt$ and $g\in (\F_2)_\K$ satisfy the following equations:
\begin{gather}
g(x_{01}, x_{13}x_{12})\, g(x_{12}x_{02}, x_{23})
=
f(x_{12},x_{23})\, g(x_{02}x_{01},x_{23}x_{13})\, g(x_{01},x_{12})
\quad \text{in } (\PB_3^1)_\K,
\label{GT^1:mixed pentagon} \tag{MP}
\\
g(x,y)^{-1} y^{\frac{\lambda - 1}{2}} g(z,y) z^{\nu} g(z,y)^{-1} y^{\frac{\lambda +1}{2}} g(x,y) x^{\nu} = 1
\quad \text{in } (\F_2)_\K,
\quad \text{with } zyx = 1,\ \lambda = \nu + 1.
\label{GT^1:octagon} \tag{O}
\end{gather}
\end{definition}
The product of two elements $(\lambda_1, f_1, g_1), (\lambda_2, f_2, g_2) \in \gtm$ is given by
\begin{equation}
\label{GRTM: composition law}
    (\lambda_1, f_1, g_1) * (\lambda_2, f_2, g_2) \eqdef (\lambda, f, g),
\end{equation}
where $\lambda \eqdef \lambda_1 \lambda_2$, $f \eqdef f_{1} \cdot f_2(x_1^{\lambda_{1}}, f_1^{-1}x_2^{\lambda_1}f_{1})$ and $g \eqdef  g_{1} \cdot g_2(x_1^{\nu_{1}}, g_1^{-1}x_2^{\lambda_1}g_{1} )$ with $\lambda_i = \nu_i +1$ for $i = 1,2$.

\medskip

Enriquez further observed that, once the Grothendieck--Teichm\"uller datum $(\lambda,f)$ is fixed, the additional module datum $g$ is determined up to a single scalar parameter. More precisely, in the discrete setting, Enriquez shows in \cite[Proposition~5.3]{Enriquez2007cyclotomic} that the defining relations are again highly rigid: necessarily $f=1$, $\lambda=\pm1$, and $g(x,y)=y^s$ for a unique integer $s$, so that $\gtm \cong \mathbb Z \rtimes \mathbb Z/2\mathbb Z$. The following is the proalgebraic analogue of \cite[Proposition~5.4]{Enriquez2007cyclotomic}.

\begin{prop}
Every element $(\lambda,f,g)\in \gtm$ can be uniquely written in the form
\[
(\lambda,f,g)=(\lambda,f,y^s f)
\]
for some $(\lambda,f)\in \gt$ and some $s\in \K$. Equivalently, there is an isomorphism of prounipotent groups
\[
\gtm \cong \gt \ltimes \K .
\]
\end{prop}

\begin{remark}
The semidirect product decomposition $\gtm \cong \gt \ltimes \K$ suggests that $\gtm$ should be viewed as a one-parameter, module-type extension of $\gt$. This invites comparison with the conjectural description of the symmetric Kashiwara--Vergne group as an extension of the Grothendieck--Teichm\"uller group by $\K$. We return to this perspective in a later section, where in \cref{thm:KV-action-F-012} we construct an injective group homomorphism
\[
\gtm \hookrightarrow \kvs(2).
\]
\end{remark}

\medskip

Every element $(\lambda,f,g)\in \gtm$ induces a moperad automorphism $(\vartheta^1,\vartheta):\hPaB^1 \to \hPaB^1$ acting trivially on objects. The operad component $\vartheta:\hPaB\to\hPaB$ is the automorphism corresponding to the Grothendieck--Teichm\"uller element $(\lambda,f)\in\gt$, determined by
\[
\vartheta(R^{1,2}) = x_{12}^{\frac{\lambda-1}{2}} \cdot R^{1,2}
\qquad \text{and} \qquad
\vartheta(\Phi^{1,2,3}) = \Phi^{1,2,3}\cdot f(x_{12},x_{23}),
\]
and these values satisfy the pentagon and hexagon relations of \cref{thm: presentation pab}. To extend $\vartheta$ to the moperad $\hPaB^1$, one sets
\[
\vartheta^1(E^{0,1}) = (E^{0,1})^\lambda
\qquad \text{and} \qquad
\vartheta^1(\Psi^{0,1,2}) = \Psi^{0,1,2}\cdot g(x_{01},x_{12}).
\]
The mixed pentagon equation~\eqref{GT^1:mixed pentagon} and octagon equation~\eqref{GT^1:octagon} are precisely the relations needed for these assignments to define a moperad automorphism; see \cite[Proposition~4.8]{calaque20moperadic}.

The following theorem is the special case of \cite[Theorem~5.13]{calaque20moperadic} in which the group $\Gamma$ is taken to be trivial. Its proof ultimately builds on \cite[Theorem~11.1.7]{FresseBook1}.

\begin{thm}
\label{thm: GT1 as PaB1 automorphisms}
Let $\Aut_0(\hPaB^1)$ denote the group of moperad automorphisms
\[
(\vartheta^1,\vartheta):\hPaB^1 \to \hPaB^1
\]
that act as the identity on objects. Then there is an isomorphism of prounipotent groups
\[
\Aut_0(\hPaB^1)\cong \gtm.
\]
\end{thm}

Via this identification, the group $\gtm$ acts freely and transitively on the set of moperad equivalences:
\[
(\varphi^1,\varphi):\hPaB^1 \longrightarrow \mathsf{CD}^+,
\]
by precomposition. Concretely, by Theorem~\ref{thm: assoc^1}, such a moperad equivalence is uniquely determined by a triple $(\mu,f,g)$, and the action of an element $(\lambda,f',g')\in \gtm$ is given by
\begin{equation}
(\mu,f,g)\ast(\lambda,f',g')
\eqdef
\bigl(\mu\lambda,\,
f\cdot f'(e^{\mu\xi_1},f^{-1}e^{\mu\xi_2}f),\,
g\cdot g'(e^{\mu\xi_1},g^{-1}e^{\lambda\xi_2}g)
\bigr).
\end{equation}

This picture admits a graded counterpart. Every moperad equivalence $(\varphi^1,\varphi):\hPaB^1 \longrightarrow \mathsf{CD}^+$
extends uniquely to a homomorphic expansion, that is, to a moperad isomorphism $(\Tilde{\varphi}^1,\Tilde{\varphi}):\hPaB^1 \to \PaCD^+$
acting trivially on objects. Accordingly, one obtains the graded Grothendieck--Teichm\"uller module group
\[
\grtm \cong \Aut_0(\PaCD^+),
\]
defined as the group of object-fixing automorphisms of the shifted moperad of parenthesized chord diagrams. This graded group was introduced and studied by Enriquez; see \cite[Section~7]{Enriquez2007cyclotomic}. For further discussion of $\grtm$ and related constructions, see \cite{enriquez2012mixed} or \cite{calaque20moperadic}.

\section{Genus Zero Kashiwara--Vergne Solutions}\label{sec: background on tder and KV solutions}
This section recalls the necessary background on genus zero Kashiwara--Vergne (KV) solutions, following Alekseev, Kawazumi, Kuno, and Naef \cite{AKKN_genus_zero, AKKN18highergenus}. We describe tangential and special automorphisms of free Lie algebras, together with a moperad of tangential automorphisms that provides a natural framework for studying KV solutions. This moperad mirrors the moperad arising from splittings of pure braid groups and braid groups with frozen strands described in the previous section.

Throughout, we write $\ass_n \eqdef \K\langle\langle x_1,\ldots,x_n\rangle\rangle$ for the degree-completed \defn{free associative algebra} on generators $x_1,\ldots,x_n$. The degree-completed \defn{free Lie algebra} $\lie_n$ is the Lie subalgebra of $\ass_n$ generated by $x_1,\ldots,x_n$, with Lie bracket given by the commutator $[f,g]=fg-gf$. The completed universal enveloping algebra of $\lie_n$ is naturally identified with $\ass_n$, that is, $\widehat{U}(\lie_n)=\ass_n$. We identify the group-like elements in the completed Hopf algebra $\widehat{U}(\lie_n)=\ass_n$ with the prounipotent group $\exp(\lie_n)\cong (\F_n)_\K$.

\begin{definition}
A \defn{tangential derivation} of $\lie_n$ is a derivation $u$ of $\lie_n$ such that
\[
u(x_i)=[x_i,a_i]
\]
for some $a_i\in \lie_n$. A \defn{special derivation} of $\lie_n$ is a tangential derivation $u$ satisfying
\[
u\Bigl(\sum_{i=1}^n x_i\Bigr)=0.
\]
\end{definition}

The collection of all tangential derivations of $\lie_n$ forms a Lie algebra, denoted $\tder_n$, with bracket given by
\[
[u,v](x_k)\eqdef u(v(x_k))-v(u(x_k));
\]
see \cite[Proposition~3.4]{AT12}. There is an isomorphism of vector spaces $\tder_n \oplus \mathfrak{a}_n \cong \lie_n^{\oplus n},$ where $\mathfrak{a}_n$ is the $n$-dimensional abelian Lie algebra generated by $x_1,\ldots,x_n$. Accordingly, we often represent a tangential derivation by a tuple of Lie words
\[
u\eqdef (a_1,\ldots,a_n)\in \lie_n^{\oplus n};
\]
see \cite[Definition~3.2]{AT12}. In this notation, if $u=(a_1,\ldots,a_n)$ and $v=(b_1,\ldots,b_n)$, then $([u,v])_k=[a_k,b_k]+u(b_k)-v(a_k).$ The vector space of special derivations $\sder_n$ is closed under this bracket and therefore forms a Lie subalgebra of $\tder_n$.

\medskip 

\begin{definition}\label{def: tangential automorphism} The group of \defn{tangential automorphisms} is the prounipotent group $\TAut_n \eqdef \exp(\tder_n)$ associated to the Lie algebra $\tder_n$. The group of \defn{special automorphisms}, denoted $\SAut_n \eqdef \exp(\sder_n)$ is exponentiated from the Lie algebra $\sder_n$. 
\end{definition}

As with tangential derivations, we can identify (as a set) elements of $\TAut_n$ with $(\exp(\lie_n))^{\times n}$, and write $F \in \TAut_n$ as a tuple $F=(f_1,f_2,\ldots,f_n)$, where each $f_i\in\exp(\lie_n)$.  There is an action of $\TAut_n$ on $\exp(\lie_n)$ given by the map $\rho:\TAut_n\rightarrow \Aut(\exp(\lie_n))$ defined on the generators~$x_i$ by \[x_i\mapsto f_i^{-1}x_if_i.\]  The group law on $\TAut_n$ is then given by (see for instance~\cite[Section~2.4]{AKKN_genus_zero})\begin{equation}
\label{eq: group law on TAut}
    (F\cdot G)_i \eqdef f_i(\rho(F)g_i),
\end{equation}
where $F\cdot G$ stands for the product in $\TAut_n$ and $f_i(\rho(F)g_i)$ is calculated in $\exp(\lie_n)$. 

\begin{definition} \label{def: cyc}
The complete, graded vector space of \defn{cyclic words} is the linear quotient  
$$\cyc_n \eqdef \ass_n/[\ass_n,\ass_n] = \ass_n/\left<ab-ba\mid \forall \; a,b\in\ass_n\right>.$$  
We denote the natural projection map by $\trace:\ass_n\rightarrow\cyc_n$. 
\end{definition} 

The adjoint action of $\tder_n$ on $\lie_n$ extends to an action of $\tder_n$ on the vector space $\ass_n$ by the Leibniz rule. This, in turn, descends to $\cyc_n$. 

The non-commutative divergence map gives another relationship between $\tder_n$ and $\cyc_n$. To define it, we note that each element $a\in\ass_n$ has a unique decomposition $$a=a_0+\partial_1(a)x_1+\ldots +\partial_n(a)x_n=a_0+\sum_{i=1}^{n}\partial_i(a)x_i$$ for some $a_0\in\K$ and $\partial_i(a)\in\ass_n$ for $1\leq i\leq n$.  In practice, $\partial_i$ picks the words of a sum that end in $x_i$ and deletes their last letter $x_i$, as well as all other words. This enables the definition of the non-commutative \emph{divergence}, a $1$-cocycle of $\tder_n$ (\cite[Proposition~3.20]{AT12}):

\begin{definition}\label{def AT div}
The non-commutative \defn{divergence} is the linear map $j:\tder_n\rightarrow \cyc_n$ defined on a tangential derivation $u=(a_1,\ldots,a_n)$ by $$j(u) \eqdef \trace\left(\sum_{i=1}^{n} \partial_i(a_i)x_i\right).$$ 
\end{definition} 
The divergence map is a $1$-cocycle of the Lie algebra $\tder_n$. In particular, for any pair $u,v$ of tangential derivations, we have $j([u,v])=u\cdot j(v)-v\cdot j(u)$. Here we write $\cdot$ for the natural action of $\tder_n$ on $\cyc_n$.  Integrating the divergence cocycle leads to the non-commutative Jacobian map: 

\begin{definition} 
The non-commutative \defn{Jacobian} is the map $J:\TAut_n\rightarrow\cyc_n$ given by setting \[J(1)=0 \quad \text{and} \quad \frac{d}{dt}\Big|_{t=0}J(e^{tu}g)=j(u)+u\cdot J(g)\] for $g\in \TAut_n$ and $u\in\tder_n$. 
\end{definition} 

The map $J$ is an additive group $1$-cocycle, that is, for any $G,H\in\TAut_n$, we have $J(G\cdot H)=J(G)+G\cdot J(H)$.

\begin{remark}
Our notation $(j,J)$ matches that of \cite{AET10}, and corresponds to $(div,j)$ in \cite{AT12}.
\end{remark}



\subsection{Operadic structures on $\lie$, $\tder$, $\sder$, $\TAut$, $\SAut$}\label{subsec:lieoperad} In this section, we define operadic structures on the spaces of tangential and special derivations, and the corresponding groups of tangential and special automorphisms. We summarize the main constructions and results below; detailed proofs and illustrative examples are provided in Appendix~\ref{sec:Operadapp}. 

We make particular note that while the collection of tangential derivations $\tder = \{\tder_n\}$ only assembles into an operad in vector spaces, the collection of special derivations $\sder = \{\sder_n\}$ assembles into an operad in Lie algebras. Moreover, in Theorem~\ref{thm:operad-SolKV}, we show that the operadic composition of genus zero KV solutions yields another KV solution, making them a colored operad in sets.

We begin by defining a linear operad structure on the free Lie algebra on $n$ generators:

\begin{definition}\label{def:LieOperad}
There is a natural right $\Sigma_n$ action on $\lie_n$ given by permuting the variables of a Lie word $f=f(x_1,\ldots,x_n)\in\lie_n$. 
In symbols, for $\sigma \in \Sigma_n$ and $f \in \lie_n$ we define \[f^\sigma(x_1,\ldots,x_n) \eqdef f(x_{\sigma^{-1}(1)},\ldots,x_{\sigma^{-1}(n)}).\]
Moreover, for $1 \leq i \leq m$ the partial composition map $\circ_i : \lie_m \oplus \lie_n \to \lie_{m+n-1}$ is defined for $f \in \lie_m$ and $g \in \lie_n$ by 
\[
(f \circ_i g)(x_1,\ldots,x_{m+n-1}) 
\eqdef 
f(x_1, \ldots, x_{i-1},\sum_{j=i}^{i+n-1}x_j,x_{i+n},\ldots,x_{m+n-1}) + g(x_i, \ldots,x_{i+n-1}).
\]
We denote by $\lie$ the symmetric sequence of free Lie algebras with these partial compositions, and call this the $\K$-\defn{linear operad of free Lie algebras}, see Proposition~\ref{prop:operad-lie} below.
\end{definition}

\begin{prop}
\label{prop:operad-lie}
    The symmetric sequence of free Lie algebras $\lie$ with the partial compositions defined above forms a linear operad.
    The same formulas endow the universal enveloping algebra $\ass$ and the vector space of cyclic words $\cyc$ with a linear operad structure, such that the quotient map $\trace: \ass \to \cyc$ is a map of linear operads.
\end{prop}

The proof is deferred to Appendix~\ref{sec:Operadapp}.

\begin{remark}
Notably $\lie$ is not an operad in Lie algebras, and as such, the operadic composition is non-canonical: the sum $\sum_{j=i}^{i+n-1} x_j$ could be replaced with any associative Lie expression in the same variables~$x_j$, such as the Baker--Campbell--Hausdorff series. 
See Remark \ref{rem:Pavol-wisdom} for more detail.
\end{remark}

Using the linear isomorphism $\tder_n \oplus \mathfrak{a}_n \cong \lie_n^{\oplus n}$, the linear operad structure on $\lie=\{\lie_n\}_{n\geq 0}$ generalizes to a linear operad of tangential derivations.  For $u=(a_1,\ldots,a_n) \in \tder_n$ and $\sigma \in \Sigma_n$ define
\[u^\sigma 
\eqdef (a_{\sigma^{-1}(1)}^\sigma,\ldots, a_{\sigma^{-1}(n)}^\sigma),\]
where $a_j^\sigma$ denotes the $\Sigma_n$ action on $\lie$.
This is the canonical action induced by the action on $\lie$ in the sense that it is the unique action which makes the following diagram commute:
\begin{equation} \label{diag:sigma-n-action}
    \begin{tikzcd}
\lie_n \arrow[r, "(-)^\sigma"] \arrow[d, "u"'] & \lie_n \arrow[d, "u^{\sigma}"] \\
\lie_n \arrow[r, "(-)^\sigma"']                & \lie_n     \end{tikzcd}
\end{equation}

For $1 \leq i \leq m$, define the composition map
\[ \circ_i : \tder_m \oplus \tder_n \to \tder_{m+n-1}\] to be the $\K$-linear map such that for $u=(a_1,\ldots, a_m)$ and $
v=(b_1,\ldots, b_n)$, the composition is
\begin{equation}\label{eq:tderOperadComp}
u \circ_i v \eqdef (a_1\circ_i 0,\ldots,a_{i-1}\circ_i 0,a_i \circ_i b_1,\ldots,a_i \circ_i b_n, a_{i+1}\circ_i 0,\ldots,a_m \circ_i 0).
\end{equation} 
Here, the $\circ_i$ on the right hand side denotes operadic composition in the operad $\lie$, and $0$ stands for $0_n \in \lie_n$.

The following proposition follows from the arguments presented in~\cite[Section 3.3]{AT12}. See also \cite{AET10} and \cite[Section 7]{AKKN_genus_zero}. 

\begin{prop}
\label{prop:tder-operad}
The family $\tder=\{ \tder_n\}_{n \geq 0}$ of tangential derivations endowed with the above $\circ_i$ operations and $\Sigma_n$-actions forms a $\K$-linear operad.  The non-commutative divergence $j : \tder \to \cyc$ is then a map of linear operads. 
\end{prop}

The collection of special derivations $\sder = \{\sder_n\}$ carries a stronger operadic structure than $\tder$: it forms an operad in Lie algebras. We state the main theorems describing this structure below; full proofs are deferred to Appendix~\ref{sec:Operadapp}.

\begin{theorem}\label{thm:sderbraid}
    The family $\sder$ of special derivations forms a linear suboperad of $\tder$, and with the operadic structure restricted from $\tder$ it is an operad in Lie algebras. The family $\ib\eqdef \{\ib_n\}_{n \geq 0}$ of infinitesimal braids injects into $\sder$ as a sub-operad in Lie algebras.
\end{theorem}

In analogy with tangential derivations, non-canonical operadic structures can be defined for the groups of tangential automorphisms $\TAut = \{ \TAut_n \}_{n \geq 0}$, to produce operads in sets.

\begin{prop}
\label{operad of TAut}
The collection $\TAut = \{ \TAut_n \}_{n \geq 0}$ has the structure of an operad in sets. 
The symmetric group action is given for $F=\exp(u) \in \TAut_n$ and $\sigma \in \Sigma_n$ by $F^\sigma \eqdef \exp(u^\sigma)$.
The partial composition maps are defined for $F=\exp(u), G=\exp(v)$ with $u \in \tder_m$ and $v \in \tder_n$, by the formula
\[
\begin{tikzcd}
\TAut_m \times \TAut_n \arrow[r, "\circ_i"] & \TAut_{m+n-1}
\end{tikzcd}
\]
\[
(F, G) 
\mapsto 
\exp(u \circ_i 0) \, \exp(0 \circ_i v) .
\] 
\end{prop}

In particular, we have $F \circ_i G = (F \circ_i 1) \, (1 \circ_i G)$, which in usual cosimplicial notation reads $$ 
F^{1,\ldots,i-1,\,i(i+1)\cdots(i+n-1),\,i+n,\ldots,m+n-1} \circ
G^{i,i+1,\ldots,i+(n-1)}.$$
As explained in \cref{sec:operadic-cohomology}, this cosimplicial structure is also present at the level of $\tder$.
In particular, if $F=\exp{(u)} \in \TAut_n, u \in \tder_n$, then $F^{i,(i+1)\hdots j}=\exp{(u^{i,(i+1)\hdots j})}$, etc. 
This construction is also natural in the sense that if $a \in \lie_n$, then we have $(F(a))^{i,(i+1)\hdots j}=F^{i,(i+1)\hdots j}(a^{i,(i+1)\hdots j})$ for instance.
Note that even though partial compositions $\circ_i$ are not group homomorphisms in general, composition with the unit $-\circ_i 1$ and $1 \circ_i -$ are, as are the symmetric group actions.

\begin{remark}
\label{rem:integration-not-compatible-with-operad}
The operad structure on $\TAut$ defined above agrees with\footnote{The order of composition in our definition is opposite to that of \cite{AKKN_genus_zero}, as we follow the definition of the set of KV solutions in \cite{AET10}, which is the inverse of that of \cite{AT12} and \cite{AKKN_genus_zero}.} with the operad composition used in \cite{AKKN_genus_zero} to construct Kashiwara--Vergne solutions corresponding to genus zero surfaces with more than three boundary components. 
It is not, however, the operad structure one would obtain by exponentiation of the partial compositions of $\tder$:
\[
\exp(u) \circ_i \exp(v) \neq \exp(u \circ_i v).
\]
where $u \in \tder_n$ and $v \in \tder_m$, and the composition on the right-hand side is as in Equation~\eqref{eq:tderOperadComp}.
These notions do, however coincide when restricted to $\sder$ and $\SAut$, where there are canonical operad structures in the categories of Lie algebras and prounipotent groups, respectively.
\end{remark}

Although the operadic structure on $ \TAut $ is non-canonical, it is natural in the sense of~\cite[Section 3]{AT12}: for any $ F \in \TAut_n $, the identity $ J(F^{i,j}) = J(F)^{i,j} $ holds for all pairs $ i, j $. 
As we compute in \cref{sec:operadic-cohomology}, this implies that the Jacobian cocycle is compatible with the operadic composition. 

\begin{prop}
\label{prop:Jacobian-operad}
    The Jacobian $ J \colon \TAut \to \cyc$ is a morphism of operads (in sets).
\end{prop}

\subsection{Moperad structures for $\sder$ and $\tder$}\label{sec: moperad of dervations}
Since $\sder$ is an operad in Lie algebras (Theorem~\ref{thm:sder-operad}), following the general shift construction presented in Example~\ref{example: shifted moperad}, one obtains a shifted $\sder$-moperad in Lie algebras, denoted $\sder^+$. The same statement integrates to groups, that is, $\SAut=\{\SAut_n\}_{n\geq 0}$ is an operad in prounipotent groups, and there is a corresponding shifted $\SAut$-moperad, denoted $\SAut^+$. 
While $\tder$ is not an operad in Lie algebras, in this section we develop an $\sder$-moperad in Lie algebras, denoted $\tder^1$, which injects into the linear moperad $\tder^+$. The construction is inspired by the $\sder$-moperad $\sder^1$, a sub-moperad of $\sder^+$, below. Both constructions integrate to moperads in prounipotent groups.

\begin{definition}
Let $\sder^1_n$ denote the semidirect product of Lie algebras $\sder^1_n \eqdef \lie_n\rtimes \sder_n$.    
For $n\geq 1$, write elements of $\sder_n^1$ as pairs $({a},u)$, where $a=a(x_1,\ldots, x_n)$ is a Lie word in $\lie_n$ and $u$ is a derivation in~$\sder_n$.
Then $[(a,u),(b,v)]=([a,b]+u(b)-v(a),[u,v])$.

Similarly, let $\SAut_n^1=(\F_n)_{\K} \rtimes \SAut_n$ denote the corresponding semidirect product of groups, consisting of pairs $(w,e^u)$, where $w=w(x_1,...,x_n)$ is a word in $(\F_n)_{\K}$, and $e^u\in \SAut_n$. If $w,w' \in (\F_n)_{\K}$ and $e^u, e^{u'} \in \SAut_n$ then $(w,e^u)(w',e^{u'})= \left(w (e^u\cdot w'), e^u e^{u'}\right)$.

The symmetric group $\Sigma_n$ acts on $\sder^1_n$ by permuting the variables: $\sigma(a,u) \eqdef (a^\sigma, u^\sigma)$. 
This action induces an action of $\Sigma_n$ on $\SAut^1_n$ also by permuting the variables.  
\end{definition}

We aim to define an $\sder$-moperad structure on $\sder^1$.
To do so, we define a Lie algebra inclusion
$$\kappa: \lie_n\langle x_1,\ldots,x_n\rangle \hookrightarrow \sder_{n+1}=\sder_{\lie_{n+1}\langle x_0,x_1,...,x_n\rangle}
.$$
We define $\kappa$ on the generators by setting $\kappa(x_i)=t^{0,i}=(x_i,0,...,0,x_0,0,...,0)$, where $x_i$ is placed in the 0-th component and $x_0$ in the $i$-th component. 
The map $\kappa$ extends uniquely to a Lie algebra homomorphism on $\lie_n$. For $a\in \lie_n$, the 0-th component of $\kappa(a)$ is always $a$, and $a$ does not involve the variable $x_0$, hence, the map~$\kappa$ is injective.

\begin{remark}
    There is an illuminating visual for understanding the $\kappa$-value of a Lie word, shown in Figure~\ref{fig:TreeReRooting}. Namely, a Lie word in $\lie_n$ is naturally represented as a rooted binary tree with leaves labeled with the numbers $\{1,2,...,n\}$ (each label may be used multiple times), and the root labeled $0$. Redefining the root of such a tree to be one of the leaves then gives a new Lie word, this time in the $(n+1)$ variables $\{x_0,x_1,...,x_n\}$. 
    We obtain an element in $\lie_{n+1}^{\oplus (n+1)}$ by summing over all ways of rooting the tree, and placing each resulting Lie word in the component numbered by the label of the new root, as shown in Figure~\ref{fig:TreeReRooting}. To finish, $\lie_{n+1}^{\oplus (n+1)}$ maps into $\tder_{n+1}$ in the natural way. This in particular implies that the 0-th component of $\kappa(a)$ is always $a$. The fact that the image of $\kappa$ is in $\sder_{n+1}$ follows from the fact that $\kappa$ is a Lie algebra map and the images of the generators lie in $\sder_{n+1}$. 
\end{remark}

The following proposition follows diagrammatically from \cite[Theorem 3.28]{WKO2}, or algebraically from Drinfeld's Lemma  \cite[Lemma after Proposition 6.1]{Drin}.

\begin{prop}
\label{prop:KappaMakesSense}
    For $u=(u_1,...,u_n)\in \sder_n$, denote $u^+ \eqdef (0,u_1,...,u_n)\in \sder_n^+\cong \sder_{n+1}$. Then, for any $a\in \lie_n$, we have $\kappa(u(a))=[u^+,\kappa(a)]$.
\end{prop}

\begin{example}
    We check $\kappa(u(a))=[u^+,\kappa(a)]$ for $u=(x_2,x_1)$ and $a=x_2$. For this calculation, recall that in $\tder_{n}$ if $u=(u_1,...,u_n)$ and $v=(v_1,...,v_n)$, then $([u,v])_k=[u_k,v_k]+u(v_k)-v(u_k)$. We compute both sides: 
    \begin{align*}
\kappa(u(x_2)) =\kappa([x_2,x_1]) &=([x_2,x_1],[x_0,x_2],-[x_0,x_1]),\\
    [u^+,\kappa(x_2)] =[(0,x_2,x_1),(x_2,0,x_0)] &=([x_2,x_1],-[x_2,x_0],[x_1,x_0]),
    \end{align*}
    and therefore $\kappa(u(x_2))=[u^+,\kappa(x_2)]$ as expected.
\end{example}

\begin{figure}
    \includegraphics[width=10cm]{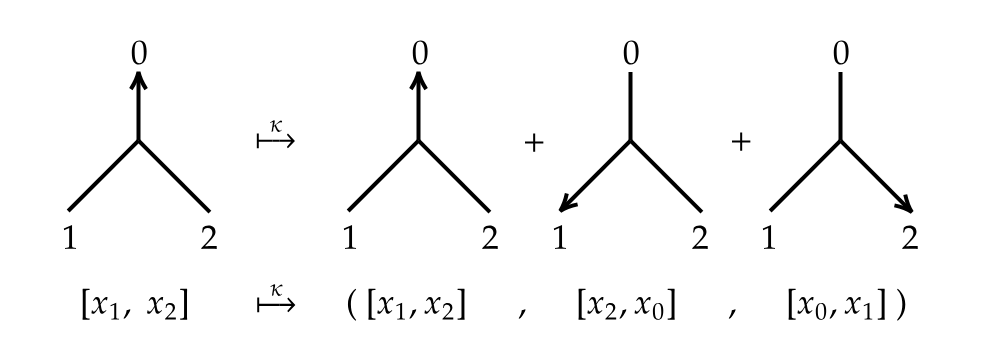}
    \caption{The value of $\kappa([x_1,x_2])\in \sder_{2+1}$, calculated by re-rooting the tree.}\label{fig:TreeReRooting}
\end{figure}

\begin{lemma}
\label{lem:iota}
There are injective Lie algebra homomorphisms
$$\iota_n:\sder_n^1 \hookrightarrow \sder^+_n\cong \sder_{n+1},$$ given on $a \in \lie_n$ and $u=(u_1,\ldots,u_n) \in \sder_n$ by 
\[ \iota_n(a,u) \eqdef \kappa(a)+ (0,u_1,\ldots,u_n).\]
\end{lemma}

\begin{proof}
    The map $\iota_n$ can be written as the sum of $\kappa:\lie_n \to \sder_{n+1}$, and the natural inclusion $(\cdot)^+:\sder_n \hookrightarrow \sder^+_n$ given by $(u_1,...,u_n)^+ \eqdef (0,u_1,...u_n)$. 
    In $\sder_n^1,$ we have $[(a,u),(b,v)]= ([a,b]+u(b) - v(a), [u,v])$. Hence, we have
    \begin{align*}
        \iota_n([(a,u),(b,v)])&=\kappa([a,b])+\kappa(u(b)) -\kappa(v(a))+ [u,v]^+\\
        &=[\kappa(a),\kappa(b)]+[u^+,\kappa(b)]+[\kappa(a),v^+]+[u^+,v^+]\\
        &=[\kappa(a)+u^+,\kappa(b)+v^+]=[\iota_n(a,u),\iota_n(b,v)].
    \end{align*}
    Here, in the second line we used~\cref{prop:KappaMakesSense}.
    To show that $\iota_n$ is injective, assume that $\iota_n(a,u)=0$. It follows immediately from the definition and injectivity of $\kappa$ that $a=0$, and then we have $\iota_n(0,u)=0$, which implies $u=0$ by the injectivity of the shift construction $(\cdot)^+$.
\end{proof}

\begin{theorem}
\label{thm:Sder1Moperad}
The maps $\{\iota_n\}_{n \geq 0}$ naturally assemble to a map of symmetric sequences in Lie algebras $\iota: \sder^1 \to \sder^+$, and the image $\iota(\sder^1)$ 
    of $\sder^1$ under the map $\iota$ is an $\sder$-submoperad of $\sder^{+}$.
\end{theorem}

\begin{proof}
Observe that the map $\iota_n$ commutes with the action of $\Sigma_n$. We need to show that the image $\iota(\sder^1)$ is closed under the monoid product and the $\sder$-action of the moperad $\sder^+$.
    Let $v=(v_0,v_1,...,v_m)=\iota_m((a,u))$ be an arbitrary element in the image of $\iota$. 
    Then, in particular, $v_0=a\in \lie\langle x_1,...,x_m\rangle$, and therefore $v-\kappa(a)\in \sder_m=\sder_{\lie\langle x_1,...,x_m\rangle}$.
    If $v-\kappa(a)=(0,u_1,...,u_m)$, then $u=(u_1,...,u_m)$. 

    Now for some $w\in \sder_n$, we compute $v\circ_i w$, where $1\leq i \leq m$. 
    Recall from Equation~\eqref{eq:tderOperadComp},~\cref{def:LieOperad}, and Section~\ref{sec:Operad_sder} that
    \[
    v \circ_i w = v \circ_i 0 + 0 \circ_i w = v^{1,\ldots,i-1,i(i+1)\cdots(i+n-1),i+n,\ldots,m+n-1}+ w^{i,i+1,...,i+n-1}.
    \]
    Since $v_0\in \lie\langle x_1,...,x_m\rangle $, therefore, $(v \circ_i w)_0=v_0 \circ_i 0\in \lie\langle x_1,...,x_{m+n-1}\rangle$. Thus, $(v \circ_i w)_0\in \lie_{m+n-1}$. 
    
    Observe that $\kappa(a \circ_i 0)=\kappa(a)\circ_i 0$. Indeed, since both sides are Lie algebra maps, it is enough to check this for the generators, which in turn is straightforward. 
    Thus, $\kappa(v_0 \circ_i 0)=\kappa(v_0)\circ_i 0$. 
    
    Now, we compute
    \[
    v\circ_i w-\kappa({v}_0\circ_i 0)=
    \left(
    v \circ_i 0-\kappa(v_0)\circ_i 0
    \right)
    +0\circ_i w
    =\iota(0,u) \circ_i 0 +0 \circ_i w.
    \]
    Therefore, we have
    \[
    v\circ_i w= \kappa(a\circ_i 0) + \iota(0,u) \circ_i 0 +0 \circ_i w = \iota(
    a \circ_i 0, u \circ_i 0 +0 \circ_i w)=\iota\left(a \circ_i 0, u\circ_i w\right),
    \]
    which is in the image of $\iota$.

    Now let $v=(v_0,v_1,...,v_m)=\iota_m((a,u))$, and $v'=(v'_0,v'_1,...,v'_n)=\iota_n((a',u'))$. 
    We check that $v\circ_0 v' \in \sder^+_{m+n}$ is in the image of $\iota$. 
    We have 
    \[
    v\circ_0 v'= v\circ_0 0+(v'_0,v'_1,...,v'_n,0,...,0).
    \]
    The 0-th component of $v\circ_0 v'$ is $(v_0\circ_0 0+v'_0)$. 
    We compute
    \begin{align*}
    v\circ_0 v'-\kappa(v_0 \circ_0 0+v'_0)
    &=\left(v \circ_0 0-\kappa(v_0\circ_0 0)\right)+\left((v_0',...,v_n',0,...,0)-\kappa(v_0')\right)\\
    &=(v-\kappa(v_0))\circ_0 0+0 \circ_0 \left(v'-\kappa(v'_0)\right) \\
    &=\iota(0,u) \circ_0 0 + 0 \circ_0 \iota(0,u')=\iota(0,u\circ_0u').
    \end{align*}
    Therefore, we have $v\circ_0v'=\iota\left(v_0 \circ_0 0+v'_0,u\circ_0u'\right)$, finishing the proof.     
\end{proof}

This structure integrates to prounipotent groups, as follows.
The operad $\sder =\{\sder_n\}_{n \geq 0}$ is an operad in degree completed Lie algebras, meaning that, in each arity, the group-like elements of $\widehat{U}(\sder_n)$ are identified with $\exp(\sder_n)$.  
Let $\SAut^1_n \eqdef (\F_n)_{\K} \rtimes \SAut_n$. 
There is an injective group homomorphism $\exp(\iota):\SAut^1_n \to \SAut^+_n\cong \SAut_{n+1}$. 
The symmetric sequence $\{\SAut^+_{n}\}_{n \geq 0}$ is an $\SAut$-moperad in prounipotent groups, and this restricts to make $\SAut^1$ an $\SAut$-moperad. 

\begin{example}
    For readers familiar with chord diagrams in the sense of finite type invariants in knot theory, we note that these structures are both part of a moperad in Hopf algebras assembled of chord diagrams on $(n+1)$ strands, as shown in Figure~\ref{fig:moperad}. For readers otherwise, this may still be helpful as a visualization.

    \begin{figure}
        \includegraphics[width=14cm]{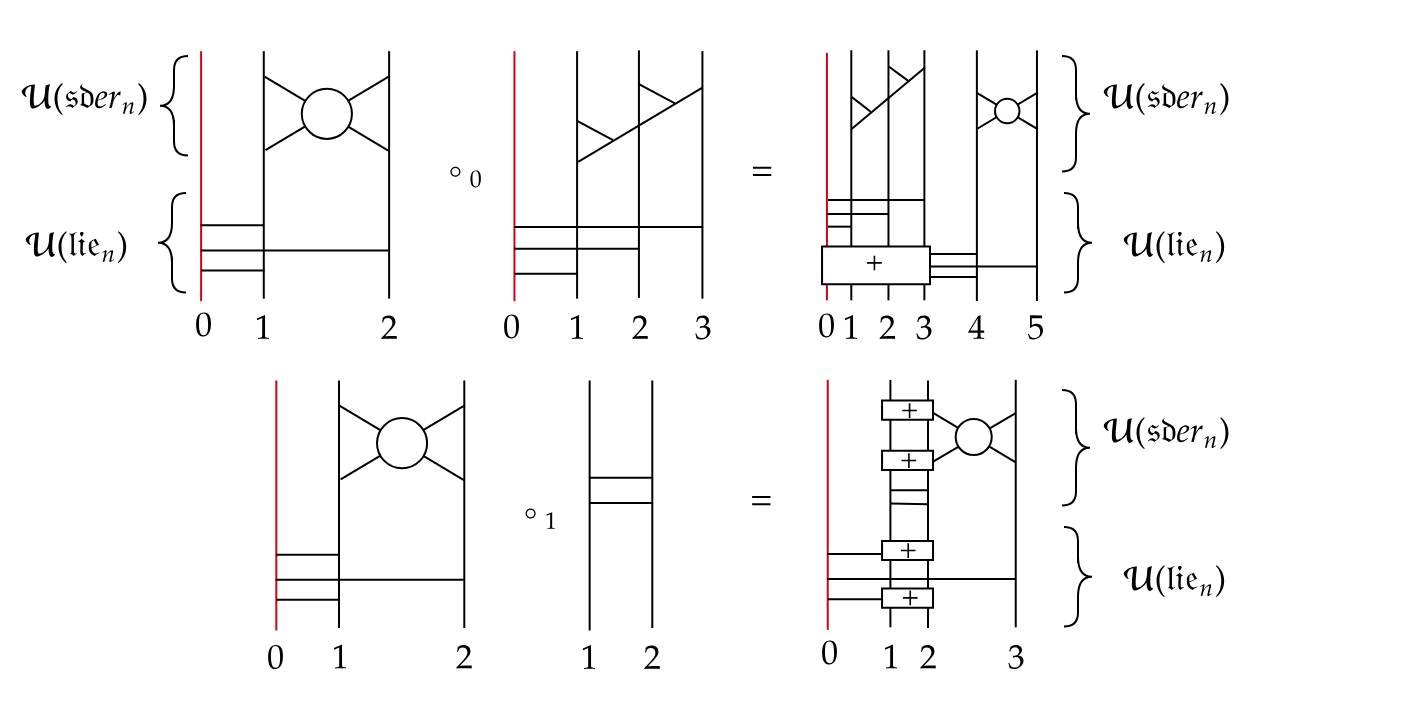}
        \caption{Moperad operations for chord diagrams.}\label{fig:moperad}
    \end{figure}
\end{example}

Expanding on Theorem~\ref{thm:sderbraid}, which states that infinitesimal braids form a suboperad of $\sder$, we have the following.

\begin{lemma}\label{lemma: image of CD+ in SAut1}
The image of $\ib^+_n\hookrightarrow \sder^+_n$ lies in $\iota(\sder^1_n)$.  
\end{lemma}

\begin{proof} 
Since the inclusion $\ib_n^+ \hookrightarrow \sder_n^+$ is a Lie algebra map, it is enough to check this for the generators.

If $0<i<j$, then $t_{i,j}$ maps to $t^{i,j}=(0,...,0,x_j,0,...,0,x_i,0,...,0)$, with $x_j$ in component $i>0$ and $x_i$ in component $j>0$. Therefore, $t^{i,j}\in \sder_n$, so $t^{i,j}=\iota\left(0,t^{i,j}\right)$.

On the other hand, $t_{0,i}$ maps to $t^{0,i}=(x_i,0,...,0,x_0,0,...,0)$, with $x_i$ in the 0-th component and $x_0$ in the $i$-th component. Then
$t^{0,i}=\kappa(x_i)$, therefore, $t^{0,i}=\iota(x_i,0)$.
\end{proof}

Finally, we introduce an $\sder$-moperad in Lie algebras, $\tder^1$, consisting of tangential derivations. The conceptual difference from the previous construction is that $\tder$ itself does not form an operad in Lie algebras, and hence there is only a shifted moperad to speak of in a linear sense. We denote this linear shifted moperad by $\tder^+$. It is the case, however, that tangential derivations admit a right $\sder$-module structure in Lie algebras. 

\begin{prop}
\label{prop:TderSderModule}
The symmetric sequence $\tder=\{\tder_n\}_{n\geq 0}$ forms a right $\sder$-module in Lie algebras, where the action $\circ_i: \tder_m \oplus \sder_n \to \tder_{n+m-1}$ is given by restricting the $i$-th linear partial compositions of $\tder$. 
\end{prop}

\begin{proof}
    We have seen in Proposition~\ref{prop:tder-operad} that these partial compositions make $\tder$ a linear operad. Thus, it is enough to show that the restrictions with $\sder$ in the second component are Lie algebra maps. For $d\in \tder_m$ and $u\in \sder_n$, we have 
    \[
    d \circ_i u = d \circ_i 0+ 0\circ_i u=d^{1,\ldots,i-1,i(i+1)\cdots(i+n-1),i+n,\ldots,m+n-1}+u^{i,i+1,\ldots,i+n-1}.
    \]
    Since $(\cdot)\circ_i 0$ and the shift $(\cdot)^{i,i+1,\ldots,i+n-1}$ are both Lie algebra maps, their sum is a Lie algebra map if the images commute in $\tder_{m+n-1}$. 
    In other words, we need to show $[d\circ_i 0,u^{i,i+1,...,i+n-1}]=0$ for any $d\in \tder_m$ and $u\in \sder_n$. To see this, notice that since $u\in \sder_n$, the value $u^{i,\ldots,i+n-1}(x_i+\cdots+x_{i+n-1})=0$, and therefore $u^{i,...,i+n-1}(d_k\circ_i 0)=0$ for all $d_k\in \lie_m$. 
    This implies that $[d\circ_i 0,u^{i,i+1,...,i+n-1}]$ -- when written as an $(m+n-1)$-tuple -- is zero in coordinates 1 through $(i-1)$, and $(i+n)$ through $(m+n-1)$. 
    For example, if $i>1$, then the first coordinate of $[d\circ_i 0,u^{i,i+1,...,i+n-1}]$ is computed as follows:
    \[
    [d_1 \circ_i 0, 0]+(d\circ_i 0) (0)-u^{i,...,i+n-1}(d_1\circ_i 0)=0+0-0.
    \]
    
    Next, if $1\leq j\leq n$, then the $(i+j-1)$-st coordinate of $[d\circ_i 0,u^{i,i+1,...,i+n-1}]$ is:
    \[
    [d_i\circ_i 0,u_j^{i,...,i+n-1}]+(d\circ_i 0)(u_j^{i,...,i+n-1})-0=[d_i\circ_i 0,u_j^{i,...,i+n-1}]+[u_j^{i,...,i+n-1},d_i\circ_i 0]=0,
    \]
    where the first equality holds because $u_j^{i,...,i+n-1}$ is a Lie word is variables $x_{i},..,x_{i+n-1}$, and the tangential derivation $d\circ_i 0$ has identical components $d_i \circ_i 0$ in positions $i$ through $(i+n-1)$, hence it acts as an inner derivation on $u_j^{i,...,i+n-1}$ (see also Section \ref{ss:symmetric-KV-solutions}). 
    This completes the proof.
\end{proof}

\begin{remark}\label{rem:Tder+SderModule}
    One can similarly define a right $\sder$-module structure on the extended symmetric sequence $\tder^+=\{\tder_n^+=\tder_{\lie_{n+1}\langle x_0,...,x_n\rangle}\}$ in which the $\sder$ action maps are defined by the partial compositions $\circ_i$ for $1\leq i \leq n$.
\end{remark}

\begin{definition}
    We define $\tder^1_n\eqdef \lie_n \rtimes \tder_n $. For $n\geq 1$, write elements of $\tder_n^1$ as pairs $({a},u)$, where $a=a(x_1,\ldots, x_n)$ is a Lie word in $\lie_n$ and $u$ is a derivation in $\tder_n$.
Then $[(a,u),(b,v)]=([a,b]+u(b)-v(a),[u,v])$. The corresponding group is $\TAut_n^1= (\F_n)_{\K} \rtimes \SAut_n$, consisting of pairs $(w,e^u)$, where $w=w(x_1,...,x_n)$ is a word in $(\F_n)_{\K}$, and $e^u\in \TAut_n$.
\end{definition}

The symmetric group $\Sigma_n$ acts on $\tder^1_n$ by permuting the variables: $\sigma(a,u) \eqdef (a^\sigma, u^\sigma)$. 
This action integrates to $\TAut^1_n$ as well.

The analogue of $\kappa$ is a (different) injective Lie algebra map
$$\lambda: \lie_n\langle x_1,\hdots, x_n \rangle \hookrightarrow \tder_{n+1}=\tder_{\lie\langle x_0,x_1,...,x_n\rangle}.$$ The map $\lambda$ is simply the inclusion of $\lie_n$ in the 0-th component, that is, $\lambda(a)=(a,0,...,0,0,0,...,0)$. The map $\lambda$ is injective, as a tuple $(a,0,...,0)$ represents the 0 derivation in $\tder_{n+1}$ if and only if $a=cx_0$ for some scalar $c$. 

\begin{prop}
\label{prop:LambdaMakesSense}
    For $d=(d_1,...,d_n)\in \tder_n$ and any $a\in \lie_n$, we have $\lambda(d(a))=[d^+,\lambda(a)]$.
\end{prop}

\begin{proof}
We compute both sides:
\begin{align*}
\lambda(d(a))&=(d(a),0,...,0),\\
[d^+,\lambda(a)]&=[(0,d_1,...,d_n),(a,0,...0)]=(d(a),0,...,0).
\end{align*}
Here in the second line we used the rule for calculating commutators of tangential derivations in the coordinate expression: $\left([u,v]\right)_k=[u_k,v_k]+u(v_k)-v(u_k)$.
\end{proof}

\begin{lemma}
\label{lem:eta}
For each $n\geq 1$, there is an injective Lie algebra homomorphism
\begin{equation}
\eta_n:\tder_n^1 \hookrightarrow \tder_n^+= \tder_{\lie_{n+1}\langle x_0,x_1,...,x_n\rangle}\cong\tder_{n+1}
\end{equation} given by 
\[ \eta_n(a,d)\eqdef\lambda(a)+d^{+} \quad \text{where}
\quad    d^+= (0,d_1,...,d_n) \] for any $a\in\lie_n$ and $d=(d_1,...,d_n)\in \tder_n$.
\end{lemma}

\begin{proof}
The proof of Lemma~\ref{lem:iota} applies verbatim, replacing $\iota$ with $\eta$.
\end{proof}

The maps $\{\eta_n\}$ assemble to a map of symmetric sequences $\eta:\tder^1 \to \tder^+$. As we saw in \cref{prop:TderSderModule} and \cref{rem:Tder+SderModule}, $\tder$ and $\tder^+$ are right $\sder$-modules in Lie algebras. The next lemma shows that this module structure restricts to $\tder^1$. 

\begin{lemma}
\label{lem:tder-is-sder-mod}
The image $\eta(\tder^1)$ in $\tder^+$ is invariant under the right $\sder$-action, hence, the action restricts to make $\tder^1$ a right $\sder$-module in Lie algebras.
\end{lemma}

\begin{proof}
For $a\in \lie_m$, $d\in \tder_m$, $u\in \sder_n$, $1\leq i\leq m$, we have
\[
\eta(a,d)\circ_i u=(\lambda(a)+d^+)\circ_i u = \lambda(a){\circ_i 0}+ d^+\circ_i u= \lambda(a{\circ_i 0})+(d\circ_i u)^+=\eta(a{\circ_i 0},d\circ_i u).
\]
\end{proof}

\begin{lemma}
\label{lem:tder-is-sder-monoid}
The image $\eta(\tder^1)$ in $\tder^+$ is closed under the linear monoid multiplication \[\circ_0: \tder^+_m \oplus \tder^+_n \to \tder^+_{m+n}.\] Moreover, for each $m,n\geq 1$ the restriction of the linear map $\circ_0:\tder^+_m \oplus \tder^+_n \to \tder^+_{m+n}$ to $\eta(\tder^1_m)\oplus \eta(\tder^1_n)$ is a Lie algebra homomorphism.
\end{lemma}

\begin{proof}
For any $a\in \lie_m$, $a'\in \lie_n$, $d\in \tder_m$ and $d'\in \tder_n$, we compute
\begin{align*}
\eta_m(a,d)\circ_0 \eta_n(a',d')
&=(a,d_1,...,d_m)\circ_0(a', d'_1,\ldots,d'_n)\\
&=\left(a \circ_0 0,\ldots,a\circ_0 0,d_1 \circ_0 0,\ldots,d_m\circ_0 0 \right)+(a', d'_1,...,d'_n,0,...,0).
\end{align*}
Note that since $a$ and $d_j$ are elements of $\lie_m\langle x_1,...,x_m\rangle$,  for $1\leq j \leq m$, it follows that $a\circ_0 0=a(x_{n+1},...,x_{m+n})$, and $d_j \circ_0 0=d_j(x_{n+1},...,x_{m+n})$. Thus, each component of $\eta_m(a,d)\circ_0 \eta_n(a',d')$ is an element of $\lie_{m+n}$, and therefore, $\eta_m(a,d)\circ_0 \eta_n(a',d')\in \eta_{m+n}(\tder^1_{m+n})$.

The monoid composition
\[\circ_0: \eta(\tder^1_m) \oplus \eta(\tder^1_n) \to \eta(\tder^1_{m+n})\]
can be described as the sum of Lie algebra maps $(\cdot)\circ_00=(\cdot)^{01\cdots n ,n+1,\ldots,n+m}$ 
and $(\cdot)^{0,\ldots,n}$. 
It is enough to check that the images of these two maps commute, that is, we need to show that
\[
\left[\left(a \circ_0 0,\ldots,a\circ_0 0,d_1\circ_0 0,\ldots,d_m\circ_0 0 \right) , (a', d'_1,...,d'_n,0,...,0)\right]=0 \quad \text{ in }\tder^1_{m+n}
\] 
for any $a,d,a',d'$. 
Since $a\circ_0 0$ and $d_i\circ_0 0$ are Lie words in $\{x_{n+1},...x_{m+n}\}$, we compute, for example, the $0$-th component of the commutator using $\left([u,v]\right)_k=[u_k,v_k]+u(v_k)-v(u_k)$. 
We obtain
\[
\left[a\circ_0 0,a'\right]+\left[a',a \circ_0 0\right]-\left[a\circ_0 0,0\right]=0,
\]
using that the first $n+1$ components of $\left(a\circ_0 0,...,a\circ_0 0,d_1\circ_0 0,...,d_n\circ_0 0 \right)$ are identical, and that  $a\circ_0 0=a(x_{n+1},...,x_{m+n})$ and all of these variables are killed by $(a', d'_1,...,d'_n,0,...,0)$. 
The first to $n$-th components of the commutator are zero for the same reason, and the last $m$ are zero trivially. 
This completes the proof.
\end{proof}

In summary, we have proven the following:
\begin{theorem}
\label{thm:Tder1Moperad}
The symmetric sequence $\tder^1$ is an $\sder$-moperad in Lie algebras, with the monoid multiplication and $\sder$-action defined by restriction from the linear moperad $\tder^+$. This structure integrates to groups, making $\TAut^1$ an $\SAut$-moperad in prounipotent groups.
\end{theorem}

\begin{proof}
    We observe that the map $\eta_n$ commutes with the action of $\Sigma_n$.
    The result then follows from the combination of \cref{lem:tder-is-sder-mod} and \cref{lem:tder-is-sder-monoid}.
\end{proof}

\subsection{Kashiwara--Vergne solutions, associators, and symmetry groups}\label{subsec:kvsymmetries}
In this section, we study genus zero Kashiwara--Vergne (KV) solutions, the higher-arity analogue of the classical KV problem introduced by Alekseev, Kawazumi, Kuno, and Naef \cite{AKKN_genus_zero,AKKN_GT_Formality,AKKN18highergenus}. Algebraically, these solutions encode compatibility conditions for iterated applications of the Baker--Campbell--Hausdorff series. 

\begin{definition}
\label{def: KV solutions}
A solution to the Kashiwara--Vergne (KV) problem of type $(0,n+1)$ is an element $F\in\TAut_n$ such that
\begin{align}
    F(e^{x_1}\cdots e^{x_n}) & =e^{x_1+\cdots +x_n}, 
    \label{SolKVI}
    \tag{SolKVI} \\
    \exists \, h \in z^2\mathbb{K}[[z]] \quad \text{such that} \quad J(F) & =\trace\left( h\left( \sum_{i=1}^{n} x_i\right)-\sum_{i=1}^{n}h(x_i)\right). \label{SolKVII}
    \tag{SolKVII}
\end{align} 
\end{definition}

\begin{remark}\label{rem:kvsol}
Our definition of KV solutions aligns with the conventions used in~\cite{AET10}, but corresponds to the inverses of the solutions appearing in~\cite{AT12} and~\cite{AKKN_genus_zero}.
\end{remark}

The formal power series $h \in z^2\mathbb{K}[[z]]$ appearing in~\eqref{SolKVII} is called the \defn{Duflo function}. It is uniquely determined by a given KV solution $F$~\cite{AET10}, and it plays an essential role in composing KV solutions. We write $\solkv(n)$ for the set of KV solutions of type $(0,n+1)$, or equivalently, of arity $n$. When it is useful to distinguish the case $n=2$ from its higher-arity analogues, we refer to elements of $\solkv(2)$ as \defn{classical KV solutions}.

The existence of classical KV solutions was first established in \cite{AM06}. In \cite{AT12}, Alekseev and Torossian gave an alternative proof for solutions of type $(0,2+1)$ by exploiting the relationship between Drinfeld associators and the associativity of the Baker--Campbell--Hausdorff series:
\[
\mathfrak{bch}(x_1,\mathfrak{bch}(x_2,x_3))
=
\mathfrak{bch}(\mathfrak{bch}(x_1,x_2),x_3).
\]
In particular, if $F\in\TAut_2$ satisfies \eqref{SolKVI}, then
\begin{equation*}
F^{12,3}F^{1,2}(\mathfrak{bch}(\mathfrak{bch}(x_1, x_2), x_3))
=
x_1+x_2+x_3
=
F^{1,23}F^{2,3}(\mathfrak{bch}(x_1, \mathfrak{bch}(x_2, x_3))).
\end{equation*}
Here we again use the standard cosimplicial notation for tangential automorphisms. Thus $F^{1,2}$ denotes the image of $F\in\TAut_2$ in $\TAut_3$, acting nontrivially on $x_1,x_2$ and fixing $x_3$. Similarly, $F^{2,3}$ acts on $x_2,x_3$, $F^{12,3}$ acts on the pair $(x_1+x_2,x_3)$, and $F^{1,23}$ acts on $(x_1,x_2+x_3)$.

Associated to such an $F$ is the special automorphism
\begin{equation}
\label{eq:KV_Ass-construction}
G_F
\eqdef F^{1,23}F^{2,3}(F^{1,2})^{-1}(F^{12,3})^{-1}
= (F \circ_2 F)(F\circ_1 F)^{-1}.
\end{equation}
This automorphism satisfies the \defn{pentagon equation} in $\TAut_4$:
\begin{equation}
\label{eq:pentagon in Taut} \tag{P}
G_F^{1,2,34} G_F^{12,3,4} = G_F^{2,3,4} G_F^{1,23,4} G_F^{1,2,3}.
\end{equation}
We will discuss this further in Section~\ref{subsec:invol}.
\medskip 

\begin{definition}[{\cite[Definition 9.1]{AT12}}]
\label{defn: KV Associator}
Given a classical KV solution $F$, the associated special automorphism
\[
G_F \eqdef F^{1,23}F^{2,3}(F^{1,2})^{-1}(F^{12,3})^{-1} \in \SAut_3
\]
is called a \defn{KV associator} if it satisfies the pentagon equation \eqref{eq:pentagon in Taut}, together with the inversion and hexagon equations
\begin{gather}
    G^{1,2,3}G^{3,2,1} =1, \label{eq:KV-inversion} \tag{U}\\
    e^{\frac{\inner^{1,2}+\inner^{1,3}}{2}}=(G^{2,3,1})^{-1} e^{\frac{\inner^{1,3}}{2}} G^{2,1,3} e^{\frac{\inner^{1,2}}{2}} (G^{1,2,3})^{-1}, \label{eq:KV-Hex1}\tag{H1} \\
    e^{\frac{\inner^{1,3}+\inner^{2,3}}{2}} = G^{3,1,2} e^{\frac{\inner^{1,3}}{2}}(G^{1,3,2})^{-1} e^{\frac{\inner^{2,3}}{2}} G^{1,2,3}. \label{eq:KV-Hex2} \tag{H2}
\end{gather}
\end{definition}
There is an injective Lie algebra homomorphism $\iota:\ib_n\to \sder_n$ given by
\[
\iota(t_{ij})=t^{ij}\eqdef (0,\ldots,\underbrace{x_j}_{i},\ldots,\underbrace{x_i}_{j},\ldots,0),
\]
see also Example~\ref{example: generating DK Lie algebras as derivations}. Exponentiating yields an injective group homomorphism $\exp(\ib_n)\hookrightarrow \SAut_n,$ and hence an operad inclusion $\iota:\CD\hookrightarrow \SAut.$  Under this inclusion, the Drinfeld associator relations map to the defining relations of a KV associator. In particular, if $(1,f)$ is a Drinfeld associator, with $f\in \exp(\lie_2)\subset \exp(\ib_3)$, then $\iota(f)\in \SAut_3$ satisfies the pentagon, inversion, and hexagon equations of Definition~\ref{defn: KV Associator}.

If the special automorphism $G_F$ lies in the image of the inclusion $\exp(\ib_3)\hookrightarrow \SAut_3$, then we say that the corresponding KV associator comes from a Drinfeld associator. In Section~9 of \cite{AT12}, the authors combine this observation with the existence of Drinfeld associators to prove the following theorem.

\begin{thm}[{\cite[Theorem~9.6]{AT12}}]
The KV problem of type $(0,2+1)$ admits solutions.
\end{thm}

In \cite[Lemma~7.3]{AKKN_genus_zero}, the authors show that solutions of type $(0,n+1)$ can be obtained operadically from solutions of type $(0,2+1)$ by iterated composition. We return to these composition operations in Appendix~\ref{sec: operad of KV solutions}.

The KV solutions of type $(0,n+1)$ naturally form a bitorsor under two symmetry groups: the group $\kv(n)$ and its graded analogue $\krv(n)$. These groups are defined by imposing natural exponential and Jacobian conditions on tangential automorphisms, and they act on opposite sides of the space of solutions.

\begin{definition}
The \defn{graded Kashiwara--Vergne group} of type $(0,n+1)$, denoted $\krv(n)$, consists of tangential automorphisms $G\in\TAut_n$ satisfying
\begin{gather}
\label{KRVI}
\tag{$\krv \text{I}$}
G(e^{x_1+\cdots + x_n})= e^{x_1+\cdots + x_n}, \\
\label{KRVII}\tag{$\krv\text{II}$}
\exists\, h(z)\in z^2\mathbb{K}[[z]] \quad \text{such that} \quad J(G)=\trace\left(
h\left(\sum_{i=1}^{n}x_i\right)-
\sum_{i=1}^{n}h(x_i)\right).
\end{gather}
The group law is given by composition of tangential automorphisms, as in \eqref{eq: group law on TAut}.
\end{definition}

\begin{definition}
The \defn{Kashiwara--Vergne group} of type $(0,n+1)$, denoted $\kv(n)$, consists of the tangential automorphisms $G\in\TAut_{n}$ satisfying: 
\begin{gather}\label{KV'I}\tag{$\kv \text{I}$}
G(e^{x_1}\cdots e^{x_n})= e^{x_1}\cdots e^{x_n}, \\
\label{KV'II}\tag{$\kv\text{II}$}
\exists \, h(z)\in z^2\mathbb{K}[[z]] \quad \text{such that} \quad J(G)=\trace\left(
h\left(\bch(x_1,\ldots,x_n)\right)-
\sum_{i=1}^{n}h(x_i)\right).
\end{gather} 
The group structure is given by composition of tangential automorphisms \eqref{eq: group law on TAut}.
\end{definition}

The defining conditions of $\kv(n)$ and $\krv(n)$ are arranged so that these groups act on the space of KV solutions from opposite sides. The corresponding bitorsor structure is the following.

\begin{thm}[{\cite[Theorem~8.4]{AKKN18highergenus}}]
\label{thm:free-transitive-krv}
The group $\krv(n)$ acts freely and transitively on the right of $\solkv(n)$ by left multiplication by the inverse. 
The group $\kv(n)$ acts freely and transitively on the left of $\solkv(n)$ by right multiplication with the inverse. 
Moreover, these two actions commute, making the set $\solkv(n)$ into a bitorsor.
\end{thm}

Explicitly, for $F \in \solkv(n)$ and $G \in \krv(n)$, the right action is given by
\[
F \cdot G \coloneqq G^{-1}F,
\]
while for $G \in \kv(n)$, the left action is
\[
G \cdot F \coloneqq FG^{-1}.
\]
Special cases of this bitorsor structure appear in~\cite[Remark~8.9]{AKKN_genus_zero},~\cite[Theorem~5.7]{AT12}, and~\cite[Proposition~8]{AET10}.

The following theorem shows that the families $\{\solkv(n)\}_{n\geq 2}$, $\{\kv(n)\}_{n\geq 2}$, and $\{\krv(n)\}_{n\geq 2}$ all carry a natural non-symmetric group-colored operadic structure. This fact is used in \cite[Section~7]{AKKN_genus_zero} and \cite{AET10}. We give further discussion, examples, and proofs in Appendix~\ref{sec:Operadapp}; see Theorem~\ref{thm:operad-SolKVapp} and Proposition~\ref{thm:kv-krv-operads}.
\begin{theorem}
\label{thm:operad-SolKV}
$ $
\begin{enumerate}[label=(\roman*)]
\item The family $\solkv \eqdef \{\solkv(n)\}_{n \geq 2}$ forms a $\mathbb{K}[[z]]$-colored non-symmetric operad. Namely for $F =(F,h_1) \in \solkv(m)$ and~$G= (G,h_2)\in \solkv(n)$ two KV solutions with $h_1=h_2$, then $F \circ_i G$ is still a KV solution.
\item Similarly, the families $\{\krv(n)\}_{n \geq 2}$ and $\{\kv(n)\}_{n \geq 2}$ form non-symmetric $\mathbb{K}[[z]]$-colored operads in the category of sets. 

\item For $F \in \solkv(2), G \in \krv(2)$, given any iterated composition of the form
\[
\tilde{F} = (\cdots((F \circ_{i_1} F) \circ_{i_2} F) \cdots \circ_{i_k} F) \in \solkv(n),
\quad
\tilde{G} = (\cdots((G \circ_{i_1} G) \circ_{i_2} G) \cdots \circ_{i_k} G) \in \krv(n),
\]
the action satisfies
\[
\tilde{F} \cdot \tilde{G} = (\cdots((F \cdot  G \circ_{i_1} F \cdot G) \circ_{i_2} F \cdot G) \cdots \circ_{i_k} F \cdot G) \in \solkv(n).
\]
where  $ F \cdot G$ as before denotes the right action of $\krv(2)$ on $\solkv(2)$. An analogous statement holds for $\kv(n)$ acting on the left.
\end{enumerate}
\end{theorem}

\begin{remark}\label{rem:krvsaut}
    Note that by (\ref{KRVI}), we have that $\krv(n) \subseteq \SAut_n$, and therefore in that case $\circ_i$ is a group homomorphism, making $\{\krv(n)\}_{n \geq 2}$ a non-symmetric $\K[[z]]$-colored operad in prounipotent groups. We can achieve the same for $\kv(n)$ by defining $\widetilde{\sder}=\{u\in\tder ~:~u(\bch(x_1,\hdots,x_n))=0\}$ with operadic composition defined the same way as in $\sder$ but with ``$+$'' replaced by ``$\bch$'', and expontentiating that to get $\widetilde{\SAut}_n$, then $\kv(n) \subseteq \widetilde{\SAut}_n$.
\end{remark}

\subsubsection{Symmetric KV solutions}
\label{ss:symmetric-KV-solutions}
An \defn{inner derivation} of a Lie algebra $\mathfrak{g}$ is a derivation arising from the adjoint action of $\mathfrak{g}$ on itself. 
The tangential derivation $t=(x_2,x_1)$ (see also \cref{example: generating DK Lie algebras as derivations}) is equivalently described as the inner derivation given by the adjoint action of $(x_1+x_2)\in\lie_2$. That is because the inner derivation $\inner:=\operatorname{ad}_{x_1+x_2}$ acts on a Lie word $a\in\lie_2$ as 
\[
\inner(a)=[a, x_1+x_2]=[a, x_1]+[a, x_2].
\]   
Indeed, $\inner(x_1)=[x_1,x_2]$ and $\inner(x_2)=[x_2,x_1]$, as in~\cref{example: generating DK Lie algebras as derivations}.

Denote by $\tw$ the tangential derivation $\tw \eqdef (0, x_1) \in \tder_2$, and set $\etw\eqdef\exp(\tw)$. 
The inner automorphism~$e^{\inner}$, together with the tangential automorphism~$\etw$, are used to define an involution on the set of KV solutions of type $(0,2+1)$.

\begin{prop}[{\cite[Proposition~8.4]{AT12}}]
\label{prop:RtauInvolution}
The automorphism $\tau : \solkv(2) \to \solkv(2)$ defined by $$\tau(F) := e^{-\inner/2}F^{2,1}\etw$$ defines an involution on the set of solutions to the KV problem. That is, if $F \in \solkv(2)$, then $\tau(F) \in \solkv(2)$ and $\tau^2(F) = F$. 
\end{prop}

\begin{remark}
Due to opposite braiding conventions with respect to those in \cite{AT12}, the automorphism $\etw$ here is denoted $R^{2,1}$ in \cite{AT12}.
\end{remark}

\begin{definition}
\label{def:SolKVsymmetric}
A KV solution $F\in\solkv(2)$ is \defn{symmetric} if $\tau(F)=F$. 
We write $\solkv^\tau(2) \subseteq \solkv(2)$ for the set of symmetric KV solutions. 
\end{definition}

As shown in \cite[Proposition~8.9]{AT12}, the subgroups 
\[
    \krvs(2) \eqdef \{ G \in \krv(2) \ | \ G^{2,1}=G \}
    \quad \text{ and } \quad
    \kvs(2) \eqdef \{ G \in \kv(2) \ | \ G^{2,1}=G \}
    \]
of $\krv(2)$ and $\kv(2)$ make $\solkv^{\tau}(2)$ into a bitorsor.
In particular, combining \cite[Theorem~5.7]{AT12} and \cite[Proposition~8.9]{AT12} we have the following.

\begin{prop}
\label{symmetry-groups-from-KV-solutions}
    Given two symmetric KV solutions $F_1,F_2 \in \solkv^{\tau}(2)$, we have
    \[
    F_1 F_2^{-1} \in \krvs(2) 
    \quad \text{and} \quad
    F_1^{-1} F_2 \in \kvs(2).
    \]
\end{prop}

\begin{proof}
    For convenience, we reprove the $\krvs$ statement directly; the proof for $\kvs$ is similar.
    First, observe that since both $F_1$ and $F_2$ satisfy~\eqref{SolKVI}, we have  
\[
F_1F_2^{-1}(e^{x_1+x_{2}}) 
= F_1(e^{x_1}e^{x_2}) 
=  e^{x_1+x_{2}},
\]  
which shows that $F_1F_2^{-1}$ satisfies~\eqref{KRVI}.
To show that $F_1F_2^{-1}$ satisfies \eqref{KRVII}, we use the $1$-cocycle property of the Jacobian.
We have
\begin{align}
\label{cocycle-step}
J(F_1 F_2^{-1}) = J(F_1) + F_1 J(F_2^{-1})
= J(F_1) - F_1 F_2^{-1} J(F_2).
\end{align}
Since $F_1$ and $F_2$ are KV solutions, we have $J(F_1)=\trace (h_1(x_1+x_2)-h_1(x_1)-h_1(x_2))$ and $J(F_2)=\trace(h_2(x_1+x_2)-h_2(x_1)-h_2(x_2))$ for some Duflo functions $h_1,h_2$ in $z^2\mathbb{K}[[z]]$.
We claim that the action of~$F_1F_2^{-1}$ on~$J(F_2)$ is trivial: 
as~$F_1F_2^{-1}$ satisfies \eqref{KRVI}, it preserves the sum~$x_1+x_2$ and thus the series~$h_2(x_1+x_2)$; moreover, it acts by conjugation on each of~$h_2(x_1)$ and~$h_2(x_2)$, which cancels under the trace. 
Therefore, \eqref{cocycle-step} becomes 
\begin{align*}
J(F_1 F_2^{-1}) 
= J(F_1) - J(F_2)
=\trace(h(x_1+x_2)-h(x_1)-h(x_2)),
\end{align*}
where $h \eqdef h_1 - h_2$ is the difference of the Duflo functions. 
We conclude that $F_1F_2^{-1}$ satisfies~\eqref{KRVII}.

Thus, we have that $F_1F_2^{-1}$ is in $\krv(2)$. 
Using that $F_1$ and $F_2$ are symmetric, one can compute $(F_1F_2^{-1})^{2,1}$ directly as follows:
\begin{align*}
    (F_1F_2^{-1})^{2,1} = (e^{\inner/2}F_1\etw^{-1})(e^{\inner/2}F_2\etw^{-1})^{-1} 
    = e^{\inner/2} F_1 F_2^{-1} e^{-\inner/2} 
    =F_1 F_2^{-1}
\end{align*}
Thus, $ F_1 F_2^{-1} \in \krvs(2)$ as required.
\end{proof}

Out of symmetric KV solutions, in \cite{AT12} the authors build special automorphisms which satisfy pentagon and hexagon equations -- the defining equations of Drinfeld associators, but in a different space -- as follows: 

\begin{prop}[{\cite[Proposition~7.1]{AT12}}]
\label{prop:pentagon-KV-associator}
Let $F\in\solkv(2)$ be a KV solution.
Then, the tangential automorphism
\begin{equation*} 
\label{eq:phi-construction}
G_F \eqdef F^{1,23}F^{2,3}(F^{1,2})^{-1}(F^{12,3})^{-1}
\end{equation*} 
is an element of $\krv(3) \subset \SAut_3$ and satisfies the pentagon equation in $\TAut_4$:
\begin{equation} 
\label{eq:pentagon in Taut}
 G^{1,2,34} G^{12,3,4} =  G^{2,3,4} G^{1,23,4} G^{1,2,3}. \tag{P}
\end{equation}  
\end{prop}

\begin{proof}
We give a short alternative proof. By \cref{thm:operad-SolKV} we know that $F\circ_1 F=F^{12,3}F^{1,2}$ and $F\circ_2 F=F^{1,23}F^{2,3}$ are in $\solkv(3)$.
    It follows that $G_F=(F\circ_2 F)(F \circ_1 F)^{-1}$ is an element of $\krv(3)$.
    Similarly, we have that both $(F\circ_1 F)\circ_1 F$ and $F \circ_2 (F \circ_2 F)$ are elements of $\solkv(4)$.
    Now observe that both sides of~\eqref{eq:pentagon in Taut} are elements of $\krv(4)$ sending $(F\circ_1 F)\circ_1 F$ to $F \circ_2 (F \circ_2 F)$. 
    Since the $\krv(4)$ action on $\solkv(4)$ is free and transitive (\cref{thm:free-transitive-krv}), the two sides must be equal. 
\end{proof}

If $F$ is a symmetric KV solution then more is true: $G_F$ and  $e^{\inner}$ satisfy all of the remaining defining equations of Drinfeld associators in $\TAut_3$, as follows. 

\begin{prop}[{\cite[Proposition~8.11]{AT12}}]
\label{prop:equations-KV-associator}
Let $F\in\solkv^{\tau}(2)$ be a symmetric KV solution, and let $G \eqdef G_F$ be the corresponding solution to the pentagon equation. 
Then $G_F \in \SAut_3$ is a KV-associator (as in Definition~\ref{defn: KV Associator}).
\end{prop}

\begin{prop}
\label{prop:symmetric-iff-associator-symmetric}
    Let $F \in \solkv(2)$ be a KV solution.
    Then we have,
    \[
    \tau(F)=F \iff G_F=G_{\tau(F)}.
    \]
\end{prop}

\begin{proof}
    Suppose that $\tau(F)=F$. 
    Then, we have
    \[
    G_F 
    = F^{1,23}F^{2,3}(F^{1,2})^{-1}(F^{12,3})^{-1}
    = \tau(F)^{1,23}\tau(F)^{2,3}(\tau(F)^{1,2})^{-1}(\tau(F)^{12,3})^{-1}
    =
    G_{\tau(F).}
    \]
    For the reverse direction, suppose that we have $G_F = G_{\tau(F)}$.
    By \cite[Proposition~8.5]{AT12}, we have
    $G_{\tau(F)}=(G_F^{3,2,1})^{-1}$.
    The fact that $F$ is a symmetric KV solution then follows from \cite[Proposition~9.3]{AT12}. 
\end{proof}

Among classical KV solutions, those whose associated KV associators arise from Drinfeld associators carry an additional Grothendieck--Teichm\"uller symmetry. Equivalently, for this distinguished class of KV associators, the natural bitorsor actions of $\kv(2)$ and $\krv(2)$ restrict to the familiar Grothendieck--Teichm\"uller symmetries.

Recall that $\gt_1 \subset \gt$ denotes the subgroup of the Grothendieck--Teichm\"uller group consisting of elements $(\lambda,f)$ with $\lambda=1$, and that $\grt_1$ is its graded analogue. In~\cite[Theorem~9]{AET10} and~\cite[Theorem~4.6]{AT12}, the authors construct injective group homomorphisms
\[
\gt_1 \hookrightarrow \kv(2)
\qquad \text{and} \qquad
\grt_1 \hookrightarrow \krv(2).
\]
These homomorphisms identify the pentagon and hexagon relations defining $\gt_1$ and $\grt_1$ with their counterparts in the KV setting, arising from iterated applications of the Baker--Campbell--Hausdorff formula.

\section{Constructing KV Solutions}\label{section: Construction of KV solutions}
In~\cite{AET10}, the authors give an explicit construction of a classical KV solution $F_\varphi$ associated to a Drinfeld associator $\varphi:\hPaB\to \CD$ with coupling constant $\mu=1$. In this section, we reinterpret that construction in moperadic terms. More precisely, to each moperad equivalence
\[
(\varphi^1,\varphi):\hPaB^1 \longrightarrow \CD^+
\]
we associate a KV solution of type $(0,2+1)$. This point of view makes it transparent that the KV equations arise from the defining relations of the $\PaB$-moperad $\hPaB^1$.

The construction is in fact more general. Given a moperad equivalence
\[
(\varphi^1,\varphi):\hPaB^1 \longrightarrow \CD^+,
\]
we associate to each fully parenthesized permutation $w\in \ob(\hPaB^1(n))$ a tangential automorphism $F_w\in \TAut_n$. We will show that the automorphism $F_{0(12)}$ is a KV solution of type $(0,2+1)$. More generally, if $w$ is a full parenthesization of the identity permutation $12\cdots n$, then the associated automorphism $F_{0w}$ is a KV solution of type $(0,n+1)$.

By Theorem~\ref{presenation of PaB1}, the $\PaB$-moperad $\hPaB^1$ is finitely presented, and hence by Theorem~\ref{thm: assoc^1}, any moperad equivalence
\[
(\varphi^1,\varphi):\hPaB^1 \longrightarrow \CD^+
\]
is uniquely determined by the values
\[
\varphi(R^{1,2})=e^{\mu t_{12}/2} \qquad \text{and} \qquad
\varphi(\Phi^{1,2,3})=f(t_{12},t_{23}),
\]
together with
\[
\varphi^1(E^{0,1})=e^{\mu t_{01}} \qquad \text{and} \qquad
\varphi^1(\Psi^{0,1,2})=g(t_{01},t_{12}).
\]
Here the pair $\bigl(\varphi(R^{1,2}),\varphi(\Phi^{1,2,3})\bigr)$ determines a Drinfeld associator, while the additional datum $g$ is constrained by equations~\eqref{eq: MP moperad iso} and~\eqref{eq: hexagons moperad iso}. Proposition~\ref{Enriquez-wisdom} tells us that, for fixed $\varphi$, there is only a one-parameter family of possible choices of $\varphi^1$, including the shifted solution $g=f$.  The construction below generalizes \cite[Proposition~28]{AET10}.

\begin{construction}
\label{construction: TAut from local data}
Fix a moperad equivalence $(\varphi^1,\varphi):\hPaB^1 \rightarrow \CD^+$ and a parenthesized permutation $w \in \ob(\hPaB^1(n))$.  
We consider the local isomorphism 
\begin{equation}
\label{eq: local iso}
\varphi^1_w:\Aut_{\hPaB^{1}(n)}(w)\cong (\PB_n^1)_{\K}\longrightarrow \exp(\mathfrak{t}_n^1)
\end{equation}
of prounipotent groups induced by $(\varphi^1,\varphi)$.

\begin{enumerate}[leftmargin=*,label=(\alph*)]
\item Using that $\PB_n^1 \cong F_n \rtimes \PB_n$ (Equation~\eqref{eq:PB1}), we restrict the isomorphism \eqref{eq: local iso} to an isomorphism of free groups
\begin{equation}
\label{eq: local iso free groups}
\barphi^1_w: (\F_n)_{\K}\longrightarrow \exp(\lie_n),
\end{equation}
which is completely determined by the image of the generators $X_1,\ldots, X_n$ of $(\F_n)_{\K}$.

\item By \cref{l:generators-normal-form}, as a morphism in $\Aut_{\hPaB^1(n)}(w)$, each generator $X_k$ of $(\F_n)_{\K}$ can be (non-uniquely) represented via operadic, monoidal, and categorical compositions of the generating morphisms of~$\hPaB^1$. 
In particular, $X_k\in\Aut_{\hPaB^1(n)}(w)$ is obtained by conjugating $E^{0,k}\in\Aut_{\hPaB^1(n)}(l^0_n( k \, 1 \, 2 \hdots k-1 \, k+1 \hdots n))$ 
by twists and associativity isomorphisms.

\item 
In turn, this determines the values $\barphi^1_w$ takes on the generators $X_k$. Indeed, a straightforward calculation shows that $\barphi^1_w(X_k) = e^{-z_k}e^{t_{0k}}e^{z_k}$ for some Lie word $z_k\in\ib^+_{n}$, as in Section~\ref{subsec:assocmoperad}. 

\item We rewrite these conjugations as $\barphi^1_w(X_k) = e^{-z_k}e^{t_{0k}}e^{z_k}=e^{-z'_k}e^{t_{0k}}e^{z'_k}$, where $z'_k\in \lie_n$. This is done via the homomorphism $\ib_n^+ \to \tder_n$, given by $t_{ij} \mapsto (0,\ldots,x_j,\ldots,x_i,\ldots,0)$, and $t_{0i}\mapsto(x_i,x_i,\ldots,x_i)$. Here $1\leq i,j\leq n$, and the values $x_j$ and $x_i$ are placed in position $i$ and $j$, respectively. 
Then, the action of $\ib^+_n$ on its Lie subalgebra $\lie_n$  (generated by $t_{01},\hdots,t_{0n}$) agrees with the action of the image in $\tder_n$. 
Hence, $\exp(t_n^+)$ acts on $(\F_n)_{\K}$ by tangential automorphisms, which gives the second equality above. This is \cite[Proposition~20]{AET10}.

\item Completing this construction for each generator $X_1,\ldots, X_n$ of $(\F_n)_{\K}\subset \Aut_{\hPaB^1(n)}(w)$, we find that $(\varphi^1,\varphi)$ restricts to a tangential automorphism of $(\F_n)_{\K}$ (also using $\exp(\lie_n) \cong (\F_n)_{\K}, e^{t_{0i}} \mapsto X_i$). Denote this tangential automorphism by \[F_{w,\varphi^1,\varphi} \eqdef (e^{z'_1},\ldots, e^{z'_n}).\]

\end{enumerate}
\end{construction}

\medskip

\begin{notation}
We denote by $F_{w,\varphi^1,\varphi} \eqdef (e^{z'_1},\ldots, e^{z'_n})$ the tangential automorphism associated to the moperad equivalence $(\varphi^1,\varphi): \hPaB^1 \rightarrow \CD^+$ and a choice of parenthesized permutation $w \in \mathrm{Ob}(\hPaB^1(n))$. To simplify notation, we will often suppress the dependence on $(\varphi^1,\varphi)$ and write $F_w$ in place of $F_{w,\varphi^1,\varphi}$ whenever the context makes the choice of expansion clear.
\end{notation}

\medskip

\begin{example}
\label{ex:classical-KV}
Let $l^0_2=((01)2)$ denote the leftmost parenthesization of the identity permutation in $\ob(\hPaB^1(2))$. The generators of the free group $(\F_2)_{\K}$, viewed as automorphisms in $\Aut_{\hPaB^1(2)}(l^0_2)$, can be presented as
\[
X_1 = 
E^{0,1} 
\quad \text{and}\quad 
X_2= (\Psi^{0,1,2})^{-1} (R^{2,1})^{-1} \Psi^{0,2,1} E^{0,2} (\Psi^{0,2,1})^{-1} R^{1,2} (\Psi^{0,1,2}).
\] 
Applying a moperad equivalence $(\varphi^1,\varphi):\hPaB^1\rightarrow \CD^+$, we compute that 
\begin{gather}
\barphi^1(X_1)=\varphi(E^{0,1})=e^{\mu t_{01}} \quad \text{and}\\
\label{X2 for l2}
\barphi^1(X_2)
=
\left(g(t_{02},t_{12})^{-1}e^{\frac{\mu t_{12}}{2}}g(t_{01},t_{12})\right)^{-1} e^{\mu t_{02}}\left(g(t_{02},t_{12})^{-1}e^{\frac{\mu t_{12}}{2}} g(t_{01},t_{12})\right). 
\end{gather} Where we note that we have written the formula for $\barphi^1(X_2)$ in terms of group multiplication.

Since $t_{01}+t_{02}+t_{12}$ is central in $\mathfrak{t}_{2}^+\cong\ib_3$, we can replace $t_{12}$ by $-t_{01}-t_{02}$ and rewrite \eqref{X2 for l2} as \begin{equation*} \big(g(t_{02},-t_{01}-t_{02})^{-1}e^{\frac{\mu (-t_{01}-t_{02})}{2}}g(t_{01},-t_{01}-t_{02})\big)^{-1} e^{\mu t_{02}}\left(g(t_{02},-t_{01}-t_{02})^{-1}e^{\frac{\mu (-t_{01}-t_{02})}{2}} g(t_{01},-t_{01}-t_{02})\right).
\end{equation*}
Thus, we obtain a tangential automorphism $F_{l^0_2,\varphi^1,\varphi}:(\F_2)_{\K}\rightarrow \exp(\lie_2)\cong(\F_2)_{\K}$ defined by restricting $(\varphi,\varphi^1)$ to $(\F_2)_{\K}$. As a two-tuple, this is given by
$$F_{l^0_2,\varphi^1,\varphi}=\left(1, g(t_{02},-t_{01}-t_{02})^{-1}e^{\frac{-\mu (t_{01}+t_{02})}{2}} g(t_{01},-t_{01}-t_{02})\right).$$ 
\end{example} 

If we let $r^1_2=(0(12))$ denote the rightmost parenthesization of the identity permutation in $\ob(\hPaB^1(2))$, a similar calculation to that in Example~\ref{ex:classical-KV} gives the tangential automorphism: \begin{equation}
\label{AET automorphism}
\tag{$F_{0(12)}$}
F_{r_{2}^{1},\varphi^1,\varphi}=\left(g(t_{01}, -t_{01}-t_{02}), g(t_{02},-t_{01}-t_{02})e^{\frac{-\mu (t_{01}+t_{02})}{2}}\right).
\end{equation} 

For the remainder of this section, we fix an equivalence of completed moperads $(\varphi^1,\varphi):\hPaB^1 \rightarrow \CD^+$ given as usual by the values 
\begin{equation}
\label{eq:values} \tag{$\dagger$}
\varphi(R^{1,2})=e^{\frac{t_{12}}{2}}, \quad \varphi(\Phi^{1,2,3})=f(t_{12},t_{23}), \quad \varphi^1(E^{0,1})=e^{t_{01}}, \quad \text{and} \quad \varphi^1(\Psi^{0,1,2})=g(t_{01},t_{12}).
\end{equation}
One of the main results of this paper (Theorem~\ref{thm: the construction is good}) is that the tangential automorphism $F_{0(12)}$ is a KV solution of type $(0,2+1)$. This is parallel to ~\cite[Theorem~4]{AET10}. We begin by addressing the first KV equation.

\begin{lemma}
\label{lem:KVI-r2}
Let $(\varphi^1,\varphi):\hPaB^1 \rightarrow \CD^+$ be an equivalence of completed moperads as in \eqref{eq:values}.
Then, the associated tangential automorphism $F_{0(12)}$ satisfies the first KV equation from Definition~\ref{def: KV solutions}:
\[ 
F_{0(12)}(e^{x_1}e^{x_2})=e^{x_1+x_2}.
\]
\end{lemma}

\begin{proof}
The morphisms  
\[
 X_1 =  \Psi^{0,1,2} E^{0,1} (\Psi^{0,1,2})^{-1} \quad \text{and} \quad
    X_2 =  (R^{2,1})^{-1}\Psi^{0,2,1} E^{0,2} (\Psi^{0,2,1})^{-1} R^{1,2}.
\]
generate a copy of the free group $(\F_2)_{\K}\subset \Aut_{\hPaB^1(2)}(0(12))$.
Using first the octagon~\eqref{eqn:O}, then the right pentagon~\eqref{eqn:RP} equation we can rewrite the product $X_2X_1$ as:
\begin{eqnarray*}
    X_2X_1 
&=& (R^{2,1})^{-1}\Psi^{0,2,1} E^{0,2}(\Psi^{0,2,1})^{-1} R^{1,2}\Psi^{0,1,2} E^{0,1} (\Psi^{0,1,2})^{-1} \\
&\stackrel{\eqref{eqn:O}}{=}& (R^{2,1})^{-1}\Psi^{0,2,1}E^{0,2}E^{02,1}(\Psi^{0,2,1})^{-1}(R^{1,2})^{-1}
\\
&\stackrel{\eqref{eqn:RP}}{=}& (R^{2,1})^{-1}\Psi^{0,2,1} \left(\Psi^{0,2,1}\right)^{-1} R^{1,2} E^{0,12} R^{2,1} \Psi^{0,2,1} (\Psi^{0,2,1})^{-1}(R^{1,2})^{-1}
\\
&=& E^{0,12}.
\end{eqnarray*} 
Note that in the application of the \eqref{eqn:RP} relation in the second reduction uses that $E^{0,2}$ commutes with $E^{02,1}$: this is an equality in $\Aut_{\PaB^1}((01)2)\cong (\PB^1_2)_{\K} \subseteq (\Br^1_2)_{\K}$, given in Remark~\ref{rmk:AltPres}, relation~ (\ref{eq:XiCompMod}).

    Applying Construction~\ref{construction: element of KV}, we obtain $F_{0(12)}(e^{x_1}e^{x_2})=\barphi^1_{0(12)}(X_2X_1)=\barphi^1_{0(12)}E^{0(12)}=e^{x_1+x_2}$, as claimed.
\end{proof}

\begin{remark}
    Since we write function composition left to right, the isomorphism $\Aut_{\hPaB^{1}(n)}(w)\cong (\PB_n^1)_{\K}$ is strictly speaking an anti-isomorphism. This leads to the (somewhat confusing) equality 
    $\barphi^1_{0(12)}(X_2X_1)=F_{0(12)}(e^{x_1}e^{x_2})$.
    On the left-hand side of this equation, $X_2X_1$ denotes the function composition of morphisms which, under the identification $\Aut_{\hPaB^1(2)}(0(12))\cong (\PB_2^1)_{\K}\cong (\F_{2})_{\K}\rtimes (\PB_2)_{\K},$ is the product of group elements $e^{x_1}e^{x_2}$ in $(\F_2)_{\K}$. 
\end{remark}

The proof of~\cref{lem:KVI-r2} illustrates how the right pentagon~\eqref{eqn:RP} and octagon~\eqref{eqn:O} relations in the $\hPaB$-moperad structure on $\hPaB^1$ correspond to the first KV equation~\eqref{SolKVI}. 
Our goal is now to prove that $F_{0(12)}$ satisfies the second KV equation~\eqref{SolKVII}.
In order to do so, we will use the following generalization of \cite[Theorem~30]{AET10}.
Recall from~\cref{operad of TAut} the operad structure on $\TAut$, given explicitly for $F \in \TAut_m$ and $G \in \TAut_n$ by the formula
\[
F\circ_{i} G
= (F \circ_i 1)(1 \circ_i G) 
= F^{1,\ldots, i-1,i(i+1)\cdots(i+n-1),i+n,\ldots,m+n-1} \circ G^{i,i+1,\ldots,i+(n-1)} \in \TAut_{m+n-1}.
\] 

\medskip

\begin{thm}
\label{thm:30-AET}
    Let $(\varphi^1,\varphi):\hPaB^1 \rightarrow \CD^+$ be an equivalence of completed moperads given by the values \eqref{eq:values}.
    Given parenthesized words $0(w)\in \hPaB^1(m)$, $0(w')\in \hPaB^1(n)$ we have the following equality of the associated tangential automorphisms:
    \[
    F_{0w} \circ_i F_{0w'} = F_{0(w\circ_i w'),}
    \]
    where the composition in the subscript is the partial composition of permutations \eqref{def: composition of permutations}.
\end{thm}

\begin{proof}
Suppose that $w \in \PaB(m)$ and $w' \in \PaB(n)$.
We write $F_{0w}=(\tilde f_1,\ldots,\tilde f_m)$ and $F_{0w'}=(\tilde g_1,\ldots,\tilde g_n)$ with $\tilde f_j \in \exp(\lie_m)$ and $\tilde g_k \in \exp(\lie_n)$, and denote $f_j \eqdef \tilde f_j\circ_i 1$ and $g_k \eqdef 1\circ_i \tilde g_k$. 
With this notation, we must prove that
\begin{equation}
\label{eq:heart-of-thm30}
    F_{0(w \circ_i w')}=( f_1 ,\ldots, f_{i-1}, f_i,\ldots, f_i, f_{i+1},\ldots, f_m)(1,\ldots,1, g_1,\ldots, g_n,1,\ldots,1),
\end{equation}
First observe that since $(\varphi^1,\varphi)$ is a morphism of moperads, the following diagram commutes
\begin{center}
\begin{tikzcd}
\Aut_{\hPaB^+(m)}(0w) \times \Aut_{\hPaB(n)}(w') \arrow[rr, "\varphi^1_{0w}\times \varphi_{w'}"] \arrow[d, "\circ_i"] &  & \exp(\ib_m^+)\times \exp(\ib_n) \arrow[d, "\circ_i"] \\
\Aut_{\hPaB^1(m+n-1)}(0(w \circ_i w')) \arrow[rr, "\varphi^1_{0(w\circ_i w')}"]           & & \exp(\ib^+_{m+n-1}).       
\end{tikzcd}
\end{center}
Restricting to the free group $(\F_m)_{\K} \times \{1\}$ we get the commutative diagram
\begin{center}
\begin{tikzcd}
(\F_m)_{\K} \arrow[rr, "\barphi^1_{0w}"] \arrow[d] &  & \exp(\lie_m) \arrow[d] \\
(\F_{m+n-1})_{\K} \arrow[rr, "\barphi^1_{0(w\circ_i w')}"]           & & \exp(\lie_{m+n-1}).       
\end{tikzcd}
\end{center}
Here, the vertical arrows are explicitly given by
\begin{equation*}
\begin{matrix}
    (\F_m)_{\K} & \longrightarrow & (\F_{m+n-1})_{\K} \\
    X_j & \mapsto & 
    \begin{cases}
        X_j & \text{ if } j < i, \\
        X_{i+n-1}\cdots X_i & \text{ if } j = i, \\
        X_{j+n-1} & \text{ if } j > i,
    \end{cases}
\end{matrix}
\quad \quad
\text{and}
\quad \quad
\begin{matrix}
    \exp(\lie_m) & \longrightarrow & \exp(\lie_{m+n-1}) \\
    e^{x_j} & \mapsto & 
    \begin{cases}
        e^{x_j} & \text{ if } j < i, \\
        e^{\sum_{\ell=i}^{i+n-1}x_\ell} & \text{ if } j = i, \\
        e^{x_{j+n-1}} & \text{ if } j > i.
    \end{cases}
\end{matrix}
\end{equation*}
Specializing to the generators $X_j$ of $(\F_{m+n-1})_{\K}$ with $j \notin \{i,...,i+n-1\}$, that is, the images of $X_j\in (\F_m)_{\K}$ for $j\neq i$, we obtain
\begin{equation*}
    F_{0(w\circ_i w')}(X_j)=
    \begin{cases}
        f_j^{-1}e^{x_j} f_j & \text{ for } j <i, \\
        f_{j-n+1}^{-1}e^{x_j} f_{j-n+1} & \text{ for } j > i+n-1.
    \end{cases}
\end{equation*}
Here the isomorphism $\exp(\lie_{m+n-1})\cong (\F_{m+n-1})_{\K}$ identifies $e^{x_j}$ with $X_j$ for $j=1,\ldots,m+n-1$.
Thus, \eqref{eq:heart-of-thm30} holds for all $X_j$, $j \notin \{i,\ldots,i+n-1\}$.

It remains to show \eqref{eq:heart-of-thm30} for $X_j$, $j \in \{i,\ldots,i+n-1\}$. 
For $X_i \in (\F_m)_{\K}$, write 
\begin{equation}
    F_{0w}(X_i)=\gamma^{-1}e^{x_i}\gamma, \quad \gamma \in \exp(\ib_m^+)
\end{equation} 
as given by part (c) of \cref{construction: TAut from local data}, and recall that we have denoted by 
\begin{equation}
    F_{0w}(X_i)=\tilde f_i^{-1}e^{x_i} \tilde f_i, \quad \tilde f_i \in \exp(\lie_m)
\end{equation}
the action given by part (d) of \cref{construction: TAut from local data}.

\begin{figure}
    \includegraphics[width=7cm]{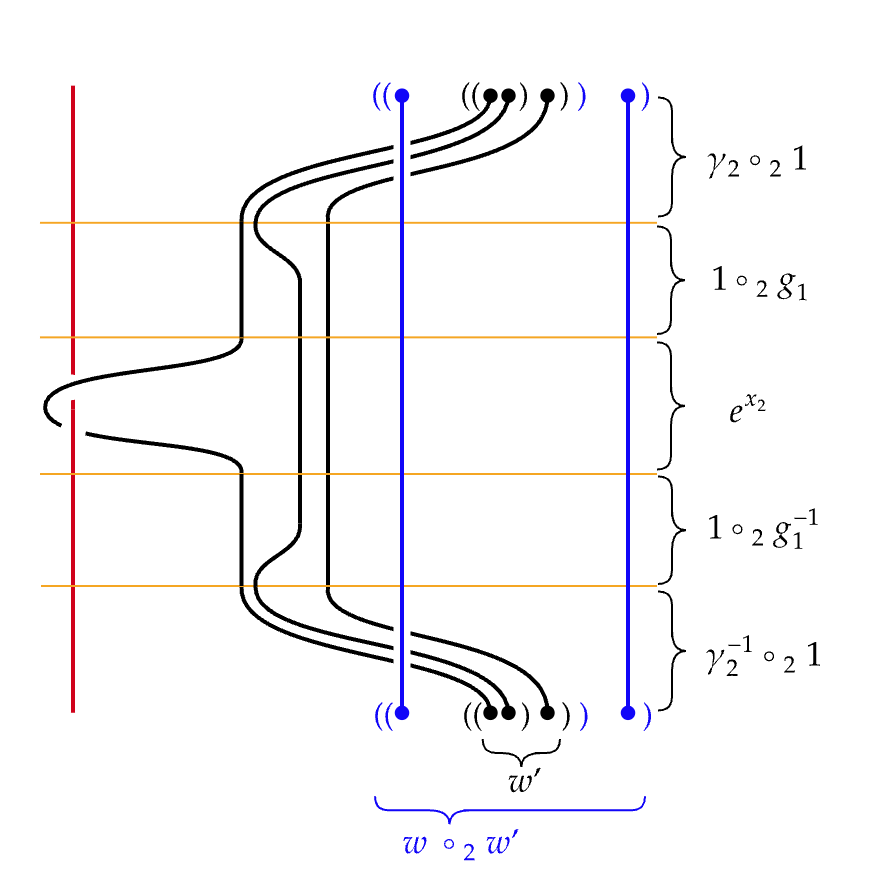}
    \caption{An example for showing the formula \eqref{eq:heart-of-thm30} for $j=i$. In this example, $w=((12)3)$, $w'=((12)3)$, and $i=2$.}\label{fig:TrickyHalf}
\end{figure}

Observe that $\tilde f_i \gamma^{-1}$  commutes with $x_i$, where the product is taken in $\ib^+_m$ via the inclusion $x_i \mapsto t_{0i}$ of $\lie_m$ into $\ib^+_m$.
The centralizer of $t_{0i}$ in $\ib^+_m$ is the set $\{\lambda t_{0i} + a^{0i,1,...,i-1,i+1,...,m}\mid \lambda \in \K \text{ and } a \in \ib_m \}$. This is a short induction argument using that $\ib_{m+1} \cong \lie_m \rtimes \ib_{m}$, as in  \cite[Proposition~51]{AET10}. Hence, we write
\begin{equation}
\label{eq:crucial-lambda}
    \tilde f_i \gamma^{-1}=e^{\lambda x_i}\alpha^{0i,1,2,\ldots,i-1,i+1,\ldots,m}, \quad \text{for some } \alpha \in \exp(\ib_m).
\end{equation}
It is possible to choose $\tilde f_i$ and $\gamma$ such that $\log \tilde f_i$ and $\log \gamma$ have no linear terms in $x_i$ (as adding a linear term in $x_i$ does not change the action on $x_i$).
Thus, taking $\log$ on both sides of Equation~\ref{eq:crucial-lambda} and comparing linear terms we get $\lambda=0$.

Further, the morphism~$X_i$ in $\Aut_{\hPaB^1(m+n)}(0(w\circ_i w'))$, shown in Figure~\ref{fig:TrickyHalf} in an example, is obtained by a composition of the action obtained from~$X_i$ in $\Aut_{\hPaB^1(m)}(0w)$, with $1_{w'}$ inserted into the $i$-th strand, with the action obtained from $X_i$ in $\Aut_{\hPaB^1(n)}(0w')$:
\begin{equation}\label{eq:ActOnXi}
    F_{0(w\circ_i w')}(X_i)=(\gamma\circ_i 1)^{-1}g_1^{-1}e^{x_i} g_1(\gamma\circ_i 1).
\end{equation}
From \eqref{eq:crucial-lambda}, we have that
\[
\gamma\circ_i 1= (\alpha^{0i(i+1)...(i+n-1),1,2,...,i-1,i+1,...,m})^{-1}f_i.
\]
Substituting this into \eqref{eq:ActOnXi}, and observing that $(\alpha^{0i(i+1)...(i+n-1),1,2,...,i-1,i+1,...,m})$ commutes with each $x_j$ for $i\leq j\leq i+n-1$, we obtain 

\begin{align*}
F_{0(w\circ_i w')}(X_i) 
& =f_i^{-1}\alpha^{0i(i+1)...(i+n-1),1,2,...,i-1,i+1,...,m}g_1^{-1}e^{x_i}g_1\left(\alpha^{0i(i+1)...(i+n-1),1,2,...,i-1,i+1,...,m}\right)^{-1}f_i  \\
& = f_i^{-1}g_1^{-1}e^{x_i} g_1 f_i 
 =(F_{0w}\circ_i F_{0w'})(X_i).
\end{align*}

The same argument shows that
\[
F_{0(w\circ_i w')}(X_{j})=f_i^{-1}g_{j-n+1}^{-1}e^{x_j}g_{j-n+1}f_i=F_{0w}\circ_i F_{0w'}(X_j)
\]
for $j=i+1,\ldots,i+n-1$.
Thus we have $F_{0(w \circ_i w')}(X_j)= (F_{0w}\circ_i F_{0w'})(X_j)$ for all $1\leqslant j \leqslant m+n-1$, completing the proof. 
\end{proof}

We illustrate the preceding theorem with an explicit example. 

\begin{example}
\label{example: F23F1,23}
Let $F=F_{0(12)}$ be the tangential automorphism from~\cref{lem:KVI-r2} and write $r_3^1=0(1(23))$ for the rightmost parenthesization of the identity permutation in $\hPaB^1(3)$.  
We calculate that 
$$F\circ_2 F= F_{0(12)}\circ_2 F_{0(12)}=F_{0(1(23))}.$$
The composition $F\circ_{2}F=F^{1,23}\cdot F^{2,3}$ is given by 
\begin{gather}
    F^{2,3}=\left(1, g(t_{01}, -t_{01}\textcolor{red}{-t_{02}-t_{03}}), g(t_{02},-t_{01}\textcolor{red}{-t_{02}-t_{03}})e^{\frac{-\mu (t_{01}\textcolor{red}{+t_{02}+t_{03}})}{2}}\right)
    \quad \text{and} \label{eq:F23} \\
    F^{1,23}=\left(g(t_{01}, -t_{01}\textcolor{red}{-t_{02}-t_{03}}), g(t_{02},-t_{01}\textcolor{red}{-t_{02}-t_{03}})e^{\frac{-\mu (t_{01}\textcolor{red}{+t_{02}+t_{03}})}{2}}, g(t_{02},-t_{01}\textcolor{red}{-t_{02}-t_{03}})e^{\frac{-\mu (t_{01}\textcolor{red}{+t_{02}+t_{03}})}{2}}\right). \label{eq:F1,23}
\end{gather}
To describe the action of $F\circ_2F$ on $X_1$ in $\Aut_{\hPaB^1(3)}(r_3^1)$, we first note that $X_1\in\Aut_{\hPaB^1(3)}(r_3^1)$ can be expressed as
\[
X_1=\Psi^{0,1,23}E^{0,1}(\Psi^{0,1,23})^{-1}
\]
 Applying the moperad equivalence $(\varphi^1,\varphi):\hPaB^1\rightarrow\CD^+$ then gives 
\begin{eqnarray*}
 \varphi^1(X_1)&=& \varphi^1\Big(\Psi^{0,1,23}E^{0,1}(\Psi^{0,1,23})^{-1}\Big)\\
 &=& g(t_{01},t_{12}+t_{13})^{-1}e^{t_{01}}g(t_{01},t_{12}+t_{13})\\
 &=& g(t_{01},-t_{01}-t_{02}-t_{03})^{-1}e^{t_{01}}g(t_{01},-t_{01}-t_{02}-t_{03})
\end{eqnarray*}
Here, the last equality is using that $(t_{12}+t_{13}+t_{23}+t_{01}+t_{02}+t_{03})$ is central in $\ib^+_3$, so $g(t_{01},t_{12}+t_{13})=g(t_{01}, -t_{01}-t_{02}-t_{03}-t_{23})$, and $t_{23}$ commutes with $t_{01}$. Comparing with Formulas (\ref{eq:F23}) and (\ref{eq:F1,23}), we see that indeed 
$\barphi_{r_3^1}^1(X_{1}) = F^{1,23}\cdot F^{2,3}(X_1).$

\end{example}

Before getting to the proof that $F_{0(12)}$ satisfies~\eqref{SolKVII}, we need one more ingredient.

\begin{prop}
\label{lem:KVII}
Let $(\varphi^1,\varphi):\hPaB^1 \rightarrow \CD^+$ be an equivalence of completed moperads given by the values~\eqref{eq:values}.
Then, the tangential automorphism $F=F_{0(12)}$ satisfies the identity
\[
f(t_{12},t_{23}) (F \circ_1 F)=(F \circ_2 F).
\]
\end{prop}

\begin{remark}
    The product $f(t_{12},t_{23})  (F \circ_1 F)$ is understood as a composition in $\TAut_3$, where $f(t_{12},t_{23})$ acts by conjugation on $\exp(\lie_3) \hookrightarrow \CD^+(3)$, identifying $e^{x_j}=e^{t_{0j}}$. In particular, if $(F\circ_1 F)(e^{t_{0j}})=z_j^{-1}e^{t_{0j}}z_j$, then $$\Big(f(t_{12},t_{23})  (F \circ_1 F)\Big)(e^{t_{0j}})=f(t_{12},t_{23})^{-1}z_j^{-1}e^{t_{0j}}z_jf(t_{12},t_{23}).$$
\end{remark}

\begin{proof}
 Theorem~\ref{thm:30-AET} implies that $F\circ_1F= F_{0((12)3)}$ and $F\circ_2 F= F_{0(1(23))}$. Our proof strategy is to calculate the images of $X_1, X_2$ and $X_3$ under $\barphi^1_{0((12)3)}$ and $\barphi^1_{0(1(23))}$ and compare the results. To do so, we use specific choices for writing $X_1, X_2$ and $X_3$ as products of the generating morphisms $R^{1,2}$, $\Psi$ and $\Phi$.
    The expressions of the generators are illustrated in Figure~\ref{fig:XisInLeft} for $0((12)3)$ and in Figure~\ref{fig:XisInRight} for $0(1(23))$, and written out in detail below. We begin with $X_1$.
    
    \begin{figure}[H]
    \includegraphics[width=12cm]{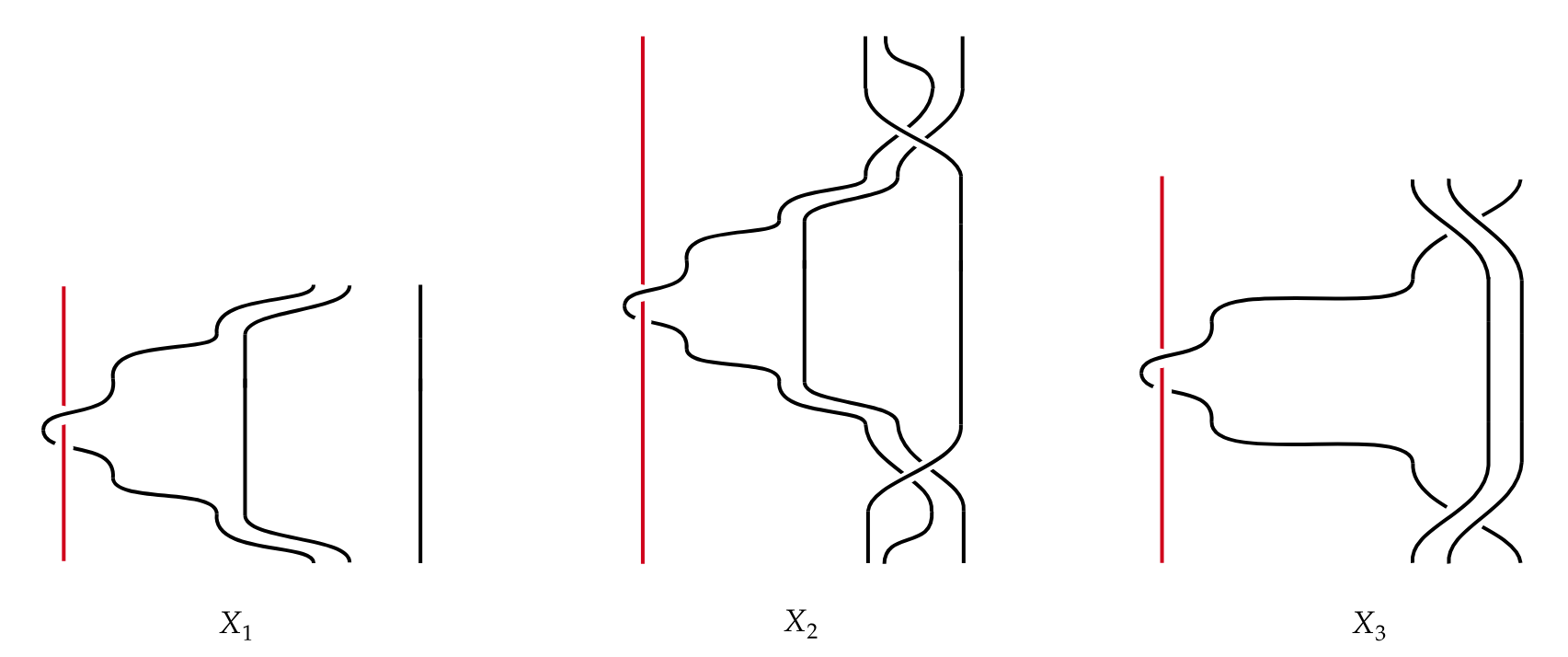}
    \caption{Expressing $X_1, X_2$ and $X_3$ as elements of $\Aut_{\PaB^1_\K(3)}\Big(0((12)3)\Big)$}.\label{fig:XisInLeft}
    \end{figure}

        \begin{figure}[H]
    \includegraphics[width=13cm]{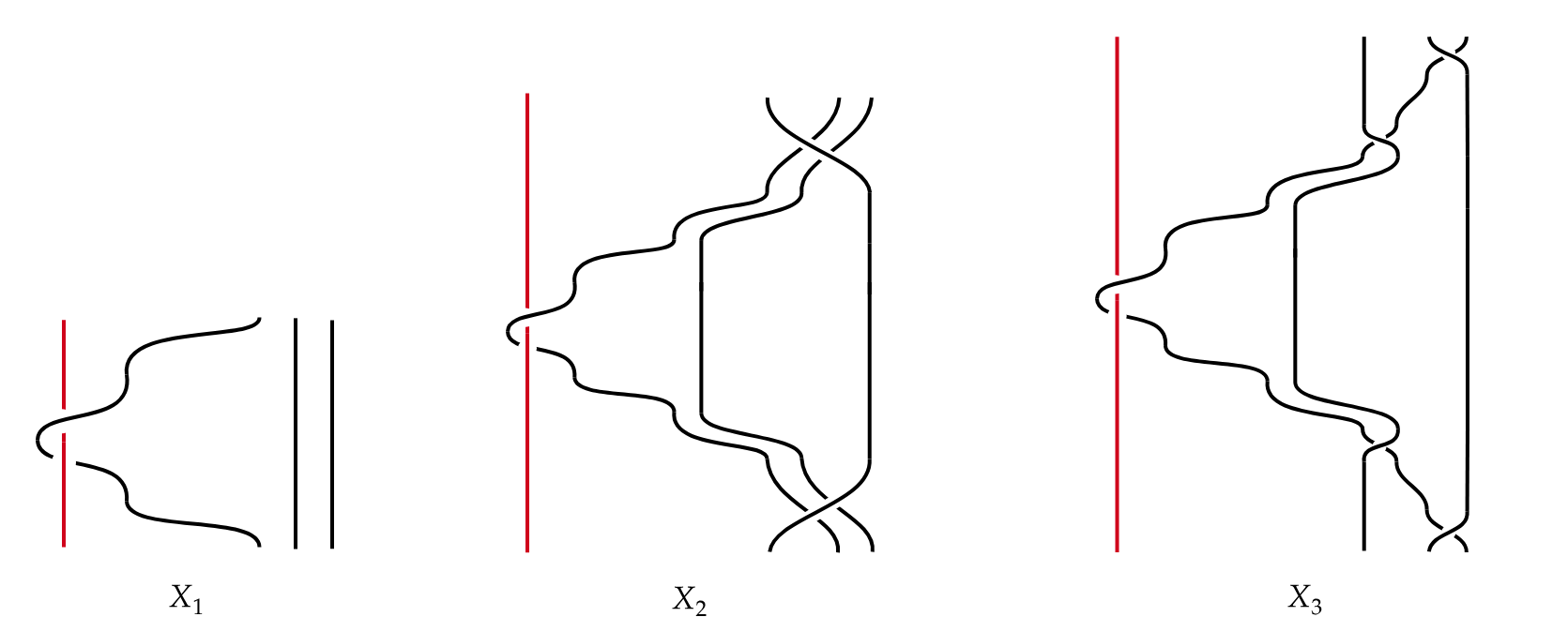}
    \caption{Expressing $X_1, X_2$ and $X_3$ as elements of $\Aut_{\PaB^1_\K(3)}\Big(0(1(23))\Big)$}.\label{fig:XisInRight}
\end{figure}

    \noindent
    {\bf Case of $X_1$.} In $\Aut_{\PaB^1_\K(3)}\Big(0((12)3)\Big)$, we have the expression 
    \begin{equation}
        X_1= \Psi^{0,12,3}\Psi^{0,1,2}E^{0,1}\left(\Psi^{0,1,2}\right)^{-1}\left(\Psi^{0,12,3}\right)^{-1} \label{eq:X1Left}
    \end{equation}
    From this expression, we follow Construction~\ref{construction: TAut from local data} to read off that the action of $(F \circ_1 F)$ on $e^{x_1}=e^{t_{01}}$ is conjugation by
    \[ z_{1_{(0((12)3))}}= g(t_{01},t_{12})g(t_{01}+t_{02},t_{13}+t_{23}).\]
    Thus, $f(t_{12},t_{23})  (F \circ_1 F)$ conjugates $e^{x_1}=e^{t_{01}}$ by
    \begin{equation}
    z_{1_{(0((12)3))}}f(t_{12},t_{23})= g(t_{01},t_{12})g(t_{01}+t_{02},t_{13}+t_{23})f(t_{12},t_{23}).\label{eq:z1Left}
    \end{equation}

    On the other hand, in $\Aut_{\PaB^1_\K(3)}\Big(0(1(23))\Big)$, we express $X_1$ as
    \begin{equation}
        X_1=\Psi^{0,1,23}E^{0,1}\left(\Psi^{0,1,23}\right)^{-1},
    \end{equation}
    and therefore, $F\circ_2F$ conjugates $e^{x_1}$ by
    \begin{equation}
    z_{1_{0(1(23)))}}=g(t_{01},t_{12}+t_{13}).
    \end{equation}
    
    We need to show that the action of $z_{1_{(0((12)3))}}f(t_{12},t_{23})$ on $e^{t_{01}}$ agrees with the action of $z_{1_{(0(1(23)))}}=g(t_{01},t_{12}+t_{13})$ on the same. First notice that $t_{01}$ commutes with $(t_{02}+t_{12})$ and with $t_{23}$, and therefore 
    \[
    [g(t_{02}+t_{12},t_{23}),t_{01}]=0.
    \]
    Therefore, the action of $z_{1_{(0(1(23)))}}=g(t_{01},t_{12}+t_{13})$ on $t_{01}$ agrees with the action of $g(t_{02}+t_{12},t_{23})g(t_{01},t_{12}+t_{13})$. 

    Recall the Mixed Pentagon relation \eqref{eqn:MP} in $\PaB^1(3)$:
    \[    \Psi^{0,1,23}\Psi^{01,2,3}=\Phi^{1,2,3}\Psi^{0,12,3}\Psi^{0,1,2}
    \]
    Therefore,
    \[    g(t_{02}+t_{12},t_{23})g(t_{01},t_{12}+t_{13})=g(t_{01},t_{12})g(t_{01}+t_{02},t_{13}+t_{23})f(t_{12},t_{23})=z_{1_{(0((12)3))}}f(t_{12},t_{23}),
    \]
    showing that 
    \[
    \Big(f(t_{12},t_{23})  (F \circ_1 F)\Big)(e^{x_1})=(F \circ_2 F)(e^{x_1}).
    \]

    \noindent
   {\bf Case of $X_2$.} Moving on to $X_2$,  in $\Aut_{\PaB^1_\K(3)}\Big(0((12)3)\Big)$, we have the expression 
\[
X_2=\left(\Phi^{1,2,3}\right)^{-1}\left(R^{1,23}\right)^{-1} \Psi^{0,23,1} \Psi^{0,2,3} E^{0,2} \left(\Psi^{0,2,3}\right)^{-1} \left(\Psi^{0,23,1}\right)^{-1} R^{1,23} \Phi^{1,2,3},
\]
from which we read (noting the cancellation at the end):
\[
z_{2_{(0((12)3))}}f(t_{12},t_{23})=g(t_{02},t_{23})g(t_{02}+t_{03},t_{12}+t_{23})e^{-\frac{t_{12}+t_{23}}{2}}
\]

On the other hand, in $\Aut_{\PaB^1_\K(3)}\Big(0(1(23))\Big)$, we express $X_2$ as:
\[
X_2=\left(R^{23,1}\right)^{-1}\Psi^{0,23,1} \Psi^{0,2,3} E^{0,2} \left(\Psi^{0,2,3}\right)^{-1}\left(\Psi^{0,23,1}\right)^{-1}R^{1,23}.
\]
Applying the moperad equivalence we get:
\[
z_{2_{(0(1(23)))}}= g(t_{02},t_{23})g(t_{02}+t_{03},t_{12}+t_{13})e^{-\frac{t_{12}+t_{13}}{2}}.
\]
Comparing $z_{2_{(0((12)3))}}$ and $z_{2_{(0(1(23)))}}$, we see that $z_{2_{(0((12)3))}}f(t_{12},t_{23})=z_{2_{(0(1(23)))}}$, and therefore 
 \[
    \Big(f(t_{12},t_{23})  (F \circ_1 F)\Big)(e^{x_2})=(F \circ_2 F)(e^{x_2}).
\]

\noindent
{\bf Case of $X_3$.} Finally, for $X_3$, we express in $\Aut_{\PaB^1_\K(3)}\Big(0((12)3)\Big)$:
\[X_3= \left(R^{3,12}\right)^{-1}\Psi^{0,3,12}E^{0,3}\left(\Psi^{0,3,12}\right)^{-1} R^{12,3},\]
and therefore,
\[ z_{3_{(0((12)3))}}f(t_{12},t_{23})= g(t_{03},t_{03}+t_{13})e^{-\frac{t_{13}+t_{23}}{2}}f(t_{12},t_{23}).\]

In $\Aut_{\PaB^1_\K(3)}\Big(0(1(23))\Big)$, we can write
\[
X_3=\left(R^{2,3}\right)^{-1}\Phi^{1,3,2}\left(R^{1,3}\right)^{-1}\Psi^{0,31,2}\Psi^{0,3,1}E^{0,3}\left(\Psi^{0,3,1}\right)^{-1}\left(\Psi^{0,31,2}\right)^{-1}R^{13}\left(\Phi^{1,3,2}\right)^{-1}R^{2,3},
\]
and therefore
\[
z_{3_{(0(1(23)))}}=g(t_{03},t_{31})g(t_{01}+t_{03},t_{12}+t_{23})e^{-\frac{t_{13}}{2}}f(t_{13},t_{23})e^{-\frac{t_{23}}{2}}
\]

To compare $z_{3_{(0((12)3))}}f(t_{12},t_{23})$ with $z_{3_{(0(1(23)))}}$, recall the Hexagon relation in $\PaB(3)$ (here we permute the indices to match our application):
\[\left(R^{3,12}\right)\Phi^{1,2,3}=\left(\Phi^{3,1,2}\right)^{-1}\left(R^{3,1}\right)^{-1}\Phi^{1,3,2}\left(R^{3,2}\right)^{-1}.\]
Applying this to the last two terms of $z_{3_{(0((12)3))}}f(t_{12},t_{23})$, we obtain:
\[
z_{3_{(0((12)3))}}f(t_{12},t_{23})=g(t_{03},t_{03}+t_{13})f(t_{13},t_{12})^{-1}e^{-\frac{t_{13}}{2}}f(t_{13},t_{23})e^{-\frac{t_{23}}{2}}.
\]
Since $(t_{01}+t_{13})$ and $t_{12}$ both commute with $t_{03}$, it follows that $g(t_{01}+t_{13},t_{12})$ acts trivially on $t_{03}$, and thus, the action of $z_{3_{(0((12)3))}}f(t_{12},t_{23})$ agrees with the action of
\[
g(t_{01}+t_{13},t_{12})g(t_{03},t_{03}+t_{13})f(t_{13},t_{12})^{-1}e^{-\frac{t_{13}}{2}}f(t_{13},t_{23})e^{-\frac{t_{23}}{2}}.\]

In turn, the first three terms of this expression can be reduced via an application of the Mixed Pentagon relation to arrive at:
\[
g(t_{03},t_{13})g(t_{01}+t_{03},t_{12}+t_{23})e^{-\frac{t_{13}}{2}}f(t_{13},t_{23})e^{-\frac{t_{23}}{2}}=z_{3_{(0(1(23)))}}.
\]
Thus, 
 \[
    \Big(f(t_{12},t_{23})  (F \circ_1 F)\Big)(e^{x_3})=(F \circ_2 F)(e^{x_3}),
\]
completing the proof.
\end{proof}

We are now in position to complete the proof that $F_{0(12)}$ is a KV solution of type $(0,2+1)$.

\begin{lemma}
\label{lem:Jacobian-r2}
Let $(\varphi^1,\varphi):\hPaB^1 \rightarrow \CD^+$ be an equivalence of completed moperads given by the values~\eqref{eq:values}.
Then, the associated tangential automorphism $F=F_{0(12)}$ satisfies the second KV equation:
    \begin{align*}
    \exists \; h \in z^2\mathbb{K}[[z]] \quad \text{such that} \quad J(F) & =\trace\left( h\left( x_1+x_2\right)-h(x_1)-h(x_2)\right). 
\end{align*} 
\end{lemma}

\begin{proof}
We compute the cosimplicial differential $d$ of $J(F)$ (see \cref{sec:operadic-cohomology}).
We use the following facts:
\begin{itemize}
\item the Jacobian is a morphism of operads (\cref{prop:Jacobian-operad}), and therefore commutes with $d$;
\item $dF=(F \circ_2 F)(F \circ_1 F)^{-1}$, see ~\cref{ex:cosimplicial-associator};
\item $f(t_{12},t_{23})=(F \circ_2 F)(F \circ_1 F)^{-1}$ from~\cref{lem:KVII};
\item $J(f(t_{12},t_{23}))=0$ by ~\cite[Proposition~22]{AET10}.
\end{itemize}
Hence, we obtain
\[
d(J(F))=J(dF)=J((F \circ_2 F)(F \circ_1 F)^{-1})=J(f(t_{12},t_{23}))=0.\]
Since the cosimplicial cohomology of $\cyc$ vanishes in degree 2 \cite[Theorem~2.8]{AT12}, there is an $h \in \cyc_1=\mathbb{K}[[z]]$ such that $J(F)=d(h)=\trace\left( h\left( x_1+x_2\right)-h(x_1)-h(x_2)\right)$, as claimed. 
\end{proof}


Combining \cref{lem:KVI-r2} and \cref{lem:Jacobian-r2}, we obtain the following result:

\begin{thm}
\label{thm: the construction is good}
Let $(\varphi^1,\varphi):\hPaB^1 \rightarrow \CD^+$ be an equivalence of completed moperads given by the values~\eqref{eq:values}.
Then, the associated tangential automorphism $F=F_{0(12)}$ is a symmetric KV solution of type~$(0,2+1)$. 
\end{thm}

\begin{proof}
    The only remaining fact to prove is that $F$ is a symmetric KV solution.
    \cref{lem:KVII} shows that its KV associator is equal to  $G_F=f(t_{12},t_{23})$.
    Since $f$ is a Drinfeld associator, it satisfies the Inversion equation \eqref{inversion}.
    We deduce that $G_F=(G_F^{3,2,1})^{-1}$.
    The conclusion then follows from \cref{prop:symmetric-iff-associator-symmetric}.
\end{proof}

In fact, \cref{construction: TAut from local data} defines a \emph{family} of genus zero KV solutions -- the suboperad of $\solkv$ generated by the KV solution $F_{0(12)}$.  These operadic compositions echo the constructions found in~\cite[Section 3]{AT12},~\cite[Appendix B1]{AET10}, and especially~\cite[Section 7]{AKKN_genus_zero}, where a gluing procedure for genus zero surfaces with boundary is used to produce families of higher-arity KV solutions. 

\begin{lemma}
\label{lem:0w-KVI}
    Let $(\varphi^1, \varphi):\hPaB^1 \rightarrow \CD^{+}$ be an equivalence of completed moperads given by the values~\eqref{eq:values}.
For any object of the form $0w=0(w) \in \hPaB^1(n)$ with $w$ a parenthesization of the identity permutation $1\cdots n$, the associated tangential automorphism $F_{0w}$ satisfies the first KV equation
\[ 
F_{0w}(e^{x_1}\cdots e^{x_n})=e^{x_1+\cdots +x_n}.
\]
\end{lemma}

\begin{proof}
First observe that by \cref{thm:30-AET}, for any $0w$ as in the statement, the tangential automorphism $F_{0w}$ can be obtained from $F_{0(12)}$ by operadic composition.
By \cref{lem:KVI-r2}, $F_{0(12)}$ satisfies~\eqref{SolKVI}, and in \cref{thm:operad-SolKV} we proved that operadic composition preserves the first KV equation. 
This completes the proof.
\end{proof}


\begin{lemma}
\label{lem:w-KVII}
    Let $(\varphi^1,\varphi):\hPaB^1\rightarrow \CD^{+}$ be an equivalence of completed moperads given by the values~\eqref{eq:values}.
    For any word $0w \in \ob(\hPaB^1(n))$, the associated tangential automorphism $F_{0w}$ satisfies the second KV equation
\begin{align*}
    \exists \; h \in\mathbb{K}[[z]] \quad \text{such that} \quad J(F_{0w}) & =\trace\left( h\left( \sum_{i=1}^{n} x_i\right)-\sum_{i=1}^{n}h(x_i)\right). 
\end{align*} 
Moreover, the Duflo function $h$ is the same for every word.
\end{lemma}

\begin{proof}
By \cref{thm:30-AET}, for any $n\geqslant 2$, the tangential automorphism $F_{r_n^1}$ can be obtained from $F_{r_2^1}=F_{0(12)}$ by operadic composition.
In \cref{thm:operad-SolKV}, we prove that operadic composition preserves the second KV equation as long as 
the Duflo function $h \in\mathbb{K}[[z]]$ is the same for all tangential automorphisms involved. 
By \cref{lem:Jacobian-r2}, $F_{r_2^1}$ satisfies \eqref{SolKVII}, and therefore so does~$F_{r_n^1}$ (see also \cite[Lemma 7.3]{AKKN_genus_zero}). In particular, any well-defined operadic composition of $F_{r_2^1}$ will have the same Duflo function as $F_{r_2^1}$.
Now, any word $0w \in \ob(\hPaB^1(n))$ can be obtained from $r_n^1$ by conjugation with the associators~$\Psi$ and~$\Phi$, and action of the symmetric group $\Sigma_n^+$. 
Since these operations preserve the Jacobian (see \cref{rem:symmetric-action-KVI}), we have $J(F_{0w})=J(F_{r_n^1})$, and the conclusion follows.
\end{proof}

Combining \cref{lem:0w-KVI} with \cref{lem:w-KVII}, we obtain the following theorem, similar in spirit to the arguments in~\cite[Appendix~B]{AET10}.

\begin{theorem}
    \label{thm:SolKV-from-moperad}
    Let $(\varphi^1,\varphi):\hPaB^1 \rightarrow \CD^{+}$ be an equivalence of completed moperads given by the values~\eqref{eq:values}.
    For any word $0w \in \ob(\hPaB^{1}(n))$ with $w$ a parenthesization of the identity permutation~$1\cdots n$, the associated tangential automorphism $F_{0w}$ is a KV solution of type $(0,n+1)$.
    Moreover, the set of such automorphisms can be identified with the suboperad of $\solkv$ generated by the symmetric KV solution $F_{0(12)}$.
\end{theorem}

\medskip

\begin{remark}
Tangential automorphisms $F_w$ for other parenthesised words (i.e.\ not of the form $0w$) are not KV solutions, as they only satisfy \eqref{SolKVI} up to conjugation by associators~$\Phi$ and~$\Psi$.
Note for instance that $F_{(01)2}$ does \emph{not} satisfy \eqref{SolKVI}.
In addition, the symmetric group action on KV solutions does not preserve~\eqref{SolKVI} either, hence $\solkv$ forms a non-symmetric operad, see \cref{rem:symmetric-action-KVI}.
\end{remark}


\subsection{Actions of $\gtm$ on KV solutions}
\label{sec: GT1 actions on constructed KV solutions}

Recall from \cref{sec: GT1} that the Grothendieck--Teichm\"uller module groups $\gtm$ and $\grtm$ are proalgebraic groups canonically identified with the groups of object-fixing automorphisms of the moperads $\hPaB^1$ and $\PaCD^+$:
\[
\gtm \cong \Aut_{0}(\hPaB^1), \qquad \grtm \cong \Aut_{0}(\PaCD^+).
\]
We denote by $\gtm_1 \subset \gtm$ and $\grtm_1 \subset \grtm$ the corresponding prounipotent subgroups, consisting of elements of the form $(1,f,g)$, equivalently, those with coupling constant $\lambda = 1$.

In this section, we show that the assignment
\[
(\varphi^1,\varphi)\longmapsto F_{r_2^1,\varphi^1,\varphi}
\]
from homomorphic expansions $(\varphi^1,\varphi):\hPaB^1\to \CD^+$ to KV solutions intertwines the natural actions of the prounipotent groups $\gtm_1$ and $\grtm_1$ with the corresponding actions of $\kv(2)$ and $\krv(2)$.

By Theorem~\ref{thm: GT1 as PaB1 automorphisms}, every element $(1,f,g)\in \gtm_1$ determines an object-fixing automorphism
\[
(\vartheta^1,\vartheta):\hPaB^1\to \hPaB^1
\]
whose values on generators are
\begin{equation}
\label{eq:values-aut} \tag{$\star$}
    \vartheta(R^{1,2}) = R^{1,2}, \quad
    \vartheta(\Phi^{1,2,3}) = f(x_{12},x_{23})\cdot\Phi^{1,2,3}, \quad
    \vartheta^1(E^{0,1}) = E^{0,1}, \quad
    \vartheta^{1}(\Psi^{0,1,2}) = g(x_{01},x_{12})\cdot\Psi^{0,1,2},
\end{equation}
where $x_{ij}$ are the pure braid generators from \eqref{eq:PureBraidGens}.

Now fix a parenthesization $0(w)$ of the identity permutation $1\cdots n$. Restricting $(\vartheta^1,\vartheta)$ to the object $0(w)\in \ob(\hPaB^1(n))$ gives a local automorphism
\[
\begin{tikzcd}
\Aut_{\hPaB^1(n)}(0w)\cong (\PB_n^1)_{\K}
\arrow[r, "\vartheta^1"]
&
(\PB_n^1)_{\K}\cong \Aut_{\hPaB^1(n)}(0w).
\end{tikzcd}
\]
From this local automorphism we will extract an automorphism of the completed free group $(\F_n)_{\K}$. In the case $w=0(12)$, this will produce an element of $\kv(2)$.

\begin{construction}
\label{construction: element of KV}
Fix a parenthesization $0w$ of the identity permutation $1\cdots n$, and let
\[
(\vartheta^{1},\vartheta):\hPaB^1\to \hPaB^1
\]
be an object-fixing automorphism with values on generators as in \eqref{eq:values-aut}.
\smallskip 

\begin{enumerate}[leftmargin=*,label=(\alph*)]
\item Restricting $(\vartheta^1,\vartheta)$ to the object $0w\in \ob(\hPaB^1(n))$ gives a local automorphism
\[
\vartheta^1_{0w}:\Aut_{\hPaB^{1}(n)}(0w)\cong (\PB_n^1)_{\K}\longrightarrow (\PB_n^1)_{\K}\cong \Aut_{\hPaB^{1}(n)}(0w).
\]
As in Construction~\ref{construction: TAut from local data}, this induces a map
\[
\bar{\vartheta}^1_{0w}:(\F_n)_{\K}\longrightarrow (\F_n)_{\K},
\]
where $(\F_n)_{\K}\subset (\PB_n^1)_{\K}$ is the free subgroup generated by $X_1,\dots,X_n$.

\item To describe this map explicitly, we recall that we can write each generator $X_i$ as a (not necessarily unique) product of the generating morphisms
\[
E^{0,i},\qquad R^{i,j},\qquad \Phi^{i,j,k},\qquad \Psi^{0,i,j}
\]
in $\Aut_{\hPaB^1(n)}(0w)$. For example, when $n=2$ and $w=0(12)$, we have
\[
X_1=\Psi^{0,1,2}E^{0,1}(\Psi^{0,1,2})^{-1},
\qquad
X_2=(R^{2,1})^{-1}\Psi^{0,2,1}E^{0,2}(\Psi^{0,2,1})^{-1}R^{1,2}.
\]

\smallskip

Applying $\vartheta^1_{0w}$ to such expressions determines the images $\bar{\vartheta}^1_{0w}(X_i)$ for $i=1,\dots,n$. Each $\bar{\vartheta}^1_{0w}(X_i)$ is a conjugate of $X_i$ by an element of $(\PB_n^1)_{\K}$, hence lies in $(\F_n)_{\K}$ because $(\F_n)_{\K}$ is a normal subgroup of $(\PB_n^1)_{\K}$. We therefore obtain an automorphism
\begin{equation}\tag{$G_{0w,\vartheta^1,\vartheta}$}
G_{0w,\vartheta^1,\vartheta}:(\F_n)_{\K}\longrightarrow (\F_n)_{\K},
\end{equation}
given on generators by
\[
X_i\longmapsto \bar{\vartheta}^1_{0w}(X_i),\qquad i=1,\dots,n.
\]

\end{enumerate}
\end{construction}

\medskip

The next lemma shows that, when $(\vartheta^{1},\vartheta):\hPaB^1\to \hPaB^1$ is defined on generators by \eqref{eq:values-aut}, the induced automorphism
\[
G_{\vartheta^1}=G_{0(12),\vartheta^1,\vartheta}:(\F_2)_{\K}\to (\F_2)_{\K}
\]
satisfies the first defining equation of the KV symmetry group $\kv(2)$.

\begin{remark}
\label{rem: braid notation for theta}
Recall that braid generators are indexed by the \emph{positions} of strands rather than their labels. To compute the images of the generators $X_1$ and $X_2$, we regard
\[
\bar{\vartheta}^1_{0(12)}(X_i)\in \Aut_{\hPaB^1(2)}(0(12))\cong (\PB_2^1)_{\K}
\subset (\Br_2^1)_{\K}.
\]
Using the expressions for $X_1$ and $X_2$ above, we obtain
\[
\bar{\vartheta}^1_{0(12)}(X_1)
=
g(x_{01},x_{12})\,x_{01}\,g(x_{01},x_{12})^{-1}
\quad \text{and} \quad 
\bar{\vartheta}^1_{0(12)}(X_2)
=
\beta_1^{-1}g(x_{01},x_{12})\,x_{01}\,g(x_{01},x_{12})^{-1}\beta_1.
\]

Here $E^{0,1}$ is identified with the pure braid $x_{01}$, and $R^{1,2}$ with the braid generator $\beta_1$. Although $\bar{\vartheta}^1_{0(12)}(X_2)$ is a pure braid, it is most naturally written as a product in $(\Br_2^1)_{\K}$, which explains the appearance of $\beta_1$ and $\beta_1^{-1}$. We also suppress the associativity morphisms $\Psi$, since they are trivial as elements of the group
\[
(\F_2)_{\K}\subseteq \Aut_{\hPaB^1(2)}(0(12))\cong (\PB_2^1)_{\K}.
\]
\end{remark}

\begin{lemma}
\label{lem:Theta-fixes-product}
Let $(\vartheta^{1},\vartheta):\hPaB^1\to \hPaB^1$ be an object-fixing automorphism defined on generators by \eqref{eq:values-aut}, and let $r_2^1=0(12)$. Then the induced automorphism
\[
G_{\vartheta^1}=G_{r_2^1,\vartheta^1,\vartheta}:(\F_2)_{\K}\longrightarrow (\F_2)_{\K}
\]
is tangential. It also satisfies the first defining equation of $\kv(2)$, namely
\[
G_{\vartheta^1}(X_2X_1)=X_2X_1.
\]
\end{lemma}

\begin{proof}
We first show that $G_{\vartheta^1}$ is a tangential automorphism. By Remark~\ref{rem: braid notation for theta}, the induced map on generators is given by
\[
\bar{\vartheta}^1_{0(12)}(X_1)
=
g(x_{01},x_{12})\,x_{01}\,g(x_{01},x_{12})^{-1}
\qquad \text{and} \qquad
\bar{\vartheta}^1_{0(12)}(X_2)
=
\beta_1^{-1}g(x_{01},x_{12})\,x_{01}\,g(x_{01},x_{12})^{-1}\beta_1.
\]

Recall that the full twist is central in the braid group. In particular, $z=x_{01}x_{12}x_{02}$ is central in $(\Br^{1}_2)_{\K}$, so $x_{12}=x_{01}^{-1}zx_{02}^{-1}.$ Substituting this into the formula for $\bar{\vartheta}^1_{0(12)}(X_1)$ gives
\[
\bar{\vartheta}^1_{0(12)}(X_1)
=
g(x_{01},x_{01}^{-1}zx_{02}^{-1})\,x_{01}\,g(x_{01},x_{01}^{-1}zx_{02}^{-1})^{-1}.
\]
Since $z$ is central, any powers of $z$ appearing in the conjugating element act trivially. Hence
\begin{equation}
\label{eq:GX1}
\bar{\vartheta}^1_{0(12)}(X_1)
=
g(x_{01},x_{01}^{-1}x_{02}^{-1})\,x_{01}\,g(x_{01},x_{01}^{-1}x_{02}^{-1})^{-1},
\end{equation}
so $\bar{\vartheta}^1_{0(12)}(X_1)$ is a conjugate of $X_1$ by an element of $(\F_2)_{\K}$.

We treat $\bar{\vartheta}^1_{0(12)}(X_2)$ similarly. Substituting $x_{12}=x_{01}^{-1}zx_{02}^{-1}$ into the conjugating term yields
\begin{equation}
\label{eq:ValueX2Step1}
\bar{\vartheta}^1_{0(12)}(X_2)
=
\beta_1^{-1}g(x_{01},x_{01}^{-1}x_{02}^{-1})\,x_{01}\,g(x_{01},x_{01}^{-1}x_{02}^{-1})^{-1}\beta_1.
\end{equation}
Now let $\omega(x_{01},x_{02})$ be any word in the free group $(\F_2)_{\K}\subseteq (\Br^{1}_2)_{\K}$. Then
\[
\beta_1^{-1}\omega(x_{01},x_{02})\beta_1
=
x_{12}^{-1}\omega(x_{12}x_{02}x_{12}^{-1},x_{01})x_{12}.
\] This follows from the braid relations; see Figure~\ref{fig:BraidConj} for an illustration.

\begin{figure}[t]
    \includegraphics[width=9cm]{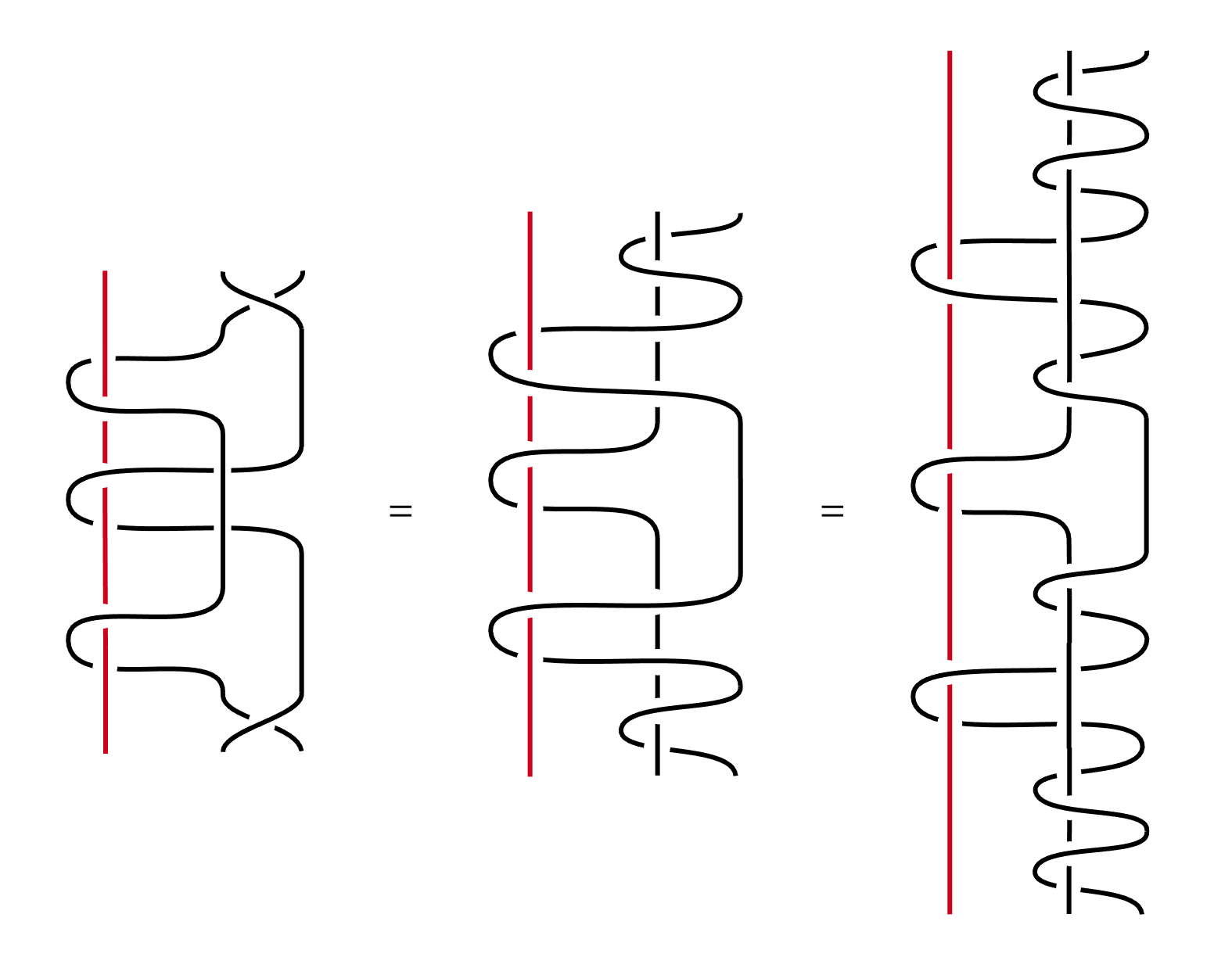}
\caption{Let $\omega(x_{01},x_{02})=x_{01}^{-1}x_{02}x_{01}$. We demonstrate the braid equivalence between $\beta_1^{-1}\omega\beta_1$ (displayed on the left), and $x_{12}^{-1} \omega(x_{12}x_{02}x_{12}^{-1},x_{01})x_{12}$ (displayed on the right).}\label{fig:BraidConj}
\end{figure}

Using again the identity $x_{12}=x_{01}^{-1}zx_{02}^{-1}$, we compute $x_{12}x_{02}x_{12}^{-1}=x_{01}^{-1}x_{02}x_{01}.$ Substituting this into \eqref{eq:ValueX2Step1}, we obtain
\begin{multline}
\label{eq:ValueX2Step2}
\bar{\vartheta}^1_{0(12)}(X_2)
=
x_{02}x_{01}\,
g\left(x_{01}^{-1}x_{02}x_{01},\,x_{01}^{-1}x_{02}^{-1}x_{01}x_{02}^{-1}\right)
x_{01}^{-1}\cdot x_{02}\cdot {}\\
{}\cdot x_{01}\,
g\left(x_{01}^{-1}x_{02}x_{01},\,x_{01}^{-1}x_{02}^{-1}x_{01}x_{02}^{-1}\right)^{-1}
x_{01}^{-1}x_{02}^{-1}.
\end{multline}
Thus $\bar{\vartheta}^1_{0(12)}(X_2)$ is also a conjugate of $X_2$ by an element of $(\F_2)_{\K}$, and hence $G_{\vartheta^1}$ is tangential.

It remains to show that $G_{\vartheta^1}(X_2X_1)=X_2X_1$. By Lemma~\ref{lemma: KV1 for braids}, we have
\[
X_2X_1=E^{0,12}.
\]
Moreover, $\vartheta^1(E^{0,1})=E^{0,1}$ by assumption, and since $E^{0,12}=E^{0,1}\circ_1 \id_2$ and $\vartheta^1$ is a map of moperads, it follows that $\vartheta^1(E^{0,12})=E^{0,12}.$
Therefore,
\[
G_{\vartheta^1}(X_2X_1)=X_2X_1,
\]
as required.
\end{proof}

\begin{remark}
If $(\vartheta^1,\vartheta):\hPaB^1\to \hPaB^1$ is an object-fixing automorphism in $\gtm$ with coupling constant $\lambda\neq 1$, then the associated tangential automorphism $G_{\vartheta^1}$ does not lie in $\kv(2)$. Indeed,
\[
G_{\vartheta^1}(X_2X_1)=(X_2X_1)^\lambda,
\]
so the scalar part of the $\gtm$-action obstructs the first KV equation. Thus only the prounipotent subgroup $\gtm_1\subset \gtm$ gives rise to KV symmetries.
\end{remark}

The next theorem shows that the tangential automorphism $G_{\vartheta^1}=G_{r_2^1,\vartheta^1,\vartheta}\in \TAut_2$ constructed from an element $(\vartheta^1,\vartheta)\in \gtm_1$ is a KV symmetry. 

\begin{theorem}
\label{thm:KV-action-F-012}
Let $(\vartheta^{1},\vartheta):\hPaB^1\to \hPaB^1$ be an automorphism in $\gtm_1$ defined on generators by \eqref{eq:values-aut}, and let $(\varphi^1,\varphi):\hPaB^1\to \CD^+$ be an equivalence of completed moperads defined on generators by \eqref{eq:values}. Then $G_{\vartheta^1}\in \kvs(2)$, and $F_{\varphi^1}\circ G_{\vartheta^1}$ is again a symmetric KV solution. Consequently, the assignment
\[
(\vartheta^1,\vartheta)\longmapsto G_{\vartheta^1}
\]
defines an injective group homomorphism $\gtm_1\hookrightarrow \kvs(2)$.
\end{theorem}

\begin{proof}
Since $(\vartheta^1,\vartheta)$ is an automorphism of completed moperads, the composite
\[
(\varphi^1,\varphi)\circ (\vartheta^1,\vartheta):\hPaB^1\to \CD^+
\]
is again an equivalence of completed moperads. Moreover, because $(\vartheta^1,\vartheta)\in \gtm_1$, the composite is still defined on generators by formulas of the form \eqref{eq:values}, and hence has coupling constant $1$.

By construction, both $F_{\varphi^1}$ and $G_{\vartheta^1}$ are obtained by restricting $(\varphi^1,\varphi)$ and $(\vartheta^1,\vartheta)$, respectively, to the free subgroup
\[
(\F_2)_{\K}\subseteq \Aut_{\hPaB^1(2)}(0(12)).
\]
It follows that $F_{\varphi^1}\circ G_{\vartheta^1}=F_{\varphi^1\circ \vartheta^1}.$ Since $\varphi^1\circ \vartheta^1$ has coupling constant $1$, the right-hand side is a symmetric KV solution by \cref{thm: the construction is good}. Therefore $F_{\varphi^1}\circ G_{\vartheta^1}$ is again a symmetric KV solution. Since both $F_{\varphi^1}$ and $F_{\varphi^1\circ \vartheta^1}$ are symmetric KV solutions, it follows from \cref{symmetry-groups-from-KV-solutions} that
\[
G_{\vartheta^1}\in \kvs(2).
\]

The fact that the resulting map $\gtm_1\longrightarrow \kvs(2)$ is a group homomorphism follows from Construction~\ref{construction: element of KV}. Indeed, the assignment $(\vartheta^1,\vartheta)\mapsto G_{\vartheta^1}$ is obtained by restricting $\vartheta^1$ first to the local automorphism group $\Aut_{\hPaB^1(2)}(0(12))$ and then to the free subgroup $(\F_2)_\K$. Since restriction is compatible with composition, we obtain
\[
G_{\vartheta^1\circ \vartheta'^1}=G_{\vartheta^1}\circ G_{\vartheta'^1},
\]
so the map is a group homomorphism.

It remains to prove injectivity. Suppose that $G_{\vartheta^1}=1$. We will show that $(\vartheta^1,\vartheta)=1$.  By \eqref{eq:GX1}, we have
\[
G_{\vartheta^1}(X_1)
=
g(x_{01},x_{01}^{-1}x_{02}^{-1})\,x_{01}\,g(x_{01},x_{01}^{-1}x_{02}^{-1})^{-1}.
\]
Since $x_{01}$ and $x_{02}$ freely generate $(\F_2)_{\K}$, the element $g(x_{01},x_{01}^{-1}x_{02}^{-1})$ commutes with $x_{01}$ only if it depends on the first argument alone. Equivalently,
\[
g(x,y)=g(x).
\]

We now show that this forces $g=1$, and hence also $f=1$.  Consider the mixed pentagon equation \eqref{GT^1:mixed pentagon}:
\begin{equation}
g(x_{01},x_{13}x_{12})\,g(x_{12}x_{02},x_{23})
=
f(x_{12},x_{23})\,g(x_{02}x_{01},x_{23}x_{13})\,g(x_{01},x_{12}).
\tag{MP}
\end{equation}
If $g$ depends only on its first argument, this becomes
\begin{equation}
g(x_{01})\,g(x_{12}x_{02})\,g(x_{01})^{-1}\,g(x_{02}x_{01})^{-1}
=
f(x_{12},x_{23}).
\end{equation}
The left-hand side is a pure braid on strands $0,1,2$ and does not involve strand $3$. Since $x_{12}$ and $x_{23}$ generate a copy of $(\F_2)_{\K}$ inside $(\PB^1_3)_{\K}$, it follows that if $f(x_{12},x_{23})$ does not involve strand $3$, then $f$ also depends only on its first argument:
\[
f(x_{12},x_{23})=f(x_{12}).
\]
But then the right-hand side does not involve strand $0$, so the left-hand side must also be unchanged after deleting strand $0$. Let $s_0:\PB_3^1\to \PB_3$ denote the strand-deletion map, extended to completions. Applying $s_0$ gives $g(x_{12})=f(x_{12}),$ and so $g=f$. Therefore
\[
g(x_{01})\,g(x_{12}x_{02})\,g(x_{01})^{-1}\,g(x_{02}x_{01})^{-1}
=
g(x_{12}).
\]
Now delete strand $2$. This yields
\[
g(x_{01})\,g(x_{01})^{-1}\,g(x_{01})^{-1}=1,
\]
and hence $g=1$. Therefore $f=1$ as well, so $(\vartheta^1,\vartheta)$ is the identity. This proves injectivity.
\end{proof}

\begin{cor}
\label{cor:KV-action-F-012}
Let $(\vartheta^{1},\vartheta):\hPaB^1\to \hPaB^1$ be an automorphism in $\gtm_1$ defined on generators by \eqref{eq:values-aut}, and let $(\varphi^1,\varphi):\hPaB^1\to \CD^+$ be an equivalence of completed moperads defined on generators by \eqref{eq:values}. Then
\[
F_{\varphi^1\circ \vartheta^1}=G_{\vartheta^1}\cdot F_{\varphi^1},
\]
where $\cdot$ denotes the action of $\kvs(2)$ on the set of symmetric KV solutions.
\end{cor}

\medskip

Every homomorphic expansion $(\varphi^1,\varphi): \hPaB^1\to \CD^+$ extends uniquely to an isomorphism $(\varphi^{1\prime},\varphi'):\hPaB^1\to \PaCD^+$ which acts trivially on objects, where $\PaCD^+$ is the shifted moperad associated to the operad $\PaCD$ (Definition~\ref{defn: PaCD}).

Now let $(\varphi^{1\prime},\varphi'):\hPaB^1\to \PaCD^+$ be such a moperad isomorphism with $\mu=1$, and let $(\vartheta^1,\vartheta):\hPaB^1\to \hPaB^1$ be an element of $\gtm\cong \Aut_0(\hPaB^1)$ with $\lambda=1$. Conjugating by $(\varphi^{1\prime},\varphi')$, we obtain an automorphism $(\varepsilon^1,\varepsilon)\in \Aut_0(\PaCD^+)\cong \grtm$ fitting into the commutative diagram
\begin{equation}
\label{diagram for GRT11}
\begin{tikzcd}
\hPaB^1 \arrow[r, "\vartheta^1"] \arrow[d, "\varphi^{1\prime}"']
& \hPaB^1 \arrow[d, "\varphi^{1\prime}"] \\
\PaCD^+ \arrow[r, dashed, "\varepsilon^1"]
& \PaCD^+.
\end{tikzcd}
\end{equation}

Since $\mu=1$ and $\lambda=1$, the automorphism $\varepsilon^1$ fixes each standard infinitesimal braid generator $t_{ij}$ of $\ib_n^+$. Thus $(\varepsilon^1,\varepsilon)$ lies in the prounipotent subgroup $\grtm_1\subset \Aut_0(\PaCD^+)$.

Because all morphisms in \eqref{diagram for GRT11} fix objects, we may restrict to the local automorphism at $r_2^1=0(12)$:
\[
\varepsilon^1_{r_2^1}:
\Aut_{\PaCD^+(2)}(r_2^1)\cong \exp(\ib_2^+)
\longrightarrow
\exp(\ib_2^+)\cong \Aut_{\PaCD^+(2)}(r_2^1).
\]

\begin{theorem}
\label{thm:KRV}
Let $(\varepsilon^1,\varepsilon):\PaCD^+\to \PaCD^+$ be an element of $\grtm_1$. Then the restriction of
\[
\varepsilon^1_{r_2^1}:\exp(\ib_2^+)\to \exp(\ib_2^+)
\]
to the subgroup $\exp(\lie_2)\subset \exp(\ib_2^+)$ is an element of $\krvs(2)$.
\end{theorem}

\begin{proof}
The automorphism $\varepsilon^1$ fits into the commutative diagram \eqref{diagram for GRT11}. Let $H_{\varepsilon^1}=F_{\varphi^1}G_{\vartheta^1}F_{\varphi^1}^{-1}\in \TAut_2$ denote the automorphism induced by restricting $\varepsilon^1_{r_2^1}$ to $\exp(\lie_2)$.

By \cref{thm: the construction is good} and \cref{thm:KV-action-F-012}, both $F_{\varphi^1}$ and $F_{\varphi^1}G_{\vartheta^1}$ are symmetric KV solutions. It follows from \cref{symmetry-groups-from-KV-solutions} that
\[
H_{\varepsilon^1}\in \krvs(2).
\]
\end{proof}

\begin{cor}
\label{cor:KRV-action-F-012}
Let $(\varphi^1,\varphi):\hPaB^1\to \CD^+$ be an equivalence of completed moperads defined on generators by \eqref{eq:values}, and let $(\varepsilon^1,\varepsilon):\PaCD^+\to \PaCD^+$ be an element of $\grtm_1$. Then
\[
F_{\varepsilon^1\circ \varphi^1}=F_{\varphi^1}\cdot H_{\varepsilon^1},
\]
where $H_{\varepsilon^1}\in \krvs(2)$ is the element induced by restricting $\varepsilon^1_{r_2^1}$ to $\exp(\lie_2)\subset \exp(\ib_2^+)$, and $\cdot$ denotes the action of $\krvs(2)$ on the set of symmetric KV solutions.
\end{cor}

\begin{remark}
Just as the family of KV solutions $\solkv=\{\solkv(n)\}_{n\geqslant 2}$ forms a colored operad in sets, the families of genus zero symmetry groups $\kv=\{\kv(n)\}_{n\geqslant 2}$ and $\krv=\{\krv(n)\}_{n\geqslant 2}$ form colored operads in groups; see Theorem~\ref{thm:operad-SolKV} and \cref{sec:kv-krv-operads}. The analogues of \cref{thm:SolKV-from-moperad} for these operads, together with \cref{cor:KV-action-F-012} and \cref{cor:KRV-action-F-012}, show that the actions of the binary symmetries associated to the word $0(12)$ are compatible with operadic composition of KV solutions.

In particular, starting from a classical element of $\gt_1$ (respectively,\ $\grt_1$), one may apply the diagonal map
\[
\gt_1\longrightarrow \gt_1^{\times k}
\qquad
(\text{respectively,\ } \grt_1\longrightarrow \grt_1^{\times k}),
\]
embed each factor into $\kv(2)$ (respectively,\ $\krv(2)$), and then compose the resulting binary operations according to a chosen rooted binary tree. In this way one recovers the trees of \cite[Theorem~1.13]{AKKN_GT_Formality}.
\end{remark}
\section{Symmetric KV Solutions from Parenthesized Braids}
\label{section: symmetric KV solutions}
In this section, we show that every symmetric KV solution determines a moperad morphism
\[
(\varphi_F^1,\varphi_F):\hPaB^1\to \TAut^+.
\]
We then show that this morphism factors through $\CD^+$ if and only if the associated KV associator is a Drinfeld associator.

\begin{remark}
Much of this section reinterprets, in the language of moperads, the constructions and results of Sections 7--9 of \cite{AT12}. However, as noted in Remark~\ref{rem:kvsol}, our conventions are closer to those of \cite{AET10}. In particular, what we call a KV solution $F\in \solkv(2)$ is the inverse of the corresponding KV solution in \cite{AT12} and \cite{AKKN_genus_zero}.
\end{remark}

\subsection{An involution on $\solkv(2)$}\label{subsec:invol}
As discussed in \cref{ss:symmetric-KV-solutions}, symmetric KV solutions are those KV solutions $F$ which are invariant under the involution
\[
\tau : \solkv(2)\to \solkv(2), \qquad \tau(F)\eqdef e^{-\inner/2}F^{2,1}\etw,
\]
where $\inner=\operatorname{ad}_{x_1+x_2}$ and $\etw=\exp(0,x_1)$. Moreover, the associated automorphism
\[
G_F = F^{1,23}F^{2,3}(F^{1,2})^{-1}(F^{12,3})^{-1}
\]
is a \emph{KV associator}, that is, it satisfies the pentagon, unit, and hexagon equations in $\SAut$:
\begin{gather}
    G^{1,2,34} G^{12,3,4} =  G^{2,3,4} G^{1,23,4} G^{1,2,3}, \label{eq:pentagon in Taut} \tag{P-KV} \\
    G^{1,2,3}G^{3,2,1} =1, \label{eq:KV-inversion} \tag{U-KV}\\
    e^{\frac{\inner^{1,2}+\inner^{1,3}}{2}}=(G^{2,3,1})^{-1} e^{\frac{\inner^{1,3}}{2}} G^{2,1,3} e^{\frac{\inner^{1,2}}{2}} (G^{1,2,3})^{-1}, \label{eq:KV-Hex1} \tag{H1-KV}\\
    e^{\frac{\inner^{1,3}+\inner^{2,3}}{2}} = G^{3,1,2} e^{\frac{\inner^{1,3}}{2}}(G^{1,3,2})^{-1} e^{\frac{\inner^{2,3}}{2}} G^{1,2,3}. \label{eq:KV-Hex2} \tag{H2-KV}
\end{gather}

Every Drinfeld associator gives rise to a KV associator in $\SAut$, via the natural inclusion $\iota:\exp(\ib_n)\hookrightarrow \SAut_n$ induced by the embedding $\ib_n\hookrightarrow \sder_n$.

\subsection{Operad and moperad maps from KV solutions}\label{subsec:MopMApFromKV}

By \cref{thm: presentation pab}, a map of completed operads $\varphi:\hPaB\to \SAut$ is determined by its values on the generators $R^{1,2}$ and $\Phi^{1,2,3}$, and these values define an operad map if and only if they satisfy the relations \eqref{eqn:P}, \eqref{eqn:H1}, and \eqref{eqn:H2} in $\SAut$. The next theorem shows that every symmetric KV solution determines such a map.

\begin{theorem}
\label{lemma: symmetric KV solutions define operad maps}
Let $F\in \solkv^{\tau}(2)$ be a symmetric KV solution, and let
\[
G_F \eqdef F^{1,23}F^{2,3}(F^{1,2})^{-1}(F^{12,3})^{-1}\in \SAut_3
\]
be the associated KV associator. Then the values
\[
\varphi_F(\Phi^{1,2,3})\eqdef G_F 
\qquad \text{and} \qquad
\varphi_F(R^{1,2})\eqdef e^{\inner/2},
\]
define an operad map $\varphi_F:\hPaB\to \SAut$.
\end{theorem}

\begin{proof}
By construction, $\varphi_F(\Phi^{1,2,3})=G_F$. Since $G_F$ satisfies the pentagon equation \eqref{eq:pentagon in Taut} in $\SAut_4$, we have
\[
\varphi_F(\Phi^{1,2,34}\Phi^{12,3,4})
=
G_{F}^{1,2,34}G_{F}^{12,3,4}
=
G_{F}^{2,3,4}G_{F}^{1,23,4}G^{1,2,3}
=
\varphi_F(\Phi^{2,3,4}\Phi^{1,23,4}\Phi^{1,2,3}),
\]
so $\varphi_F$ preserves the relation \eqref{eqn:P}.

Similarly, since $F$ is symmetric, the associated KV associator $G_F$ satisfies the hexagon equation \eqref{eq:KV-Hex1}. Therefore
\begin{align*}
\varphi_F(R^{1,23})
=
e^{\frac{\inner^{1,2}+\inner^{1,3}}{2}}
&=
(G^{2,3,1})^{-1}e^{\frac{\inner^{1,3}}{2}}G^{2,1,3}e^{\frac{\inner^{1,2}}{2}}(G^{1,2,3})^{-1} \\
&=
\varphi_F\bigl((\Phi^{2,3,1})^{-1}R^{1,3}\Phi^{2,1,3}R^{1,2}(\Phi^{1,2,3})^{-1}\bigr),
\end{align*}
so $\varphi_F$ preserves the relation \eqref{eqn:H1}. The second hexagon relation \eqref{eqn:H2} follows in the same way from \eqref{eq:KV-Hex2}. Thus $\varphi_F$ preserves the defining relations of $\hPaB$, and therefore defines an operad map.
\end{proof}

\begin{cor}
\label{prop:operad-factors-Drinfeld-associator}
Let $F$ be a symmetric KV solution. Then the operad map $\varphi_F:\hPaB\to \SAut$ factors through $\CD$ if and only if there exists a Drinfeld associator $(1,f)\in \K^\times\times \exp(\ib_3)$ such that
\[
G_F=\iota(f).
\]
\end{cor}

\begin{proof}
Suppose first that $\varphi_F$ factors through $\CD$. Then
\[
G_F=\varphi_F(\Phi^{1,2,3})=\iota(f)
\]
for some $f\in \exp(\ib_3)$, since $\iota:\CD\hookrightarrow \SAut$ is an injective operad morphism. Since $G_F$ satisfies the pentagon, inversion, and hexagon equations of a KV associator, it follows that $f$ satisfies the corresponding Drinfeld associator equations. Hence $(1,f)$ is a Drinfeld associator.

Conversely, suppose that $G_F=\iota(f)$ for some Drinfeld associator $(1,f)$. Define a map $\tilde{\varphi}_F:\hPaB\to \CD$ by setting
\[
\tilde{\varphi}_F(R^{1,2})\eqdef e^{t_{12}/2}
\qquad \text{and} \qquad
\tilde{\varphi}_F(\Phi^{1,2,3})\eqdef f.
\]
Since $(1,f)$ is a Drinfeld associator, these values satisfy the defining relations of $\hPaB$ in $\CD$ by \cref{cor: operadic associators}. Thus $\tilde{\varphi}_F$ is an operad morphism. By construction,
\[
\iota\circ \tilde{\varphi}_F(R^{1,2})=e^{\inner/2}=\varphi_F(R^{1,2})
\qquad\text{and}\qquad
\iota\circ \tilde{\varphi}_F(\Phi^{1,2,3})=\iota(f)=G_F=\varphi_F(\Phi^{1,2,3}),
\]
so $\varphi_F=\iota\circ \tilde{\varphi}_F$. Hence $\varphi_F$ factors through $\CD$.
\end{proof}

\medskip
Given $F\in \solkv^{\tau}(2)$, we will extend the operad morphism $\varphi_F:\hPaB\to \SAut$ from Theorem~\ref{lemma: symmetric KV solutions define operad maps} to a morphism of moperads 
\[
(\varphi_F^1,\varphi_F):\hPaB^1\to \TAut^+.
\]
The additional generator $E^{0,1}\in \hPaB^1(2)$ is sent to a distinguished tangential automorphism $\Theta$, defined below.

\begin{definition}
Let $\theta \eqdef \operatorname{ad}_{\bch(x_1,x_2)}$ denote the inner derivation of $\lie_2$ given by the Baker--Campbell--Hausdorff series; explicitly, \[\theta(a)=[a,\bch(x_1,x_2)]\] for $a\in \lie_2$. We write $\Theta \eqdef \exp(\theta)$ for the corresponding tangential automorphism.
\end{definition}

The automorphism $\Theta$ is closely related to the tangential automorphism $\etw\in \TAut_2$ which appeared in the definition of the involution on $\solkv(2)$. Recall that $\etw$ is the tangential automorphism determined by
\[
\etw(x_1)=x_1
\qquad \text{and} \qquad
\etw(x_2)=e^{-x_1}x_2e^{x_1}.
\]
The following proposition records the identities relating $\Theta$, $\etw$, and a KV solution $F$; it is a reformulation of Propositions 8.1--8.3 of \cite{AT12}.

\begin{prop}
\label{lemma: properties of B}
The automorphisms $\Theta$ and $\etw$ satisfy:
\begin{gather}
\Theta = \etw^{2,1}\etw^{1,2}, \label{eq:theta-B}\\
\etw^{1,2}\etw^{1,3}\etw^{2,3} =\etw^{2,3}\etw^{1,3}\etw^{1,2}, \label{eq:YB} \\
\etw^{1,23}=\etw^{1,2}\etw^{1,3}. 
\label{eq:BB-B}
\end{gather}
Moreover, if $F\in \TAut_2$ satisfies the first KV equation \eqref{SolKVI}, then
\begin{gather}
e^{\inner} = F \Theta F^{-1}, \label{eq:ThetaAndT}\\
(F^{1,2})^{-1}\etw^{12,3}F^{1,2}
=\etw^{1,3}\etw^{2,3}. 
\label{eq:braided-F}
\end{gather}
\end{prop}

\medskip

Recall from Section~3 that $\TAut^1$ is an $\SAut$-moperad in prounipotent groups, while $\TAut$ itself does not form an operad in groups, since its partial compositions are not group homomorphisms. Accordingly, once the operad map
\[
\varphi_F:\hPaB\to \SAut
\]
has been constructed, the extension to $\hPaB^1$ is naturally formulated as a map of $\hPaB$-modules in prounipotent groups, or equivalently as a map of $\hPaB$-moperads in sets.

\begin{theorem}
\label{prop: symmetric KV solutions give KV solutions}
Let $F\in \solkv^{\tau}(2)$ be a symmetric KV solution. Then the assignment
\[
\varphi^1_F(E^{0,1})\eqdef \Theta^{0,1},
\qquad
\varphi^1_F(\Psi^{0,1,2})\eqdef F^{1,2},
\]
extends $\varphi_F$ to a map of $\hPaB$-modules in prounipotent groups
\[
(\varphi_F^1,\varphi_F):\hPaB^1\longrightarrow \TAut^+.
\]
Equivalently, $(\varphi_F^1,\varphi_F)$ is a map of $\hPaB$-moperads in sets.
\end{theorem}

\begin{proof}
The presentation of $\PaB^1$ as a $\PaB$-moperad in~\cref{presenation of PaB1} says that, in order to extend the operad map $\varphi_{F}:\hPaB\rightarrow \SAut$ defined in~\cref{lemma: symmetric KV solutions define operad maps} to a map $(\varphi^{1}_{F},\varphi_{F}):\hPaB^1 \longrightarrow \TAut^+$, it is sufficient to define values $\varphi^1_{F}(E^{0,1}) \eqdef \Theta$ and $\varphi^1_{F}(\Psi^{0,1,2}) \eqdef F$ and verify that these values satisfy the defining equations~\eqref{eqn:MP}, \eqref{eqn:O} and~\eqref{eqn:RP} of $\PaB^1$. 
For $\Theta$ and $F$ as before, the mixed pentagon equation~\eqref{eqn:MP} gives 
\begin{equation*}
\varphi^1_{F}\Big(\Psi^{0,1,23}\Psi^{01,2,3}\Big) 
=
F^{1,23}F^{2,3} 
= G_{F}^{1,2,3} F^{12,3}F^{1,2}
=\varphi^1_{F} \left(\Phi^{1,2,3} \Psi^{0,12,3} \Psi^{0,1,2}\right).
\end{equation*} Thus $\varphi^1_{F}$ preserves the mixed pentagon equation~\eqref{eqn:MP}.

To verify that the values $\varphi^1_{F}(E^{0,1})$ and $\varphi^1_{F}(\Psi^{0,1,2})$ preserve the octagon equation~\eqref{eqn:O}, we will use the hexagon identities from Proposition~\ref{prop:equations-KV-associator}. In particular, if we write out a cyclic permutation of \eqref{eq:KV-Hex1},
\begin{equation}
e^{\frac{\inner^{3,12}}{2}} = (G^{1,2,3})^{-1} \, e^{\frac{\inner^{3,2}}{2}} \, G^{1,3,2} \, e^{\frac{\inner^{3,1}}{2}} \, (G^{3,1,2})^{-1}
\end{equation} and multiplying $e^{\frac{\inner^{3,12}}{2}}e^{\frac{\inner^{12,3}}{2}} = e^{\frac{\inner^{3,1} + \inner^{3,2}}{2}}e^{\frac{\inner^{1,3} + \inner^{2,3}}{2}} = e^{\inner^{1,3} + \inner^{2,3}}$, we get:
\begin{eqnarray*}\label{eq:KV-Hex2-squared}
e^{\inner^{1,3} + \inner^{2,3}} &=&
\left(
(G^{1,2,3})^{-1} e^{\frac{\inner^{3,2}}{2}} G^{1,3,2} e^{\frac{\inner^{3,1}}{2}} (G^{3,1,2})^{-1}
\right) 
\left(
G^{3,1,2}  e^{\frac{\inner^{1,3}}{2}}  (G^{1,3,2})^{-1}  e^{\frac{\inner^{2,3}}{2}}  G^{1,2,3}
\right)\\
&=& (G^{1,2,3})^{-1} e^{\frac{\inner^{3,2}}{2}} G^{1,3,2} e^{\frac{\inner^{3,1}}{2}}  e^{\frac{\inner^{1,3}}{2}}  (G^{1,3,2})^{-1}  e^{\frac{\inner^{2,3}}{2}}  G^{1,2,3}.
\end{eqnarray*} If we expand out each of the instances of $G$, 
 we arrive at 
\begin{equation}
\label{eq:KV-Hex2-alt-expanded}
\begin{aligned}
e^{\inner^{1,3}+\inner^{2,3}}
&=
\bigl( F^{12,3}\,F^{1,2}\,(F^{2,3})^{-1}\,(F^{1,23})^{-1} \bigr)
\;\;
e^{\frac{\inner^{3,2}}{2}}
\;\;
\bigl( F^{1,32}\,F^{3,2}\,(F^{1,3})^{-1}\,(F^{13,2})^{-1} \bigr)
\\
&\quad\;
e^{\frac{\inner^{3,1}}{2}}
\;\;
e^{\frac{\inner^{1,3}}{2}}
\;\;
\bigl( F^{13,2}\,F^{1,3}\,(F^{3,2})^{-1}\,(F^{1,32})^{-1} \bigr)
\;\;
e^{\frac{\inner^{2,3}}{2}}
\\
&\quad\;
\bigl( F^{1,23}\,F^{2,3}\,(F^{1,2})^{-1}\,(F^{12,3})^{-1} \bigr).
\end{aligned}
\end{equation} After conjugating by $F^{12,3}$ this becomes: 
\begin{equation}
\label{eq:KV-Hex2-alt-conjugated}
\begin{aligned}
(F^{12,3})^{-1}  e^{\inner^{1,3}+\inner^{2,3}}  F^{12,3}
&=
F^{1,2}(F^{2,3})^{-1}(F^{1,23})^{-1}
e^{\frac{\inner^{3,2}}{2}} \\
&\quad
 F^{1,32}F^{3,2}(F^{1,3})^{-1}(F^{13,2})^{-1} 
e^{\frac{\inner^{3,1}}{2}}  e^{\frac{\inner^{1,3}}{2}} \\
&\quad\
 F^{13,2}F^{1,3}(F^{3,2})^{-1}(F^{1,32})^{-1} 
e^{\frac{\inner^{2,3}}{2}} \\
&\quad
 F^{1,23}\,F^{2,3}\,(F^{1,2})^{-1}.
\end{aligned}
\end{equation} 
Since $F^{1,2}$ commutes with $(F^{12,3})^{-1}  e^{\inner^{1,3}+\inner^{2,3}}  F^{12,3}$, we can conjugate~\eqref{eq:KV-Hex2-alt-conjugated} by $F^{1,2}$ to get
\begin{equation}
\begin{aligned}
(F^{12,3})^{-1}  e^{\inner^{1,3}+\inner^{2,3}}  F^{12,3}
&=
(F^{2,3})^{-1}(F^{1,23})^{-1}
e^{\frac{\inner^{3,2}}{2}} \\
&\quad
 F^{1,32}F^{3,2}(F^{1,3})^{-1}(F^{13,2})^{-1} 
e^{\frac{\inner^{3,1}}{2}}  e^{\frac{\inner^{1,3}}{2}} \\
&\quad\
 F^{13,2}F^{1,3}(F^{3,2})^{-1}(F^{1,32})^{-1} 
e^{\frac{\inner^{2,3}}{2}} \\
&\quad
 F^{1,23}\,F^{2,3}.
\end{aligned}
\end{equation} 
Further, since $e^{\frac{\inner^{i,j}}{2}}$ commutes with the $F^{k,ij}$ and $F^{ij,k}$, this can be simplified to 
\begin{equation}
\begin{aligned}
(F^{12,3})^{-1}  e^{\inner^{1,3}+\inner^{2,3}}  F^{12,3}
&=
(F^{2,3})^{-1}
e^{\frac{\inner^{3,2}}{2}} 
 F^{3,2}(F^{1,3})^{-1} 
e^{\frac{\inner^{3,1}}{2}}  e^{\frac{\inner^{1,3}}{2}} 
 F^{1,3}(F^{3,2})^{-1} 
e^{\frac{\inner^{2,3}}{2}} 
F^{2,3} 
\end{aligned}
\end{equation} 
Using that $e^{\frac{\inner^{3,1}}{2}}  e^{\frac{\inner^{1,3}}{2}} = e^{\inner^{1,3}}$, we can now apply \eqref{eq:ThetaAndT} on both sides 
to arrive at
\begin{equation}
\label{eq:KV-Hex2-alt-expanded-B32B23}
\begin{aligned}
\Theta^{12,3}
&=
\bigl( (F^{2,3})^{-1} e^{\frac{\inner^{2,3}}{2}} F^{3,2} \bigr)
\Theta^{1,3}
\bigl( (F^{3,2})^{-1} e^{\frac{\inner^{2,3}}{2}} F^{2,3} \bigr).
\end{aligned}
\end{equation}Passing through the isomorphism $\TAut_3\cong\TAut_2^+$, we can now check:
\begin{align*}
    \varphi^{1}_{F}(E^{01,2})
    = \Theta^{01,2} 
&= \etw^{2,1}\Theta^{0,2} \etw^{1,2}\\
&=
(F^{1,2})^{-1}e^{\frac{\inner^{2,1}}{2}}F^{2,1}\Theta^{0,2}(F^{2,1})^{-1}e^{\frac{\inner^{1,2}}{2}}F^{1,2} \\
&=
\varphi^1_{F}
\left(
(\Psi^{0,1,2})^{-1} R^{2,1} 
\Psi^{0,2,1} E^{0,2} 
(\Psi^{0,2,1})^{-1} R^{1,2} \Psi^{0,1,2}
\right).
\end{align*}
Therefore, $\varphi_F^1$ preserves the octagon equation~\eqref{eqn:O}.

It remains to see that $\varphi_F^1$ preserves the right pentagon~\eqref{eqn:RP}. We first rewrite the expression $\Theta^{12,3}\Theta^{1,2}$ as follows:
\begin{align*}
\Theta^{12,3}\Theta^{1,2}
&= \etw^{3,2}\etw^{3,1}\left(\etw^{1,3}\etw^{2,3}\etw^{2,1}\right)\etw^{1,2}\\
&\overset{(1)}{=} \etw^{3,2}\etw^{3,1}\etw^{2,1}\left(\etw^{2,3}\etw^{1,3}\etw^{1,2}\right)\\
&\overset{(2)}{=} \etw^{3,2}\etw^{3,1}\etw^{2,1}\etw^{1,2}\etw^{1,3}\etw^{2,3}\\
&\overset{(3)}{=} \etw^{3,2}\left((F^{3,2})^{-1}\etw^{32,1}F^{3,2}\right)
\left((F^{3,2})^{-1}\etw^{1,32}F^{3,2}\right)\etw^{2,3}\\
&= \etw^{3,2}(F^{3,2})^{-1}\Theta^{1,32}F^{3,2}\etw^{2,3}.
\end{align*}
Here $(1)$ and $(2)$ are applications of the Yang--Baxter relation \eqref{eq:YB}, while $(3)$ uses \eqref{eq:braided-F} together with the relation
\[
F^{2,3}\etw^{1,2}\etw^{1,3}(F^{2,3})^{-1}=\etw^{1,23},
\]
permuted by $(2\,3)$.

Because $F$ is symmetric, we have $\etw^{3,2}=(F^{2,3})^{-1}e^{\frac{\inner^{3,2}}{2}}F^{3,2}$ and $\etw^{1,2}=(F^{2,1})^{-1}e^{\frac{\inner^{1,2}}{2}}F^{1,2}$. Putting this together, and passing through the isomorphism $\TAut_3\cong  \TAut^+_2$ to relabel the variables, we have: 
\begin{align*}
    \varphi^{1}_{F}(E^{01,2} E^{0,1}) 
    = \Theta^{01,2}\Theta^{0,1}
    &= \etw^{2,1}(F^{2,1})^{-1}\Theta^{0,21}F^{2,1}\etw^{1,2}\\
    &= (F^{1,2})^{-1}e^{\frac{\inner^{2,1}}{2}}F^{2,1}(F^{2,1})^{-1}\Theta^{0,21}F^{2,1}(F^{2,1})^{-1}e^{\frac{\inner^{1,2}}{2}}F^{1,2}\\
    &= (F^{1,2})^{-1}e^{\frac{\inner^{2,1}}{2}}\Theta^{0,21}e^{\frac{\inner^{1,2}}{2}}F^{1,2}\\
    &= \varphi_{F}^{1}((\Psi^{0,1,2})^{-1} R^{2,1} E^{0,21} R^{1,2} \Psi^{0,1,2}). 
\end{align*} 
It follows that $\varphi_F^1$ preserves the right pentagon equation \eqref{eqn:RP}.

In summary $(\varphi^1_{F},\varphi_{F}):\hPaB^1\rightarrow \TAut^+$ is a map of moperads (in sets) and $\PaB$-modules in prounipotent groups.
\end{proof}

In the case where the KV associator $G_{F} \eqdef F^{1,23}F^{2,3}(F^{1,2})^{-1}(F^{12,3})^{-1}$ is a Drinfeld associator, the corresponding moperad morphism $(\varphi^1_F, \varphi_F): \hPaB^1 \to \TAut^+$ factors through a moperad equivalence $\hPaB^1 \to \CD^+$. 
The converse also holds, as the following corollary shows.

\begin{cor}
\label{cor: moperad map factors if KV Ass is Drinf Ass}
Let $F$ be a symmetric KV solution. Then the moperad map $(\varphi_F^1,\varphi_F):\hPaB^1\to \TAut^+$ factors through a moperad equivalence $(\varphi_F^{\prime 1},\varphi_F'):\hPaB^1\to \CD^+$ if and only if the corresponding KV associator $G_F$ is in the $\iota$-image of a Drinfeld associator.
\end{cor}

\begin{proof}
Suppose first that $(\varphi_F^1,\varphi_F)$ factors as
\[
\begin{tikzcd}
\hPaB^1
\arrow[r, "\varphi_F^{\prime 1}"]
\arrow[rr, dashed, bend right=15, "\varphi_F^1"']
&
\CD^+ \arrow[r, hook, "\iota"]
&
\TAut^+,
\end{tikzcd}
\]
where $(\varphi_F^{\prime 1},\varphi_F'):\hPaB^1\to \CD^+$ is a moperad equivalence. Then the underlying operad map $\varphi_F':\hPaB\to \CD$ is an operad equivalence. Hence, by Theorem~\ref{cor: operadic associators}, the element $f \eqdef \varphi_F'(\Phi^{1,2,3}) \in \exp(\ib_3)$ is a Drinfeld associator. Since
\[
G_F=\varphi_F(\Phi^{1,2,3})=\iota\!\left(\varphi_F'(\Phi^{1,2,3})\right)=\iota(f),
\]
it follows that $G_F$ is the $\iota$-image of a Drinfeld associator.

Conversely, suppose that $(\varphi_F^1,\varphi_F):\hPaB^1\to \TAut^+$ is defined on generators by
\[
\varphi_F(\Phi^{1,2,3})\coloneqq G_F=\iota(f), \ \varphi_F(R^{1,2})\coloneqq \iota\!\left(e^{t_{12}/2}\right)=e^{\mathbf t/2}, \ \varphi_F^1(E^{0,1})\coloneqq \Theta^{0,1}, \ \text{and} \  \varphi_F^1(\Psi^{0,1,2})\coloneqq F^{1,2}.
\]
We aim to define a factorization \[
\begin{tikzcd}
\hPaB^1
\arrow[r, dashed, "\varphi_F^{\prime 1}"]
\arrow[rr, bend right=15, "\varphi_F^1"']
&
\CD^+ \arrow[r, dashed, "\varphi_F^{\prime \prime 1}"]
&
\TAut^+,
\end{tikzcd}
\] in which  $(\varphi_F^{\prime 1},\varphi_F'):\hPaB^1\xrightarrow{\cong}\CD^+$ is a moperad equivalence. The assumption that $G_F$ lies in the $\iota$-image of a Drinfeld associator $(1,f)$ implies, via Corollary~\ref{prop:operad-factors-Drinfeld-associator}, that the underlying operad map $\varphi_F:\hPaB \to \SAut$ factors through an operad equivalence $\varphi_F':\hPaB \xrightarrow{\cong}\CD.$  By Lemma~\ref{lemma: inclusion PaB1 to PaB+}, this operad equivalence induces a moperad equivalence
\[
(\varphi_F^{\prime 1},\varphi_F'):\hPaB^1 \xrightarrow{\cong}\CD^+,
\]
obtained by precomposing the shifted equivalence $(\varphi_F^+,\varphi_F'):\hPaB^+\xrightarrow{\cong}\CD^+$ with the inclusion $(\rho^1,\id):\hPaB^1 \hookrightarrow \hPaB^+$. Moreover, this map is explicitly given on moperad generators by
\[
\varphi_F^{\prime 1}(E^{0,1})=e^{t_{01}}
\qquad \text{and} \qquad
\varphi_F^{\prime 1}(\Psi^{0,1,2})=f(t_{01},t_{12}). 
\]

We now aim to define a map $(\varphi_F^{\prime\prime 1},\varphi_F^{\prime\prime}):\CD^+\to \TAut^+$ so that $(\varphi_F^{\prime\prime 1},\varphi_F^{\prime\prime})\circ (\varphi_F^{\prime 1},\varphi_F^{\prime})= (\varphi_F^1,\varphi_F)$. 
Write \[\bar{F}^{1,2}:=\iota(\varphi_F^{\prime 1}(\Psi^{0,1,2})) = \iota(f(t_{01},t_{12})) \quad \text{and} \quad \bar{\Theta}:=\iota(\varphi_F^{\prime 1}(E^{0,1})) =\iota(e^{t_{01}}).\] The value $\bar{F}^{1,2}$ gives rise to the same pentagon element as $F^{1,2}$, namely
\[\bar{F}^{1,23}\bar{F}^{2,3}(\bar{F}^{1,2})^{-1}(\bar{F}^{12,3})^{-1} = G_F = F^{1,23}F^{2,3}(F^{1,2})^{-1}(F^{12,3})^{-1}. \]
By Theorem~7.5 of \cite{AT12}, $\bar{F}^{1,2}$ is also a KV solution, and then Proposition~7.2 of \cite{AT12} implies that
\[
F^{1,2}=e^{\lambda \inner}\bar{F}^{1,2}
\]
for some $\lambda\in \K$.  

Now, by \eqref{eq:ThetaAndT},
\[
\Theta^{0,1}
=(F^{1,2})^{-1}e^{\inner}F^{1,2}
=(e^{\lambda\inner}\bar F^{1,2})^{-1}e^{\inner}(e^{\lambda\inner}\bar F^{1,2})
=(\bar{F}^{1,2})^{-1}e^{-\lambda\inner}e^{\inner}e^{\lambda\inner}\bar{F}^{1,2}
=(\bar{F}^{1,2})^{-1}e^{\inner}\bar{F}^{1,2},
\]
since $e^{\lambda\inner}$ commutes with $e^{\inner}$. Thus the original value $\Theta^{0,1}$ may be rewritten in terms of the same tangential automorphism $\bar{F}^{1,2}$ appearing in the factored map.  One can use this, together with the fact that $\Theta$ and $\bar{\Theta}$ must satisfy the equations \eqref{eqn:RP} and \eqref{eqn:O} to write $\Theta$ as an expression in terms of $\bar{\Theta}$. It follows that one then defines the moperad map $(\varphi_F^{\prime\prime 1},\varphi_F^{\prime\prime}):\CD^+\to \TAut^+$ by appropriately twisting the values $\bar{F}^{1,2}:=\iota(\varphi_F^{\prime 1}(\Psi^{0,1,2})) = \iota(f(t_{01},t_{12}))$ and $\bar{\Theta}:=\iota(\varphi_F^{\prime 1}(E^{0,1})) =\iota(e^{t_{01}})$ so that $(\varphi_F^{\prime\prime 1},\varphi_F^{\prime\prime})\circ (\varphi_F^{\prime 1},\varphi_F^{\prime}) = (\varphi_F^{1},\varphi_F)$ as required. 

\end{proof}

\begin{example}
    In \cite[Thm.~1~(3)]{SeveraWillwacher11}, it is shown that the Alekseev--Torossian associator is in fact a Drinfeld associator. In the language of the present section, this means that the symmetric KV solution corresponding to the Alekseev--Torossian construction gives rise to a moperad map $(\varphi_F^1,\varphi_F):\hPaB^1\to \TAut^+$
    which factors through a moperad equivalence $\hPaB^1 \xrightarrow{\sim } \CD^+.$ Thus, in this case, the operadic and moperadic structures arising from the KV solution are already governed by a genuine Drinfeld associator, rather than only by a KV associator in $\SAut$.
\end{example}

\subsection{Actions of $\gtm$ on KV associators}
\label{subsec:GT1-KV-action}
Precomposition by an automorphism of $\hPaB^1$ acts not only on symmetric KV solutions, but also on the associated $\PaB$-module maps
\[
(\varphi_F^1,\varphi_F):\hPaB^1\to \TAut^+.
\]
The first proposition shows that, on the generators $\Psi^{0,1,2}$ and $\Phi^{1,2,3}$, this induced action recovers the usual $\kvs(2)$-action on a symmetric KV solution $F$ and the corresponding action on its KV associator $G_F$. We then show that, when the map factors through $\CD^+$, this reduces to the classical $\gt_1$-action on Drinfeld associators.

To motivate the induced action on KV associators, let $F$ be a KV solution and let $G\in \kv(2)$. Writing $G\cdot F = FG^{-1}$, the corresponding KV associator is
\begin{align}
\begin{split}
\label{KV-associators-action}
\Phi_{G\cdot F}
&=(FG^{-1})^{1,23}(FG^{-1})^{2,3}\bigl((FG^{-1})^{1,2}\bigr)^{-1}\bigl((FG^{-1})^{12,3}\bigr)^{-1} \\
&=F^{1,23}(G^{-1})^{1,23}F^{2,3}(G^{-1})^{2,3}G^{1,2}(F^{1,2})^{-1}
   G^{12,3}(F^{12,3})^{-1} \\
&=F^{1,23}F^{2,3}\,\Phi_{G^{-1}}\,(F^{1,2})^{-1}(F^{12,3})^{-1}.
\end{split}
\end{align}
Thus the action of $\kv(2)$ on KV solutions induces an action on KV associators. A similar construction gives the corresponding $\krv(2)$-action; see \cite[Equation~(27)]{AT12}. This is the \emph{Drinfeld twist} action.

\begin{prop}
\label{prop:GTM-action-on-KV-associators}
Let $F$ be a symmetric KV solution, and let $(\varphi_F^1,\varphi_F):\hPaB^1\to \TAut^+$ be the associated $\PaB$-module map. Let $\gamma=(1,f_\gamma,g_\gamma)\in \gtm_1$, and let
$(\vartheta_\gamma^1,\vartheta_\gamma):\hPaB^1\to \hPaB^1$ be the corresponding object-fixing automorphism. Then the composite
\[
(\varphi_F^1\circ \vartheta_\gamma^1,\varphi_F\circ \vartheta_\gamma):\hPaB^1\to \TAut^+
\]
is the $\PaB$-module map associated to the symmetric KV solution $G_\gamma\cdot F,$ where $G_\gamma\in \kvs(2)$ is the image of $\gamma$ under the inclusion $\gtm_1\hookrightarrow \kvs(2)$. In particular,
\[
(\varphi_F^1\circ \vartheta_\gamma^1)(\Psi^{0,1,2})=G_\gamma\cdot F
\quad\text{and}\quad
(\varphi_F\circ \vartheta_\gamma)(\Phi^{1,2,3})=\Phi_{\,G_\gamma\cdot F}.
\]
\end{prop}

\begin{proof}
By \eqref{eq:values-aut}, the automorphism $(\vartheta_\gamma^1,\vartheta_\gamma)$ acts on the generators by
\[
\vartheta_\gamma^1(\Psi^{0,1,2})=\Psi^{0,1,2}\cdot g_\gamma(x_{01},x_{12})
\qquad \text{and} \qquad
\vartheta_\gamma(\Phi^{1,2,3})=\Phi^{1,2,3}\cdot f_\gamma(x_{12},x_{23}).
\]
If we now apply $(\varphi_F^1,\varphi_F)$, we obtain $(\varphi_F^1\circ \vartheta_\gamma^1)(\Psi^{0,1,2}) =F^{1,2}\cdot \varphi_F^1\!\bigl(g_\gamma(x_{01},x_{12})\bigr).$
By the identification of $\gtm_1$ with a subgroup of $\kvs(2)$ from Theorem~\ref{thm:KV-action-F-012}, the element $\varphi_F^1(g_\gamma(x_{01},x_{12}))$ is exactly $G_\gamma^{-1}$, so the right-hand side is $G_\gamma\cdot F$.

Similarly, the value on $\Phi^{1,2,3}$ is $(\varphi_F\circ \vartheta_\gamma)(\Phi^{1,2,3})= G_F\cdot \varphi_F\!\bigl(f_\gamma(x_{12},x_{23})\bigr),$
and this is precisely the KV associator $\Phi_{G_\gamma\cdot F}$ by the induced action formula for KV associators. Hence the composite $(\varphi_F^1\circ \vartheta_\gamma^1,\varphi_F\circ \vartheta_\gamma)$ is the module map associated to $G_\gamma\cdot F$.
\end{proof}

When $(\varphi_F^1,\varphi_F)$ factors through a moperad equivalence $\hPaB^1 \xrightarrow{\sim} \CD^+,$ the preceding action reduces to the classical $\gt_1$-action on Drinfeld associators. This is the moperadic analogue of \cite[Proposition~9.13]{AT12}. 

\begin{prop}
\label{prop:GT1-reduces-to-GT}
Let $\varphi^1: \hPaB^1 \to \TAut^+$ and $\vartheta^1: \hPaB^1 \to \hPaB^1$ be as in the previous proposition. 
Suppose moreover that $\varphi^1$ factors through $\CD^+$.
Then, the composite
\[
\begin{tikzcd}
(\varphi^1 \circ \vartheta^1,\varphi \circ \vartheta): \hPaB^1\arrow[r, "\vartheta^1"] & \hPaB^1 \arrow[r, "\varphi^1"] & \CD^+ \hookrightarrow \TAut^+
\end{tikzcd}
\]
is such that $\varphi^1\circ \vartheta^1(\Psi^{0,1,2})$ and $\varphi \circ \vartheta (\Phi^{1,2,3})$ are both given by Equation~\eqref{eqn: action of gt on Drinf ass}.
In other words, the induced action of $\gtm$ on Drinfeld associators coincides with the classical $\gt_1$ action. 
\end{prop}

\begin{proof}
By Proposition~\ref{prop: moperad isos are associators}, we know that $(\varphi^1, \varphi)$ factors through a composition
\[
\begin{tikzcd}
\hPaB^1 \arrow[r, "\rho^1"] & \hPaB^+ \arrow[r, "\varphi^+"] & \CD^+,
\end{tikzcd}
\] where $\rho^1$ sends the associativity isomorphism $\Psi^{0,1,2}$ to the shifted associativity isomorphism $\Phi^{1,2,3}$, and the generators of the $\hPaB$-moperad $\hPaB^1$ are defined to be
\[
\varphi^1(\Psi^{0,1,2}) = f_\varphi(t_{01}, t_{12}), \quad \varphi^1(E^{0,1}) = e^{t_{01}}.
\]

Precomposing with $(\vartheta^1, \vartheta)$ gives a new $\hPaB$-moperad map 
$(\varphi^1 \circ \vartheta^1, \varphi \circ \vartheta)$, whose underlying operad map is:
\[(\varphi \circ \vartheta) (R^{1,2}) = e^{\frac{t_{12}}{2}} \quad \text{and} \quad (\varphi \circ \vartheta) (\Phi^{1,2,3}) = f_\varphi(t_{12}, t_{23}) \cdot f_\vartheta(e^{t_{12}}, f_\varphi(t_{12}, t_{23})^{-1} e^{t_{23}} f_\varphi(t_{12}, t_{23})).\] This corresponds to the action of an element $(1, f_\vartheta)\in \gt_1$ and a Drinfeld associator $(1, f_\varphi)$ as in \eqref{eqn: action of gt on Drinf ass}.

At the moperad level, we also have:
\[
(\varphi^1 \circ \vartheta^1)(E^{0,1}) = e^{t_{01}} \quad \text{and} \quad 
(\varphi^1 \circ \vartheta^1)(\Psi^{0,1,2}) 
= f_\varphi(t_{01}, t_{12}) \cdot f_\vartheta(\varphi^1(x_{01}), \varphi^1(x_{12})),
\]
where the pure braids $x_{01}, x_{12} \in \PB_2^1$ are represented as morphisms in $\Aut_{\hPaB^1(2)}((01)2)$ as:
\[
x_{01} = \id_{01} \circ_1 E^{0,1} \quad \text{and} \quad 
x_{12} = \Psi_{0,1,2}^{-1} \cdot (\id_{01} \circ_1 (R^{2,1}R^{1,2})) \cdot \Psi_{0,1,2}.
\]
The images under $\varphi^1$ are:
\[
\varphi^1(x_{01}) = e^{t_{01}} \quad \text{and} \quad
\varphi^1(x_{12}) = f_\varphi(t_{01}, t_{12})^{-1} e^{t_{12}} f_\varphi(t_{01}, t_{12}).
\]
Hence,
\[
(\varphi^1 \circ \vartheta^1)(\Psi^{0,1,2}) = f_\varphi(t_{01}, t_{12}) \cdot f_\vartheta(e^{t_{01}}, f_\varphi(t_{01}, t_{12})^{-1} e^{t_{12}} f_\varphi(t_{01}, t_{12})).
\]
This shows that precomposition with $(\vartheta^1, \vartheta)$ gives the expected $\gt_{1}$-action on the associator $f_\varphi$, when viewed as a moperad map $\hPaB^1 \to \CD^+$.
\end{proof}

\appendix
\section{Operadic Structures and Kashiwara--Vergne Solutions}
\label{sec:Operadapp}
\subsection{Operadic structure on Lie algebras and $\tder$}
In this section, we provide proofs and examples of some of the statements discussed in Section~\ref{sec: background on tder and KV solutions}. Recall Definition~\ref{def:LieOperad}, which gives a right $\Sigma_n$ action on $\lie_n$ as well as a partial composition map, which are given by, for any $f \in \lie_m, g \in \lie_n, \sigma \in \Sigma_m, 1 \leq i \leq m$,
\begin{gather*}
    f^\sigma(x_1,\ldots,x_n) \eqdef f(x_{\sigma^{-1}(1)},\ldots,x_{\sigma^{-1}(n)}), \\
(f \circ_i g)(x_1,\ldots,x_{m+n-1}) 
\eqdef 
f(x_1, \ldots, x_{i-1},\sum_{j=i}^{i+n-1}x_j,x_{i+n},\ldots,x_{m+n-1}) + g(x_i, \ldots,x_{i+n-1}).
\end{gather*}

As we show below in Proposition~\ref{prop:operad-lie}, the symmetric sequence of free Lie algebras with these partial compositions is a linear operad denoted $\lie$.

\begin{example}
Consider the Lie monomials
$f(x_1,x_2) = [x_1,x_2] \in \lie_2$ and $g(x_1,x_2) = [x_1,[x_1,x_2]] \in \lie_2$.
Then $f \circ_1 g \in \lie_3$ is computed by substituting $x_1 \mapsto x_1 + x_2$ in $f$ and adding $g(x_1,x_2)$ to the result:
\[
(f \circ_1 g)(x_1,x_2,x_3) = f(x_1 + x_2, x_3) + g(x_1,x_2) = [x_1 + x_2, x_3] + [x_1, [x_1, x_2]].\]
Similarly, composing at the second input gives:
\[
(f \circ_2 g)(x_1,x_2,x_3) = f(x_1, x_2 + x_3) + g(x_2,x_3) = [x_1, x_2 + x_3] + [x_2,[x_2,x_3]].\]
\end{example}

\begin{prop}
    The symmetric sequence of free Lie algebras $\lie$ with the partial compositions above forms a linear operad.
\end{prop}

\begin{proof}
Observe first that both the $\Sigma_n$ action and the $\circ_i$ operations are linear maps.
We proceed to verify that they satisfy axioms (i)--(iii) from \cref{defn: operad}. For (i), we observe that the element $0 \in \lie_1$ is such that $f \circ_i 0 = 0 \circ_1 f = f$ for any $f \in \lie$.
   To check (ii), associativity,  let $f \in \lie_m, \, g \in \lie_n$, and $h \in \lie_k$ be three Lie words, and let $1 \leq j \leq m$. For $1 \leq i \leq j-1$, one finds that $(f \circ_j g) \circ_i h$ and $(f \circ_i h)\circ_{j+k-1} g$ are both equal to 
\[ f(x_1, \ldots,x_{i-1},\sum_{\ell=i}^{i+k-1}x_\ell,x_{i+k},\ldots,\sum_{\ell=j+k-1}^{j+k+n-2}x_\ell,\ldots,x_{m+n+k-2}) + g(x_{j+k-1},\ldots,x_{j+k+n-2})+ h(x_i,\ldots,x_{i+k-1}).\]
   For $j \leq i \leq j+n-1$,  one finds that $(f \circ_j g) \circ_i h)$ and $f\circ_j (g \circ_{i-j+1} h)$ are both equal to 
    \[ f(x_1, \ldots, x_{j-1},
    \sum_{\ell=j}^{j+n+k-2}x_\ell,
    \ldots,x_{m+n+k-2}) + 
    g(x_j,\ldots,\sum_{\ell=i}^{i+k-1}x_\ell,\ldots,x_{j+n+k-2})+ 
    h(x_i,\ldots,x_{i+k-1}).\]
    Finally, if $j+k \leq i \leq m+n-1$, then $(f \circ_j g) \circ_i h$ and $(f \circ_{i-n+1} h) \circ_j g$ are both equal to
    \[ f(x_1, \ldots, x_{j-1},
    \sum_{\ell=j}^{j+n-1}x_\ell,x_{j+n},
    \ldots,x_{i-1}, \sum_{\ell=i}^{i+k-1}x_\ell, x_{i+k},\ldots, x_{m+n+k-2}) + 
    g(x_j,\ldots, x_{j+n-1})+ 
    h(x_i,\ldots,x_{i+k-1}).\]
    
\medskip 

    It remains to check (iii), equivariance. Let $f,g$ be two Lie words as before, let $\sigma \in \Sigma_m$ and $\tau \in \Sigma_n$ be two permutations, and let $1 \leq i \leq m$.  Then, one computes that
    \begin{eqnarray*}
         (f^\sigma \circ_{\sigma^{-1}(i)} g^\tau)(x_1,\ldots,x_{m+n-1}) 
         &=&
        f(x_{\sigma^{-1}(1)},\ldots,x_{\sigma^{-1}(m)})\circ_{\sigma^{-1}(i)}
        g(x_{\tau^{-1}(1)},\ldots,x_{\tau^{-1}(n)}) \\
        & = & (f\circ_i g)(x_{(\sigma \circ_i \tau)^{-1}(1)},\ldots,x_{(\sigma \circ_i \tau)^{-1}(m+n-1)}) \\
        &=& (f\circ_{i} g)^{\sigma \circ_i \tau}(x_1,\ldots,x_{m+n-1}).
    \end{eqnarray*}
    
    We conclude that $\lie$ is a linear operad.
\end{proof}

\begin{proof}[{Proof of \cref{prop:operad-lie}}]
    The same formulas and the same proof work \emph{mutatis mutandis} for the universal enveloping algebra $\ass$ and the vector space of cyclic words $\cyc$.
    Checking that the quotient map $\trace : \ass \to \cyc$ is a morphism of linear operads can be verified directly: for $f \in \ass_m$, $g \in \ass_n$, $1 \leqslant i \leqslant m$ and $\sigma \in \Sigma_m$, we clearly have $\trace(f^\sigma)=\trace(f)^\sigma$ and $\trace(f\circ_i g)=\trace(f) \circ_i \trace(g)$.
\end{proof}

As discussed in \cref{subsec:lieoperad}, this in turn gives an operadic structure on $\tder$. 
Namely, for any $u=(a_1,\ldots, a_m) \in \tder_m, v=(b_1,\ldots, b_n) \in \tder_n, \sigma \in \Sigma_m, 1 \leq i \leq m$, we have 
\begin{gather*}
    u^\sigma 
\eqdef (a_{\sigma^{-1}(1)}^\sigma,\ldots, a_{\sigma^{-1}(n)}^\sigma), \\
u \circ_i v \eqdef (a_1\circ_i 0,\ldots,a_{i-1}\circ_i 0,a_i \circ_i b_1,\ldots,a_i \circ_i b_n, a_{i+1}\circ_i 0,\ldots,a_m \circ_i 0).
\end{gather*}
Here $a_j^\sigma$ denotes the $\Sigma_m$ action on $\lie$, $\circ_i$ on the right hand side above denotes operadic composition in the operad $\lie$, and $0$ stands for $0_n \in \lie_n$. We next focus on the operadic structure of $\sder_n \subseteq \tder_n$.

Recall that the universal enveloping algebra of the free Lie algebra can be identified with the vector space of free associative monomials in $x_1,\ldots, x_n$, that is $\widehat{U}(\lie_n)=\ass_n$. Since the universal enveloping algebra functor is monoidal, the operadic composition on $\lie$ induces a  linear operad structure on $\ass=\{\ass_n\}_{n\geq 1}$ via the same substitution formulas. This structure further descends to a well-defined operad on the sequence of cyclic words $\cyc =\{ \cyc_n\}_{n\geq 1}$, where $\cyc_n = \ass_n / [\ass_n, \ass_n]$ is the linear quotient by commutators.

\begin{lemma}
\label{lem:trace-is-operadic}
The operad structure on $\ass=\{\ass_n\}_{n\geq 1}$ descends to $\cyc =\{ \cyc_n\}_{n\geq 1}$ and the quotient map $\trace: \ass \to \cyc$ is a map of linear operads.
\end{lemma}
\begin{proof}
To verify that $\trace$ is an operad map, we need to check compatibility with symmetric group actions and operadic composition. Let $f$ be a word in $\ass_n$, and let $\sigma$ be a permutation in $\Sigma_n$. Then, we clearly have $\trace(f^\sigma)=(\trace(f))^\sigma$. It is similarly immediate to check that $\trace(f\circ_i g)=\trace(f)\circ_i \trace(g)$.
\end{proof}

Using the linear isomorphism $\tder_n \oplus \mathfrak{a}_n \cong \lie_n^{\oplus n}$, the linear operad structure on $\lie=\{\lie_n\}_{n\geq 1}$ restricts to a linear operad of tangential derivations. For $1 \leq i \leq m$, define the composition map
\[ \circ_i : \tder_m \oplus \tder_n \to \tder_{m+n-1}\] to be the $\K$-linear map which takes $u=(a_1,\ldots, a_m)$ and $
v=(b_1,\ldots, b_n)$ to
\begin{equation}\label{eq:tderOperadCompapp}
u \circ_i v \eqdef (a_1\circ_i 0,\ldots,a_{i-1}\circ_i 0,a_i \circ_i b_1,\ldots,a_i \circ_i b_n, a_{i+1}\circ_i 0,\ldots,a_m \circ_i 0).
\end{equation} 
Here, the $\circ_i$ on the right hand side denotes operadic composition in the operad $\lie$, and $0$ stands for $0_n \in \lie_n$.

For any $u=(a_1,\ldots,a_n) \in \tder_n$ and $\sigma \in \Sigma_n$ define
\[u^\sigma 
\eqdef (a_{\sigma(1)}^\sigma,\ldots, a_{\sigma(n)}^\sigma),\]
where $a_j^\sigma$ denotes the $\Sigma_n$ action on $\lie$.
This is the canonical action induced by the one on $\lie$ in the sense that it is the unique action which makes the following diagram 
\begin{equation} \label{diag:sigma-n-action}
    \begin{tikzcd}
\lie_n \arrow[r, "(-)^\sigma"] \arrow[d, "u"'] & \lie_n \arrow[d, "u^{\sigma}"] \\
\lie_n \arrow[r, "(-)^\sigma"']                & \lie_n      
\end{tikzcd}
\end{equation} commute.  

\begin{prop}
\label{prop:tder-operadapp}
The family $\tder=\{ \tder_n\}_{n \geq 1}$ of tangential derivations endowed with the above $\circ_i$ operations and $\Sigma_n$-actions forms a $\K$-linear operad. 
\end{prop}

\begin{proof}
We need check the axioms (i)--(iii) from \cref{defn: colored operad}.  The tangential derivation $0 \in \tder_1$ is such that $u \circ_i 0 = 0 \circ_1 u = u$ for any $u \in \tder$, therefore axiom (i) is satisfied. For the associativity axioms (ii), we let $u=(a_1,\ldots,a_m)$, $v=(b_1,\ldots,b_n)$ and $w=(c_1,\ldots,c_k)$ be tangential derivations in $\lie_m,\lie_n$ and $\lie_k$, respectively, and we select $1 \leq i< j \leq m$.
   We have
   \begin{eqnarray*}
       (u \circ_j v)\circ_i w 
&=& (a_1 \circ_j 0,\ldots,a_j \circ_j b_1,\ldots,a_j \circ_j b_n,\ldots,a_m \circ_j 0) \circ_i w \\
       &=& (
       (a_1 \circ_j 0) \circ_i 0,\ldots,
       (a_i \circ_j 0)\circ_i c_1,\ldots,
       (a_i \circ_j 0)\circ_i c_k,\ldots, \\
       &&
       (a_j \circ_j b_1)\circ_i 0,\ldots,
       (a_j \circ_j b_n)\circ_i 0,\ldots,
       (a_m \circ_j 0)\circ_i 0)\\
       &=& (
       (a_1 \circ_i 0) \circ_{j+k-1} 0,\ldots,
       (a_i \circ_i c_1)\circ_{j+k-1} 0,\ldots,
       (a_i \circ_i c_k)\circ_{j+k-1} 0,\ldots, \\
       &&
       (a_j \circ_i 0)\circ_{j+k-1} b_1,\ldots,
       (a_j \circ_i 0)\circ_{j+k-1} b_n,\ldots,
       (a_m \circ_i 0)\circ_{j+k-1} 0)\\
       &=& (u \circ_i w) \circ_{j+k-1} v \ ,
   \end{eqnarray*}
   since by \cref{prop:operad-lie} the $\circ_i$ composition maps of the operad $\lie$ satisfy axiom (ii).
   The proofs of the other cases are similar.
   As to the equivariance axiom (iii), let $u \in \tder_m$ and $v \in \tder_n$ be two tangential derivations as before, let $\sigma$ and $\tau$ be elements of $\Sigma_m$ and $\Sigma_n$, respectively, and select $1 \leq i \leq m$.
   Then, we have
   \begin{eqnarray*}
     u^\sigma \circ_{\sigma(i)} v^\tau &=& (
       a_{\sigma(1)}^\sigma\circ_{\sigma(i)} 0,\ldots,
       a_{\sigma(i)}^\sigma \circ_{\sigma(i)} b_{\tau(1)}^\tau,\ldots,
       a_{\sigma(i)}^\sigma \circ_{\sigma(i)} b_{\tau(n)}^\tau,\ldots,
       a_{\sigma(m)}^\sigma \circ_{\sigma(i)} 0)\\
       &=& (
       (a_{\sigma(1)}\circ_{i} 0)^{\sigma \circ_{\sigma(i)} \tau},\ldots,
       (a_{\sigma(i)} \circ_{i} b_{\tau(1)})^{\sigma \circ_{\sigma(i)} \tau},\ldots,
       (a_{\sigma(i)} \circ_{i} b_{\tau(n)})^{\sigma \circ_{\sigma(i)} \tau},\ldots,
       (a_{\sigma(m)} \circ_{i} 0)^{\sigma \circ_{\sigma(i)} \tau})\\
       &=& (u \circ_{i} v)^{\sigma \circ_{\sigma(i)} \tau} \ ,
   \end{eqnarray*}
   where we used the fact that the partial composition maps in $\lie$ satisfy equivariance (\cref{prop:operad-lie}).
   We conclude that $\tder$ is a linear operad.
\end{proof}

\begin{example}
Consider the tangential derivations
\[
u = ([x_1,x_2],\, 0) \in \tder_2 \quad \text{and} \quad v = (x_1,\, x_2) \in \tder_2.
\]
Using the operadic composition formula~\eqref{eq:tderOperadCompapp}, we compute
\[
u \circ_1 v 
= \big( [x_1 + x_2, x_3] + x_1,\, [x_1 + x_2, x_3] + x_2,\, 0 \big) \in \tder_3.
\]
\end{example}

\begin{example}\label{example: cosimplicial elements of tder we need later}
Let $u = (f_1(x_1,x_2),\, f_2(x_1,x_2)) \in \tder_2 $ be an arbitrary tangential derivation. 
By composing $u$ with identities via the operad structure on $\tder$, we obtain a family of higher-arity tangential derivations that are used to construct explicit examples in the paper. 
These compositions recover the cosimplicial framework used in \cite[Section~3]{AT12} and \cite[Section~7]{AKKN_genus_zero}, see \cref{sec:operadic-cohomology}. For these 
\begin{align*} 
u^{n-1,n}  & \eqdef 0_{n-2} \circ_{n-2} u 
= (0,\ldots,0,\, f_1(x_{n-1},x_n),\, f_2(x_{n-1},x_n)), \\
u^{n-2,(n-1)n} & \eqdef 0_{n-1} \circ_{n-1}(u \circ_2 0_2) 
= (0,\ldots,0,\, f_1(x_{n-2},x_{n-1}+x_n),\, f_2(x_{n-2},x_{n-1}+x_n),\, f_2(x_{n-2},x_{n-1}+x_n)), \\  
u^{2,3\cdots n} &\eqdef 0_2 \circ_2 (u \circ_2 0_{n-2}) 
= (0,\, f_1(x_2,x_3+\cdots+x_n),\, f_2(x_2,x_3+\cdots+x_n),\, \ldots,\, f_2(x_2,x_3+\cdots+x_n)), \\ 
u^{1,2\cdots n} &\eqdef u \circ_2 0_{n-1} 
= (f_1(x_1,x_2+\cdots+x_n),\, f_2(x_1,x_2+\cdots+x_n),\, \ldots,\, f_2(x_1,x_2+\cdots+x_n)).
\end{align*}
\end{example}

Following \cite[Sec.~3.3]{AT12}, we can see that the action of $\tder$ on $\cyc$ is operadic.

\begin{lemma}\label{lem:divergence-is-operadic}
    The non-commutative divergence $j : \tder \to \cyc$ is a map of linear operads. 
\end{lemma}

\begin{proof}
    We need to check compatibility with the symmetric group actions and operadic compositions. 
    Let $u=(a_1,\ldots,a_n)$ be a tangential derivation in $\tder_n$, and let $\sigma$ be a permutation in $\Sigma_n$.
    Then, we have that 
    \begin{eqnarray*}
        j(u^\sigma) & = & \trace\left(\sum_{k=1}^{n} \partial_k(a_{\sigma^{-1}(k)}(x_{\sigma(1)},\ldots,x_{\sigma(n)}))x_k\right) \\
        & = & \trace\left(\sum_{k=1}^{n} \partial_{\sigma(k)}(a_{k}(x_{\sigma(1)},\ldots,x_{\sigma(n)}))x_{\sigma(k)}\right) \\
        & = & \trace\left(\left(\sum_{k=1}^{n} \partial_k(a_k)x_k\right)^\sigma\right) \\
        & = & j(u)^\sigma,
    \end{eqnarray*}
    where we made use in the last line of the fact that $\trace$ commutes with symmetric group actions (\cref{lem:trace-is-operadic}).
    This shows that $j$ respects the symmetric group actions.
    Consider now the operadic composition $u\circ_i v$ of two tangential derivations $u=(a_1,\ldots,a_m)$ in $\tder_m$ and $v=(b_1,\ldots,b_n)$ in $\tder_n$.
    Let us write $f(x_1,\ldots,x_m)\eqdef \sum_{k=1}^{m} \partial_k(a_k)x_k$ and $g(x_1,\ldots,x_n) \eqdef \sum_{k=1}^{n} \partial_k(b_k)x_k$ for the two associative words whose traces are $j(u)$ and $j(v)$, respectively.
    We have 
    \[
    j(u) \circ_i j(v) = \trace(f) \circ_i \trace(g) = \trace(f \circ_i g)=\trace(f \circ_i 0)+\trace(0 \circ_i g),
    \]
    since the trace respects operadic composition (\cref{lem:trace-is-operadic}).
On the other hand, we have 
\[
    j(u\circ_i v) 
    = j(u \circ_i 0 +0 \circ_i v)
    = j(u \circ_i 0) + j(0 \circ_i v).
    \]
Now, we claim that $j(u \circ_i 0)=\trace(f \circ_i 0)$ and $j(0 \circ_i v)=\trace(0 \circ_i g)$.
This will complete the proof, since combined with the previous two equalities, it shows that the non-commutative divergence $j$ respects operadic composition. 
On the one hand, we have by definition
\begin{equation} 
\label{eq:divergence-is-operadic}
    j(u \circ_i 0)= 
\trace\left( 
\sum_{k=1}^{i-1} \partial_k(a_k \circ_i 0)x_k + 
\sum_{k=i}^{i+n-1} \partial_k(a_i \circ_i 0)x_k + 
\sum_{k=i+n}^{m+n-1} \partial_k(a_{k-n+1} \circ_i 0)x_k 
\right).
\end{equation}
On the other hand, we have 
\[
f \circ_i 0 = f\left( x_1,\ldots,x_{i-1},\sum_{k=1}^{i+n-1}x_k,x_{i+n},\ldots,x_m\right).
\]
Remarking that we have $\partial_k(a_i \circ_i 0)=\partial_l(a_i \circ_i 0)$ for any $k,l \in \{i,\ldots,i+n-1\}$, one sees directly that $f \circ_i 0$ is equal to the expression in the brackets on the right-hand side of Equation~\ref{eq:divergence-is-operadic}.
The proof that $j(0 \circ_i v)=\trace(0 \circ_i g)$ is similar.
\end{proof} 

The collection of tangential derivations $\tder = \{ \tder_n \}_{n \geq 1}$ only forms a $\K$-linear operad. Analogously, it is possible to endow the collection $\TAut=\{\TAut_n\}_{n \geq 1}$ with a structure of an operad in sets, however, the definition of the partial compositions is non-canonical. Below we choose a definition which is analogous to, but not exponentiated from, the partial compositions on  $\tder$.

\begin{prop}\label{operad of TAut app}
The collection $\TAut = \{ \TAut_n \}_{n \geq 1}$ inherits a symmetric group action from the symmetric sequence $\tder = \{ \tder_n \}_{n \geq 1}$, and forms an operad in sets. 
The partial composition maps are defined as follows:
\[
\begin{tikzcd}
\TAut_n \times \TAut_m \arrow[r, "\circ_i"] & \TAut_{n+m-1}
\end{tikzcd}
\]
\[
F \circ_i G =  (F \circ_i 1)(1 \circ_i G)
= 
F^{1,\ldots,i-1,\,i(i+1)\cdots(i+n-1),\,i+n,\ldots,m+n-1} 
G^{i,i+1,\ldots,i+(n-1)}
\]
\end{prop}
\begin{proof}
The partial composition
\[
F \circ_i G = (F \circ_i 1)(1 \circ_i G)
\]
is modeled on the corresponding operation in $ \tder $, but defined directly at the group level. The terms
\[
F^{1,\ldots,i-1,\,i(i+1)\cdots(i+n-1),\,i+n,\ldots,m+n-1}
\quad\text{and}\quad
G^{i,i+1,\ldots,i+(n-1)}
\]
represent the image of $F\in\TAut_n$ and $G\in\TAut_m$ under the natural group homomorphisms which include $\TAut_n$ and $\TAut_m$ into $ \TAut_{n+m-1}$ (cf.\cite[Section 3]{AT12} or ~\cite[Section~5.2]{AET10}). The product is then defined as group multiplication in $\TAut_{n+m-1}$. It follows that the map \[
\begin{tikzcd}
\TAut_n \times \TAut_m \arrow[r, "\circ_i"] & \TAut_{n+m-1},
\end{tikzcd}
\] defined by sending $(F,G)\mapsto F\circ_iG$ defines an associative, unital and equivariant map of sets. 
\end{proof}

\begin{remark}
The operad structure on $\TAut$ defined above is non-canonical, and agrees\footnote{The order of composition in our definition is the opposite of that of \cite{AKKN_genus_zero}, since our definition of the set of KV solutions agrees with \cite{AET10}, which is the inverse of that of \cite{AKKN_genus_zero}.} with the operad composition used in \cite{AKKN_genus_zero} to construct Kashiwara--Vergne solutions corresponding to genus zero surfaces with more than three boundary components. 
It does not, however, agree with the operad structure one would obtain by exponentiation of the partial compositions of $\tder$:
\[
\exp(u) \circ_i \exp(v) \neq \exp(u \circ_i v).
\]
where $u \in \tder_n$ and $v \in \tder_m$, and the composition on the right-hand side is as in Equation~\eqref{eq:tderOperadCompapp}.
These notions do, however coincide when restricted to $\sder$ and $\SAut$, where there are canonical operad structures in the categories of Lie algebras and pro-unipotent groupoids.
\end{remark}

\subsection{The operad of special derivations $\sder$}\label{sec:Operad_sder}
While tangential derivations naturally form an operad in vector spaces, they do not preserve a Lie algebra structure under operadic composition. In contrast, the subspace of special derivations forms an operad in Lie algebras, as the direct sum of special derivations is closed under the Lie bracket.

\begin{theorem}
\label{thm:sder-operad}
    The family $\sder$ of special derivations forms a linear suboperad of $\tder$, and $\sder$ is, moreover, an operad in Lie algebras.
\end{theorem}

\begin{proof}
We need to show that $\sder$ is stable under symmetric group actions, partial compositions, and that it contains the unit of $\tder$.  The fact that $0 \in \tder_1$ is in $\sder_1$ is immediate. 
To check the stability under the $\Sigma_n$ action, let $u \in \sder_n$ and observe that since the sum $\sum_{i=1}^{n}x_i \in \lie_n$ is invariant under the action of $\Sigma_n$, diagram (\ref{diag:sigma-n-action}) gives the equality for any $\sigma \in \Sigma_n$,
\[ u\left(\sum_{i=1}^{n}x_i\right)^\sigma = u^\sigma \left(\sum_{i=1}^{n}x_i\right). \]
Since the left hand side is zero by hypothesis, it is also the case of the right hand side, and we have $u^\sigma \in \sder_n$.
It remains to show stability under partial composition. 
Let $u \in \sder_m$ and $v \in \sder_n$ be two special derivations.
Making the change of variables $y_1 \eqdef x_1, \ldots, y_i \eqdef \sum_{\ell=i}^{i+n-1}x_\ell, \ldots, y_m \eqdef x_{m+n-1}$ and $z_1 \eqdef x_i, \ldots, z_n \eqdef x_{i+n-1}$, we have
\[u\circ_i v \left(\sum_{j=1}^{m+n-1}x_j\right)=u \circ_i 0 \left(\sum_{j=1}^{m+n-1}x_j\right) + 0 \circ_i v \left(\sum_{j=1}^{m+n-1}x_j\right)=u\left(\sum_{k=1}^{m}y_k\right)+v\left(\sum_{\ell=1}^{n}z_\ell\right)=0. \]
Therefore, $u\circ_i v \in \sder_{m+n-1}$ and $\sder$ forms a linear suboperad of $\tder$.

To show that $\sder$ is an operad in the category of (degree complete) Lie algebras, we need to show that the $\circ_i$ composition maps and the $\Sigma_n$ actions are morphisms of Lie algebras. The fact that the $\Sigma_n$ actions are morphisms of Lie algebras is straightforward to check. 
For the partial compositions, given $u,u' \in \tder_m$ and $v,v' \in \tder_n$, we want to show that 
\begin{equation}
\label{eq:circ-lie-alg}
[u \circ_i v, u' \circ_i v']=[u,u']\circ_i [v,v'].
\end{equation}
Writing $u\circ_i v=u \circ_i 0 + 0 \circ_i v$, the left hand side of~\eqref{eq:circ-lie-alg} gives four terms, while the right hand side gives two. 
We first show that $[u\circ_i0,u'\circ_i 0]=[u,u']\circ_i 0$; a similar argument shows that $[0\circ_i v,0\circ_i v']=0\circ_i[v,v']$.
Writing $u=(a_1,\ldots,a_m)$ and $u'=(a_1',\ldots,a_m')$, we have $$[u,u']=uu'-u'u=(c_1,\ldots,c_m) \quad \text{where} \quad c_j=u(a_j')-u'(a_j)+[a_j,a_j'].$$
Then, we have $$[u,u']\circ_i 0=(\tilde c_1,\ldots, \tilde c_{i-1},\tilde c_{i},\ldots,\tilde c_i,\tilde c_{i+1},\ldots,\tilde c_m)$$ 
where $\tilde c_j=c_j(x_1,\ldots,\sum_{j=i}^{i+n-1}x_j,x_{i+n},\ldots,x_{m+n-1})$.
Similarly we have 
$$u\circ_i 0=(\tilde a_1,\ldots, \tilde a_{i-1},\tilde a_{i},\ldots,\tilde a_i,\tilde a_{i+1},\ldots,\tilde a_m),
\quad u'\circ_i 0=(\tilde a_1',\ldots, \tilde a_{i-1}',\tilde a_{i}',\ldots,\tilde a_i',\tilde a_{i+1}',\ldots,\tilde a_m')$$
and
$$[u\circ_i 0,u'\circ_i 0]=(\tilde d_1,\ldots, \tilde d_{i-1},\tilde d_{i},\ldots,\tilde d_i,\tilde d_{i+1},\ldots,\tilde d_m) \quad \text{where} \quad \tilde d_j=(u\circ_i 0)(\tilde a_j')-(u'\circ_i 0)(\tilde a_j)+[\tilde a_j,\tilde a_j'].$$
Thus our assertion is equivalent to the equalities $\tilde c_j = \tilde d_j$ for all $j$.
We readily have that $$[a_j,a_j'](x_1,\ldots,\sum_{j=i}^{i+n-1}x_j,x_{i+n},\ldots,x_{m+n-1})=[\tilde a_j,\tilde a_j'].$$
It remains to show that $$u(a_j')(x_1,\ldots,\sum_{j=i}^{i+n-1}x_j,x_{i+n},\ldots,x_{m+n-1})=(u\circ_i 0) (\tilde a_j'),$$
a similar statement holds for $u'(a_j)$.
Indeed, $u(a_j'(x_1,\ldots,x_m))$ is the sum over all occurrences of generators $x_k$ in the word $a_j'$, of the word $a_j'$ with this occurrence replaced by the commutator $[x_k,a_k]$.
Performing the substitution of variables $$(x_1,\ldots,x_m) \mapsto (x_1,\ldots,\sum_{j=i}^{i+n-1}x_j,x_{i+n},\ldots,x_{m+n-1})$$
in this Lie word gives the new Lie word $(u\circ_i 0) (\tilde a_j')$.

It remains to show that $[u\circ_i 0,0\circ_i v']=0$; the proof that $[0\circ_i v,u' \circ_i 0]=0$ is similar.
Let us write $u\circ_i 0=(\tilde a_1,\ldots, \tilde a_{i-1},\tilde a_{i},\ldots,\tilde a_i,\tilde a_{i+1},\ldots,\tilde a_m)$  where $\tilde a_j = a_j(x_1,\ldots,\sum_{j=i}^{i+n-1}x_j,x_{i+n},\ldots,x_{m+n-1})$ and $0\circ_i v'=(0,\ldots,0,\tilde b_1',\ldots, \tilde b_n', 0, \ldots, 0)$ where $\tilde b_j'=b_j'(x_{i},\ldots,x_{i+n-1})$.
We have 
\begin{eqnarray*}
    [u\circ_i 0,0\circ_i v'] & = & ((u\circ_i 0)(0)-(0 \circ_i v')(\tilde a_1)+[\tilde a_1,0]), \\
    &  & \ldots, \\
    &  & ((u\circ_i 0)(\tilde b_1')-(0 \circ_i v')(\tilde a_i)+[\tilde a_i,\tilde b_1']) \\
    &  & \ldots, \\
    &  & ((u\circ_i 0)(\tilde b_n')-(0 \circ_i v')(\tilde a_i)+[\tilde a_i,\tilde b_n']) \\
    &  & \ldots, \\
    &  & ((u\circ_i 0)(0)-(0 \circ_i v')(\tilde a_m)+[\tilde a_m,0])
\end{eqnarray*}

The terms of the form $((u\circ_i 0)(0)-(0 \circ_i v')(\tilde a_j)+[\tilde a_j,0])$ are equal to $-(0 \circ_i v')(\tilde a_j)$.
This is in turn the negative sum over all occurrences of generators in $\tilde{a}_j$ of the bracket with the corresponding entry of $0\circ_i v'$.
These terms are zero, including the occurrences of the sum $x_i+\cdots+x_{i+n-1}$ since $v'$ is a special derivation (i.e.\ we have $(0 \circ_i v')(\sum_{j=i}^{i+n-1}x_j)=0$).
Considering the terms of the form $((u\circ_i 0)(\tilde b_j')-(0 \circ_i v')(\tilde a_i)+[\tilde a_i,\tilde b_j'])$, a similar phenomenon occurs for $(0\circ_i v')(\tilde a_i)$ which is equal to $0$. 
Then, the remaining terms $((u\circ_i 0)(\tilde b_j')+[\tilde a_i,\tilde b_j'])$ cancel thanks to the Jacobi identity.
Therefore, we have $[u\circ_i 0,0\circ_i v']=0$ and the proof is complete.
\end{proof}

\begin{remark}
\label{rem:Pavol-wisdom}
The following observation was communicated to us by Pavol {\v S}evera.
    Let us choose a family $l=\{l_n(x_1,\ldots,x_n)\}_{n \geqslant 2}$ of Lie words $l_n \in \lie_n$ such that 
    $$l_m(x_1,\ldots,x_{i-1},l_n(x_i,\ldots,x_{i+n-1}),x_{i+n},\ldots,x_{m+n-1})=l_{m+n-1}(x_1,\ldots,x_{m+n-1}).$$
    Then, one can define an alternative linear operad structure on $\lie$ by the formula 
    \begin{gather*}
    (f \circ_i g)(x_1,\ldots,x_{m+n-1}) 
    \eqdef 
    f(x_1, \ldots, x_{i-1},l_n(x_i,\ldots,x_{i+n-1}),x_{i+n},\ldots,x_{m+n-1}) + g(x_i, \ldots,x_{i+n-1}).
    \end{gather*}
    We write $\lie^l$ for this operad.
    It induces an operad structure $\tder^l$ on $\tder$ by the same formulas as above. 
    Moreover, one can define the Lie algebra $\sder^l$ of \defn{special derivations} with respect to $l$, which are tangential derivations $u \in \tder_n$ such that 
    $$u(l_n(x_1,\ldots,x_n)) = 0.$$
    The proof of \cref{thm:sder-operad} carries \emph{mutatis mutandis} and shows that $\sder^l$ is a linear operad in Lie algebras. 
    Taking $l_n=x_1+\cdots + x_n$, one recovers the above operad structures, while taking $l_n=\bch(x_1,\ldots,x_n)$, one gets other operad structures that we denote $\lie^\bch$, $\tder^\bch$ and $\sder^\bch$.
\end{remark}

Under prounipotent completion, the standard decomposition $\PB_{n+1}\cong \mathsf{F}_n\rtimes \PB_n$
induces a corresponding decomposition of Lie algebras $\ib_{n+1}\cong \lie_n \rtimes \mathfrak{t}_n.$ This motivates the appearance of special derivations coming from the action of $\mathfrak{t}_n$ on the free Lie algebra $\lie_n$. 

\begin{example}
\label{example: generating DK Lie algebras as derivations}
The element $t=(x_2,x_1)\in\tder_2$ is the tangential derivation that acts on the generators of $\lie_2$ via: \[t(x_1)=[x_1,x_2] \quad \text{and} \quad t(x_2)=[x_2,x_1].\] The element $t$ is a special derivation since $t(x_1+x_2)=[x_1,x_2]+[x_2,x_1] =0$. 
\end{example}
Generalizing Example~\ref{example: generating DK Lie algebras as derivations}, for $1\le i<j\le n$ let
\[
t^{ij}=(0,\ldots,\underbrace{x_j}_{i},\ldots,\underbrace{x_i}_{j},\ldots,0)\in \tder_n.
\]
Then $t^{ij}$ is a special derivation satisfying
\[
t^{ij}(x_i)=[x_i,x_j]
\qquad \text{and} \qquad
t^{ij}(x_j)=[x_j,x_i].
\]
The derivations $t^{ij}\in \sder_n$ span a Lie subalgebra isomorphic to the Drinfeld--Kohno Lie algebra $\mathfrak t_n$ described in Section~\ref{subsec:chorddiag}; see \cite[Proposition 3.11]{AT12}. The following proposition shows that this inclusion $\ib_n\hookrightarrow \sder_n$ extends to an inclusion of operads.

\begin{prop}
\label{prop:DK-operad}
    The family $\ib\eqdef \{\ib_n\}_{n \geq 0}$ of infinitesimal braids is a suboperad of $\sder$.
\end{prop}

\begin{proof}
We need to check that each $\ib_n$ is stable under the action of $\Sigma_n$ and that $\ib$ is stable under operadic composition. It suffices to check these properties on generators.  Let $t_{ij}$ be a generator of $\ib_n$, and let $\sigma$ be a permutation in $\Sigma_n$. A direct computation shows that $(t^{ij})^{\sigma}=t^{\sigma^{-1}(i)\sigma^{-1}(j)}$ and thus $\ib$ is stable under the symmetric group actions.

Given $t^{ij}\in\sder_m$, $t^{kl}\in\sder_n$ and $\alpha\in\{1,\ldots,m\}$, the composition rule in the operad $\sder$ says that $t^{ij} \circ_\alpha t^{kl}=t^{ij} \circ_\alpha 0 +0 \circ_\alpha t^{kl}$ with  \[ 
    t^{ij} \circ_\alpha 0 
    = 
    \begin{cases}
        t^{(i+n-1)(j+n-1)} & \text{ if } \alpha < i,\\
        \sum_{\beta=1}^{n}t_{(i+\beta-1)(j+n-1)} & \text{ if } \alpha = i,\\
        t^{i(j+n-1)} & \text{ if } i<\alpha<j,\\
        \sum_{\beta=1}^{n}t^{i(j+\beta-1)} & \text{ if } \alpha=j,\\
        t^{ij} & \text{ if } \alpha > j,
    \end{cases}
    \]
    and $0 \circ_\alpha t^{kl} = t^{(k+\alpha-1)(l+\alpha-1)}$ for all $\alpha$. But this is precisely the definition of $t_{ij} \circ_\alpha t_{kl}$ in $\ib$. It follows that the inclusion $\ib\hookrightarrow \sder$ is stable under operadic composition and completes the proof.
\end{proof}

\begin{remark}
Related results appear in \cite{Willwacher15} and \cite{SeveraWillwacher11}, where analogous structures are studied in the context of graph complexes.
\end{remark}

Consider the linearization of $\krv(n)$, the Lie subalgebra $\lkrv_n$ of $\tder_n$, called the \defn{graded Kashiwara-Vergne Lie algebra}, which consists of pairs $(u,r) \in \tder_n \times \K[[z]]$ satisfying the equations
\begin{gather}
    u\left( \sum_{i=1}^{n}x_i\right) =  0 \ , \tag{krv1} \label{eq:Lie1}\\
    j(u) = \trace\left(r\left(\sum_{i=1}^{n}x_i\right)-\sum_{i=1}^{n}r(x_i)\right) \ . \tag{krv2} \label{eq:Lie2}
\end{gather}

\begin{prop}
    The family $\lkrv \eqdef \{\lkrv_n\}_{n \geq 0}$ of graded Kashiwara--Vergne Lie algebras is a colored suboperad of $\sder$.
\end{prop}

\begin{proof}
    We need only check that the symmetric group action and operadic composition preserve the second equation~\eqref{eq:Lie2}.
    Let $u \in \lkrv_n$ and $\sigma \in \Sigma_n$. 
    Since the Jacobian $j$ is a morphism of operads (\cref{prop:Jacobian-operad}), we have $j(u^\sigma)=j(u)^\sigma=j(u)$ since the RHS of~\eqref{eq:Lie2} is invariant under the action of $\Sigma_n$.

    Let $u \in \lkrv_m$ and $v \in \lkrv_n$ and $1 \leqslant i \leqslant m$.
    Suppose moreover that both~$u$ and~$v$ share the same Duflo function~$r$.
    Then, writing $\omega \eqdef x_i + \cdots x_{i+n-1}$ and using again the fact that $j$ is a morphism of operads, we have
    \begin{eqnarray*}
        j(u\circ_i v)&=&j(u \circ_i 0 + 0 \circ_i v) \\
        &=& j(u) \circ_i 0 + 0 \circ_i j(v) \\
        &=&  
\trace\left(
r\left(\sum_{j=1}^{m+n-1}x_j\right)-
\sum_{j=1}^{i-1}r(x_j)-r(\omega)-\sum_{j=i+n}^{m+n-1}r(x_j)
\right)
+
\trace 
\left(r(\omega)-\sum_{j=i}^{i+n-1}r(x_j)\right) \\
&=&\trace \left( r\left(\sum_{j=i}^{m+n-1}x_j\right) -\sum_{j=i}^{m+n-1}r(x_j)\right),
    \end{eqnarray*} 
    which finishes the proof.
\end{proof}

This operad integrates to a group, giving the operad in groups $\krv$ described in \cref{sec:kv-krv-operads}.

\begin{remark}
    The same works for the linearization $\lkv$ of the family of groups $\kv$; they form a suboperad in Lie algebras of $\sder^\bch$.
    However it does not form a suboperad of $\sder$. 
    This explains the apparent dissymmetry between the operads $\krv$ and $\kv$ presented below in \cref{sec:kv-krv-operads}.
\end{remark}

The isomorphism $\PB_n^1\cong\PB_{n+1}\cong \F_n\rtimes \PB_n$ described in Lemma~\ref{lemma: iso PB frozen to PB n+1} induces, after prounipotent completion, an isomorphism of Lie algebras\begin{equation}\mathfrak{t}^1_{n}\cong\ib_{n+1}\cong \lie_n\langle t_{01},\ldots, t_{0n}\rangle \rtimes \mathfrak{t}_{n}.\end{equation}

\begin{notation}
There is an isomorphism of Lie algebras $ \ib_n^1 \cong \ib_{n+1} $ obtained by reindexing the generators $ t_{0i} \mapsto t_{1i} $ and $ t_{ij} \mapsto t_{i+1,j+1} $ for $ 1 \leq i < j \leq n $. To align with our convention of denoting moperads obtained via a shifting construction using a superscript $ (-)^+ $, we will write $ \ib_n^+ $ in place of $ \ib_n^1 $ throughout this paper.
\end{notation}

\subsection{Operadic composition of KV solutions}\label{sec: operad of KV solutions}

We first consider the lift to~$\TAut$ of the operad structure on~$\tder$, and show that it is indeed an operad.

\begin{proof}[{Proof of \cref{operad of TAut}}]
We proceed to verify that the $\Sigma_n$ action and the $\circ_i$ operations satisfy axioms (i)--(iii) from \cref{defn: operad}. For (i), we observe that the element $1=\exp(0) \in \TAut_1$ is such that $F \circ_i 1 = 1 \circ_1 F = F$ for any $F \in \TAut$.
To check (ii), associativity,  let $F=\exp(u) \in \TAut_m, \, G=\exp(v) \in \TAut_n$, and $H=\exp(w) \in \TAut_k$ be three tangential automorphisms, and let $1 \leq j \leq m$. 
For $1 \leq i \leq j-1$, we compute
\begin{eqnarray*}
       \bch(\bch(u \circ_j 0,0\circ_j v)\circ_i 0,0\circ_i w)
       &=& 
       \bch(\bch((u \circ_j 0) \circ_i 0,(0\circ_j v)\circ_i 0),0\circ_i w)\\
       &=& 
       \bch(\bch((u \circ_i 0) \circ_{j+k-1} 0,(0\circ_i 0)\circ_{j+k-1} v),0\circ_i w)\\
       &=& 
       \bch((u \circ_i 0) \circ_{j+k-1} 0,\bch(0\circ_{j+k-1} v,0\circ_i w))\\
       &=& 
       \bch((u \circ_i 0) \circ_{j+k-1} 0,\bch(0\circ_i w,0\circ_{j+k-1} v))\\
       &=& 
       \bch(\bch((u \circ_i 0) \circ_{j+k-1} 0,(0\circ_i w)\circ_{j+k-1} 0),0\circ_{j+k-1} v)\\
       &=& \bch(\bch(u \circ_i 0,0\circ_i w)\circ_{j+k-1} 0, 0 \circ_{j+k-1} v),
\end{eqnarray*}
which shows that $(F \circ_j G) \circ_i H$ and $(F \circ_i H)\circ_{j+k-1} G$ are equal.
Here, we made use of the following properties: compositions with $1$ are group homomorphisms in $\TAut$, operadic composition in $\tder$ is associative, the $\bch$ product is associative (twice), and the fact that we have $[0 \circ_{j+k-1} v,0\circ_i w]=0$ in $\tder_{m+n+k-2}$. 

For $j \leq i \leq j+n-1$,  a similar, slightly simpler computation gives
\begin{eqnarray*}
       \bch(\bch(u \circ_j 0,0\circ_j v)\circ_i 0,0\circ_i w)
       &=& 
       \bch(\bch((u \circ_j 0) \circ_i 0,(0\circ_j v)\circ_i 0),0\circ_i w)\\
       &=& 
       \bch(\bch(u \circ_j (0 \circ_{i-j+1} 0),0\circ_j (v\circ_{i-j+1} 0)),0\circ_i w)\\
       &=& 
       \bch(u \circ_j 0,\bch(0\circ_j (v\circ_{i-j+1} 0),0\circ_j(0\circ_{i-j+1} w)))\\
       &=& \bch(u \circ_j 0,0 \circ_j \bch(v \circ_{i-j+1} 0, 0 \circ_{i-j+1} w),
\end{eqnarray*}
which shows that $(F \circ_j G) \circ_i H=F\circ_j (G \circ_{i-j+1} H)$.
Finally, if $j+k \leq i \leq m+n-1$, a computation similar to the first one above shows that $(F \circ_j G) \circ_i H=(F \circ_{i-n+1} H) \circ_j G$.
    
\medskip 

    It remains to check (iii), equivariance. 
    Let $f,g$ be two Lie words as before, let $\sigma \in \Sigma_m$ and $\tau \in \Sigma_n$ be two permutations, and let $1 \leq i \leq m$. 
    Using the fact that the $\Sigma_n$-actions on $\TAut$ are group homomorphisms, and that operadic composition in $\tder$ is $\Sigma_n$-equivariant, we have
    \begin{eqnarray*}
        \bch(u \circ_i 0, 0 \circ_i v)^{\sigma \circ_i \tau}
        =
        \bch((u \circ_i 0)^{\sigma \circ_i \tau}, (0 \circ_i v)^{\sigma \circ_i \tau}) 
        =
        \bch(u^\sigma \circ_{\sigma^{-1}(i)} 0,0 \circ_{\sigma^{-1}(i)} v^\tau),
    \end{eqnarray*}
    which shows that $(F\circ_{i} G)^{\sigma \circ_i \tau}=(F^\sigma \circ_{\sigma^{-1}(i)} G^\tau)$, as desired.
\end{proof}

In \cite[Section 7]{AKKN_genus_zero} and \cite{AET10} the authors use the operadic structure above on tangential automorphisms to show that by operadically composing a KV solution of type $(0,2+1)$ with itself, one obtains KV solutions of type $(0,n+1)$. 
More generally, we can form a group-colored, non-symmetric operad of KV solutions. 

Recall from \cref{def: KV solutions} that each KV solution $F$ of type $(0,n+1)$ determines a Duflo function $h \in \mathbb{K}[[z]]$ which arises in the second KV equation \eqref{SolKVII}.
Let \[
\solkv(n) \eqdef \left\{ (F, h) \,\middle|\, F \in \TAut_n \text{ satisfies~\eqref{SolKVI} and~\eqref{SolKVII} for } h \in \mathbb{K}[[z]] \right\}.
\]
The sequence $\solkv =\{\solkv(n)\}_{n\geq 2}$, forms a $\mathbb{K}[[z]]$-colored sequence in $\K$-vector spaces.
Given a pair $F= (F,h_1) \in \solkv(m)$ and $G=(G,h_2) \in \solkv(n)$ one can define the operadic composition $F\circ_iG\in \solkv(m+n-1)$ if $h_1=h_2$. Before we state this as a theorem, let us illustrate with an example. 

\begin{example}
\label{example: KV solution of type 3}
Let $ F \in \TAut_2$ be a KV solution of type $(0,2+1)$. Then the operadic composition inherited from $\TAut$ defines
\[
F \circ_2 F 
= 
(F \circ_2 1)(1 \circ_2 F) 
= 
F^{1,23}F^{2,3}.
\]
Since $F$ satisfies the first KV equation~\eqref{SolKVI}, i.e.\ $F(e^{x_1} e^{x_2}) = e^{x_1 + x_2}$, the composition $F \circ_2 F$ also satisfies it.  
Explicitly, we compute:
\[
(F \circ_2 F)(e^{x_1} e^{x_2} e^{x_3}) = F^{1,23}F^{2,3}(e^{x_1} e^{x_2} e^{x_3}) = F^{1,23}\left(e^{x_1} e^{x_2+x_3}\right) = e^{x_1 + x_2 + x_3}.
\]
To show that $F\circ_2F$ satisfies \eqref{SolKVII}, we recall that since $J$ is a 1-cocycle on $\TAut$, we have:
\[
J(F^{1,23}F^{2,3}) = J(F^{1,23}) + F^{1,23} \cdot J(F^{2,3}).
\]
Then, since 
\begin{align*}
    J(F^{2,3}) &= \trace\left(h(x_2 + x_3) - h(x_2) - h(x_3)\right) \quad \text{and} \\
    J(F^{1,23}) &= \trace\left(h(x_1 + (x_2 + x_3)) - h(x_1) - h(x_2+x_3)\right),
\end{align*}
the sum becomes: 
\begin{align*}
    J(F^{1,23}) + F^{1,23} \cdot J(F^{2,3}) 
    = \trace\left(h(x_1 + x_2 + x_3) - h(x_1) - h(x_2)-h(x_3)\right).
\end{align*}
Here, we used the fact that $h \in z^2 \K[[z]]$ is conjugation-equivariant, that $F^{1,23}$ conjugates $x_2$ and $x_3$ by the same element, and the fact that the trace is additive. More generally, given a KV solution $F$ of type $(0,2+1)$, one can then consider the non-symmetric operad generated by this solution  (See \cite[Lemma~7.3]{AKKN_genus_zero}.).
\end{example}

\begin{theorem}
\label{thm:operad-SolKVapp}
The family $\solkv \eqdef \{\solkv(n)\}_{n \geq 2}$ forms a $\mathbb{K}[[z]]$-colored non-symmetric operad. 
\end{theorem}

\begin{proof}
Let $F =(F,h_1) \in \solkv(m)$ and~$G= (G,h_2)\in \solkv(n)$ be two KV solutions with $h_1=h_2$. It suffices to show that the operadic composition $F \circ_i G$ is still a KV solution. We need to check that the composite $F \circ_i G$ satisfies the two equations in~\cref{def: KV solutions}.
To simplify notation, let us write~$\omega \eqdef x_{i}+\ldots+x_{i+n-1}$.
Since both $F$ and $G$ satisfy \eqref{SolKVI}, we have
\begin{align*}
(F \circ_i G)(e^{x_1}\cdots e^{x_{m+n-1}})
& =  (F \circ_i 1)(1 \circ_i G)(e^{x_1}\cdots e^{x_{m+n-1}}) \\
& = (F \circ_i 1)(e^{x_1}\cdots e^{x_{i-1}}e^{\omega}e^{x_{i+n}}\cdots e^{x_{m+n-1}}) \\ 
& = 
e^{x_1+\cdots +x_{m+n-1}}. 
\end{align*} 
Here, we are using the fact that $(1 \circ_i G)$ acts trivially on the basis elements $x_j$, for $j\not\in\{i,i+1,\ldots,i+n-1\}$. 
It follows that $F\circ_iG$ satisfies~\eqref{SolKVI}.

It remains to check $F \circ_i G$ satisfies~\eqref{SolKVII}. Write $h=h_1=h_2$ for the common Duflo function of $F$ and $G$. 
We first note that, since $G$ satisfies~\eqref{SolKVII}, we have
$$J(1 \circ_i G)
=\trace\left(h(\omega)-\sum_{j=i}^{i+n-1}h(x_j)\right).$$ 
We claim that $F \circ_i 1$ acts trivially on $J(1 \circ_i G)$.
To see this, we note that the tangential automorphism
$$F \circ_i 1=F^{1,\ldots,i-1,\,i(i+1)\cdots(i+n-1),\,i+n,\ldots,m+n-1}$$ acts by the same Lie word $f \in \lie_{m+n-1}$ on all of the generators $x_i,\ldots,x_{i+n-1}$.
That is, we have
$(F \circ_i 1)(x_j)=f^{-1}x_j f$ for all $i \leqslant j \leqslant i+n-1$. 
Therefore we also have $(F \circ_i 1)(\omega)=f^{-1}(\omega) f$.
Using that $h \in z^2 \K[[z]]$ is conjugation-equivariant (i.e.\ we have $h(g^{-1} z g) = g^{-1} h(z) g$ for all invertible $g \in \lie$), and that conjugation terms cancel out after taking the trace, we can then compute
\begin{align*}
(F \circ_i 1) \cdot J(1 \circ_i G)  
= \trace\left(h(f^{-1}(\omega) f)-\sum_{j=i}^{i+n-1}h(f^{-1}x_j f)\right) 
= \trace\left(h(\omega)-\sum_{j=i}^{i+n-1}h(x_j)\right),
\end{align*}
which proves the claim.
Combining this result with the facts that $J$ satisfies the $1$-cocycle condition and that both $F$ and $G$ satisfy~\eqref{SolKVII}, we now have
\begin{align*}
J(F \circ_i G)  
& =J(F \circ_i 1)+(F \circ_i 1) \cdot J(1 \circ_i G) \\
&= J(F \circ_i 1)+J(1 \circ_i G) \\
& = 
\trace\left(
h\left(\sum_{j=1}^{m+n-1}x_j\right)-
\sum_{j=1}^{i-1}h(x_j)-h(\omega)-\sum_{j=i+n}^{m+n-1}h(x_j)
\right)
+
\trace 
\left(h(\omega)-\sum_{j=i}^{i+n-1}h(x_j)\right) \\
& =\trace \left( h\left(\sum_{j=i}^{m+n-1}x_j\right) -\sum_{j=i}^{m+n-1}h(x_j)\right).
\end{align*} 
Here, the third equality follows from the fact that the trace function is additive. 
We conclude that $F\circ_iG$ satisfies~\eqref{SolKVII}, and therefore we have $F \circ_i G = (F\circ_i G,h) \in \solkv(m+n-1)$ as claimed.
\end{proof}

\begin{remark}
\label{rem:trivial-action-on-J}
    In fact, the argument above shows that for any pair of tangential automorphisms $F,G$, we have $(F \circ_i 1) \cdot J(1 \circ_i G)= J(1 \circ_i G)$.
\end{remark}

\begin{remark}
\label{rem:symmetric-action-KVI}
The natural symmetric groups action on $\TAut$ does \emph{not} make the operad $\solkv$ into a symmetric operad. 
The reason is that given a KV solution $F \in \solkv(n)$ and a permutation $\sigma \in \Sigma_n$, the tangential automorphism $F^{\sigma}$ is not in general a KV solution.
The problem is that $F^{\sigma}$ satisfies~\eqref{SolKVII}, but not necessarily~\eqref{SolKVI}.
To see that $F^{\sigma}$ satisfies~\eqref{SolKVII}, observe that when $F \in \solkv(n)$ is a KV solution, then $J(F)$ is $\Sigma_n$-equivariant.
Therefore, since $J$ commutes with the $\Sigma_n$ actions, we get
\[
J(F^\sigma)=J(F)^{\sigma}=J(F).
\]
However, there is no guarantee in general that $F^\sigma$ satisfies~\eqref{SolKVI}. 
This can be seen already for classical KV solutions: given an $F \in \solkv(2)$, it is not true that $F^{2,1}(e^{x_1}e^{x_2})=e^{x_1+x_2}$. 
However, we do have $F^{2,1}(e^{x_2}e^{x_1})=e^{x_2+x_1}$.
In \cref{section: symmetric KV solutions}, we consider the involution on the set of KV solutions $\solkv(2)$ introduced in \cite{AT12}, which involves $F^{2,1}$ (see \cref{prop:RtauInvolution}).
\end{remark}

\subsection{Operadic composition of the KV symmetry groups}
\label{sec:kv-krv-operads}
Recall from Section~\ref{subsec:kvsymmetries} that $\krv(n)$ consists of pairs $(F,h)$, where $F\in \TAut_n$ satisfies \eqref{KRVI} and \eqref{KRVII}, and $h(z)\in z^2\K[[z]]$ is the associated Duflo function; similarly, $\kv(n)$ is defined using \eqref{KV'I} and \eqref{KV'II}. The operadic composition of tangential automorphisms induces corresponding compositions on these symmetry groups.

The purpose of this subsection is to show that the families
\[
\kv=\{\kv(n)\}_{n\ge 2}
\qquad\text{and}\qquad
\krv=\{\krv(n)\}_{n\ge 2}
\]
form non-symmetric, $\K[[z]]$-colored operads in sets. We also explain how the left and right actions of $\kv(2)$ and $\krv(2)$ interact with operadic composition on the suboperads of $\solkv$ generated by a single element. The proofs follow the constructions and results of \cite{AKKN_genus_zero}; see also Remark~\ref{rem:krvsaut}.
\begin{prop}
\label{thm:kv-krv-operads}
The sequences $\krv$ and $\kv$ form non-symmetric $\K[[z]]$-colored operads in the category of sets.
\end{prop}

\begin{proof} 
This proof is almost identical to \cref{thm:operad-SolKVapp}.
Let $(F, h_1) \in \krv(m) $ and let $(G, h_2) \in \krv(n)$, it suffices to show that the operadic composition $ (F \circ_i G, h) \in \krv(m+n-1) $ whenever $h_1 = h_2=h$. 

Let us denote by $\omega \eqdef x_i+\cdots+x_{i+n-1}$. 
Using the fact that $F$ and $G$ satisfy~\eqref{KRVI} we have 
\begin{align*}
F\circ_i G \left(e^{x_1+\cdots +x_{m+n-1}}\right)
& =(F\circ_i 1)(1 \circ_i G)
\left(e^{x_1+\cdots +x_{m+n-1}}\right) \\
& =(F\circ_i 1)\left(e^{x_1+\cdots +x_{i-1}+\omega+x_{i+n}+\cdots+x_{m+n-1}}\right) \\
& =e^{x_1+\cdots +x_{m+n-1}}.
\end{align*} 
Here, we are using the fact that $1 \circ_i G$ acts trivially on $x_j$, $j \neq \{i,\ldots,i+n-1\}$.
Therefore, $F \circ_i G$ satisfies equation~\eqref{KRVI}. 
It remains to verify equation \eqref{KRVII}. 
Since~$F$ and~$G$ satisfy~\eqref{KRVII} and the Jacobian~$J$ is a morphism of operads (\cref{prop:Jacobian-operad}), we have 
\begin{align*}
J(F \circ_i 1)=J(F)\circ_i 1
&=\trace\left(
h\left(\sum_{j=1}^{m+n-1}x_j\right)
-\sum_{j=1}^{i-1}h(x_j)-h(\omega)- \sum_{j=i+n}^{m+n-1}h(x_j)\right), \text{ and } \\
J(1 \circ_i G)=1 \circ_i J(G)
&=\trace\left(h(\omega)-\sum_{j=i}^{i+n-1}h(x_j)\right). 
\end{align*}  
Following similar reasoning to that in the proof of \cref{thm:operad-SolKV}, we compute that 
\begin{align*}
J((F \circ_i 1)(1 \circ_i G))
&=J(F \circ_i 1) + (F \circ_i 1)J(1 \circ_i G) \\
&= \trace\left(
h\left(\sum_{j=1}^{m+n-1}x_j\right)
-\sum_{j=1}^{i-1}h(x_j)-h(\omega)- \sum_{j=i+n}^{m+n-1}h(x_j)\right)+\trace\left(h(\omega)-\sum_{j=i}^{i+n-1}h(x_j)\right)\\
&=\trace\left(h\left(\sum_{j=1}^{m+n-1}x_j\right)-\sum_{j=i}^{m+n-1}h(x_j)\right).
\end{align*}
Therefore the composite $F\circ_i G=(F\circ_i 1)(1 \circ_i G)$ satisfies \eqref{KRVII}, and thus $F\circ_iG=(F\circ_i G,h)$ is an element of~$\krv(m+n-1)$. 
The proof that $\kv=\{\kv(n)\}$, $n\geq 2$, forms an operad is similar. 
\end{proof}

\begin{prop}
\label{prop:KRV-action-compatible}
Let $F \in \solkv(2)$ and $G \in \krv(2)$, with $F \cdot G = G^{-1}F \in \solkv(2)$ under the right action of $ \krv(2) $ on $ \solkv(2)$. Then for any iterated composition of the form
\[
\tilde{F} = (\cdots((F \circ_{i_1} F) \circ_{i_2} F) \cdots \circ_{i_k} F) \in \solkv(n),
\quad
\tilde{G} = (\cdots((G \circ_{i_1} G) \circ_{i_2} G) \cdots \circ_{i_k} G) \in \krv(n),
\]
the action satisfies
\[
\tilde{F} \cdot \tilde{G} = (\cdots((G^{-1}F \circ_{i_1} G^{-1}F) \circ_{i_2} G^{-1}F) \cdots \circ_{i_k} G^{-1}F) \in \solkv(n).
\]
An analogous statement holds for $\kv(n)$ acting on the left.
\end{prop}

\begin{proof}
Since operadic composition in $\TAut$ is associative and equivariant, it suffices to verify the identity for a single elementary composition:
    \[
(G \circ_2 G)(F \circ_2 F) = G^{-1}F \circ_2 G^{-1}F.
\]
Using cosimplicial notation, this is equivalent to verifying:
\begin{equation}
\label{eq:compatibility-factorization}
((G^{-1})^{1,23}(G^{-1})^{2,3})(F^{1,23}F^{2,3}) 
=  ((G^{-1})^{1,23}F^{1,23})((G^{-1})^{2,3}F^{2,3}).
\end{equation}
To prove this identity, it is enough to show that $F^{1,23}$ and $(G^{-1})^{2,3}$ commute.
We verify this explicitly on the generators $ x_1, x_2, x_3 \in \lie_3 $. We will denote $G^{-1}=(g_1,g_2) \in \krv(2) \subseteq \TAut_2$, suppressing the associated Duflo element.

Since $(G^{-1})^{2,3}$ acts trivially on $ x_1 $, we compute:
\[
(F^{1,23}(G^{-1})^{2,3})(x_1) 
= F^{1,23}(x_1) 
= f_1^{-1}(x_1, x_2 + x_3) \cdot x_1 \cdot f_1(x_1, x_2 + x_3).
\]
On the other hand,
\begin{align*}
((G^{-1})^{2,3}F^{1,23})(x_1) 
&= f_1^{-1}(x_1, (G^{-1})^{2,3}(x_2 + x_3)) \cdot (G^{-1})^{2,3}(x_1) \cdot f_1(x_1, (G^{-1})^{2,3}(x_2 + x_3)) \\
&= f_1^{-1}(x_1, x_2 + x_3) \cdot x_1 \cdot f_1(x_1, x_2 + x_3),
\end{align*}
where we used the fact that since $G^{-1} \in \krv(2)$, so $ (G^{-1})^{2,3}(x_2 + x_3)=x_2+x_3$.

\smallskip

We now compute the action on $ x_2 $. 
First, we have
\begin{align*}
(F^{1,23} (G^{-1})^{2,3})(x_2) 
&= F^{1,23}((G^{-1})^{2,3}(x_2)) \\
&= F^{1,23}(g_2^{-1}(x_2,x_3) \cdot x_2 \cdot g_2(x_2,x_3)) \\
&= g_2^{-1}(F^{1,23}(x_2), F^{1,23}(x_3)) \cdot F^{1,23}(x_2) \cdot g_2(F^{1,23}(x_2), F^{1,23}(x_3)).
\end{align*}
Using $F^{1,23}(x_2) = f_2^{-1}(x_1, x_2 + x_3) \cdot x_2 \cdot f_2(x_1, x_2 + x_3)$, and the similar formula for $ F^{1,23}(x_3) $, the expression becomes:
\begin{align*}
&= f_2^{-1}(x_1, x_2 + x_3) \cdot 
g_2^{-1}(x_2, x_3) \cdot x_2 \cdot g_2(x_2,x_3) \cdot 
f_2(x_1, x_2 + x_3).
\end{align*}

\smallskip

On the other hand:
\begin{align*}
((G^{-1})^{2,3} \cdot F^{1,23})(x_2) 
&= (G^{-1})^{2,3}(F^{1,23}(x_2)) \\
&= (G^{-1})^{2,3}(f_2^{-1}(x_1, x_2 + x_3) \cdot x_2 \cdot f_2(x_1, x_2 + x_3)) \\
&= f_2^{-1}(x_1, (G^{-1})^{2,3}(x_2 + x_3)) \cdot (G^{-1})^{2,3}(x_2) \cdot f_2(x_1, (G^{-1})^{2,3}(x_2 + x_3)) \\
&= f_2^{-1}(x_1, x_2 + x_3) \cdot g_2^{-1}(x_2, x_3) \cdot x_2 \cdot g_2(x_2,x_3) \cdot f_2(x_1, x_2 + x_3),
\end{align*}
where we have again used that $(G^{-1})^{2,3}$ fixes $x_2 + x_3$ and acts on $ x_2 $ by conjugation. 

\smallskip
The actions on $ x_3 $ follow by identical reasoning. Therefore, $ F^{1,23} \cdot (G^{-1})^{2,3} = (G^{-1})^{2,3} \cdot F^{1,23} $, and the equality~\eqref{eq:compatibility-factorization} holds. 
The analogous result for $\kv$ is proven similarly.

\end{proof}

\begin{remark}
    For KV solutions associated with a moperad equivalence $\hPaB^1 \to \CD^+$, this result follows from~\cref{thm:SolKV-from-moperad} and~\cref{cor:KRV-action-F-012}.
\end{remark}
\section{Cosimplicial Structures and Operadic Cohomology}
\label{sec:operadic-cohomology}

We recall the construction of the cosimplicial cohomology of a multiplicative operad \cite{merkulov2019grothendieck}, and we apply it to the operads $\lie$ and $\tder$.
Let $\calP$ be an operad with an element $e \in \calP(2)$ such that $e \circ_1 e = e \circ_2 e$. 
Then, $\calP$ admits a cosimplicial structure given by the maps $d_i : \calP(n) \to \calP(n+1)$ defined by
\[ 
d_i(\mu)
\eqdef 
\begin{cases}
    e \circ_2 \mu & \text{if } i=0, \\
    \mu \circ_i e & \text{if } i \in [n], \\
    e \circ_1 \mu & \text{if } i=n+1.
\end{cases}
\]
If moreover $\calP$ is an operad in vector spaces, the cosimplicial maps assemble into a differential $d : \calP(n) \to \calP(n+1)$ defined by 
\[ d(\mu) \eqdef \sum_{i=0}^{n+1}(-1)^i d_i(\mu)\]
which turns the shifted totalization of $\calP$
\[ \Simp^\bullet(\calP) \eqdef \bigoplus_{n \geq 1} \calP(n)[-n]\]
into a chain complex. 
Its cohomology $H^\bullet(\Simp^\bullet(\calP))$ is called the \defn{cosimplicial cohomology} of the operad $\calP$.

\subsubsection{Alekseev--Torossian cohomologies}
Consider the element $e\eqdef 0 \in \lie_2$.
We clearly have 
$$e\circ_1 e = 0 = e \circ_2 e.$$
Applying the definition, we get cosimplicial maps $\lie_n \to \lie_{n+1}$ given by 
\[ 
\delta_i(f)
\eqdef 
\begin{cases}
    f(x_2,\ldots,x_{n+1})& \text{if } i=0, \\
    f(x_1,\ldots,x_i+x_{i+1},\ldots, x_{n+1}) & \text{if } i \in [n], \\
    f(x_1,\ldots,x_n)& \text{if } i=n+1.
\end{cases}
\]
Taking instead the operad $\lie^\bch$ we get
\[ 
\tilde \delta_i(f)
\eqdef 
\begin{cases}
    f(x_2,\ldots,x_{n+1})& \text{if } i=0, \\
    f(x_1,\ldots,\bch(x_i,x_{i+1}),\ldots, x_{n+1}) & \text{if } i \in [n], \\
    f(x_1,\ldots,x_n)& \text{if } i=n+1.
\end{cases}
\]
We recover the cohomologies defined in \cite[Section~2.3]{AT12} as the cosimplicial cohomologies of the operads $\lie$ and $\lie^\bch$. 
These cohomologies are also used in \cite[Prop.~26-27]{AET10}.
Alternatively, a general construction from the sole datum of the $\Sigma_n$-action was given by {\v S}evera--Willwacher in~\cite{severaWillwacher2011}, while a geometric construction has been given by Vergne in~\cite{Vergne2012}.
Moreover, these structures carry through to the linear operads $\ass$ and $\cyc$.

Similarly, consider the element $e\eqdef (0,0) \in \tder_2$. Then we have $e\circ_1 e = (0,0,0) = e \circ_2 e$. Applying the definition, we get cosimplicial maps $d_i : \tder_n \to \tder_{n+1}$ given by 
\[ 
d_i(u)
\eqdef 
\begin{cases}
    (0,a_1(x_2,\ldots,x_{n+1}),\ldots,a_n(x_2,\ldots,x_{n+1})) & \text{if } i=0, \\
    (a_1(x_1,\ldots,x_i+x_{i+1},\ldots,x_n),\ldots,a_i,a_i,\ldots,a_n) & \text{if } i \in [n], \\
    (a_1(x_1,\ldots,x_{n}),\ldots,a_n(x_1,\ldots,x_{n}),0) & \text{if } i=n+1.
\end{cases}
\]
We recover the cohomology defined in \cite[Section~3.2]{AT12} as the cosimplicial cohomology of the operad $\tder$.
On infinitesimal braids, we recover its cosimplicial cohomology studied by Willwacher in~\cite{Willwacher15}. 

In usual cosimplicial notation, $u\circ_i e$ is denoted $u^{1,\ldots,i(i+1),\ldots,n}$, and more generally a natural notation for $u \circ_i 0$ is $u^{1,\ldots, i-1,i(i+1)\cdots(i+n-1),i+n,\ldots,m+n-1}$.
Similarly $e\circ_1 v \in \tder_{n+1}$ is usually denoted $v^{1,\ldots,n}$, and more generally a natural notation for $0 \circ_i v$ is $v^{i,i+1,\ldots,i+(n-1)}$.
Since we have $u \circ_i v = u \circ_i 0 + 0 \circ_i v$, in ``cosimplicial terms'' the partial compositions maps can be written
\begin{equation}
\label{eq:cosimplicial-composition}
    u \circ_i v = u^{1,\ldots, i-1,i(i+1)\cdots(i+n-1),i+n,\ldots,m+n-1}+v^{i,i+1,\ldots,i+(n-1)}.
\end{equation}

\begin{example}
\label{ex:cosimplicial-associator-tder}
    Consider $u \in \tder_2$. 
    Then, we have in $\tder_3$
    \begin{align*}
        du
        &=d_0(u)-d_1(u)+d_2(u)-d_3(u) \\
        &=0\circ_2 u-u \circ_1 0+u \circ_2 0-0\circ_1 u \\
        &=u \circ_2 0+0\circ_2 u-u \circ_1 0-0\circ_1 u \\
        &= u\circ_2 u - u\circ_1 u.
    \end{align*}
\end{example}

\subsubsection{Cosimplicial structure on tangential automorphisms}

The element $1=\exp(0,0) \in \TAut_2$ plays for the operad $\TAut$ the same role as the element $(0,0) \in \tder_2$ plays for $\tder$, thus the operad structure on $\TAut$ induces a cosimplicial structure given by composing with~$1$.
This comes from the analogous structure on $\tder$, in the sense that for $F=\exp{(u)} \in \TAut_n, u \in \tder_n$, we have $F^{i,(i+1)\hdots j}=\exp{(u^{i,(i+1)\hdots j})}$. 

Similar to~\eqref{eq:cosimplicial-composition} above, partial composition in $\TAut$ can be written via cosimplicial maps: we have 
$$ 
F\circ_i G = F^{1,\ldots,i-1,\,i(i+1)\cdots(i+n-1),\,i+n,\ldots,m+n-1} \circ
G^{i,i+1,\ldots,i+(n-1)}.$$

Now observe that, even though partial compositions $\circ_i$ on $\tder$ are not morphisms of Lie algebras in general, compositions with zero $-\circ_i 0$ and $0 \circ_i -$ are.
Therefore cosimplicial maps are morphisms of Lie algebras, and behave well with respect to the Jacobian~\cite[Section 3]{AT12}: for any $ F \in \TAut_n $, the identity $ J(F^{i,j}) = J(F)^{i,j} $ holds for all pairs $ i, j $.
See~\cite[Proposition~24]{AET10} and the discussion preceding~\cite[Lemma~7.3]{AKKN_genus_zero} for more details.
A consequence of this fact is that the Jacobian map is a morphism of operads in sets.


\begin{proof}[{Proof of \cref{prop:Jacobian-operad}}]
    This is a direct computation; using that $J$ commutes with cosimplicial maps together with \cref{rem:trivial-action-on-J}, we have that
    \begin{eqnarray*}
        J(F \circ_i G) &=& J(F \circ_i 1) + (F \circ_i 1) J(1 \circ_i G) \\
        &=& J(F \circ_i 1) + J(1 \circ_i G) \\
        &=& J(F) \circ_i 0 + 0 \circ_i J(G) \\
        &=& J(F) \circ_i J(G).
    \end{eqnarray*}
\end{proof}

Similarly to the case of tangential derivations, one can define a ``differential'' on $F \in \TAut_n$ by the formula
\[ 
d(F) 
\eqdef 
\prod_{i=0}^{n+1} d_{i}(F)^{\left((-1)^{i}\right)}.
\]
\begin{example}
\label{ex:cosimplicial-associator}
Consider $F \in \TAut_2$.
Then, we have 
\begin{align*}
    dF
    &=d_0(F)d_1(F)^{-1}d_2(F)d_3(F)^{-1} \\
    &=(1\circ_2 F)(F \circ_1 1)^{-1}(F \circ_2 1)(1\circ_1 F)^{-1} \\
    &=(F \circ_2 1)(1\circ_2 F)(1\circ_1 F)^{-1}(F \circ_1 1)^{-1}\\
    &= (F\circ_2 F) (F \circ_1 F)^{-1}
\end{align*}
in $\TAut_3$.
\end{example}


\bibliographystyle{amsalpha}
\bibliography{bib.bib}

\end{document}